\documentclass[
 a4paper,
 11pt,
 chapterprefix,
 bibtotoc,
 idxtotoc,
 pointlessnumbers,
 english,
 french
]{scrbook}

\usepackage[latin1]{inputenc}
\usepackage[english,french]{babel}
\usepackage{array}

\usepackage[T1]{fontenc}
\usepackage{mathrsfs}
\usepackage{graphics,psfrag}
\usepackage{graphicx}
\usepackage{latexsym}
\usepackage{amssymb}
\usepackage{amsmath}
\usepackage{theorem}
\usepackage[mathscr]{eucal}
\usepackage[all]{xy}
\usepackage{epsfig}
\usepackage{enumerate}
\usepackage{stmaryrd}
\usepackage{makeidx}

\outer\long\def\COUIC#1{}

\RequirePackage{color}
\definecolor{myred}{rgb}{0.75,0,0}
\definecolor{mygreen}{rgb}{0,0.5,0}
\definecolor{myblue}{rgb}{0,0,0.65}

\RequirePackage{ifpdf}
\ifpdf
  \IfFileExists{pdfsync.sty}{\RequirePackage{pdfsync}}{}
  \RequirePackage[pdftex,
   colorlinks = true,
   urlcolor = myblue,
   citecolor = mygreen,
   linkcolor = myred,
   pdftitle = {Manuscript},
   pdfauthor = {Daniel Juteau},
   pagebackref,
   pdfpagemode=None,
   bookmarksopen=true]{hyperref}
\else
  \RequirePackage[hypertex]{hyperref}
\fi

\RequirePackage{ae, aecompl, aeguill}

\usepackage{url}

\newtheorem{theo}{Theorem}[section]
\newtheorem{prop}[theo]{Proposition}
\newtheorem{lem}[theo]{Lemma}
\newtheorem{cor}[theo]{Corollary}
\newtheorem{defi}[theo]{Definition}
\newtheorem{ass}[theo]{Assumption}
\newtheorem{remark}[theo]{Remark}
\newtheorem{example}[theo]{Example}

\def\wh{\widehat}
\def\wt{\widetilde}
\def\ov{\overline}
\def\un{\underline}
\def\red#1{{\color{myred}{\un {#1}}}}

    \def\AM{{\mathbb{A}}}
  \def\bG{{\mathfrak b}}  
\def\CG{{\mathfrak C}}    \def\CM{{\mathbb{C}}}
\def\DG{{\mathfrak D}}    
    \def\EM{{\mathbb{E}}}
    \def\FM{{\mathbb{F}}}
  \def\gG{{\mathfrak g}}  \def\GM{{\mathbb{G}}}
    
\def\IG{{\mathfrak I}}    
    
    \def\KM{{\mathbb{K}}}
\def\LG{{\mathfrak L}}  \def\lG{{\mathfrak l}}  \def\LM{{\mathbb{L}}}
  \def\mG{{\mathfrak m}}  
\def\NG{{\mathfrak N}}  \def\nG{{\mathfrak n}}  \def\NM{{\mathbb{N}}}
\def\OG{{\mathfrak O}}    \def\OM{{\mathbb{O}}}
\def\PG{{\mathfrak P}}    \def\PM{{\mathbb{P}}}
    \def\QM{{\mathbb{Q}}}
    \def\RM{{\mathbb{R}}}
\def\SG{{\mathfrak S}}  \def\sG{{\mathfrak s}}  
\def\TG{{\mathfrak T}}  \def\tG{{\mathfrak t}}  
  \def\uG{{\mathfrak u}}

\def\XG{{\mathfrak X}}    
\def\YG{{\mathfrak Y}}    
    \def\ZM{{\mathbb{Z}}}

    \def\AC{{\mathcal{A}}}
    \def\BC{{\mathcal{B}}}
    \def\CC{{\mathcal{C}}}
    \def\DC{{\mathcal{D}}}
    \def\EC{{\mathcal{E}}}
    \def\FC{{\mathcal{F}}}
    \def\GC{{\mathcal{G}}}
    \def\HC{{\mathcal{H}}}
    \def\IC{{\mathcal{I}}}
    \def\JC{{\mathcal{J}}}
    \def\KC{{\mathcal{K}}}
    \def\LC{{\mathcal{L}}}
    \def\MC{{\mathcal{M}}}
    \def\NC{{\mathcal{N}}}
    \def\OC{{\mathcal{O}}}
    \def\PC{{\mathcal{P}}}
    \def\QC{{\mathcal{Q}}}
    \def\RC{{\mathcal{R}}}
    \def\SC{{\mathcal{S}}}
    \def\TC{{\mathcal{T}}}
    \def\UC{{\mathcal{U}}}
    \def\VC{{\mathcal{V}}}

\def\Irm{{\mathrm{I}}}

\def\Rrm{{\mathrm{R}}}    
\def\Srm{{\mathrm{S}}}

\def\Iti{{\tilde{I}}}    
    
\def\Kti{{\tilde{K}}}

    \def\NCt{{\tilde{\mathcal{N}}}}

\def\Sti{{\tilde{S}}}

\def\Xti{{\tilde{X}}}

\def\Cba{{\bar{C}}}

\def\a{\alpha}
\def\b{\beta}
\def\g{\gamma}
\def\G{\Gamma}
\def\d{\delta}
\def\D{\Delta}
\def\e{\varepsilon}

\def\l{\lambda}
\def\L{\Lambda}

\def\s{\sigma}

\def\t{\tau}

        \def\alpt{{\tilde{\alpha}}}

\def\epsb{{\boldsymbol{\varepsilon}}}

\def\psib{{\boldsymbol{\psi}}}

\DeclareMathOperator{\Ad}{{\mathrm{Ad}}}

\DeclareMathOperator{\Aut}{{\mathrm{Aut}}}

\DeclareMathOperator{\codim}{codim}
\DeclareMathOperator{\Coker}{{\mathrm{Coker}}}

\DeclareMathOperator{\Cone}{Cone}

\DeclareMathOperator{\End}{{\mathrm{End}}}

\DeclareMathOperator{\Gal}{{\mathrm{Gal}}}
\DeclareMathOperator{\Hom}{{\mathrm{Hom}}}

\DeclareMathOperator{\ic}{{\mathbf{IC}}}

\DeclareMathOperator{\id}{{\mathrm{Id}}}
\DeclareMathOperator{\im}{{\mathrm{Im}}}

\DeclareMathOperator{\Ind}{{\mathrm{Ind}}}

\DeclareMathOperator{\Irr}{{\mathrm{Irr}}}
\DeclareMathOperator{\Ker}{{\mathrm{Ker}}}
\DeclareMathOperator{\Loc}{{\mathrm{Loc}}}

\DeclareMathOperator{\Proj}{{\mathrm{Proj}}}

\DeclareMathOperator{\res}{{\mathrm{res}}}

\DeclareMathOperator{\RHOM}{R\underline{Hom}}
\DeclareMathOperator{\rg}{R\G}
\DeclareMathOperator{\rgc}{R\G_c}

\DeclareMathOperator{\Sh}{{\mathrm{Sh}}}
\DeclareMathOperator{\Soc}{Soc}
\DeclareMathOperator{\Spec}{{\mathrm{Spec}}}

\DeclareMathOperator{\Tr}{{\mathrm{Tr}}}
\DeclareMathOperator{\tr}{{\mathrm{tr}}}
\DeclareMathOperator{\Top}{Top}
\DeclareMathOperator{\Tor}{{\mathrm{Tor}}}

\newcommand{\elem}[1]{\stackrel{#1}{\longto}}
\newcommand{\map}[1]{\stackrel{#1}{\to}}
\newcommand{\leftelem}[1]{\stackrel{#1}{\longfrom}}

\def\Imp{\Longrightarrow}
\def\iff{\Leftrightarrow}
\def\Iff{\Longleftrightarrow}
\def\to{\rightarrow}

\def\longto{\longrightarrow}
\def\longfrom{\longleftarrow}
\def\injto{\hookrightarrow}
\def\rtordu{\rightsquigarrow}
\def\isom{\stackrel{\sim}{\to}}

\def\fonctio#1#2#3#4{\begin{array}{ccc}
{#1} & \longto & {#2} \\
{#3} & \longmapsto & {#4} 
\end{array}}

\def\incl{\hspace{0.05cm}{\subset}\hspace{0.05cm}}

\def\DS{\displaystyle}

\def\lexp#1#2{\kern\scriptspace\vphantom{#2}^{#1}\kern-\scriptspace#2}
\def\le{\hspace{0.1em}\mathop{\leqslant}\nolimits\hspace{0.1em}}
\def\ge{\hspace{0.1em}\mathop{\geqslant}\nolimits\hspace{0.1em}}

\mathchardef\inferieur="321E
\mathchardef\superieur="321F

\def\eqna{\begin{eqnarray*}}
\def\endeqna{\end{eqnarray*}}

\def\ie{{\textit{i}.\textit{e}.}}

\def\iff{ if and only if }

\def\op{\text{op}}

\def\pr{{\mathrm{pr}}}

\def\p{{}^p}
\def\pp{{}^{p_+}\!}
\def\0{{}^0\!}

\def\rs{\text{rs}}
\def\di{\mathrm{div}}
\def\tors{{\mathrm{tors}}}
\def\free{{\mathrm{free}}}
\def\mini{{\mathrm{min}}}
\def\jmin{{j_\mini}}
\def\io{{i_0}}
\def\triv{{\mathrm{triv}}}
\def\reg{{\mathrm{reg}}}
\def\jreg{{j_\reg}}
\def\isubreg{{i_\subreg}}

\def\subreg{{\mathrm{subreg}}}

\def\Sing{{\mathrm{Sing}}}

\def\impsi{\NG^0}

\def\longue{{\mathrm{lg}}}
\def\courte{{\mathrm{sh}}}
\def\hauteur{{\mathrm{ht}}}

\def\qed{\hfill$\boxempty$}
\newenvironment{proof}{\noindent\emph{Proof.}}{\qed \\}

\def\octa#1#2#3#4#5#6{
\xymatrix @!=.4cm{
&&&\\
\\
&& #3 \ar[ddr] \ar[uur]\\
&&&&&\\
& #2 \ar[uur] \ar[dr] && #5 \ar[ddr] \ar[ur]\\
&& #4 \ar[drr] \ar[ur]\\
#1 \ar[uur] \ar[urr] &&&& #6 \ar[ddr] \ar[drr]\\
&&&&&&\\
&&&&&
}
}

\begin{document}

\shorthandoff{;:!?}

\frontmatter

\thispagestyle{empty}

{\selectlanguage{french}

\begin{flushright}
UFR de mathématiques\\
\'Ecole Doctorale de Mathématiques de Paris centre\\
Université Paris 7 -- Denis Diderot\\
\hrulefill
\end{flushright}

\vfill

\begin{center}
{\sc\huge Thèse de Doctorat}
\vfill
{\Large présentée par}
\vfill
{\sc\huge Daniel Juteau}
\vfill
{\Large pour obtenir le grade de\\
Docteur de Mathématiques de l'Université Paris 7 -- Denis Diderot}
\vfill
\hrulefill
\vfill
{\Huge Correspondance de Springer modulaire et matrices de décomposition
\vfill
---
\vfill
Modular Springer correspondence and decomposition matrices}
\vfill
\hrulefill
 \vfill
\end{center}

\begin{flushleft}
{\Large Thèse soutenue le 11 décembre 2007, devant le jury composé de~:}

~

{\large
\begin{tabular}{ll}
Cédric {\sc Bonnafé} (co-directeur)& Chargé de Recherche à l'Université de Franche-Comté\\
Michel {\sc Brion}& Directeur de Recherche à l'Université de Grenoble 1\\
Michel {\sc Broué}& Professeur à l'Université Paris 7\\
Bernard {\sc Leclerc}& Professeur à l'Université de Caen\\
Jean {\sc Michel}& Directeur de Recherche à l'Université Paris 7\\
Raphaël {\sc Rouquier} (co-directeur)& Directeur de Recherche à l'Université Paris 7\\
& Professeur à l'Université d'Oxford\\
Wolfgang {\sc Soergel} (rapporteur)& Professeur à l'Université Albert-Ludwigs de Freiburg\\
Tonny {\sc Springer}& Professeur à l'Université d'Utrecht
\end{tabular}
}

~

\vfill

~

{\large Rapporteur non présent à la soutenance~:

~

\begin{tabular}{ll}
Bao Châu {\sc Ngô} & Professeur à l'Université Paris-Sud (Paris 11)\\
& Membre de l'Institute for Advanced Study à Princeton
\end{tabular}
}
\end{flushleft}
}

\newpage
\thispagestyle{empty}

\mainmatter
{\selectlanguage{french}
\tableofcontents

\addchap{Remerciements}

Tout d'abord, je suis infiniment reconnaissant à Cédric Bonnafé et
Raphaël Rouquier pour la qualité exceptionnelle de leur
encadrement. J'ai grandement apprécié leurs immenses qualités humaines
et mathématiques. Je mesure la chance que j'ai eue de les avoir comme
directeurs. Ils ont tous deux été d'une très grande disponibilité,
même quand la distance nous séparait. Je les remercie pour le sujet
qu'ils m'ont donné, leurs conseils, leurs encouragements, leur
soutien, et la liberté qu'ils m'ont accordée. Je pense que c'était
parfois éprouvant de m'avoir comme étudiant, notamment quand il a
fallu relire mon manuscrit jusqu'à la dernière minute ! Je les
remercie aussi, ainsi que leurs épouses Anne et Meredith, pour les
moments de convivialité que nous avons partagés.

\medskip

Je remercie très sincèrement les deux rapporteurs, Bao Châu Ngô et
Wolfgang Soergel, pour le temps qu'ils ont passé à lire ce travail et
à écrire les rapports. Plus particulièrement, je remercie Bao Châu Ngô
pour avoir suggéré d'étudier les faisceaux pervers en caractéristique
positive sur les courbes modulaires et les variétés de Shimura en vue
d'applications arithmétiques (j'espère acquérir le bagage nécessaire à
cette fin)~; et je remercie Wolfgang Soergel pour sa lecture très minutieuse
et ses questions qui ont permis d'améliorer le manuscrit.

\medskip

Je remercie vivement chaque membre du jury pour avoir
accepté d'y participer. C'est un très grand honneur pour moi de
pouvoir réunir un jury constitué de mathématiciens de tout premier
plan, dans les domaines de la géométrie et des représentations. La
présence de chacun d'entre eux me fait très plaisir.

\medskip

Les nombreux voyages que j'ai dû faire pendant ma thèse m'ont permis
de rencontrer des personnes qui m'ont beaucoup apporté, humainement et
mathématiquement. Je les remercie pour nos discussions, et les moments
que nous avons partagés, dans les différents endroits où j'ai séjourné
(Paris, Besançon, Yale, Lausanne, Oxford, Nancy et Hendaye), que ce
soit dans le milieu mathématique ou non.  De même pour les conférences
auxquelles j'ai participé, notamment à Luminy et à Oberwolfach. Merci
aussi à toutes les personnes qui m'ont donné l'occasion de présenter
mes résultats dans différents séminaires.  Merci à ceux qui m'ont
hébergé, et pour la chaleur de leur accueil. Enfin, merci à tous ceux
qui se seront déplacés pour ma soutenance, et à ceux qui auront
contribué à mon pot de thèse. Je préfère remercier chacun de vive
voix, plutôt que citer  une longue liste ici.

\medskip

Je remercie les différentes institutions qui ont
permis tous ces déplacements : l'Institut de Mathématiques de Jussieu,
l'\'Ecole Doctorale de Sciences Mathématiques de Paris centre, le
Laboratoire de Mathématiques de Besançon,
l'Université d'Oxford via le réseau européen Liegrits, l'\'Ecole
Polytechnique Fédérale de Lausanne et l'Université de Yale.

\medskip

Ces années de thèse ont aussi été marquées par une dure maladie en
2006, avec un traitement lourd, pendant plusieurs mois. Je suis
extrêmement reconnaissant aux personnes qui m'ont entouré pendant
toute cette période~: mes parents, Souki et Sebas, qui n'ont pas
compté leur temps et leur énergie pour me soutenir, et Germain qui a
fait des milliers de kilomètres pour être là le plus possible. Un
grand merci à Miguel qui est venu des Philippines. Merci aussi à tous
les parents et amis qui ont eu des attentions pour moi, de près ou de
loin. Enfin, je suis bien sûr profondément reconnaissant à toutes les
personnes qui ont contribué au progrès de la médecine, aux docteurs
qui m'ont examiné, opéré et suivi, et à tout le personnel hospitalier
qui m'a soigné avec gentillesse et attention. Je leur dois ma
guérison.

L'été suivant, le soutien de mes proches m'a une nouvelle fois été
indispensable, pour finir la rédaction dans les temps. Sans eux, je
n'aurais sans doute pas pu soutenir cette année.

Indépendamment de ces deux épisodes, je dois un très grand merci à ma
famille, qui m'a beaucoup donné, et qui a su m'aimer tel que je suis.

\medskip

Merci enfin à Germain pour tout ce qu'il me fait vivre.

\addchap{Introduction (français)}

\section*{Contexte et vue d'ensemble}

En 1976, Springer a introduit une construction géométrique des
représentations irréductibles des groupes de Weyl, qui a eu une
profonde influence et de nombreux développements ultérieurs, culminant
avec la théorie des faisceaux-caractères de Lusztig, qui permet de
calculer les valeurs des caractères de groupes réductifs
finis. Dans cette thèse, nous définissons une correspondance de
Springer pour les représentations modulaires des groupes de Weyl, et
en établissons quelques propriétés, répondant ainsi à une question
posée par Springer lui-même.

\subsection*{Représentations modulaires, matrices de décomposition}

La théorie des représentations modulaires des groupes finis,
initiée et développée par Brauer à partir du début des années 1940,
est l'étude des représentations des groupes finis sur un corps de
caractéristique $\ell > 0$. Lorsque $\ell$ divise l'ordre du groupe,
la catégorie des représentations n'est plus semi-simple.

Soit $W$ un groupe fini.
On fixe une extension finie $\KM$ du corps $\QM_\ell$ des
nombres $\ell$-adiques. Soit $\OM$ son anneau de valuation. On note
$\mG = (\varpi)$ l'idéal maximal de $\OM$, et $\FM$ le corps résiduel
(qui est fini de caractéristique $\ell$). Le triplet $(\KM,\OM,\FM)$
est ce qu'on appelle un système modulaire, et on suppose qu'il est
assez gros pour tous les groupes finis que nous rencontrerons
(c'est-à-dire que tous les $\KM W$-modules simples sont absolument
simples, et de même pour $\FM$). La lettre $\EM$ désignera l'un des
anneaux de ce triplet.

Pour une catégorie abélienne $\AC$, on note $K_0(\AC)$ son groupe de
Grothendieck. Lorsque $\AC$ est la catégorie des $A$-modules de type
fini, où $A$ est un anneau, on adopte la notation $K_0(A)$. La
sous-catégorie pleine formée des objets projectifs de $\AC$ sera notée
$\Proj \AC$. 

Le groupe de Grothendieck $K_0(\KM W)$ est libre de base $([E])_{E
\in\Irr \KM W}$, où $\Irr \KM W$ désigne l'ensemble des classes
d'isomorphisme de $\KM$-modules simples. De même pour $K_0(\FM
W)$. Pour $F \in \Irr \FM W$, on note $P_F$ une enveloppe projective
de $F$. Alors $([P_F])_{F\in\Irr \FM W}$ est une base de
$K_0(\Proj\:\FM W)$. La réduction modulo $\mG$ définit un
isomorphisme de $K_0(\Proj\:\OM W)$ sur $K_0(\Proj\:\FM W)$.

Nous allons définir le triangle $cde$ \cite{Serre}.
\[
\xymatrix{
K_0(\Proj\:\FM W) \ar[rr]^c \ar[dr]_e && K_0(\FM W)\\
& K_0(\KM W) \ar[ur]_d
}
\]
Le morphisme $c$ est induit par l'application qui à chaque $\FM
W$-module projectif associe sa classe dans $K_0(\FM W)$.  Le foncteur
d'extensions des scalaires à $\KM$ induit un morphisme de
$K_0(\Proj\:\OM W)$ vers $K_0(\KM W)$. Le morphisme $e$ s'obtient en
composant avec l'inverse de l'isomorphisme canonique de
$K_0(\Proj\:\OM W)$ sur $K_0(\Proj\:\FM W)$.  Le morphisme $d$ est un
peu plus délicat. Soit $E$ un $\KM W$-module. On peut choisir un
réseau $E_\OM$ dans $E$ stable par $W$. L'image de $\FM \otimes_\OM
E_\OM$ dans $K_0(\FM W)$ ne dépend pas du choix du réseau
\cite{Serre}. Le morphisme $d$ est induit par l'application
qui à $E$ associe $[\FM \otimes_\OM E_\OM]$.

On définit la matrice de décomposition
$D^W = (d^W_{E,F})_{E \in \Irr \KM W,\ F \in \Irr \FM W}$ par
\[
d([E]) = \sum_{F \in \Irr \FM W} d^W_{E,F} [F]
\]

Un des plus grands problèmes en théorie des représentations modulaires
est de déterminer ces nombres de décomposition $d^W_{E,F}$ explicitement
pour des classes intéressantes de groupes finis. Ce problème est
ouvert pour le groupe symétrique.

Le triangle $cde$ peut s'interpréter en termes de caractères
ordinaires et modulaires (de Brauer). Nous renvoyons à \cite{Serre}. Lorsque les
caractères ordinaires sont connus (c'est le cas pour le groupe
symétrique), la connaissance de la matrice de décomposition équivaut à
la détermination des caractères de Brauer.

Il y a des variantes de ce problème lorsqu'on sort du cadre des
groupes finis. On peut par exemple considérer les représentations modulaires des
algèbres de Hecke. Ces algèbres sont des déformations d'algèbres de
groupes de réflexions, et jouent un rôle très important dans la
théorie des représentations des groupes finis de type de Lie. Elles
sont définies de manière générique, et on peut regarder ce qu'il se
passe lorsqu'on spécialise un ou plusieurs paramètres.

On peut aussi s'intéresser aux représentations rationnelles d'un
groupe réductif en caractéristique positive, ou les représentations
d'une algèbre de Lie réductive, ou de groupes quantiques (déformations
d'algèbres enveloppantes), etc. Dans ce cas, on considère les
multiplicités des simples dans des classes d'objets particuliers dont
on connaît les caractères.

\subsection*{Faisceaux pervers et représentations}

Depuis trois décennies, l'utilisation de méthodes géométriques a
permis des progrès spectaculaires dans de nombreux domaines de la
théorie des représentations. Nous nous intéresserons ici tout
particulièrement aux représentations des groupes de Weyl et des
groupes finis de type de Lie.

Il y a une trentaine d'années, Springer est parvenu à construire
géométriquement toutes les représentations irréductibles ordinaires des groupes
de Weyl, dans la cohomologie de certaines variétés liées aux éléments
nilpotents de l'algèbre de Lie correspondante \cite{SPRTRIG, SPRWEYL}.
Cette découverte a eu un retentissement considérable. De nombreuses
autres constructions ont été proposées par la suite. Par exemple,
Kazhdan et Lusztig ont proposé une approche topologique \cite{KL},
et Slodowy a construit ces représentations par monodromie
\cite{SLO1}. Au début des années 1980, l'essor de la cohomologie
d'intersection a permis de réinterpréter la correspondance de Springer
en termes de faisceaux pervers \cite{LusGreen,BM}.

Lusztig a prolongé ce travail en étudiant une correspondance de
Springer généralisée, ainsi que des complexes de cohomologie
d'intersection sur un groupe réductif $G$ ou son algèbre de Lie $\gG$,
qu'il appelle complexes admissibles ou faisceaux-caractères
\cite{ICC,CS1,CS2,CS3,CS4,CS5}. Si $G$ est muni d'une structure
$\FM_q$-rationnelle définie par un endomorphisme de Frobenius $F$, les
faisceaux-caractères $F$-stables donnent lieu à des fonctions
centrales sur le groupe fini $G^F$, qui sont très proches des
caractères irréductibles. La matrice de transition entre ces deux
bases est décrite par une transformation de Fourier. Ainsi, ces
méthodes géométriques ont permis de déterminer les caractères de
groupes réductifs finis (au moins lorsque le centre est connexe).

Jusqu'ici, à ma connaissance, on n'a pas utilisé ces méthodes pour
étudier les représentations modulaires des groupes de Weyl ou des
groupes finis de type de Lie. Pourtant, cela a été le cas dans au
moins deux autres situations modulaires. Soergel \cite{Soergel} a
converti un problème sur la catégorie $\OC$ en un problème sur les
faisceaux pervers à coefficients modulo $\ell$ sur des variétés de
Schubert. D'autre part, Mirkovic et Vilonen ont établi une équivalence de
catégories entre les représentations rationnelles d'un groupe réductif
sur un anneau quelconque $\L$ et les faisceaux pervers à coefficients
$\L$ sur le dual de Langlands, défini sur $\CM$. La topologie
classique permet d'utiliser des coefficients arbitraires. Ce travail
permet d'ailleurs de donner une définition intrinsèque du dual de
Langlands en construisant sa catégorie de représentations. Nous y
reviendrons dans la section \og Perspectives \fg.

\subsection*{Correspondance de Springer modulaire}

Il était tentant de chercher un lien entre représentations modulaires
des groupes de Weyl et faisceaux pervers modulo $\ell$ sur les
nilpotents. Autrement dit, de chercher à définir une correspondance de
Springer modulaire. Il est vrai que la construction de
Lusztig-Borho-MacPherson utilise le théorème de décomposition de
Gabber \cite{BBD}, qui n'est pas valable en caractéristique
$\ell$. Mais Hotta et Kashiwara \cite{HK} ont une approche \emph{via}
une transformation de Fourier pour les $\DC$-modules dans le cas où le
corps de base est celui des complexes, ce qui évite de recourir au
théorème de décomposition. De plus, la transformation de
Fourier-Deligne permet de considérer un corps de base de
caractéristique $p$ et des coefficients $\ell$-adiques \cite{BRY}. Dans cette
thèse, nous définissons une correspondance de Springer en utilisant
la transformation de Fourier-Deligne avec des coefficients modulo
$\ell$. De plus, nous introduisons une matrice de décomposition pour
les faisceaux pervers sur les nilpotents, et nous la comparons à la
matrice de décomposition du groupe de Weyl. Par ailleurs, nous
calculons de façon purement géométrique certains nombres de
décomposition. Nous constatons que certaines propriétés des nombres de
décomposition des groupes de Weyl peuvent être vues comme le reflet de
propriétés géométriques. Par exemple, la règle de suppression de lignes et
de colonnes de James peut s'expliquer par une règle similaire de Kraft
et Procesi \cite{KP1} sur les singularités nilpotentes, une fois qu'on
a déterminé la correspondance de Springer modulaire pour $GL_n$ (ce
que nous ferons dans cette thèse).

\section*{Contenu détaillé}

\subsection*{Préliminaires et exemples}

Dans le chapitre \ref{chap:preliminaries}, nous faisons des rappels
sur les faisceaux pervers sur $\KM$, $\OM$ et $\FM$ et donnons
quelques compléments, qui nous seront utiles par la suite. Nous
insistons en particulier sur les aspects spécifiques à $\OM$ et
$\FM$. Par exemple, sur $\OM$ nous n'avons pas une perversité
autoduale, mais une paire de perversités, $p$ et $p_+$, échangées par
la dualité. De plus, nous étudions l'interaction entre les paires de
torsion et les $t$-structures (voir à ce sujet \cite{HRS}), et aussi
avec les situations de recollement. Dans ce passage quelque peu
technique, on trouvera de nombreux triangles distingués qui seront
utilisés par la suite. Le point clé est que la réduction (dérivée)
modulo $\ell$ ne commute pas aux troncations en général. Nous donnons
aussi quelques compléments sur les extensions perverses
$\p j_!$, $\p j_{!*}$, $\p j_*$ (sur la tête et le socle, et sur le
comportement vis-à-vis des multiplicités). Finalement, nous
définissons les nombres de décomposition pour les faisceaux
pervers. Nous sommes particulièrement intéressés par le cas d'une
$G$-variété ayant un nombre fini d'orbites~: c'est le cône nilpotent
que nous avons en vue.

Dans le chapitre \ref{chap:examples}, nous donnons quelques exemples
de faisceaux pervers, et en particulier de complexes de cohomologie
d'intersection sur $\EM$. Nous rappelons les propriétés des morphismes
propres et semi-petits (resp. petits). En particulier, le complexe de
cohomologie d'intersection d'une variété admettant une petite
résolution est obtenu par image directe.

Ensuite nous introduisons la notion d'équivalence lisse de
singularités, et rappelons que la cohomologie d'intersection locale
est un invariant pour cette équivalence.

Puis nous étudions les singularités coniques, où la cohomologie
d'intersection locale se ramène à un calcul de la cohomologie d'une
variété (l'ouvert complémentaire du sommet du cône), et, plus
généralement, le cas d'une variété affine munie d'une $\GM_m$-action
contractant tout sur l'origine.

Il est naturel de se demander quand le complexe de cohomologie
d'intersection se réduit au faisceau constant $\EM$ (de telle sorte que la
variété vérifie la dualité de Poincaré usuelle). On parle alors de
variété $\EM$-lisse. Un exemple typique de variété $\KM$-lisse
(resp. $\FM$-lisse) est fourni par le quotient
d'une variété lisse par un groupe fini (resp. par un groupe fini
d'ordre premier à $\ell$).

Dans la dernière section de ce chapitre, nous étudions les
singularités simples. Une variété normale $X$ a des singularités
rationnelles si on a une résolution $\pi : \Xti \to X$ avec
$R^i \pi_* \OC_X = 0$ pour $i > 0$.
Sur $\CM$, les surfaces à point double rationnel sont (à équivalence
analytique près) les quotients du plan affine
par un sous-groupe fini de $SL_2(\CM)$. Elles sont classifiées par les
diagrammes de Dynkin simplement lacés. On peut interpréter les autres
types en considérant en plus l'action d'un groupe de symétries. On
associe à chaque diagramme de Dynkin $\G$ un diagramme homogène
$\wh\G$, et un groupe de symétries $A(\G)$. Dans le cas où $\G$ est
déjà homogène, on a $\wh\G = \G$ et $A(\G) = 1$. Soit $\wh\Phi$ un
système de racines de type $\wh\G$. On note $P(\wh\Phi)$ le réseau des
poids, et $Q(\wh\Phi)$ le réseau radiciel. Soit $H$
le sous-groupe fini de $SL_2(\CM)$ associé à $\wh\G$. On montre que
\[
H^2\left((\AM^2\setminus\{0\})/H,\ZM\right) \simeq P(\wh\Phi)/Q(\wh\Phi)
\]
avec une action naturelle de $A(\G)$. Grâce aux résultats du chapitre
\ref{chap:preliminaries}, cela nous permet de comparer les faisceaux
pervers en caractéristiques $0$ et $\ell$.

\subsection*{Calculs de nombres de décomposition}

Jusqu'au chapitre \ref{chap:dec}, notre but est de calculer certains
nombres de décomposition pour les faisceaux pervers $G$-équivariants
sur la variété nilpotente, par des méthodes géométriques.

Introduisons d'abord quelques notations. Les faisceaux pervers simples sur $\KM$
(resp. $\FM$) sont paramétrés par l'ensemble $\NG_\KM$
(resp. $\NG_\FM$) des paires $(x,\rho)$ (à conjugaison près) constituées d'un élément
nilpotent $x$ et d'un caractère $\rho\in\Irr \KM A_G(x)$ (resp.
$\rho \in \Irr \FM A_G(x)$), où $A_G(x)$ est le groupe fini des
composantes du centralisateur de $x$ dans $G$. On utilisera la notation
\[
\left(d_{(x,\rho),(y,\s)}\right)_{(x,\rho)\in\NG_\KM,\ (y,\s)\in\NG_\FM}
\]
pour la matrice de décomposition de ces faisceaux pervers. Dans le cas
de $GL_n$, tous les $A_G(x)$ sont triviaux, si bien qu'on peut oublier
$\rho$ qui est toujours $1$, et les orbites nilpotentes sont
paramétrées par l'ensemble $\PG_n = \{\l\vdash n\}$ des partitions de
$n$. Dans ce cas, la matrice de décomposition sera notée
\[
(d_{\l,\mu})_{\l,\mu\in\PG_n}
\]

Quant à la matrice de décomposition pour le groupe de Weyl $W$, on la notera
\[
\left(d_{E,F}^W\right)_{E\in\Irr\KM W,\ F\in\Irr\FM W}
\]

Pour le groupe symétrique $\SG_n$, les $\KM\SG_n$-modules simples sont
les modules de Specht $S^\l$, pour $\l \in \PG_n$. Ils sont définis
sur $\ZM$, et munis d'une forme bilinéaire symétrique définie sur $\ZM$. La
réduction modulaire du module de Specht, que l'on notera encore
$S^\l$, est donc munie elle aussi d'une forme bilinéaire
symétrique. Le quotient de $S^\l$ par le radical de cette forme
bilinéaire symétrique est soit nul, soit un $\FM\SG_n$-module
simple. L'ensemble des $\mu$ tels que ce quotient soit non nul
(celui-ci sera alors noté $D^\mu$) est
l'ensemble $\PG_n^{\ell\text{-reg}}$ des partitions de $n$ qui sont
$\ell$-régulières (dont chaque partie est répétée au plus $\ell - 1$
fois). Les $D^\mu$, pour $\mu\in\PG_n^{\ell\text{-reg}}$, forment un
système de représentants des classes d'isomorphisme de
$\FM\SG_n$-modules simples. La matrice de décomposition du groupe
symétrique $\SG_n$ sera notée plutôt
\[
\left(d_{\l,\mu}^{\SG_n}\right)_{\l\in\PG_n,\ \mu\in\PG_n^{\ell\text{-reg}}}
\]

Pour l'algèbre de Schur 
\[
S_\EM(n) = S_\EM(n,n) = \End_{\EM\SG_n}\left(\bigoplus_{\l \vdash n}
\Ind_{\EM \SG_\l}^{\EM \SG_n} \EM\right)
\]
on notera la matrice de décomposition
\[
\left(d_{\l,\mu}^{S(n)}\right)_{\l,\mu \in\PG_n}
\]
Il est connu que
\[
d_{\l,\mu}^{S(n)} = d_{\l',\mu'}^{\SG_n}
\]
pour $\l\in\PG_n,\ \mu\in\PG_n^{\ell\text{-reg}}$, où $\l'$ désigne la
partition transposée. Nous allons voir que
\[
d_{\l,\mu} = d^{\SG_n}_{\l',\mu'} = d_{\l,\mu}^{S(n)}
\]
pour $\l\in\PG_n,\ \mu\in\PG_n^{\ell\text{-reg}}$,
et nous conjecturons que
\[
d_{\l,\mu} = d^{S(n)}_{\l,\mu}
\]
pour \emph{toutes} les partitions $\l$, $\mu$ de $n$
(voir à ce sujet les remarques dans la dernière section de cette introduction).

Dans le chapitre \ref{chap:min}, nous calculons la cohomologie entière
de l'orbite nilpotente (non triviale) minimale $\OC_\mini$ dans une
algèbre de Lie simple $\gG$ sur le corps des nombres complexes. En
réalité, les résultats et méthodes de ce chapitre sont valables pour
un corps de base de caractéristique $p > 0$, à condition de travailler
avec la cohomologie étale et de prendre pour coefficients les entiers
$\ell$-adiques.

La cohomologie rationnelle de $\OC_\mini$ est déjà connue.
La dimension de $\OC_\mini$ est $d = 2h^\vee - 2$, où $h^\vee$ est le
nombre de Coxeter dual. La première moitié de la cohomologie est donnée par
\[
\t_{\leqslant d - 1} \rg(\OC_\mini,\QM) \simeq \bigoplus_{i=1}^k \QM[-2(d_i - 2)]
\]
où $k$ est le nombre de racines simples longues, et 
$d_1 \leqslant \ldots \leqslant d_k \leqslant \ldots \leqslant d_n$
sont les degrés de $W$ ($n$ étant le nombre total de racines simples).
L'autre moitié s'en déduit par dualité de Poincaré.

C'est donc la torsion qui nous intéresse.
Si $\Phi$ est le système de racines de $\gG$, et
$\Phi'$ le sous-système engendré par les racines simples longues (pour une
certaine base), alors la cohomologie moitié de $\OC_ \mini$ est
\[
H^d(\OC_\mini,\ZM) \simeq P^\vee(\Phi')/Q^\vee(\Phi')
\]
Nous verrons par la suite que ce résultat est lié à la réduction
modulaire de la représentation naturelle du groupe de Weyl $W'$ de $\Phi'$.

En dehors de la cohomologie moitié, nous n'avons pas d'expression
uniforme pour la partie de torsion de la cohomologie de
$\OC_\mini$. En revanche, on sait toutefois que c'est le conoyau d'une
matrice dont les coefficients sont déterminés explicitement par
l'ensemble ordonné des racines longues dans $\Phi$, qui est nivelé par
cohauteur (hauteur de la coracine), ce qui nous permet de faire le
calcul dans tous les types. En dehors de la cohomologie moitié, nous
constatons que les seuls nombres premiers qui interviennent dans la
torsion de la cohomologie sont mauvais. Cela revient à dire que les
fibres du complexe de cohomologie d'intersection entière sont sans
$\ell$-torsion pour $\ell$ bon (uniquement pour la perversité $p$,
précisément pas pour $p_+$, où la cohomologie moitié intervient).
Nous ignorons si l'on peut trouver une interprétation (peut-être
homologique) en termes de théorie des représentations à ces groupes de
mauvaise torsion.

Dans le chapitre \ref{chap:dec}, nous calculons certains nombres de
décomposition pour les faisceaux pervers $G$-équivariants sur la
variété nilpotente, en utilisant d'une part les résultats précédents,
et d'autre part des résultats géométriques que l'on peut trouver dans
la littérature.

Tout d'abord, nous déterminons les nombres de décomposition associés aux
classes régulière et sous-régulière (le centralisateur d'un élément
sous-régulier peut être non connexe). Comme dans la section sur les
singularités simples, on associe au type $\G$ de $G$ un diagramme
homogène $\wh\Phi$ et un groupe de symétries $A := A(\G)$ qui est
isomorphe à $A_G(x_\subreg)$ lorsque $G$ est adjoint. On a
\[
d_{(x_\reg,1),(x_\subreg,\rho)}
= [\FM \otimes_\ZM (P(\wh\Phi)/Q(\wh\Phi)) : \rho]
\]
pour tous les $\rho$ dans $\Irr \FM A$. On calcule cette multiplicité
dans tous les types, pour chaque nombre premier $\ell$ et pour chaque
$\rho \in \Irr\FM A$.
En ce qui concerne les classes minimales et triviales, il découle des
résultats du chapitre \ref{chap:min} que
\[
d_{(x_\mini,1),(0,1)}
= \dim_\FM \FM \otimes_\ZM \left( P^\vee(\Phi')/Q^\vee(\Phi') \right)
\]
Nous donnons également cette multiplicité dans tous les types.
\`A titre d'exemple, voyons ce qu'il se passe pour $GL_n$. On trouve
\[
d_{(n),(n-1,1)} = d_{(21^{n-2}),(1^n)} =
\begin{cases}
1 \text{ si } \ell \mid n\\
0 \text{ sinon}
\end{cases}
\]
ce qui est compatible avec notre conjecture faisant le lien avec
l'algèbre de Schur. Nous avons un autre résultat qui va dans ce sens.
Les nombres de décomposition de l'algèbre de Schur vérifient la
propriété suivante. Si $\l$ et $\mu$ sont deux partitions de $n$ dont
les $r$ premières lignes et les $s$ premières colonnes sont
identiques, et si $\l_1$ et $\mu_1$ désignent les partitions (d'un
entier $n_1$ plus petit) obtenues à partir de $\l$ et $\mu$ en
supprimant ces lignes et ces colonnes, on a

\[
d^{S(n)}_{\l,\mu} = d^{S(n_1)}_{\l_1,\mu_1}
\]

Kraft et Procesi ont montré que les singularités des adhérences
des orbites nilpotentes dans $GL_n$ vérifient une propriété
similaire \cite{KP1}. Avec les mêmes notations, on a

\[
\codim_{\ov\OC_{\l_1}} \OC_{\mu_1} = \codim_{\ov\OC_\l} \OC_\mu
\quad\text{ et }\quad
\Sing(\ov\OC_{\l_1}, \OC_{\mu_1}) = \Sing(\ov\OC_\l, \OC_\mu)
\]

Nous en déduisons que les nombres de décomposition $d_{\l,\mu}$
vérifient la même propriété~:

\[
d_{\l,\mu} = d_{\l_1,\mu_1}
\]

Si $\l > \mu$ sont deux partitions de $n$ adjacentes pour l'ordre de
dominance (c'est-à-dire s'il n'existe pas de partition $\nu$ telle que
$\l > \nu > \mu$), Kraft et Procesi utilisent le résultat sur les
lignes et les colonnes pour ramener la détermination de la singularité
de $\ov\OC_\l$ le long de $\OC_\mu$ aux cas extrêmes
$(\l,\mu) = ((m), (m - 1, 1))$ et $(\l,\mu) = ((2, 1^{m-2}), (1^m))$,
pour un $m$ plus petit. Les dégénérescences minimales en type $A_n$
sont donc toutes de types $A_m$ (une singularité simple de type $A_m$)
ou $a_m$ (une singularité minimale de type $a_m$), pour des $m$ plus
petits.

Comme, dans $GL_n$, tous les $A_G(x)$ sont triviaux, cela suffit pour
déterminer le nombre de décomposition $d_{\l,\mu}$ lorsque $\l$ et
$\mu$ sont adjacentes. Dans ce cas, on a bien~:

\[
d^{S(n)}_{\l,\mu} = d_{\l,\mu}
\]

Kraft et Procesi ont aussi démontré que les singularités des
adhérences des orbites nilpotentes dans les types classiques vérifient
une règle de suppression des lignes et des colonnes \cite{KP2}. Il
faut traiter les cas orthogonaux et symplectiques à la fois. Ils en
déduisent le type des singularités des dégénérescences minimales dans
ce cas. Ils ne trouvent que des singularités simples et minimales de
types classiques, à une exception près. Plus précisément, dans le cas
de la codimension deux, on a (à équivalence lisse près) une
singularité de type $A_k$, $D_k$ ou $A_k \cup A_k$, cette dernière
étant la réunion non normale de deux singularités simples de type de
type $A_k$, s'intersectant transversalement en leur point
singulier. Lorsque la codimension est strictement supérieure à $2$, on
a une singularité minimale de type $b_k$, $c_k$ ou $d_k$. En
supprimant des lignes et des colonnes, on peut toujours se ramener à
ces cas irréductibles. Dans cet article, Kraft et Procesi déterminent
quelles adhérences d'orbites sont normales dans les types classiques,
ce qui était leur but.

Nous pouvons utiliser leurs résultats pour déterminer d'autres nombres
de décomposition dans les types classiques, mais pour pouvoir le faire
dans tous les cas il faudrait aussi déterminer les systèmes locaux qui
interviennent. Quoi qu'il en soit, pour une dégénérescence minimale
$\ov\OC \supset \OC'$ en type classique, on peut toujours déterminer
la quantité suivante~:
\[
\sum_{\rho \in \Irr \FM A_G(x_{\OC'})}
d_{(x_\OC,1),(x_{\OC'},\rho)}
\]
(Dans les types classiques, les $A_G(x)$ sont de la forme $(\ZM/2)^k$,
donc abéliens, et tous les $\rho \in \Irr \FM A_G(x_{\OC'})$ sont de
degré $1$.) En particulier, on peut dire quand les
$d_{(x_\OC,1),(x_{\OC'},\rho)}$ sont nuls pour tous les $\rho \in
\Irr \FM A_G(x_{\OC'})$. Une étude plus précise devrait pouvoir
suffire à déterminer tous les nombres de décomposition de ce type.

Un autre résultat de Kraft et Procesi, concernant la décomposition
spéciale de la variété nilpotente \cite{KP3}, nous permet de montrer la
nullité de certains nombres de décomposition, dans les types
classiques, lorsque $\ell \neq 2$. Dans \cite{LusSpec}, Lusztig a
introduit un sous-ensemble de $\Irr \KM W$, constitué des
représentations dites spéciales. Les classes nilpotentes spéciales
sont les classes $\OC$ telles que la représentation $\chi$ associée à
$(\OC,\KM)$ par la correspondance de Springer est spéciale.
D'autre part, Spaltenstein a introduit dans \cite{SpalDual} une
application décroissante de l'ensemble des classes nilpotentes dans
lui-même, telle que $d^3 = d$ (c'est une involution sur son
image). L'image de $d$ est précisément l'ensemble des classes
spéciales. Les variétés localement fermées
\[
\widehat \OC = \ov \OC \setminus
\mathop{\bigcup_{\OC' \text{ spéciale}}}\limits_{\ov \OC'\subset\ov \OC} \ov \OC'
\]
où $\OC$ parcourt l'ensemble des classes spéciales, forment une
partition de la variété nilpotente $\NC$. Elles sont appelées pièces
spéciales. Ainsi, chaque classe nilpotente est incluse dans une unique
pièce spéciale. Lusztig a attaché à chaque classe spéciale $\OC$ un
quotient canonique $\ov A_G(x_\OC)$ du groupe fini $A_G(x_\OC)$, et
conjecturé que la pièce spéciale $\wh\OC$ est le quotient d'une
variété lisse par $\ov A_G(x_\OC)$. Une conséquence de cette
conjecture est que $\wh\OC$ est $\KM$-lisse, mais en fait cela donne
plus d'information~: en particulier, la conjecture implique que
$\wh\OC$ est $\FM$-lisse dès que $\ell$ ne divise pas l'ordre de ce
groupe $\ov A_G(x_\OC)$. Dans \cite{KP3}, Kraft et Procesi montrent
cette conjecture dans les types classiques. Nous en déduisons que,
dans les types classiques, on a
\[
d_{(x_\OC,1),(x_{\OC'},\rho)} = 0
\]
dès que $\ell > 2$, lorsque $\OC$ est une classe spéciale, $\OC'$ une
classe incluse dans la pièce spéciale $\wh\OC$, et
$\rho \in \Irr \FM A_G(x_{\OC'})$. Une étude plus détaillée
permettrait peut-être de déterminer les nombres de décomposition
lorsque $\ell = 2$.

Faisons une remarque supplémentaire.
Dans un autre article \cite{Kraft}, Kraft résout le
problème de la normalité des adhérences d'orbites nilpotentes dans
$G_2$. On y trouve l'information suivante, qui n'est pas couverte par
les résultats précédents~: $\ov\OC_{10}$ a une singularité de type
$A_1$ en $\OC_8$, où $\OC_i$ désigne l'unique orbite nilpotente de
dimension $i$ dans l'algèbre de Lie $\gG$ d'un groupe simple $G$ de type
$G_2$. Comme $A_G(x_8) = 1$ (on note $x_i$ est un représentant de $\OC_i$),
cela nous permet de déterminer le nombre de décomposition
$d_{(x_{10},1),(x_8,1)}$~:
\[
d_{(x_{10},1),(x_8,1)} = 
\begin{cases}
1 \text{ si } \ell = 2,\\
0 \text{ sinon}
\end{cases}
\]

Une étude plus détaillée de cet article permettrait peut-être
de retrouver d'autres nombres de décomposition géométriquement dans
$G_2$. Quoi qu'il en soit, en utilisant la correspondance de Springer
modulaire nous pourrons déterminer toute la matrice de décomposition
pour $G_2$ lorsque $\ell = 3$, et toute la matrice sauf une colonne
lorsque $\ell = 2$. Pour $\ell > 3$, $\ell$ ne divise pas l'ordre du
groupe de Weyl, et la matrice de décomposition est l'identité~; je
pense que c'est vrai pour dans tous les types, mais il faudra étudier
la notion de cuspidalité. Au moins, la partie de la matrice
de décomposition correspondant au groupe de Weyl est bien l'identité,
comme nous le verrons.

\subsection*{Correspondance de Springer modulaire et matrices de décomposition}

Dans la suite de la thèse, nous définissons une correspondance de
Springer modulaire et en établissons quelques propriétés, notamment le
fait qu'elle préserve les nombres de décomposition.
Comme le théorème de décomposition de Gabber \cite{BBD} n'est pas vrai
dans le cadre modulaire, nous nous inspirons de l'approche de Kashiwara et
Brylinski \cite{BRY}, utilisant une transformation de Fourier.

Dans le chapitre \ref{chap:fourier}, nous introduisons la
transformation de Fourier-Deligne en suivant un article de Laumon
\cite{Lau}. Nous détaillons les preuves, et vérifions que tout se
passe bien lorsque les coefficients sont $\KM$, $\OM$ ou $\FM$.

Le chapitre \ref{chap:springer} est le c{\oe}ur de cette thèse.  Nous
commençons par rappeler le contexte géométrique de la correspondance
de Springer, qui est celui de la résolution simultanée de Grothendieck $\pi$
des singularités des fibres du quotient adjoint. Prenant la fibre en
zéro, on retrouve la résolution de Springer $\pi_\NC$ du cône
nilpotent $\NC$.

Puis nous introduisons les faisceaux pervers $\EM\KC_\rs$, $\EM\KC$ et
$\EM \KC_\NC$, respectivement sur l'ouvert $\gG_\rs$ des éléments
réguliers semi-simples, sur $\gG$ tout entière, et sur le fermé $\NC$
des éléments nilpotents. Nous avons le diagramme commutatif à carrés
cartésiens suivant~:

\[
\xymatrix@=1.5cm{
\tilde\gG_\rs
\ar[d]_{\pi_\rs}
\ar@<-.5ex>@{^{(}->}[r]^{\tilde j_\rs}
\ar@{}[dr] | {\DS\boxempty_\rs}
&
\tilde\gG
\ar[d]^\pi
\ar@{}[dr] | {\DS\boxempty_\NC}
&
\NCt
\ar@<.5ex>@{_{(}->}[l]_{i_\NCt}
\ar[d]^{\pi_\NC}
\\
\gG_\rs
\ar@<-.5ex>@{^{(}->}[r]_{j_\rs}
&
\gG
&
\NC
\ar@<.5ex>@{_{(}->}[l]^{i_\NC}
}
\]

On note $r$ le rang de $G$, et $\nu$ le nombre de racines positives
dans $\Phi$. On pose

\begin{gather*}
\EM\KC_\rs = {\pi_\rs}_! \EM_{\tilde \gG_\rs} [2\nu + r]\\
\EM\KC     = \pi_!       \EM_{\tilde \gG} [2\nu + r]\\
\EM\KC_\NC = {\pi_\NC}_! \EM_\NCt [2\nu]
\end{gather*}

On a
\begin{gather*}
\EM\KC = \p {j_\rs}_{!*} \EM\KC_\rs\\
\EM\KC_\NC = i_\NC^* \EM\KC [-r]
\end{gather*}

Le morphisme $\pi$ est propre et petit, génériquement un $W$-torseur
(au-dessus de $\gG_\rs$), et sa restriction $\pi_\NC$ aux nilpotents
est semi-petite.

Ensuite, nous définissons une correspondance de Springer modulaire
utilisant une transformation de Fourier-Deligne. Pour $E\in\Irr\KM W$,
le faisceau pervers associé par la correspondance de Springer à
la Brylinski est $\TC(E) = \FC(\p {j_\rs}_{!*}(E[2\nu + r])$.
Cela permet de définir une application injective
\[
\Psi_\KM : \Irr \KM W \longto \NG_\KM
\]
On notera $\impsi_\KM$ son image.
Nous procédons de même pour la correspondance de Springer modulaire. \`A
$F\in \Irr \FM W$, on associe
\[
\TC(F) = \FC(\p {j_\rs}_{!*}(F[2\nu + r]))
\]
C'est un faisceau pervers simple $G$-équivariant sur $\NC$. Il est
donc de la forme $\p \JC_{!*} (\OC_F,\LC_F)$ pour une certaine paire
$(\OC_F,\LC_F)$ appartenant à l'ensemble $\NG_\FM$ des paires
$(\OC,\LC)$ constituées d'une orbite nilpotente $\OC$ et d'un
$\FM$-système local $G$-équivariant sur $\OC$. On identifiera
$\NG_\FM$ à l'ensemble des paires $(x,\rho)$ où $x \in \NC$ et
$\rho \in \Irr \FM A_G(x)$, à conjugaison près. On obtient donc une
application injective
\[
\Psi_\FM : \Irr \FM W \longto \NG_\FM
\]

On notera $\impsi_\FM$ son image.

Ensuite, nous montrons que la matrice de décomposition du groupe de
Weyl $W$ peut être extraite de la matrice de décomposition pour les
faisceaux pervers $G$-équivariants sur la variété nilpotente, en ne
gardant que les lignes qui sont dans l'image de la correspondance de
Springer ordinaire, et les colonnes qui sont dans l'image de la
correspondance de Springer modulaire. Plus précisément, nous montrons
que, pour tous $E\in\Irr\KM W$ et $F\in\Irr\FM W$, on a
\[
d^W_{E,F} = d_{\Psi_\KM(E),\Psi_\FM(F)}
\]

Puis nous déterminons la correspondance de Springer modulaire lorsque
$G = GL_n$. On a alors~:
\[
\impsi_\FM = \PG_n^{\ell\text{-res}}
\]
où $\PG_n^{\ell\text{-res}}$ est l'ensemble des partitions
$\ell$-restreintes de $n$, c'est-à-dire dont la transposée est
$\ell$-régulière.
\[
\forall \l \in \PG_n^{\ell\text{-reg}},\qquad
\Psi_\FM(D^\l) = \l'
\]

En particulier, pour $\l\in\PG_n$ et $\mu\in\PG_n^{\ell\text{-reg}}$,
on a~:
\[
d^{\SG_n}_{\l,\mu} = d_{\l',\mu'}
\]
de telle sorte qu'on peut voir la règle de suppression des lignes et
des colonnes de James comme une conséquence du résultat géométrique
de Kraft et Procesi sur les singularités nilpotentes.

\section*{Perspectives}

Les thèmes de réflexion pour prolonger ce travail ne manquent pas.

\subsection*{Géométrie des orbites nilpotentes}

Cette thèse a révélé de nouveaux liens entre la théorie des
représentations des groupes de Weyl et la géométrie des classes
nilpotentes. On peut s'attendre à de nouvelles interactions entre les
deux domaines.

Par exemple, nous avons remarqué que la règle de suppression de
lignes et de colonnes de James peut s'expliquer géométriquement grâce
au résultat de Kraft et Procesi sur les singularités nilpotentes. 

Du côté des représentations, Donkin a trouvé une généralisation de cette
règle \cite{DonkinGen}. Je m'attends à une généralisation similaire du
côté de la géométrie des adhérences d'orbites nilpotentes (on devrait
trouver une singularité produit).

\subsection*{Détermination de la correspondance de Springer modulaire,
ensembles basiques}

Une question se pose naturellement à propos de la correspondance de
Springer modulaire. Supposons pour d'abord pour simplifier que $\ell$ ne divise pas les
$A_G(x)$. Alors on peut identifier $\NG_\FM$ et $\NG_\KM$ à un
ensemble de paramètres $\PG$ commun. A-t-on alors $\impsi_\FM \subset \impsi_\KM$~?

Supposons que ce soit le cas.
Dans ce cas, pour chaque $F \in \Irr \FM W$, il existe un unique
$E \in \Irr \KM W$ tel que $\Psi_\KM(E) = \Psi_\FM(F)$.
Cela permet de construire un ensemble basique pour $W$ et montre de
façon géométrique la triangularité de la matrice de décomposition de $W$.

Même si $\ell$ divise l'ordre de certains $A_G(x)$, on peut s'en
sortir en choisissant un ensemble basique pour chaque $A_G(x)$ (qui
est de la forme $(\ZM/2)^k$ ou un groupe symétrique $\SG_k$,
$k\leqslant 5$, pour $G$ adjoint). En fait, pour tous ces groupes, il
y a un choix canonique.

Dans l'autre sens, la connaissance d'un ensemble basique pour $W$ et une
propriété de triangularité pour un ordre compatible avec les
adhérences des orbites associés par la correspondance de Springer
permet de déterminer la correspondance de Springer modulaire.

Nous avons pu déterminer la correspondance de Springer modulaire de
$GL_n$, ainsi que pour les groupes de rang inférieur ou égal à trois,
pour cette raison (il faut faire attention pour $G_2$ car il y a une paire
cuspidale en caractéristique zéro).

Si on arrivait à montrer que $\impsi_\FM \subset \impsi_\KM$, il
serait intéressant de déterminer l'ensemble basique qu'on obtient, et de
le comparer à l'ensemble basique canonique de \cite{GeckRouquier},
lorsque celui-ci est défini (c'est-à-dire quand $\ell$ ne divise pas
les $A_G(x)$).

\subsection*{Correspondance de Springer modulaire généralisée,
  faisceaux-caractères modulaires}

Dans la correspondance de Springer originale, $\impsi_\KM$ contient
toujours les paires de la forme $(\OC,\KM)$, mais en général
$\impsi_\KM$ est strictement inclus dans $\NG_\KM$. La principale
motivation de Lusztig dans \cite{ICC} est de comprendre les paires manquantes. 
Ce travail se poursuit dans la série d'articles sur les faisceaux
caractères, qui permet de décrire les caractères des groupes finis de
type de Lie.

Il est clair qu'une des premières choses à faire pour continuer le
travail de cette thèse est d'étudier les notions d'induction et de
restriction, de cuspidalité, de définir une correspondance de Springer
généralisée, et de la déterminer dans tous les cas. Peut-être
verra-t-on apparaître de nouveaux objets combinatoires pour les types
classiques (des $\ell$-symboles~?).

J'espère que tout cela débouchera sur une théorie des
faisceaux-caractères modulaires permettant d'étudier les
représentations modulaires des groupes finis de type de Lie. Dans le
dernier chapitre, nous présentons brièvement quelques calculs sur
$\sG\lG_2$.

\subsection*{Détermination de fibres de cohomologie d'intersection}

La détermination des fibres de cohomologie d'intersection (soit sur
$\OM$ pour la perversité $p_+$, soit sur $\FM$) pour les
adhérences d'orbites nilpotentes suffirait pour connaître la matrice
de décomposition pour les faisceaux pervers, et donc (si l'on a déterminé
la correspondance de Springer modulaire) celle du groupe de Weyl. Dans
cette thèse, nous déterminons cette correspondance pour $GL_n$. On a
transformé le problème des matrices de décomposition du groupe
symétrique en un problème topologique et géométrique, où il n'est plus
fait mention du groupe de Weyl.

Bien sûr, ce problème est sans aucun doute très difficile. En
caractéristique zéro, la détermination des fibres de cohomologie
d'intersection pour les nilpotents passe par la correspondance de
Springer, les formules d'orthogonalité des fonctions de Green
(voir l'algorithme que Shoji utilise dans \cite{ShojiF4}
pour le type $F_4$, qui est repris dans d'autres travaux comme
\cite{BeSpa} pour les types $E_6$, $E_7$, $E_8$, et généralisé
dans \cite[\S 24]{CS5}). Il est peu problable qu'il existe un tel
algorithme en caractéristique $\ell$.

\subsection*{Variétés de Schubert}

Les faisceaux pervers à coefficients modulo $\ell$ n'avaient jamais
été utilisés, à ma connaissance, pour étudier directement les
représentations modulaires des groupes de Weyl, mais en revanche ils
apparaissent dans au moins deux autres contextes en théorie des
représentations.
Le premier de ces contextes est celui de la théorie de
Kazhdan-Lusztig, et donc de variétés de Schubert.

Cette fois, on considère un groupe réductif complexe $G$ sur $k$ de
caractéristique $\ell$, et on veut en étudier les
représentations rationnelles. Pour tout poids $\l$ dans $X(T)$, on a un
module induit $\nabla(\l)$. S'il est non nul, il a un socle simple
$L(\l)$, et toutes les représentations simples de $G$ peuvent être
construites de cette manière.
On veut déterminer les multiplicités $[\nabla(\l) : L(\mu)]$ pour des
poids $\l$, $\mu$ dans $X(T)$ tels que $\nabla(\l)$ et $\nabla(\mu)$
soient non nuls. Lusztig \cite{LusPb} a proposé une conjecture pour
ces multiplicités dans le cas $\ell > h$ (l'analogue pour $G$ défini sur $\CM$ avait été
conjecturé dans \cite{KL1}), faisant un lien avec les faisceaux
pervers sur le dual de Langlands $\GC$ de $G$.

\`A l'époque où Soergel écrivait \cite{Soergel}, on
savait \cite{AJS} que cette conjecture était vraie pour $\ell$
\og assez grand \fg, mais, en dehors des types $A_1$, $A_2$, $A_3$,
$B_2$ et $G_2$, il n'y avait pas un seul nombre premier $\ell$ dont on sût
s'il était assez grand~! On espère qu'il suffit de prendre $\ell$ plus
grand que le nombre de Coxeter $h$. Soergel montre que, si $\ell > h$,
alors une partie de la conjecture de Lusztig (pour les poids \og
autour du poids de Steinberg \fg) est équivalente au fait que
$\pi_{s*}\p\JC_{!*}(S_w,\FM)$ est semi-simple pour toute réflexion
simple $s$ et tout élément $w$ du groupe de Weyl $W$, où
$\pi_s$ est le morphisme quotient $\GC/\BC \to \GC/\PC_s$ (on dénote 
par $\BC$ un sous-groupe de Borel de $\GC$, et par $\PC_s$
le sous-groupe parabolique minimal contenant $\BC$ correspondant à $s$). Pour $\KM$ au
lieu de $\FM$, cela résulte du théorème de décomposition. De plus, il
définit de manière unique pour chaque $x$ dans $W$ un faisceau pervers
indécomposable $\LC_x$, dont les fibres de cohomologie encodent des
multiplicités.

\`A la fin de leur article original \cite{KL}, Kazhdan et Lusztig
mentionnent le cas de $Sp_4$. On a deux éléments de longueur trois
dans le groupe de Weyl. Parmi les deux variétés de Schubert
correspondantes, l'une est lisse, et l'autre a un lieu singulier de
codimension deux. Plus précisément, cette dernière est un fibré en
droites projectives sur une singularité simple de type $A_1$. On sait
(et nous le verrons dans cette thèse) que la cohomologie
d'intersection se comporte différemment pour $\ell = 2$ dans ce
cas. Je remercie Geordie Williamson pour avoir attiré mon attention
sur ce point.

Ainsi, on connaissait depuis longtemps des exemples avec de la $2$-torsion
dans les types non simplement lacés. Ce n'est que récemment que Braden
a montré qu'il y avait de la torsion dans les types $A_7$ et
$D_4$ (il l'a annoncé à la rencontre \og Algebraische Gruppen \fg à
Oberwolfach en 2004). Encore plus récemment, Geordie Williamson (un
étudiant de Soergel) est parvenu à obtenir des résultats
positifs. Dans \cite{WillLowRank}, ce dernier développe une procédure
combinatoire (basée sur le $W$-graphe), qui montre qu'il n'y a pas de
$\ell$-torsion, pour $\ell$ bon et différent de $2$, sous certaines
conditions qui sont très souvent vérifiées en petit rang. En
particulier, il montre que c'est le cas pour tout $\ell \neq 2$
dans les types $A_n$, $n < 7$. Ainsi, la conjecture de Lusztig (pour
les poids autour du poids de Steinberg) est vérifiée pour $SL_n$, $n \leqslant 7$,
dès que $\ell > n$.

Je pense qu'on pourrait trouver d'autres exemples de torsion dans les
variétés de Schubert, en utilisant les résultats de \cite{BP}. Dans
cet article, Brion et Polo décrivent les singularités génériques de
certaines variétés de Schubert (paraboliques), et en déduisent en
particulier une façon efficace de calculer certaines polynômes de
Kazhdan-Lusztig. Même dans le cas où les polynômes étaient déjà
connus, Brion et Polo donnent une information géométrique plus précise, 
qu'on pourrait utiliser pour calculer la torsion dans la cohomologie
d'intersection. Dans les cas où leurs résultats s'apppliquent, Brion
et Polo décrivent la singularité transverse comme l'adhérence de l'orbite d'un
vecteur de plus haut poids dans un module de Weyl pour un certain
sous-groupe réductif contenant $T$. On pourrait traiter ces
singularités comme on l'a fait pour la classe minimale. Ils décrivent
aussi une généralisation avec des multicônes.

\subsection*{Grassmannienne affine}

Le deuxième contexte où l'on a déjà utilisé des faisceaux pervers à
coefficients quelconques est celui des grassmanniennes affines. 
Mirkovi\'c et Vilonen donnent dans \cite{MV} une version géométrique
de l'isomorphisme de Satake. Ils construisent une équivalence entre la
catégorie des représentations rationnelles d'un groupe réductif $G$
sur un anneau quelconque $\L$ et une catégorie de faisceaux pervers
équivariants sur la grassmannienne affine du dual de Langlands de $G$,
défini sur $\CM$. \'Etant donné le lien entre la variété nilpotente et
la grassmannienne affine pour $G = GL_n$ \cite{LusGreen}, il me semble
maintenant que cela doit impliquer notre conjecture sur la coïncidence
de la matrice de décomposition pour les faisceaux pervers sur le cône
nilpotent avec celle de l'algèbre de Schur en type $A$ (il y a
peut-être quelques compatibilités à vérifier). Je remercie
les personnes qui m'ont signalé cet article, à commencer par George Lusztig
lui-même. Cependant, je pense qu'il serait aussi intéressant
d'explorer l'approche que nous proposons dans le dernier chapitre, qui
est un premier pas dans l'étude des faisceaux-caractères modulaires.

L'article de \cite{MV} suggère que les matrices de décomposition de
faisceaux pervers équivariants sur la grassmannienne affine a une
interprétation en termes de représentations.
En ce qui concerne la détermination géométrique concrète de ces nombres de
décomposition, dans un type quelconque,
cette thèse donne déjà quelques résultats, en utilisant \cite{MOV}. En
effet, la plupart des dégénérations minimales sont des singularités
Kleiniennes ou minimales, pour lesquelles nos résultats s'appliquent
directement. Dans les types non simplement lacés, on trouve aussi des
singularités que les auteurs appellent \og quasi-minimales \fg, de
types qu'ils désignent par $ac_2$, $ag_2$, et $cg_2$. Il serait intéressant de faire les
calculs sur les entiers dans ce cas. Par exemple, Malkin, Ostrik et
Vybornov conjecturent que les singularités de types $a_2$, $ac_2$ et $ag_2$
(resp. $c_2$ et $cg_2$) sont deux à deux non équivalentes. La
cohomologie d'intersection rationnelle ne permet pas de les
séparer. Mais peut-être pourrait-on les séparer en travaillant sur les
entiers. De la même manière, on pourrait trouver des preuves plus
simples de non-lissité (voir la dernière section de leur article, où
il font des calculs de multiplicités équivariantes). Par exemple, les
singularités de type $c_n$ et $g_2$ sont rationnellement lisses mais
non $\FM_2$-lisses. Il faudrait faire les calculs pour les
singularités quasi-minimales.

D'une manière générale, je pense que les faisceaux pervers sur les
entiers et modulo $\ell$ sont encore sous-utilisés, et qu'ils seront
amenés à jouer un rôle de plus en plus important, notamment en théorie des
représentations.

}

{\selectlanguage{english}

\addchap{Introduction}

\section*{Context and overview}

In 1976, Springer introduced a geometrical construction of irreducible
representations of Weyl groups, which had a deep influence and many
later developments, which lead to Lusztig's theory of character
sheaves, which enables one to calculate character values for finite
reductive groups. In this thesis, we define a Springer correspondence
for modular representations of Weyl groups, and establish some of its
properties, thus answering a question raised by Springer himself.

\subsection*{Modular representations, decomposition matrices}

The modular representation theory of finite groups, initiated and
developed by Brauer since the early 1940's, is concerned with
representations of finite groups over fields of characteristic $\ell >
0$. When $\ell$ divides the order of the group, the category of
representations is no longer semi-simple.

Let $W$ be a finite group. We fix a finite extension $\KM$ of the
field $\QM_\ell$ of $\ell$-adic numbers. Let $\OM$ be its valuation
ring. We denote by $\mG = (\varpi)$ the maximal ideal of $\OM$, and by
$\FM$ the residual field (which is finite of characteristic
$\ell$). The triplet $(\KM,\OM,\FM)$ is called an $\ell$-modular
system, and we assume that it is large enough for $W$ (that is, we
assume that all simple $\KM W$-modules are absolutely simple, and
similarly for $\FM$). The letter $\EM$ will denote either of the rings
of this triplet.

For an abelian category $\AC$, we denote by $K_0(\AC)$ its
Grothendieck group. When $\AC$ is the category of finitely generated
$A$-modules, where $A$ is a ring, we adopt the notation $K_0(A)$. The
full subcategory consisting of the projective objects of $\AC$ will be
denoted $\Proj \AC$.

The Grothendieck group  $K_0(\KM W)$ is free with basis $([E])_{E
\in\Irr \KM W}$, where $\Irr \KM W$ denotes the set of isomorphism
classes of simple $\KM$-modules; and similarly for $K_0(\FM
W)$. For $F \in \Irr \FM W$, let $P_F$ be a projective cover
of $F$. Then $([P_F])_{F\in\Irr \FM W}$ is a basis of
$K_0(\Proj\:\FM W)$. Reduction modulo $\mG$ defines an
isomorphism from $K_0(\Proj\:\OM W)$ onto $K_0(\Proj\:\FM W)$.

We will define the $cde$ triangle \cite{Serre}.
\[
\xymatrix{
K_0(\Proj\:\FM W) \ar[rr]^c \ar[dr]_e && K_0(\FM W)\\
& K_0(\KM W) \ar[ur]_d
}
\]
The morphism $c$ is induced by the function which maps every
projective $\FM W$-module on its class in $K_0(\FM W)$. The functor
``extension of scalars to $\KM$'' induces a morphism from
$K_0(\Proj\:\OM W)$ into $K_0(\KM W)$. The morphism $e$ is obtained by
composing with the inverse of the canonical isomorphism from
$K_0(\Proj\:\OM W)$ onto $K_0(\Proj\:\FM W)$.  The morphism $d$ is an
little more difficult to define. Let $E$ be a $\KM W$-module. One can choose a
$W$-stable $\OM$-lattice $E_\OM$ in $E$. The image of $\FM \otimes_\OM
E_\OM$ in $K_0(\FM W)$ does not depend on the choice of the lattice
\cite{Serre}. The morphism $d$ is induced by the function which maps
$E$ to $[\FM \otimes_\OM E_\OM]$.

We define the decomposition matrix
$D^W = (d^W_{E,F})_{E \in \Irr \KM W,\ F \in \Irr \FM W}$ by
\[
d([E]) = \sum_{F \in \Irr \FM W} d^W_{E,F} [F]
\]

One of the main problems in modular representation theory is to
determine these decomposition numbers $d^W_{E,F}$ explicitly for
interesting classes of finite groups. This problem is open for the
symmetric group.

The $cde$ triangle can be interpreted in terms of ordinary and modular
(or Brauer) characters. We refer to \cite{Serre}. When the ordinary
characters are known (which is the case for the symmetric group), the
knowledge of the decomposition matrix is equivalent to the
determination of Brauer characters.

There are variants of this problem when one leaves the framework of
finite groups. One can for example consider the modular
representations of Hecke algebras. These algebras are deformations of
reflection group algebras, and play a key role in the representation
theory of finite groups of Lie type. They are defined generically, and
one can study what happens for special values of the parameters.

One can also study the rational representations of a reductive group in
positive characteristic, or the representations of a reductive Lie
algebra, or of quantum groups (deformations of enveloping algebras),
etc. In that case, one considers the multiplicities of simple objects
in some classes of particular objects, whose characters are known.

\subsection*{Perverse sheaves and representations}

In the last three decades, geometric methods have lead to dramatic
progress in many parts of representation theory. We will be
particularly concerned with representations of Weyl groups and finite
groups of Lie type.

In 1976, Springer managed to construct geometrically all the
irreducible ordinary representations of a Weyl group in the cohomology
of certain varieties associated with the nilpotent elements of the
corresponding Lie algebra \cite{SPRTRIG, SPRWEYL}. This discovery had
a huge impact. Many other constructions were subsequently proposed by
other mathematicians. For example, Kazhdan and Lusztig proposed a
topological approach \cite{KL}, and Slodowy constructed these
representations by monodromy \cite{SLO1}. At the beginning of the
1980's, the blossoming of intersection cohomology permitted to
reinterpret Springer correspondence in terms of perverse
sheaves \cite{LusGreen, BM}.

Lusztig extended this work by studying a generalized Springer
correspondence, and some intersection cohomology complexes on a
reductive group $G$ or its Lie algebra $\gG$, which he calls
admissible complexes or character sheaves
\cite{ICC,CS1,CS2,CS3,CS4,CS5}. If $G$ is endowed with an
$\FM_q$-rational structure defined by a Frobenius endomorphism $F$,
then the $F$-stable character sheaves yield central functions on the
finite group $G^F$, which are very close to the irreducible
characters. The transition matrix between these two bases is described
by a Fourier transform. Thus, these geometric methods were powerful
enough to determine the ordinary characters of finite reductive groups
(at least for those with connected cent re).

As far as I know, up to now these methods have not been used to study
the modular representations of Weyl groups or of finite groups of Lie
type. Yet, this was the case in at least two other modular
situations. Soergel \cite{Soergel} transformed a problem about the
category $\OC$ into a problem for perverse sheaves with modular
coefficients on Schubert varieties. Secondly, Mirkovic and Vilonen
established an equivalence of categories between the rational
representations of a reductive group over an arbitrary ring $\L$ and
the perverse sheaves with $\L$ coefficients on the Langlands dual
group, defined over $\CM$. The classical topology allows one to use
arbitrary coefficients. By the way, this work gives an intrinsic
definition of the Langlands dual group by defining its category of
representations. We will come back to this in the section about
perspectives.

\subsection*{Modular Springer correspondence}

It was tempting to look for a link between modular representations of
Weyl groups and perverse sheaves modulo $\ell$ on the nilpotents. In
other words, to try to define a modular Springer correspondence. Of
course, the construction of Lusztig-Borho-MacPherson uses Gabber's
decomposition theorem \cite{BBD}, which no longer holds in
characteristic $\ell$. But Hotta and Kashiwara \cite{HK} have an
approach \emph{via} a Fourier transform for $\DC$-modules, in the case
of the base field is $\CM$, and this permits to avoid the use of the
decomposition theorem. Moreover, the Fourier-Deligne transform allows
one to consider a base field of characteristic $p$ and $\ell$-adic
coefficients \cite{BRY}. In this thesis, we define a modular Springer
correspondence  by using a Fourier-Deligne transform with modular
coefficients. Besides, we define a decomposition matrix for perverse
shaves on the nilpotents, and we compare it with the decomposition
matrix of the Weyl group. Moreover, we calculate certain decomposition
numbers in a purely geometric way. We will see that certain properties
of decomposition numbers for Weyl groups can be considered as the
shadow of some geometric properties. For example, the James's row and
column removal rule can be explained by a similar rule about nilpotent
singularities obtained by Kraft and Procesi \cite{KP1}, once the
modular Springer correspondence has been determined (which we will
also do in this thesis).

\section*{Detailed contents}

\subsection*{Preliminaries and examples}

In Chapter \ref{chap:preliminaries}, we review perverse sheaves over
$\KM$, $\OM$ and $\FM$ which we will use subsequently. We particularly
insist on the aspects which are specific to $\OM$ and $\FM$. For
example, over $\OM$ there is no self-dual perversity, but a pair of
perversities, $p$ and $p_+$, exchanged by the duality. Moreover, we
study the interaction between torsion pairs and $t$-structures (about
this subject, see \cite{HRS}), and also with recollement situations. 
In this part, which is somewhat technical, one can find many
distinguished triangles which will be used in the sequel. The key
point is that the (derived) reduction modulo $\ell$ does not commute
with truncations in general. We also give some complements about the
perverse extensions $\p j_!$, $\p j_{!*}$, $\p j_*$ (on the top and
the socle, and on the behavior with respect to multiplicities).
Finally, we define decomposition numbers for perverse sheaves. We are
particularly interested by the case of a $G$-variety with a finite
number of orbits: we have the nilpotent cone in mind.

In Chapter \ref{chap:examples}, we give some examples of perverse
sheaves, and in particular of intersection cohomology complexes over
$\EM$. We recall the properties of proper small (resp. semi-small)
morphisms. In particular, the intersection cohomology complex of a
variety having a small resolution is obtained by direct image.

Afterwards, we introduce the notion of smooth equivalence of
singularities, and recall that the local intersection cohomology is an
invariant for that equivalence.

Then we study conical singularities, where the local intersection
cohomology is reduced to the calculation of the cohomology of a
variety (the open complement to the vertex of the cone), and, more
generally, we consider the case of an affine variety endowed with a
$\GM_m$-action contracting everything onto the origin.

It is a natural question to ask when the intersection cohomology
complex is reduced to the constant sheaf $\EM$ (so that the variety
satisfies usual Poincaré duality). In that case, we say that the
variety is $\EM$-smooth. A typical example of $\KM$-smooth
(resp. $\FM$-smooth) variety is given by the quotient of a smooth
variety by a finite group (resp. a finite group with order prime to
$\ell$).

In the last section of this chapter, we study simple singularities. A
normal variety $X$ has rational singularities if it has a resolution
$\pi : \Xti \to X$ with $R^i \pi_* \OC_X = 0$ for $i > 0$.
Over $\CM$, the surfaces with a rational double point are (up to
analytic equivalence) the quotients of the affine plane
by a finite subgroup of $SL_2(\CM)$. They are classified by
simply-laced Dynkin diagrams. One can interpret the other types
by considering the action of a group of symmetries. One
associates to each Dynkin diagram $\G$ a homogeneous diagram
$\wh\G$, and a group of symmetries $A(\G)$. In the case where $\G$ is
already homogeneous, we have $\wh\G = \G$ and $A(\G) = 1$. Let
$\wh\Phi$ be a root system of type $\wh\G$. We denote by $P(\wh\Phi)$
the weight lattice, and by $Q(\wh\Phi)$ the root lattice. Let $H$
be the finite group of $SL_2(\CM)$ associated to $\wh\G$. We show that
\[
H^2\left((\AM^2\setminus\{0\})/H,\ZM\right) \simeq P(\wh\Phi)/Q(\wh\Phi)
\]
with a natural action of $A(\G)$. Thanks to the results of Chapter
\ref{chap:preliminaries}, this allows us to compare perverse sheaves
in characteristic $0$ and in characteristic $\ell$.

\subsection*{Calculation of decomposition numbers}

Until Chapter \ref{chap:dec}, our aim is to calculate certain
decomposition numbers for $G$-equivariant perverse sheaves
on the nilpotent variety, by geometrical methods.

Let us first introduce some notation. The simple perverse sheaves over $\KM$
(resp. $\FM$) are parametrized by the set $\NG_\KM$
(resp. $\NG_\FM$) of pairs $(x,\rho)$ (up to conjugacy) consisting of
a nilpotent element $x$ and a character $\rho\in\Irr \KM A_G(x)$ (resp.
$\rho \in \Irr \FM A_G(x)$), where $A_G(x)$ is the finite group of
components of the centralizer of $x$ in $G$. We will denote by
\[
\left(d_{(x,\rho),(y,\s)}\right)_{(x,\rho)\in\NG_\KM,\ (y,\s)\in\NG_\FM}
\]
the decomposition matrix of these perverse sheaves. In the case
of $GL_n$, all the $A_G(x)$ are trivial, so that one can forget about
$\rho$ which is always $1$, and the nilpotent orbits are
parametrized by the set $\PG_n = \{\l\vdash n\}$ of all partitions of
$n$. In that case, the decomposition matrix will be denoted by
\[
(d_{\l,\mu})_{\l,\mu\in\PG_n}
\]

On the other hand, the decomposition matrix for the Weyl group $W$
will be denoted by
\[
\left(d_{E,F}^W\right)_{E\in\Irr\KM W,\ F\in\Irr\FM W}
\]

For the symmetric group $\SG_n$, the simple $\KM\SG_n$-modules are the 
Specht modules $S^\l$, for $\l \in \PG_n$. They are defined
over $\ZM$, and endowed with a symmetric bilinear form defined over $\ZM$. The
modular reduction of the Specht module, which we will still denote by
$S^\l$, is thus also endowed with a symmetric bilinear form.
The quotient of $S^\l$ by the radical of that symmetric bilinear form
is either zero, or a simple $\FM\SG_n$-module. The set of partitions
$\mu$ such that this quotient is non-zero (we then denote it by $D^\mu$)
is the set $\PG_n^{\ell\text{-reg}}$ of partitions of $n$ which are
$\ell$-regular (each entry is repeated at most $\ell - 1$ times).
The $D^\mu$, for $\mu\in\PG_n^{\ell\text{-reg}}$, form a complete
set of representatives of isomorphism classes of simple
$\FM\SG_n$-modules. The decomposition matrix of the symmetric group
$\SG_n$ will be rather denoted by
\[
\left(d_{\l,\mu}^{\SG_n}\right)_{\l\in\PG_n,\ \mu\in\PG_n^{\ell\text{-reg}}}
\]

For the Schur algebra
\[
S_\EM(n) = S_\EM(n,n) = \End_{\EM\SG_n}\left(\bigoplus_{\l \vdash n}
\Ind_{\EM \SG_\l}^{\EM \SG_n} \EM\right)
\]
we will denote the decomposition matrix by
\[
\left(d_{\l,\mu}^{S(n)}\right)_{\l,\mu \in\PG_n}
\]
It is known that
\[
d_{\l,\mu}^{S(n)} = d_{\l',\mu'}^{\SG_n}
\]
for $\l\in\PG_n,\ \mu\in\PG_n^{\ell\text{-reg}}$, where $\l'$ stands for
the conjugate partition. We will see that
\[
d_{\l,\mu} = d^{\SG_n}_{\l',\mu'} = d_{\l,\mu}^{S(n)}
\]
for $\l\in\PG_n,\ \mu\in\PG_n^{\ell\text{-reg}}$,
and we conjecture that
\[
d_{\l,\mu} = d^{S(n)}_{\l,\mu}
\]
for \emph{all} partitions $\l$, $\mu$ of $n$
(see the remarks in the last section of this introduction).

In Chapter \ref{chap:min}, we calculate the integral cohomology
of the minimal (non-trivial) nilpotent orbit $\OC_\mini$ in a simple
Lie algebra $\gG$ over the fields of complex numbers. Actually,
the results and methods of this chapter are still valid for a
base field of characteristic $p > 0$, if one works with
étale cohomology with coefficients in the $\ell$-adic integers.

The rational cohomology of $\OC_\mini$ is already known.
The dimension of $\OC_\mini$ is $d = 2h^\vee - 2$, where $h^\vee$ is the
dual Coxeter number. The first half of the cohomology is given by
\[
\t_{\leqslant d - 1} \rg(\OC_\mini,\QM) \simeq \bigoplus_{i=1}^k \QM[-2(d_i - 2)]
\]
where $k$ is the number of long simple roots, and 
$d_1 \leqslant \ldots \leqslant d_k \leqslant \ldots \leqslant d_n$
are the degrees of $W$ (and $n$ is the total number of simples roots).
The other half can be deduced Poincaré duality.

Therefore, we are mainly interested by the torsion.
If $\Phi$ is the root system of $\gG$, and
$\Phi'$ is the root subsystem generated by the long simple roots (for some
choice of basis), then the middle cohomology of $\OC_ \mini$ is
\[
H^d(\OC_\mini,\ZM) \simeq P^\vee(\Phi')/Q^\vee(\Phi')
\]
We will see in the sequel that this result is linked to the
modular reduction of the natural representation of the Weyl group $W'$
of $\Phi'$.

Apart from the middle cohomology, we have no uniform formula for the
torsion part of the cohomology of $\OC_\mini$. But we know that it is
the cokernel of a matrix whose coefficients are determined explicitly
by the poset of the long roots in $\Phi$, which is leveled by the
coheight (height of the coroot), which allows us to make the
calculation in any type. Apart from the middle cohomology, we observe
that the prime numbers dividing the torsion are bad. This amounts to
say that the intersection cohomology stalks with integer coefficients
are without $\ell$-torsion when $\ell$ is good (only for the
perversity $p$, precisely not for $p_+$, where the middle cohomology
appears. We do not know whether one can find an interpretation (maybe
homological) in terms of representation theory to these finite groups
of bad torsion.

In Chapter \ref{chap:dec}, we calculate certain decomposition numbers
for $G$-equivariant perverse sheaves on the
nilpotent variety, using on the one hand the preceding results,
and on the other hand geometric results that can be found in the
literature.

First of all, we determine the decomposition numbers associated to
the regular and subregular classes (the centralizer of a subregular
nilpotent element is not necessarily connected). As in the section on
simple singularities, we associate to the type $\G$ of $G$ a
homogeneous diagram $\wh\Phi$ and a group of symmetries $A := A(\G)$
which is isomorphic to $A_G(x_\subreg)$ when $G$ is adjoint. We have
\[
d_{(x_\reg,1),(x_\subreg,\rho)}
= [\FM \otimes_\ZM (P(\wh\Phi)/Q(\wh\Phi)) : \rho]
\]
for all $\rho$ in $\Irr \FM A$. we calculate this multiplicity
in all types, for each prime number $\ell$ and for each
$\rho \in \Irr\FM A$.
For the minimal and trivial classes, we deduce from the results
of Chapter \ref{chap:min} that
\[
d_{(x_\mini,1),(0,1)}
= \dim_\FM \FM \otimes_\ZM \left( P^\vee(\Phi')/Q^\vee(\Phi') \right)
\]
We also give this multiplicity in all types.
As an example, let us see what happens for $GL_n$. We find
\[
d_{(n),(n-1,1)} = d_{(21^{n-2}),(1^n)} =
\begin{cases}
1 \text{ if } \ell \mid n\\
0 \text{ otherwise}
\end{cases}
\]
which is compatible with our conjecture making a link with the Schur
algebra. We have another result in that direction.  The decomposition
numbers for the Schur algebras satisfy the following property.  If
$\l$ and $\mu$ are two partitions of $n$ whose $r$ first lines and $s$
first columns are identical, and if $\l_1$ and $\mu_1$ are the
partitions (of a smaller integer $n_1$) obtained from $\l$ and $\mu$
by suppressing these lines and columns, we have
\[
d^{S(n)}_{\l,\mu} = d^{S(n_1)}_{\l_1,\mu_1}
\]

Kraft and Procesi have shown that the singularities of the closures
of nilpotent orbits in $GL_n$ satisfy a
similar property \cite{KP1}. With the same notation, we have
\[
\codim_{\ov\OC_{\l_1}} \OC_{\mu_1} = \codim_{\ov\OC_\l} \OC_\mu
\quad\text{ and }\quad
\Sing(\ov\OC_{\l_1}, \OC_{\mu_1}) = \Sing(\ov\OC_\l, \OC_\mu)
\]

We deduce that the decomposition numbers $d_{\l,\mu}$ also satisfy
that property:
\[
d_{\l,\mu} = d_{\l_1,\mu_1}
\]

If  $\l > \mu$ are two adjacent partitions of $n$ for the dominance
order (that is, if there is no partition $\nu$ such that
$\l > \nu > \mu$), Kraft and Procesi use the result on
lines and columns to reduce the determination of the singularity
of $\ov\OC_\l$ along $\OC_\mu$ to the extreme cases
$(\l,\mu) = ((m), (m - 1, 1))$ and $(\l,\mu) = ((2, 1^{m-2}), (1^m))$,
for a smaller integer $m$. The minimal degenerations in type $A_n$
are thus all of type $A_m$ (a simple  singularity of type $A_m$)
or $a_m$ (a minimal singularity of type $a_m$), for smaller integers $m$.

Since, in $GL_n$, all the $A_G(x)$ are trivial, this is enough to
determine the decomposition number $d_{\l,\mu}$ when $\l$ and
$\mu$ are adjacent. In that case, we have:
\[
d^{S(n)}_{\l,\mu} = d_{\l,\mu}
\]
as expected.

Kraft and Procesi have also shown that the singularities of the closures
of nilpotent orbits classical types satisfy a row and column removal rule
\cite{KP2}. They must deal with orthogonal and symplectic groups
simultaneously. They deduce the singularity type of the minimal
degenerations in that case. They find only simple and minimal
singularities of classical types, with only one exception. More
precisely, in the codimension two case, we have (up to smooth
equivalence) a singularity of type $A_k$, $D_k$ or $A_k \cup A_k$,
the latter being the non-normal union of two simple singularities of type
$A_k$, intersecting transversely at the singular point.
When the codimension is greater than $2$, we have a
minimal singularity of type $b_k$, $c_k$ or $d_k$. By
suppressing these lines and columns, one can always reduce to these
irreducible cases. In this article, Kraft and Procesi determine
which orbit closures are normal in classical types, which was their goal.

We can also use their results to determine other decomposition numbers
in classical types, but to do so in all cases, one should also
determine the local systems which appear. In any case, for a minimal
degeneration $\ov\OC \supset \OC'$ classical type, one can always
determine the following quantity:
\[
\sum_{\rho \in \Irr \FM A_G(x_{\OC'})}
d_{(x_\OC,1),(x_{\OC'},\rho)}
\]
(In classical types, the $A_G(x)$ are of the form $(\ZM/2)^k$,
and thus abelian, so all the $\rho \in \Irr \FM A_G(x_{\OC'})$ are of
degree $1$.) In particular, one can tell when the
$d_{(x_\OC,1),(x_{\OC'},\rho)}$ are zero for all the $\rho \in
\Irr \FM A_G(x_{\OC'})$. A more detailed study should be enough to
determine all the decomposition numbers of this type.

Another result of Kraft and Procesi, about the special decomposition
of the nilpotent variety \cite{KP3}, allows us to show that certain
decomposition numbers are zero, in the classical types,
when $\ell \neq 2$. In \cite{LusSpec}, Lusztig introduced
a subset of $\Irr \KM W$, whose elements are called special
representations. The special nilpotent classes are the classes
$\OC$ such that the representation $\chi$ associated to
$(\OC,\KM)$ by the Springer correspondence is special.
On the other hand, Spaltenstein introduced in \cite{SpalDual} an order
reversing map from the set of nilpotent classes to itself,
such that $d^3 = d$ (it is an involution on its
image). The image of $d$ is precisely the set of special classes.
The locally closed subvarieties
\[
\widehat \OC = \ov \OC \setminus
\mathop{\bigcup_{\OC' \text{ special}}}\limits_{\ov \OC'\subset\ov \OC} \ov \OC'
\]
where $\OC$ runs over the set of special classes, form a
partition of the nilpotent variety $\NC$. They are called special pieces.
Thus each nilpotent class is contained in a unique
special piece. Lusztig attached to each special class $\OC$ a
canonical quotient $\ov A_G(x_\OC)$ of the finite group $A_G(x_\OC)$, and
conjectured that the special piece $\wh\OC$ is the quotient of a smooth
variety by $\ov A_G(x_\OC)$. A consequence of this
conjecture is that $\wh\OC$ is $\KM$-smooth, but actually it
gives more information: in particular, the conjecture implies that
$\wh\OC$ is $\FM$-smooth as soon as $\ell$ does not divide the order
of the group $\ov A_G(x_\OC)$. In \cite{KP3}, Kraft and Procesi show
that this conjecture holds for classical types. We deduce that,
in classical types, we have
\[
d_{(x_\OC,1),(x_{\OC'},\rho)} = 0
\]
for $\ell > 2$, when $\OC$ is a special class, $\OC'$ is a
class contained in the special piece $\wh\OC$, and
$\rho \in \Irr \FM A_G(x_{\OC'})$. A more detailed study
could maybe give the decomposition numbers when $\ell = 2$.

Let us make one more remark.
In another article \cite{Kraft}, Kraft solves the normality problem
for closures of nilpotent orbits in
$G_2$. He gives the following information, which is not covered by
the preceding results: $\ov\OC_{10}$ has a simple singularity of type
$A_1$ along $\OC_8$, where $\OC_i$ denotes the unique nilpotent class of
dimension $i$ in the Lie algebra $\gG$ of a simple group $G$ of type
$G_2$. Since $A_G(x_8) = 1$ (we denote by $x_i$ a representative of $\OC_i$),
this allows us to determine the decomposition number
$d_{(x_{10},1),(x_8,1)}$:
\[
d_{(x_{10},1),(x_8,1)} = 
\begin{cases}
1 \text{ if } \ell = 2,\\
0 \text{ otherwise}
\end{cases}
\]

A more detailed study of this article would maybe yield more
decomposition numbers in a geometrical way. In any case, using the
modular Springer correspondence we will be able to determine the whole
decomposition matrix when $\ell = 3$, and all the matrix but one
column when $\ell = 2$. For $\ell > 3$, $\ell$ does not divide the
order of the Weyl group, and the decomposition matrix is the identity;
I think that this holds in any type, but one will need the notion of
cuspidality. At least, the part of the decomposition matrix
corresponding to the Weyl group is the identity matrix, as we shall see.

\subsection*{Modular Springer correspondence and decomposition matrices}

In the sequel of the thesis, we define a modular Springer
correspondence and establish some of its properties, notably the fact
that it preserves decomposition numbers. Since Gabber's decomposition
theorem \cite{BBD} is no longer true in the modular case, we are
inspired by the approach of Kashiwara and Brylinski \cite{BRY}, using
a Fourier transform.

In Chapter \ref{chap:fourier}, we introduce the Fourier-Deligne
transform, following an article by Laumon
\cite{Lau}. We give detailed proofs, and check that everything is fine
when we take $\KM$, $\OM$ or $\FM$ coefficients.

Chapter \ref{chap:springer} is the core of this thesis. First, we
recall the geometric context of Springer correspondence, which is
Grothendieck's simultaneous resolution $\pi$ of the singularities of
the fibers of the adjoint quotient. Taking the fiber at zero, we
recover Springer's resolution $\pi_\NC$ of the nilpotent cone $\NC$.

Then we introduce the perverse sheaves $\EM\KC_\rs$, $\EM\KC$ and
$\EM \KC_\NC$, respectively on the open subvariety $\gG_\rs$ of the
regular semi-simple elements, on $\gG$ itself, and on the closed
subvariety $\NC$ of nilpotent elements. We have the following diagram
with cartesian squares:

\[
\xymatrix@=1.5cm{
\tilde\gG_\rs
\ar[d]_{\pi_\rs}
\ar@<-.5ex>@{^{(}->}[r]^{\tilde j_\rs}
\ar@{}[dr] | {\DS\boxempty_\rs}
&
\tilde\gG
\ar[d]^\pi
\ar@{}[dr] | {\DS\boxempty_\NC}
&
\NCt
\ar@<.5ex>@{_{(}->}[l]_{i_\NCt}
\ar[d]^{\pi_\NC}
\\
\gG_\rs
\ar@<-.5ex>@{^{(}->}[r]_{j_\rs}
&
\gG
&
\NC
\ar@<.5ex>@{_{(}->}[l]^{i_\NC}
}
\]

Let $r$ be the rank of $G$, and $\nu$ the number of positive roots in
$\Phi$. We set

\begin{gather*}
\EM\KC_\rs = {\pi_\rs}_! \EM_{\tilde \gG_\rs} [2\nu + r]\\
\EM\KC     = \pi_!       \EM_{\tilde \gG} [2\nu + r]\\
\EM\KC_\NC = {\pi_\NC}_! \EM_\NCt [2\nu]
\end{gather*}

We have
\begin{gather*}
\EM\KC = \p {j_\rs}_{!*} \EM\KC_\rs\\
\EM\KC_\NC = i_\NC^* \EM\KC [-r]
\end{gather*}

The morphism $\pi$ is proper and small, generically a $W$-torsor
(above $\gG_\rs$), and its restriction $\pi_\NC$ to the nilpotents is
semi-small.

Afterwards, we define a modular Springer correspondence, using a
Fourier-Deligne transform. To $E\in\Irr\KM W$, the Springer
correspondence à la Brylinski associates the perverse sheaf
$\TC(E) = \FC(\p {j_\rs}_{!*}(E[2\nu + r])$.
This defines an injective map
\[
\Psi_\KM : \Irr \KM W \longto \NG_\KM
\]
We will denote by $\impsi_\KM$ its image.
We proceed similarly for the modular Springer correspondence.
To $F\in \Irr \FM W$, we associate
\[
\TC(F) = \FC(\p {j_\rs}_{!*}(F[2\nu + r]))
\]
and this defines an injective map
\[
\Psi_\FM : \Irr \FM W \longto \NG_\FM
\]

We will denote by $\impsi_\FM$ its image.

Then, we show that the decomposition matrix of the Weyl group $W$ can
be extracted from the decomposition matrix for $G$-equivariant
perverse sheaves on the nilpotent variety, by keeping only the lines
which are in the image of the ordinary Springer correspondence, and
the columns which are in the image of the modular Springer
correspondence. More precisely, we show that, for all
$E\in\Irr\KM W$ and $F\in\Irr\FM W$, we have
\[
d^W_{E,F} = d_{\Psi_\KM(E),\Psi_\FM(F)}
\]

Finally, we determine the modular Springer correspondence when
$G = GL_n$. We have:
\[
\impsi_\FM = \PG_n^{\ell\text{-res}}
\]
where $\PG_n^{\ell\text{-res}}$ is the set of
$\ell$-restricted partitions of $n$, that is, whose conjugate is
$\ell$-regular.
\[
\forall \l \in \PG_n^{\ell\text{-reg}},\qquad
\Psi_\FM(D^\l) = \l'
\]

In particular, for $\l\in\PG_n$ and $\mu\in\PG_n^{\ell\text{-reg}}$,
we have:
\[
d^{\SG_n}_{\l,\mu} = d_{\l',\mu'}
\]
so that James's row and column removal rule can be seen as a
consequence of the geometric result of Kraft and Procesi about
nilpotent singularities.

\section*{Perspectives}

There are many themes I can explore to extend the present work.

\subsection*{Geometry of the nilpotent orbits}

This thesis has revealed new links between the representation theory
of Weyl group and the geometry of nilpotent classes. One can expect
new interactions between these two areas.

For example, we observed that James's row and column removal rule can
be explained geometrically by the result of Kraft and Procesi about
nilpotent singularities.

On the representation theoretic side, Donkin found a generalization of
this rule \cite{DonkinGen}. I expect a similar generalization on the
geometrical side (one should find a product singularity).

\subsection*{Determination of the modular Springer correspondence,
basic sets}

A question arises naturally about the modular Springer
correspondence. For simplicity, first suppose that $\ell$ does not
divide the orders of the groups $A_G(x)$. Then one can identify
$\NG_\FM$ and $\NG_\KM$ to a common set of parameters $\PG$. Does one
have $\impsi_\FM \subset \impsi_\KM$ in that case ?

Let us assume that it is the case.
Then, for each $F \in \Irr \FM W$, there is a unique
$E \in \Irr \KM W$ such that $\Psi_\KM(E) = \Psi_\FM(F)$.
This defines a basic set for $W$ and shows in a geometrical way the
triangularity of the decomposition matrix $W$.

Even if $\ell$ divides the order of some $A_G(x)$, this question still
makes sense if we choose a basic set for each $A_G(x)$ (which is of
the form $(\ZM/2)^k$, or a symmetric group $\SG_k$,
$k\leqslant 5$, for $G$ adjoint). In fact, for all these groups,
there is a canonical choice.

In the other direction, the knowledge of a basic set for $W$ and a
triangularity property compatible with the order of the orbits through
the Springer correspondence allows to determine the modular Springer
correspondence.

We could determine the modular Springer correspondence of $GL_n$, and
in rank up to three, for this reason (one has to be careful for $G_2$
because there is one cuspidal pair in characteristic zero).

If we could show that $\impsi_\FM \subset \impsi_\KM$, it would be
interesting to determine the basic set that we obtain, and to compare
it with the canonical basic set of \cite{GeckRouquier},
when the latter is well defined (that is, when $\ell$ does not divide
the $A_G(x)$).

\subsection*{Generalized modular Springer correspondence,
  modular character sheaves}

In the original Springer correspondence, $\impsi_\KM$ contains all the
pairs of the form $(\OC,\KM)$, but in general $\impsi_\KM$ is strictly
contained in $\NG_\KM$. The main motivation of Lusztig in \cite{ICC}
is to understand these missing pairs. This work is extended in the
series of articles about character sheaves, which allows to compute
character values of finite groups of Lie type.

Clearly, one of the first things to do to continue the work of this
thesis would be to study the notions of induction and restriction, of
cuspidality, and to define a generalized modular Springer
correspondence, and to determine it in all cases. Perhaps some new
combinatorial objects could appear for classical types ($\ell$-symbols ?).

I hope that this will lead to a theory of modular character sheaves,
with a link with the modular representation theory of finite groups of
Lie type. In the last chapter, we present briefly some calculations
for $\sG\lG_2$.

\subsection*{Determination intersection cohomology stalks}

The determination of the intersection cohomology stalks (either over
$\OM$ for the perversity $p_+$, or over $\FM$) for the nilpotent orbit
closures would be enough to determine the decomposition matrix for
perverse sheaves, and thus for the Weyl group (if the modular Springer
correspondence has been determined). In this thesis, we determine this
correspondence for $GL_n$. We have translated the problem of the
decomposition matrices of the symmetric group into a geometrical and
topological problem, where there is no mention of the Weyl group.

Of course, this problem is certainly very difficult. In characteristic
zero, the determination of the intersection cohomology stalks for the
nilpotents goes through the Springer correspondence and the
orthogonality relations for Green functions (see the algorithm that
Shoji uses in \cite{ShojiF4} for the type $F_4$, which is used again
in other works like \cite{BeSpa} for the types $E_6$, $E_7$, $E_8$,
and generalized in \cite[\S 24]{CS5}). It is unlikely that such an
algorithm exists in characteristic $\ell$.

\subsection*{Schubert varieties}

As far as I know, perverse sheaves modulo $\ell$ had never been used
to study directly the modular representations of Weyl groups, but they
were used in at least two other contexts in representation theory.
The first of these is concerned with Kazhdan-Lusztig theory, and thus
Schubert varieties.

This time, we consider a complex reductive group $G$ over $k$ of
characteristic $\ell$, and we want to study its rational
representations. For each weight $\l$ in $X(T)$, we have an induced
module  $\nabla(\l)$. If it is non-zero, then it has a simple socle
$L(\l)$, and all simple representations of $G$ can be obtained in this
way. We want to determine the multiplicities $[\nabla(\l) : L(\mu)]$ for
weights $\l$, $\mu$ in $X(T)$ such that $\nabla(\l)$ and $\nabla(\mu)$
are non-zero. Lusztig \cite{LusPb} proposed a conjecture for
these multiplicities in the case $\ell > h$ (the analogue for $G$
defined over $\CM$ had been conjectured in \cite{KL1}), making a link
with the perverse sheaves with the Langlands dual $\GC$ of $G$.

When Soergel wrote \cite{Soergel}, this conjecture was known to be
true when $\ell$ is ``large enough'' \cite{AJS}. Nevertheless, apart
from the types $A_1$, $A_2$, $A_3$, $B_2$ and $G_2$, there was no
single prime number $\ell$ which was known to be large enough !
It is hoped that it is enough to take $\ell$ greater than the Coxeter
number $h$. Soergel shows that, if $\ell > h$, then part of Lusztig's
conjecture (for the weights ``around the Steinberg weight'') is
equivalent to the fact that $\pi_{s*}\p\JC_{!*}(S_w,\FM)$ is
semi-simple for each simple reflection $s$ and each element $w$ of the
Weyl group $W$, where $\pi_s$ is the quotient morphism
$\GC/\BC \to \GC/\PC_s$ (we denote by $\BC$ a Borel subgroup of $\GC$,
and by $\PC_s$ the minimal parabolic subgroup containing $\BC$
corresponding to $s$). For $\KM$ instead of $\FM$, this results from
the decomposition theorem. Moreover, he defines for each $x$ in $W$ an
indecomposable perverse sheaf $\LC_x$, whose cohomology stalks encode
multiplicities.

At the end of their original article \cite{KL}, Kazhdan and Lusztig
mention the case of $Sp_4$. We have two elements of length three in
the Weyl group. Among the two corresponding Schubert varieties, one is
smooth, and the other one has a singular locus of codimension
two. More precisely, the latter is a $\PM^1$ bundle over a simple
singularity of type $A_1$. It is known (and we will see this in the
thesis) that the intersection cohomology is different for $\ell = 2$
in this case. I thank Geordie Williamson for explaining this to me.

So, examples with $2$-torsion have been known for a long time in
simply-laced types. Only recently, Braden found examples of
$2$-torsion in types $A_7$ and $D_4$ (he announced this result at the
meeting ``Algebraische Gruppen'' in Oberwolfach in 2004).
Even more recently, Geordie Williamson (a student of Soergel) obtained
positive results. In \cite{WillLowRank}, he develops a combinatorial
procedure (based on the $W$-graph), which shows that there is no
$\ell$-torsion, for $\ell$ good and different from $2$, under certain
conditions which are very often satisfied in small rank. In
particular, he shows that it is the case for all $\ell \neq 2$
in  types $A_n$, $n < 7$. Thus, Lusztig's conjecture (for
the weights around the Steinberg weight) is satisfied for $SL_n$,
$n \leqslant 7$, as soon as $\ell > n$.

I think that one could find other examples of torsion in Schubert
varieties, using the results in \cite{BP}. In this article, Brion and
Polo describe the generic singularities of certain (parabolic)
Schubert varieties, and deduce in particular an efficient way to
calculate certain Kazhdan-Lusztig polynomials. Even in the cases where
these polynomials were already known, Brion and Polo give a more
precise geometrical description, that one could use to calculate the
torsion in the local intersection cohomology. For the cases where
their results apply, Brion and Polo describe the transverse
singularity as the closure of the orbit of a highest weight vector in
a Weyl module, for a certain reductive subgroup containing $T$. One
could treat these singularities in the same way as the minimal
class. They also describe a generalization with multicones.

\subsection*{Affine Grassmannians}

The second context where perverse sheaves with arbitrary coefficients
were used is about affine Grassmannians. In \cite{MV}, 
Mirkovi\'c and Vilonen give a geometric version of Satake isomorphism.
They construct an equivalence between the category of rational
representations of a reductive group $G$ over any ring $\L$ and a
category of equivariant perverse sheaves on the affine Grassmannian of
the Langlands dual of $G$, defined over $\CM$. Given the link between
the singularities of the nilpotent variety and the ones of the affine
Grassmannian for $G = GL_n$ \cite{LusGreen}, it seems to me now that
this should imply our conjecture about the equality between the
decomposition matrix for perverse sheaves on the nilpotent variety and
for the Schur algebra (there may be some compatibilities that have to
be checked). I thank the mathematicians who told me about this
article, and particularly George Lusztig. However, I think it would
also be interesting to explore the approach that we propose in the
last chapter, which is a first step in the study of modular character
sheaves.

The article of \cite{MV} suggests that decomposition matrices for
equivariant perverse sheaves on the affine Grassmannian have a
representation theoretic interpretation. This thesis
can be used to determine concretely some of these
decomposition numbers, using \cite{MOV}. Indeed, most of the minimal
degenerations are either Kleinian or minimal singularities, for which
our results apply directly. In non-homogeneous types, one can also
find other singularities, that the authors call ``quasi-minimal'', of
types $ac_2$, $ag_2$, and $cg_2$. It would be interesting to determine
the intersection cohomology stalks over the integers in this case.
For example, Malkin, Ostrik and Vybornov conjecture that the
singularities of types $a_2$, $ac_2$ and $ag_2$
(resp. $c_2$ and $cg_2$) are pairwise non-equivalent. The rational
intersection cohomology is not enough to distinguish them. However,
they might have a different local intersection cohomology over the
integers. One could also obtain simpler proofs for non-smoothness
(see the last section of their article, where they calculate
equivariant multiplicities). For example, the singularities of type
$c_n$ and $g_2$ are rationally smooth, but not $\FM_2$-smooth.
One should do the calculations for quasi-minimal singularities.

More generally, I think that perverse sheaves over the integers and
modulo $\ell$ are still underused, and that their role will be more
and more important in the years to come, notably in representation theory.

\chapter{Perverse sheaves over $\KM$, $\OM$, $\FM$}\label{chap:preliminaries}

\section{Context}\label{sec:context}

In all this thesis, we fix on the one hand a prime number $p$
and an algebraic closure $\ov\FM_p$ of the prime field with $p$ elements,
and for each power $q$ of $p$, we denote by $\FM_q$ the unique subfield
of $\ov\FM_p$ with $q$ elements. On the other hand, we fix a prime number $\ell$
distinct from $p$, and a finite extension $\KM$ of the field $\QM_\ell$
of $\ell$-adic numbers, whose valuation ring we denote by $\OM$.
Let $\mG = (\varpi)$ be the maximal ideal of $\OM$, and let $\FM = \OM/\mG$
be its residue field (which is finite of characteristic $\ell$).
In modular representation theory, a triplet such as $(\KM,\OM,\FM)$
is called an $\ell$-modular system. The letter $\EM$ will often be used
to denote either of these three rings.

Let $k$ denote $\FM_q$ or $\ov\FM_p$ (sometimes we will allow $k$ to
be the field $\CM$ of complex numbers instead). We will consider only
separated $k$-schemes of finite type, and morphisms of
$k$-schemes. Such schemes will be called varieties. If $X$ is a
variety, we will say ``$\EM$-sheaves on $X$'' for ``constructible
$\EM$-sheaves on $X$''.  We will denote by $\Sh(X,\EM)$ the noetherian
abelian category of $\EM$-sheaves on $X$, and by $\Loc(X,\EM)$ the
full subcategory of $\EM$-local systems on $X$. If $X$ is connected,
these correspond to the continuous representations of the étale
fundamental group of $X$ at any base point.

Let $D^b_c(X,\EM)$ be the bounded derived category of $\EM$-sheaves as
defined by Deligne.  The category $D^b_c(X,\EM)$ is triangulated, and
endowed with a $t$-structure whose heart is equivalent to the abelian
category of $\EM$-sheaves, because the following condition is
satisfied \cite{DELIGNE, BBD}.
\begin{equation}
\begin{array}{l}
\text{For each finite extension } k' \text{ of } k \text{ contained in } \ov\FM_p,\\
\text{the groups $H^i(\Gal(\ov\FM_p/k'),\ZM/\ell)$, $i\in\NM$, are finite.}
\end{array}
\end{equation}
We call this $t$-structure the \emph{natural} $t$-structure on
$D^b_c(X,\EM)$. The notion of $t$-structure will be recalled in the
next section. For triangulated categories and derived categories, we
refer to \cite{WEIBEL,KS2}.

We have internal operations $\otimes^\LM_\EM$ and $\RHOM$ on $D^b_c(X,\EM)$,
and, if $Y$ is another scheme, for $f : X \to Y$ a morphism we have triangulated functors
\begin{gather*}
f_!,\ f_* : D^b_c(X,\EM) \to D^b_c(Y,\EM)\\
f^*,\ f^! : D^b_c(Y,\EM) \to D^b_c(X,\EM)
\end{gather*}
We omit the letter $R$ which is normally used (\emph{e.g.} $Rf_*$,
$Rf_!$) meaning that we consider derived functors. For the functors
between categories of sheaves, we will use a $0$ superscript, as in
$\0 f_!$ and $\0 f_*$, following \cite{BBD}.

We will denote by 
\[
\DC_{X,\EM} : D^b_c(X,\EM)^\op \to D^b_c(X,\EM)
\]
the dualizing functor $\DC_{X,\EM} (-) = \RHOM(-,a^!\EM)$, where $a:X\to\Spec k$
is the structural morphism.

We have a modular reduction functor
$\FM \otimes^\LM_\OM (-) : D^b_c(X,\OM) \to D^b_c(X,\FM)$, which we will
simply denote by $\FM(-)$.
It is triangulated, and it commutes with the functors 
$f_!,\ f_*,\ f^*,\ f^!$ and the duality.
Moreover, it maps a torsion-free sheaf to a sheaf, and a torsion sheaf to a complex
concentrated in degrees $-1$ and $0$.

By definition, we have $D^b_c(X,\KM) = \KM \otimes_\OM D^b_c(X,\OM)$,
and $\Sh(X,\KM) = \KM \otimes_\OM \Sh(X,\OM)$. The functor
$\KM \otimes_\OM (-) : D^b_c(X,\OM) \to D^b_c(X,\KM)$ is exact.

In this chapter, we are going to recall the construction of the
perverse $t$-structure on $D^b_c(X,\EM)$ for the middle perversity $p$
(with two versions over $\OM$, where we have two perversities $p$ and
$p_+$ exchanged by the duality). We will recall the main points of the
treatment of $t$-structures and recollement of \cite{BBD}, to which we
refer for the details. However, in this work we emphasize the aspects
concerning $\OM$-sheaves, and we give some complements.

Before going through all these general constructions, let us already see
what these perverse sheaves are. They form an abelian full subcategory
$\p \MC(X,\EM)$ of $D^b_c(X,\EM)$. If $\EM$ is $\KM$ or $\FM$,
then this abelian category is artinian and noetherian, and its simple
objects are of the form $j_{!*}(\LC[\dim V)])$, where $j : V \to X$ is
the inclusion of a smooth irreducible subvariety, $\LC$ is an
irreducible locally constant constructible $\EM$-sheaf on $V$, and
$j_{!*}$ the intermediate extension functor. If $\EM = \OM$, the
abelian category is only noetherian.
In any case, $\p \MC(X,\EM)$ is the intersection of the full
subcategories
$\p D^{\leqslant 0}(X,\EM)$ and $\p D^{\geqslant 0}(X,\EM)$
of $D^b_c(X,\EM)$, where, if $A$ is a complex in $\DC^b_c(X,\EM)$, we have
\begin{gather}
\label{def:p <= 0}
A\in \p D^{\leqslant 0}(X,\EM) \Iff
\text{for all points $x$ in $X$, }
\HC^i i_x^* A = 0
\text{ for all } i > - \dim(x)\\
\label{def:p >= 0}
A\in \p D^{\geqslant 0}(X,\EM) \Iff
\text{for all points $x$ in $X$, }
\HC^i i_x^! A = 0
\text{ for all } i < - \dim(x)
\end{gather}
Here the points are not necessarily closed,
$i_x$ is the inclusion of $x$ into $X$, and
$\dim(x) = \dim \ov{\{x\}} = \deg\tr(k(x)/k)$.

The pair $(\p D^{\leqslant 0}, \p D^{\geqslant 0})$ is a $t$-structure
on $D^b_c(X,\EM)$, and $\p \MC(X,\EM)$ is its \emph{heart}.

When $\EM$ is a field (\ie\ $\EM = \KM$ or $\FM$),
the duality $\DC_{X,\EM}$ exchanges
$\p D^{\leqslant 0}(X,\EM)$ and $\p D^{\geqslant 0}(X,\EM)$,
so it induces a self-duality on $\p\MC(X,\EM)$.

However, when $\EM = \OM$, this is no longer true.
The perversity $p$ is no longer self-dual. The duality
exchanges the $t$-structure defined by the middle perversity $p$ with
the $t$-structure
$(\pp D^{\leqslant 0}(X,\OM),\pp D^{\geqslant 0}(X,\OM))$
defined by
\begin{equation}
\label{def:p+ <= 0}
A\in \pp D^{\leqslant 0}(X,\OM) \Iff
\text{for all points $x$ in $X$, }
\begin{cases}
\HC^i i_x^* A = 0 \text{ for all } i > - \dim(x) + 1\\
\HC^{- \dim(x) + 1} i_x^* A \text{ is torsion}
\end{cases}
\end{equation}
\begin{equation}
\label{def:p+ >= 0}
A\in \pp D^{\geqslant 0}(X,\OM) \Iff
\text{for all points $x$ in $X$, }
\begin{cases}
\HC^i i_x^! A = 0 \text{ for all } i < - \dim(x)\\
\HC^{- \dim(x)} i_x^! A \text{ is torsion-free}
\end{cases}
\end{equation}
The definition of torsion (resp. torsion-free) objects is given in
Definition \ref{def:torsion}. 

We say that this $t$-structure is defined by the perversity $p_+$, and
that the duality exchanges $p$ and $p_+$. We denote by $\pp\MC(X,\OM)
= \pp D^{\leqslant 0}(X,\OM) \cap \pp D^{\geqslant 0}(X,\OM)$ the
heart of the $t$-structure defined by $p_+$, and we call its objects
$p_+$-perverse sheaves, or dual perverse sheaves. This abelian
category is only artinian.

The $t$-structures defined by $p$ and $p_+$ determine each other
(see \cite[\S 3.3]{BBD}). We have

\begin{equation}\label{p+ <= 0}
A \in \pp D^{\leqslant 0}(X,\OM) \Iff
A \in \p D^{\leqslant 1}(X,\OM) \text{ and }
\p H^1 A \text{ is torsion}
\end{equation}
\begin{equation}
A \in \pp D^{\geqslant 0}(X,\OM) \Iff
A \in \p D^{\geqslant 0}(X,\OM) \text{ and }
\p H^0 A \text{ is torsion-free}
\end{equation}
\begin{equation}
A \in \p D^{\leqslant 0}(X,\OM) \Iff
A \in \pp D^{\leqslant 0}(X,\OM) \text{ and }
\pp H^0 A \text{ is divisible}
\end{equation}
\begin{equation}
A \in \p D^{\geqslant 0}(X,\OM) \Iff
A \in \pp D^{\geqslant -1}(X,\OM) \text{ and }
\pp H^{-1} A \text{ is torsion}
\end{equation}

If $A$ is $p$-perverse, then it is also $p_+$-perverse if and only if
$A$ is torsion-free in $\p\MC(X,\OM)$. If $A$ is $p_+$-perverse,
then $A$ is also $p$-perverse if and only if $A$ is divisible in
$\pp\MC(X,\OM)$.
Thus, if $A$ is both $p$- and $p_+$-perverse, then $A$ is without
torsion in $\p\MC(X,\OM)$ and divisible in $\pp\MC(X,\OM)$.

In the next sections, we will recall
why $(\p  D^{\leqslant 0}, \p  D^{\geqslant 0})$ (resp. the two versions
with $p$ and $p_+$ if $\EM = \OM$) is indeed a $t$-structure on
$D^b_c(X,\EM)$. We refer to \cite{BBD} for more details, however their treatment
of the case $\EM = \OM$ is quite brief, so we give some complements.
The rest of the chapter is organized as follows.

First, we recall the definition of $t$-categories and their main
properties.  Then we see how they can be combined with torsion
theories.  Afterwards, we recall the notion of recollement of
$t$-categories, stressing on some important properties, such as the
construction of the perverse extensions $\p j_!$, $\p j_{!*}$ and $\p
j_*$ with functors of truncation on the closed part. Then again, we study the
connection with torsion theories.  Already at this point, we have six
possible extensions (the three just mentioned, in the two versions $p$
and $p_+$). We also study the heads and socles of the extensions
$\p j_!$, $\p j_{!*}$ and $\p j_*$, and show that the intermediate
extension preserves decomposition numbers.

Then we see how those constructions show that the definitions of the
last section give indeed $t$-structures on the triangulated
categories $D^b_c(X,\EM)$, first fixing a stratification, and then
taking the limit.  Then, we stick to this case, where we have functors
$\KM \otimes^\LM_\OM (-)$ and $\FM \otimes^\LM_\OM (-)$ (we did not
try to axiomatize this setting), and we study the connection between
modular reduction and truncation.  If we take a complex $A$ over
$\OM$, for each degree we have three places where we can truncate its
reduction modulo $\varpi$, because $\HC^i(\FM A)$ has pieces coming from
$\HC^i_\tors(A)$, $\HC^i_\free(A)$ and $\HC^{i + 1}_\tors(A)$.  So, in
a recollement situation, we have 9 possible truncations.

Finally, we introduce decomposition numbers for perverse sheaves, and
particularly in the $G$-equivariant setting. We have in mind
$G$-equivariant perverse sheaves on the nilpotent variety.

The relation between modular reduction and truncation is really one of
the main technical points of this thesis. For example, the fact that
the modular reduction does not commute with the intermediate extension
means that the reduction of a simple perverse sheaf will not
necessarily be simple, that is, that we have can have non-trivial
decomposition numbers.

\section{$t$-categories}

Let us begin by recalling the definition of a $t$-structure on a triangulated
category.

\begin{defi}\label{def:ts}
A $t$-category is a triangulated category $\DC$, endowed with two strictly
full subcategories $\DC^{\leqslant 0}$ and $\DC^{\geqslant 0}$, such that,
if we let $\DC^{\leqslant n} = \DC^{\leqslant 0}[-n]$ and
$\DC^{\geqslant n} = \DC^{\geqslant 0}[-n]$, we have
\begin{enumerate}[(i)]
\item For $X$ in $\DC^{\leqslant 0}$ and $Y$ in $\DC^{\geqslant 1}$,
we have $\Hom_\DC(X,Y) = 0$.

\item $\DC^{\leqslant 0} \subset \DC^{\leqslant 1}$ and
$\DC^{\geqslant 0} \supset \DC^{\geqslant 1}$.

\item For each $X$ in $\DC$, there is a distinguished triangle $(A,X,B)$
in $\DC$ with $A$ in $\DC^{\leqslant 0}$ and $B$ in $\DC^{\geqslant 1}$.
\end{enumerate}

We also say that $(\DC^{\leqslant 0},\DC^{\geqslant 0})$ is a $t$-structure on $\DC$.
Its \emph{heart} is the full subcategory $\CC := \DC^{\leqslant 0} \cap \DC^{\geqslant 0}$.
\end{defi}

Let $\DC$ be a $t$-category.

\begin{prop}
\begin{enumerate}[(i)]
\item The inclusion of $\DC^{\leqslant n}$ (resp. $\DC^{\geqslant n}$) in $\DC$
has a right adjoint $\t_{\leqslant n}$ (resp. a left adjoint $\t_{\geqslant n}$).

\item For all $X$ in $\DC$, there is a unique
$d \in \Hom(\t_{\geqslant 1} X, \t_{\leqslant 0} X [1])$
such that the triangle
\[
\t_{\leqslant 0} X \longto X \longto \t_{\geqslant 1} X \elem{d}
\]
is distinguished. Up to unique isomorphism, this is the unique triangle
$(A,X,B)$ with $A$ in $\DC^{\leqslant 0}$ and $B$ in
$\DC^{\geqslant 1}$.

\item Let $a \leqslant b$. Then, for any $X$ in $\DC$, there is a
  unique morphism
  $\t_{\geqslant a} \t_{\leqslant b} X \to \t_{\leqslant b} \t_{\geqslant a} X$
  such that the following diagram is commutative.
\[
\xymatrix{
\t_{\leqslant b} X \ar[r] \ar[d] & X \ar[r]& \t_{\geqslant a} X \\
\t_{\geqslant a} \t_{\leqslant b} X \ar[rr]^\sim&& \t_{\leqslant b}
\t_{\geqslant a} X \ar[u]
}
\]
It is an isomorphism.
\end{enumerate}
\end{prop}

For example, if $\AC$ is an abelian category and $\DC$ is its derived category,
the natural $t$-structure on $\DC$ is the one for which $\DC^{\leqslant n}$ 
(resp. $\DC^{\geqslant n}$)
is the full subcategory of the complexes $K$ such that $H^i K = 0$ for
$i > n$ (resp. $i < n$). For $K = (K^i, d^i : K^i \to K^{i + 1})$ in $\DC$,
the truncated complex $\t_{\leqslant n} K$ is the subcomplex
$\cdots \to K^{n - 1} \to \Ker d^n \to 0 \to \cdots$ of $K$.
The heart is equivalent to the abelian category $\AC$ we started with.
Note that, in this case, the cone of a morphism $f : A \to B$ between
two objects of $\AC$ is a complex concentrated in degrees $-1$ and $0$.
More precisely, we have $H^{-1}(\Cone f) \simeq \Ker f$ and
$H^0(\Cone f) \simeq \Coker f$. In particular, we have a triangle
$(\Ker f [1],\ \Cone f,\ \Coker f)$.

If we abstract the relations between $\AC$ and $\DC(\AC)$, we get the notion
of admissible abelian subcategory of a triangulated category $\DC$, and a $t$-structure
on $\DC$ precisely provides an admissible abelian subcategory by taking the heart.

More precisely, let $\DC$ be a triangulated category and $\CC$ a full subcategory
of $\DC$ such that $\Hom^i(A,B) := \Hom(A,B[i])$ is zero for $i < 0$ and $A,B$ in $\CC$.
We have the following proposition, which results from the octahedron axiom.

\begin{prop}
Let $f : X \to Y$ in $\CC$. We can complete $f$ into a distinguished triangle
$(X,Y,S)$. Suppose $S$ is in a distinguished triangle $(N[1],S,C)$ with $N$ and
$C$ in $\CC$. Then the morphisms $N \to S[-1] \to X$ and $Y \to S \to C$,
obtained by composition from the morphisms in the two triangles above, are respectively
a kernel and a cokernel for the morphism $f$ in $\CC$.
\end{prop}

Such a morphism will be called $\CC$-\emph{admissible}. In a distinguished triangle 
$X\elem{f} Y\elem{g} Z \elem{d}$ on objects in $\CC$, the morphisms
$f$ and $g$ are admissible, $f$ is a kernel of $g$, $g$ is a cokernel of $f$,
and $d$ is uniquely determined by $f$ and $g$. A short exact sequence
in $\CC$ will be called \emph{admissible} if it can be obtained
from a distinguished triangle in $\DC$ by suppressing the degree one morphism.

\begin{prop}
Suppose $\CC$ is stable by finite direct sums. Then the following conditions are equivalent.
\begin{enumerate}[(i)]
\item $\CC$ is abelian, and its short exact sequences are admissible.

\item Every morphism of $\CC$ is $\CC$-admissible.
\end{enumerate}
\end{prop}

A full abelian $\CC$ subcategory of $\DC$, such that $\Hom^{-1}_\DC(\CC,\CC) = 0$,
satisfying the equivalent conditions of the proposition, is called admissible.
We will now see that $t$-structures provide admissible
abelian subcategories.

\begin{theo}
The heart $\CC$ of a $t$-category $\DC$ is an admissible abelian subcategory
of $\DC$, stable by extensions.
The functor
$H^0 := \t_{\geqslant 0} \t_{\leqslant 0} \simeq \t_{\leqslant 0}
\t_{\geqslant 0} : \DC \to \CC$ is
a cohomological functor.
\end{theo}

Let $\DC_i$ ($i = 1,2$) be two $t$-categories, and let $\e_i : \CC_i \to \DC_i$
denote the inclusion functors of their hearts. Let $T : \DC_1 \to \DC_2$
be a triangulated functor. Then we say that $T$ is
right $t$-exact if $T(\DC_1^{\leqslant 0}) \subset \DC_2^{\leqslant 0}$,
left $t$-exact if $T(\DC_1^{\geqslant 0}) \subset \DC_2^{\geqslant 0}$,
and $t$-exact if it is both left and right exact.

\begin{prop}\label{prop:pT}
\begin{enumerate}
\item If $T$ is left (resp. right) $t$-exact, then the additive functor
$\p T := H^0 \circ T \circ \e_1$ is left (resp. right) exact.

\item Let $(T^*,T_*)$ be a pair of adjoint triangulated functors, with
$T^* : \DC_2 \to \DC_1$ and $T_* : \DC_1 \to \DC_2$.
Then $T^*$ is right $t$-exact if and only if $T_*$ is left $t$-exact,
and in that case $(\p T^*,\p T_*)$ is a pair of adjoint functors between
$\CC_1$ and $\CC_2$.
\end{enumerate}
\end{prop}

\section{Torsion theories and $t$-structures}\label{sec:tt ts}

\begin{defi}\label{def:tt}
Let $\AC$ be an abelian category. A torsion theory on $\AC$ is a pair
$(\TC, \FC)$ of full subcategories such that
\begin{enumerate}[(i)]
\item for all objects $T$ in $\TC$ and $F$ in $\FC$, we have
\begin{equation}\label{eq:hom torsion theory}
\Hom_\AC(T,F) = 0
\end{equation}
\item for any object $A$ in $\AC$, there are objects $T$ in $\TC$ and $F$ in $\FC$
such that there is a short exact sequence
\begin{equation}\label{eq:ses torsion theory}
0 \longto T \longto A \longto F \longto 0
\end{equation}
\end{enumerate}
\end{defi}

Then the short exact sequence \ref{eq:ses torsion theory} is functorial. We obtain
functors $(-)_\tors : \AC \to \TC$ and $(-)_\free : \AC \to \FC$.

Examples of torsion theories arise with $\OM$-linear abelian categories.

\begin{defi}\label{def:torsion}
Let $\AC$ be an $\OM$-linear abelian category. An object $A$ in $\AC$ is \emph{torsion}
if $\varpi^N 1_A$ is zero for some $N \in \NM$, and it is
\emph{torsion-free} (resp. \emph{divisible}) if $\varpi.1_A$ is a monomorphism
(resp. an epimorphism).
\end{defi}

\begin{prop}\label{prop:hom tors div}
Let $\AC$ be an $\OM$-linear abelian category.
\begin{enumerate}[(i)]
\item If $T \in \AC$ is torsion and $F \in \AC$ is torsion-free, then we have
\begin{equation}
\Hom_\AC(T,F) = 0
\end{equation}
\item If $Q \in \AC$ is divisible and $T \in \AC$ is torsion, then we have
\begin{equation}
\Hom_\AC(Q,T) = 0
\end{equation}
\end{enumerate}
\end{prop}

\begin{proof}
\emph{(i)}\ Let $f\in \Hom_\AC(T,F)$.
Let $N \in \NM$ such that $\varpi^N.1_T = 0$. Then we have
$(\varpi^N.1_F) f = f (\varpi^N.1_T) = 0$, and consequently $f = 0$,
since $\varpi^N.1_F$ is a monomorphism.

\medskip

\emph{(ii)}\ Let $g\in \Hom_\AC(Q,T)$.
Let $N \in \NM$ such that $\varpi^N.1_T = 0$. Then we have
$g (\varpi^N.1_Q) = (\varpi^N.1_T) g = 0$, and consequently $g = 0$,
since $\varpi^N.1_Q$ is an epimorphism.
\end{proof}

\begin{prop}\label{prop:tors div}
Let $A$ be an object in $\AC$.
\begin{enumerate}
\item If $A$ is noetherian, then $A$ has a greatest torsion subobject
$A_\tors$, the quotient $A/A_\tors$ has no torsion
and $\KM A \simeq \KM A/A_\tors$.

\item If $A$ is artinian, then $A$ has a greatest divisible
subobject $A_\di$, the quotient $A/A_\di$ is a torsion object
and we have $\KM A \simeq \KM A_\di$.
\end{enumerate}
\end{prop}

\begin{proof}
In the first case,
the increasing sequence of subobjects $\Ker \varpi^n.1_A$ must
stabilize, so there is an integer $N$ such that
$\Ker \varpi^n.1_A = \Ker \varpi^N.1_A$ for all $n \geqslant N$.
We set $A_\tors := \Ker \varpi^N.1_A$. This is clearly a torsion object,
since it is killed by $\varpi^N$. Now let $T$ be a torsion subobject
of $A$. It is killed by some $\varpi^k$, and we can assume
$k \geqslant N$. Thus $T \subset \Ker \varpi^k.1_A = \Ker \varpi^N.1_A = A_\tors$.
This shows that $A_\tors$ is the greatest torsion subobject of $A$.
We have $\Ker \varpi.1_{A/A_\tors} = \Ker \varpi^{N+1}.1_A / \Ker \varpi^N.1_A = 0$
which shows that $A/A_\tors$ is torsion-free. Applying the exact functor
$\KM \otimes_\OM -$ to the short exact sequence
$0 \to A_\tors \to A \to A/A_\tors \to 0$,
we get $\KM A \simeq \KM A/A_\tors$.

In the second case, the decreasing sequence of subobjects $\im \varpi^n.1_A$
must stabilize, so there is an integer $N$ such that
$\im \varpi^n.1_A = \im \varpi^N.1_A$ for all $n \geqslant N$.
We set $A_\di := \im \varpi^N.1_A$. We have
$\im \varpi.1_{A_\di} = \im \varpi^{N+1}.1_A = \im \varpi^N.1_A = A_\di$,
thus $A_\di$ is divisible. We have
$\im \varpi^n.1_{A/A_\di} = \im \varpi^n.1_A / \im \varpi^N.1_A = 0$
for $n \geqslant N$. Hence $A/A_\di$ is a torsion object.
Applying the exact functor
$\KM \otimes_\OM -$ to the short exact sequence
$0 \to A_\di \to A \to A/A_\di \to 0$,
we get $\KM A_\di \simeq \KM A$.
\end{proof}

\begin{prop}\label{prop:torsion theory tors div}
Let $\AC$ be an $\OM$-linear abelian category.
We denote by $\TC$ (resp. $\FC$, $\QC$) the full subcategory
of torsion (resp. torsion-free, divisible) objects in $\AC$.
If $\AC$ is noetherian (resp. artinian), then $(\TC, \FC)$ (resp. $(\QC,\TC)$)
is a torsion theory on $\AC$.
\end{prop}

\begin{proof}
This follows from Propositions \ref{prop:hom tors div} and \ref{prop:tors div}
\end{proof}

We want to discuss the combination of $t$-structures with torsion theories.

\begin{prop}\label{prop:tt ts}
Let $\DC$ be a triangulated category with a $t$-structure
$(\p \DC^{\leqslant 0}, \p \DC^{\geqslant 0})$, with heart $\CC$,
truncation functors $\p \t_{\leqslant i}$ and $\p \t_{\geqslant i}$,
and cohomology functors $\p H^i : \DC \to \CC$, and suppose
that $\CC$ is endowed with a torsion theory $(\TC,\FC)$.
Then we can define a new $t$-structure
$(\pp \DC^{\leqslant 0}, \pp \DC^{\geqslant 0})$ on $\DC$ by
\[
\pp \DC^{\leqslant 0} = \{A \in \p \DC^{\leqslant 1} \mid  \p H^1(A) \in \TC\}
\]
\[
\pp \DC^{\geqslant 0} = \{A \in \p \DC^{\geqslant 0} \mid  \p H^0(A) \in \FC\}
\]
\end{prop}

\begin{proof}
Let us check the three axioms for $t$-structures given in Definition \ref{def:ts}.

\medskip

(i)  Let $A \in \pp \DC^{\leqslant 0}$ and $B \in \pp \DC^{\geqslant 1}$.
Then we have
\[
\begin{array}{rcll}
\Hom_\DC(A,B)
&=& \Hom_\DC(\p \t_{\geqslant 1} A, \p \t_{\leqslant 1} B)
&\text{since $A \in \p \DC^{\leqslant 1}$ and $B \in \p \DC^{\geqslant 1}$}
\\
&=& \Hom_\CC(\p H^1 A, \p H^1 B) = 0
&\text{by (\ref{eq:hom torsion theory}), since $\p H^1 A \in \TC$ and $\p H^1 B \in \FC$}
\end{array}
\]

\medskip

(ii) We have
$\pp \DC^{\leqslant 0} \subset \p \DC^{\leqslant 1} \subset \pp \DC^{\leqslant 1}$
and 
$\pp \DC^{\geqslant 0} \supset \p \DC^{\geqslant 1} \supset \pp \DC^{\geqslant 1}$.

\medskip

(iii) Let $A \in \DC$. By (\ref{eq:ses torsion theory}),
there are objects $T \in \TC$ and $F \in \FC$ such that we have
a short exact sequence
\[
0 \longto T \longto \p H^1 A \longto F \longto 0
\]
By \cite[Proposition 1.3.15]{BBD} there is a distinguished triangle
\[
A' \elem{a} A \elem{b} A'' \elem{d} A'[1]
\]
such that $A' \in \p \DC^{\leqslant 1}$
and $A'' \in \p \DC^{\geqslant 1}$,  $\p H^1 A' \simeq T$ and
$\p H^1 A'' \simeq F$, and thus $A' \in \pp \DC^{\leqslant 0}$
and $A'' \in \pp \DC^{\geqslant 1}$.
\end{proof}

We denote by $\CC^+$ the heart of this new $t$-structure, by
$\pp H^i : \DC \to \CC^+$ the new cohomology functors, and by
$\pp \t_{\leqslant i}$, $\pp \t_{\geqslant i}$ the new truncation functors.

We may also use the following notation. For the notions attached to the initial
$t$-structure, we may drop all the $p$, and for the new $t$-structure one
may write $i_+$ instead of $i$, as follows: $(\DC^{\leqslant i_+}, \DC^{\geqslant i_+})$,
$H^{i_+}$, $\t_{\leqslant i_+}$, $\t_{\geqslant i_+}$.

Note that $\CC^+$ is endowed with a torsion theory, namely $(\FC,\TC[-1])$.
We can do the same construction, and we find that $\CC^{++} = \CC [-1]$.
We recover the usual shift of $t$-structures.

By definition, we have functorial distinguished triangles
\begin{equation}\label{eq:tri tors}
\t_{\leqslant i} \longto \t_{\leqslant i_+} \longto H^{i + 1}_\tors (-) [- i - 1]
\end{equation}
and
\begin{equation}\label{eq:tri free}
\t_{\leqslant i_+} \longto \t_{\leqslant i + 1} \longto H^{i + 1}_\free (-) [- i - 1]
\end{equation}

\begin{example}
If $\DC$ is an $\OM$-linear triangulated category, then its heart $\CC$ is also $\OM$-linear.
If $\CC$ is noetherian (resp. artinian), then it is naturally endowed with a torsion theory
by Proposition \ref{prop:torsion theory tors div}, and the preceding considerations apply.
\end{example}

\section{Recollement}\label{sec:recollement}

The recollement (gluing) construction consists roughly in a way to construct
a $t$-structure on some derived category of sheaves on a topological space
(or a ringed topos) $X$, given $t$-structures on derived categories of sheaves on
$U$ and on $F$, where $j : U \to X$ is an open subset of $X$, and
$i : F \to X$ its closed complement. This can be done in a very general axiomatic
framework \cite[\S 1.4]{BBD}, which can be applied to both the complex topology
and the étale topology. The axioms can even be applied to non-topological
situations, for example for representations of algebras. Let us recall the definitions 
and main properties of the recollement procedure.

So let $\DC$, $\DC_U$ and $\DC_F$ be three triangulated categories, and let
$i_* : \DC_F \to \DC$ and $j^* : \DC \to \DC_U$ be triangulated functors.
It is convenient to set $i_! = i_*$ and $j^! = j^*$. We assume that
the following conditions are satisfied.

\begin{ass}\label{ass:recollement}
\begin{enumerate}[(i)]
\item $i_*$ has triangulated left and right adjoints, denoted by $i^*$ and
  $i^!$ respectively.

\item $j^*$ has triangulated left and right adjoints, denoted by $j_!$
  and $j_*$ respectively.

\item We have $j^* i_* = 0$. By adjunction, we also have $i^* j_! = 0$
  and $i^! j_* = 0$ and, for $A$ in $\DC_F$ and $B$ in $\DC_U$, we
  have
\[
\Hom(j_! B, i_* A) = 0 \text{ and } \Hom(i_* A, j_* B) = 0
\]

\item For all $K$ in $\DC$, there exists $d : i_* i^* K \to j_! j^* K
  [1]$ (resp. $d : j_* j^* K \to i_* i^! K [1]$), necessarily unique,
  such that the triangle $j_! j^* K \to K \to i_* i^* K \map{d}$
  (resp. $i_* i^! K \to K \to j_* j^* K \map{d}$) is distinguished.

\item The functors $i_*$, $j_!$ and $j_*$ are fully faithful: the
  adjunction morphisms $i^* i_* \to \id \to i^! i_*$ and $j^* j_* \to
  \id \to j^* j_!$ are isomorphisms.
\end{enumerate}
\end{ass}

Whenever we have a diagram
\begin{equation}\label{eq:recollement setup}
\xymatrix{
\DC_F
\ar[r]^{i_*}
&
\DC
\ar@<-3ex>[l]_{i^*}
\ar@<3ex>[l]_{i^!}
\ar[r]^{j^*}
&
\DC_U
\ar@<-3ex>[l]_{j_!}
\ar@<3ex>[l]_{j_*}
}
\end{equation}
such that the preceding conditions are satisfied, we say that we are
in a situation of recollement.

Note that for each recollement situation, there is a dual recollement
situation on the opposite derived categories. Recall that the opposite
category of a triangulated category $\TC$ is also triangulated, with
translation functor $[-1]$, and distinguished triangles the triangles
$(Z,Y,X)$, where $(X,Y,Z)$ is a distinguished triangle in $\TC$.
One can check that the conditions in \ref{ass:recollement} are satisfied
for the following diagram, where the roles of $i^*$ and $i^!$ (resp.
$j_!$ and $j_*$) have been exchanged.

\begin{equation}\label{eq:dual recollement setup}
\xymatrix{
\DC_F^\op
\ar[r]^{i_*}
&
\DC^\op
\ar@<-3ex>[l]_{i^!}
\ar@<3ex>[l]_{i^*}
\ar[r]^{j^*}
&
\DC_U^\op
\ar@<-3ex>[l]_{j_*}
\ar@<3ex>[l]_{j_!}
}
\end{equation}
We can say that there is a ``formal duality'' in the axioms of a
recollement situation, exchanging the symbols $!$ and $*$. Note that,
in the case of $D^b_c(X,\EM)$, the duality $\DC_{X,\EM}$ really
exchanges these functors.

If $\UC \map{u} \TC \map{q} \VC$ is a sequence of triangulated
functors between triangulated categories such that $u$ identifies
$\UC$ with a thick subcategory of $\TC$, and $q$ identifies $\VC$ with
the quotient of $\TC$ by the thick subcategory $u(\UC)$, then we say
that the sequence $0 \to \UC \map{u} \TC \map{q} \VC \to 0$ is exact.

\begin{prop}
The sequences
\begin{gather*}
0 \longfrom \DC_\FC \leftelem{i^*} \DC \leftelem{j_!} \DC_U \longfrom 0\\
0 \longto \DC_\FC \elem{i_*} \DC \elem{j^*} \DC_U \longto 0\\
0 \longfrom \DC_\FC \leftelem{i^!} \DC \leftelem{j_*} \DC_U \longfrom 0
\end{gather*}
are exact.
\end{prop}

Suppose we are given a $t$-structure
$(\DC_U^{\leqslant 0},\DC_U^{\geqslant 0})$ on $\DC_U$,
and a $t$-structure
$(\DC_F^{\leqslant 0},\DC_F^{\geqslant 0})$ on $\DC_F$.
Let us define
\begin{gather}
\label{eq:recollement def <= 0}
\DC^{\leqslant 0} := \{K \in \DC \mid j^*K \in \DC_U^{\leqslant 0}
\text{ and } i^*K \in \DC_F^{\leqslant 0} \} \\
\label{eq:recollement def >= 0}
\DC^{\geqslant 0} := \{K \in \DC \mid j^*K \in \DC_U^{\geqslant 0}
\text{ and } i^!K \in \DC_F^{\geqslant 0} \}
\end{gather}

\begin{theo}\label{th:recollement}
With the preceding notations, $(\DC^{\leqslant 0}, \DC^{\geqslant 0})$
is a $t$-structure on $\DC$.
\end{theo}

We say that it is obtained from those on
$\DC_U$ and $\DC_F$ by \emph{recollement} (gluing).

Now suppose we are just given a $t$-structure on $\DC_F$. Then we can
apply the recollement procedure to the degenerate $t$-structure
$(\DC_U,0)$ on $\DC_U$ and to the given $t$-structure on $\DC_F$. The
functors $\t_{\leqslant p}$ ($p \in \ZM$) relative to the $t$-structure obtained on
$\DC$ will be denoted $\t^F_{\leqslant p}$. The functor
$\t^F_{\leqslant p}$ is right adjoint to the inclusion of the full
subcategory of $\DC$ whose objects are the $X$ such that $i^* X$ is in
$\DC_F^{\leqslant p}$. We have a distinguished triangle
$(\t^F_{\leqslant p} X, X, i_* \t_{> p} i^* X)$. The $H^p$ cohomology
functors for this $t$-structure are the $i_* H^p i^*$. 

Dually, one can define the functor $\t^F_{\geqslant p}$ using the
degenerate $t$-structure $(0,\DC_U)$ on $\DC_U$. It is left adjoint to
the inclusion of $\{X \in \DC \mid i^! X \in \DC_F^{\geqslant p}\}$ in
$\DC$, we have distinguished triangles
$(i_* \t_{< p} i^! X, X, \t^F_{\geqslant p} X)$,
and the $H^p$ are the $i_* H^p i^!$.

Similarly, if we are just given a $t$-structure on $\DC_U$, and if we
endow $\DC_F$ with the degenerate $t$-structure $(\DC_F, 0)$
(resp. $(0, \DC_F)$), we can define a $t$-structure on $\DC$ for which
the functors $\t_{\leqslant p}$ (resp. $\t_{\geqslant p}$), denoted by 
$\t^U_{\leqslant p}$ (resp. $\t^U_{\geqslant p}$), yield distinguished
triangles $(\t^U_{\leqslant p}, X, j_* \t_{> p} j^* X)$
(resp. $(j_! \t_{< p} j^* X, X, \t^U_{\geqslant p} X)$),
and for which the $H^p$ functors are the $j_* H^p j^*$
(resp. $j_! H^p j^*$).

Moreover, we have
\begin{equation}
\t_{\leqslant p} = \t_{\leqslant p}^F \t_{\leqslant p}^U
\text{ and }
\t_{\geqslant p} = \t_{\geqslant p}^F \t_{\geqslant p}^U
\end{equation}

An \emph{extension} of an object $Y$ of $\DC_U$ is an object $X$ of
$\DC$ endowed with an isomorphism $j^* X \isom Y$. Such an isomorphism
induces morphisms $j_! Y \to X \to j_* Y$ by adjunction. If an
extension $X$ of $Y$ is isomorphic, as an extension, to
$\t_{\geqslant p}^F j_! Y$ (resp. $\t_{\leqslant p}^F j_* Y$), then
the isomorphism is unique, and we just write $X = \t_{\geqslant p}^F j_! Y$
(resp. $\t_{\leqslant p}^F j_* Y$).

\begin{prop}
Let $Y$ in $\DC_U$ and $p$ an integer. There is, up to unique
isomorphism, a unique extension $X$ of $Y$ such that $i^* X$ is in
$\DC_F^{\leqslant p - 1}$ and
$i^! X$ is in $\DC_F^{\geqslant p + 1}$. It is
$\t^F_{\leqslant p - 1} j_* Y$, and this extension of $Y$ is
canonically isomorphic to $\t^F_{\geqslant p + 1} j_! Y$.
\end{prop}

Let $\DC_m$ be the full subcategory of $\DC$ consisting in the objects
$X$ such that $i^* X \in \DC_F^{\leqslant p - 1}$ and
$i^! X \in \DC_F^{\leqslant p + 1}$. The functor $j^*$ induces an
equivalence $\DC_m \to \DC_U$, with quasi-inverse
$\t^F_{\leqslant p - 1} j_* = \t^F_{\geqslant p + 1} j_!$, which will
be denoted $j_{!*}$.

Let $\CC$, $\CC_U$ and $\CC_F$ denote the hearts of the $t$-categories
$\DC$, $\DC_U$ and $\DC_F$. We will use the notation $\p T$ of
Proposition \ref{prop:pT}, where $T$ is one of the functors of the
recollement diagram \ref{eq:recollement setup}.
By definition of the $t$-structure of $\DC$, $j^*$ is $t$-exact, $i^*$
is right $t$-exact, and $i^!$ is left $t$-exact. Applying
Proposition \ref{prop:pT}, we get

\begin{prop}
\begin{enumerate}[(i)]
\item The functors $j_!$ and $i^*$ are right $t$-exact, the functors
$j^*$ and $i_*$ are $t$-exact, and the functors $j_*$ and $i^!$ are
  left $t$-exact.

\item $(\p j_!, \p j^*, \p j_*)$ and $(\p i^*, \p i_*, \p i^!)$ form
  two sequences of adjoint functors. 
\end{enumerate}
\end{prop}

\begin{prop}
\begin{enumerate}[(i)]
\item
The compositions $\p j^* \p i_*$, $\p i^* \p j_!$ and $\p i^! \p j_*$ are
zero. For $A$ in $\CC_F$ and $B$ in $\CC_U$, we have
\[
\Hom(\p j_! B, \p i_* A) = \Hom(\p i_* A, \p j_* B) = 0
\]

\item For any object $A$ in $\CC$, we have exact sequences
\begin{gather}
\p j_! \p j^* A \longto A \longto \p i_* \p i^* A \longto 0\\
0 \longto \p i_* \p i^! A \longto A \longto \p j_* \p j^* A
\end{gather}

\item If we identify $\CC_F$ with its essential image by the fully
  faithful functor $\p i_*$, which is a thick subcategory of $\CC$,
  then for any object $A$ in $\CC$, $\p i^* A$ is the largest quotient
  of $A$ in $\CC_F$, and $\p i^! A$ is the largest subobject of $A$ in $\CC_F$.
\end{enumerate}

\item The functor $\p j^*$ identifies $\CC_U$ with the quotient of
  $\CC$ by the thick subcategory $\CC_F$.

\item For any object $A$ in $\CC$, we have exact sequences
\begin{gather}
0 \longto \p i_* H^{-1} i^* A \longto \p j_! \p j^* A \longto A
\longto \p i_* \p i^* A \longto 0\\
0 \longto \p i_* \p i^! A \longto A \longto \p j_* \p j^* A \longto \p
i_* H^1 i^! A \longto 0
\end{gather}
\end{prop}

Since $j^*$ is a quotient functor of triangulated categories, the
composition of the adjunction morphisms $j_! j^* \to \id \to j_* j^*$
comes from a unique morphism of functors $j_! \to j_*$. Applying
$j^*$, we get the identity automorphism of the identity functor.

Similarly, since the functor $\p j^*$ is a quotient functor of abelian
categories, the composition of the adjunction morphisms $\p j_! \p j^*
\to \id \to \p j_* \p j^*$ comes from a unique morphism of functors
$\p j_! \to \p j_*$. Applying $\p j^*$, we get the identity automorphism
of the identity functor.

Let $\p j_{!*}$ be the image of $\p j_!$ in $\p j_*$. We have a factorization
\begin{equation}
j_! \longto \p j_! \longto \p j_{!*} \longto \p j_* \longto j_*
\end{equation}

\begin{prop}
We have
\begin{gather}
\p j_! = \t^F_{\geqslant 0}\ j_! = \t^F_{\leqslant -2}\ j_*\\
\p j_{!*} = \t^F_{\geqslant 1}\ j_! = \t^F_{\leqslant -1}\ j_*\\
\p j_* = \t^F_{\geqslant 2}\ j_! = \t^F_{\leqslant 0}\ j_*
\end{gather}
\end{prop}

For $A$ in $\CC$, the kernel and cokernel of $\p j_! A \to \p j_* A$
are in $\CC_F$. More precisely, we have the following Yoneda splice of
two short exact sequences.

\begin{equation}
\xymatrix{
0\ar[r] &
i_* H^{-1} i^* j_* A \ar[r] &
\p j_! A \ar[dr] \ar[rr] &&
\p j_* A \ar[r] &
i_* H^0 i^* j_* A \ar[r] &
0\\
&&& \p j_{!*} A \ar[ur] \ar[dr]\\
&& 0 \ar[ur] && 0
}
\end{equation}

\begin{cor}
For $A$ in $\CC_U$, $j_{!*} A$ is the unique extension $X$ of $A$ in
$\DC$ such that $i^* X$ is in $\DC_F^{\leqslant -1}$ and $i^! X$ is in
$\DC_F^{\geqslant 1}$. Thus it is the unique extension of $A$ in $\CC$
with no non-trivial subobject or quotient in $\CC_F$.

Similarly, $\p j_! A$ (resp. $\p j_* A$) is the unique extension $X$
of $A$ in $\DC$ such that $i^* X$ is in $\DC_F^{\leqslant -2}$
(resp. $\DC_F^{\leqslant 0}$) and $i^! X$ is in $\DC_F^{\geqslant 0}$
(resp. $\DC_F^{\geqslant 2}$). In particular, $\p j_! A$ (resp. $\p
j_* A$) has no non-trivial quotient (resp. subobject) in $\CC_F$.
\end{cor}

\begin{prop}
The simple objects in $\CC$ are the $\p i_* S$, with $S$ simple in
$\CC_F$, and the $j_{!*} S$, for $S$ simple in $\CC_U$.
\end{prop}

\section{Torsion theories and recollement}\label{sec:tt recollement}

Suppose we are in a recollement situation, and that we are given
torsion theories $(\TC_F, \FC_F)$ and
$(\TC_U, \FC_U)$ of $\CC_F$ and $\CC_U$. Then we can define a torsion theory
$(\TC, \FC)$ on $\CC$ by
\begin{gather}
\TC = \{ K \in \CC \mid \p i^* K \in \TC_F \text{ and } j^* K \in \TC_U \}\\
\FC = \{ K \in \CC \mid \p i^! K \in \FC_F \text{ and } j^* K \in \FC_U \}
\end{gather}
Using these torsion theories on $\CC$, $\CC_F$ and $\CC_U$, one can
define new $t$-structures on $\DC$, $\DC_F$ and $\DC_U$, with the
superscript $p_+$. Then the new $t$-structure on $\DC$ is obtained by
recollement from the new $t$-structures on $\DC_F$ and $\DC_U$.

Moreover, we have six interesting functors from $\CC_U \cap \CC_U^+$ to $\DC$
\begin{equation}
\p j_!
= \p\t^F_{\leqslant -2}\ j_* 
= \p\t^F_{\geqslant 0}\ j_!
\end{equation}
\begin{equation}
\pp j_!
= \p\t^F_{\leqslant -2_+}\ j_*
= \p\t^F_{\geqslant 0_+}\ j_!
\end{equation}
\begin{equation}
\p j_{!*}
= \p\t^F_{\leqslant -1}\ j_*
= \p\t^F_{\geqslant 1}\ j_!
\end{equation}
\begin{equation}
\pp j_{!*}
= \p\t^F_{\leqslant -1_+}\ j_*
= \p\t^F_{\geqslant 1_+}\ j_!
\end{equation}
\begin{equation}
\p j_*
= \p\t^F_{\leqslant 0}\ j_*
= \p\t^F_{\geqslant 2}\ j_!
\end{equation}
\begin{equation}
\pp j_*
= \p\t^F_{\leqslant 0_+}\ j_*
= \p\t^F_{\geqslant 2_+}\ j_!
\end{equation}

The first of these functors has image in $\CC$, the last one in $\CC^+$,
and the other four in $\CC \cap \CC^+$.

By \ref{eq:tri tors} and \ref{eq:tri free} and the description above,
we have functorial triangles
\begin{equation}
\p j_! \longto \pp j_! \longto \p i_* \p H^{-1}_\tors i^* j_* [1] \rtordu
\end{equation}
\begin{equation}
\pp j_! \longto \p j_{!*} \longto \p i_* \p H^{-1}_\free i^* j_* [1] \rtordu
\end{equation}
\begin{equation}
\p j_{!*} \longto \pp j_{!*} \longto \p i_* \p H^0_\tors i^* j_* \rtordu
\end{equation}
\begin{equation}
\pp j_{!*} \longto \p j_* \longto \p i_* \p H^0_\free i^* j_* \rtordu
\end{equation}
\begin{equation}
\p j_* \longto \pp j_* \longto \p i_* \p H^1_\tors i^* j_* [-1] \rtordu
\end{equation}

\section{Complements on perverse extensions}

Assume we are in a recollement situation with $\EM$-linear categories,
where $\EM$ is a field, and that $\CC_F$, $\CC_U$ and $\CC$ are
artinian and noetherian. Moreover, we assume that any simple object in
these three hearts has a finite-dimensional endomorphism algebra.  In
general, the functor $j_{!*}$, which we also denote by $\p j_{!*}$, is
fully faithful and sends a simple on a simple, but is not necessarily
exact. However, we will prove two simple but useful results, the first
one about heads and socles, and the second one which says that the
multiplicities of the simples are preserved by $j_{!*}$.

\subsection{Top and socle of perverse extensions}\label{subsec:top socle}

\begin{prop}\label{prop:top socle}
Let $A$ be an object of $\CC_U$. Then we have
\begin{align}
\Soc\: \p j_* A \simeq \Soc j_{!*} A \simeq j_{!*} \Soc A\\ 
\Top\: \p j_! A \simeq \Top j_{!*} A \simeq j_{!*} \Top A
\end{align}
\end{prop}

\begin{proof}
Let $S$ be a simple object in $\CC$. 
Then either $S \simeq j_{!*} j^* S$ or $S \simeq i_* i^* S$.

Suppose we are in the first case. Then we have
\[
\begin{array}{rcll}
\Hom_{\CC}(S, \Soc \p j_* A)
&\simeq& \Hom_{\CC}(S, \p j_* A)
&\text{because the socle is the largest}\\
&&&\text{semisimple subobject}\\
&\simeq& \Hom_{\CC_U}(j^* S, A)
&\text{by adjunction of the pair $(j^*,\p j_*)$}\\
\\
\Hom_{\CC}(S, \Soc j_{!*}A)
&\simeq& \Hom_{\CC}(j_{!*} j^* S, j_{!*}A)
&\text{by assumption and because the socle}\\
&&&\text{is the largest semisimple subobject}\\
&\simeq& \Hom_{\CC_U}(j^* S, A)
&\text{because $j_{!*}$ is fully faithful}\\
\\
\Hom_{\CC}(S, j_{!*} \Soc A)
&\simeq& \Hom_{\CC}(j_{!*} j^* S, j_{!*} \Soc A)
&\text{by assumption}\\
&\simeq& \Hom_{\CC_U}(j^* S, \Soc A)
&\text{because $j_{!*}$ is fully faithful}\\
&\simeq& \Hom_{\CC_U}(j^* S, A)
&\text{because the socle is the largest}\\
&&&\text{semisimple subobject}\\
&&&\text{and $j^*S$ is simple}\\
\end{array}
\]
Taking dimensions and dividing by $\dim \End_{\CC(U)}(S)$, we find 
\[
[\Soc\: \p j_* A : S]
= [\Soc j_{!*} A : S]
= [j_{!*} \Soc A : S]
= [\Soc A : j^* S]
\]
Similarly,  we have
\[
\begin{array}{rcll}
\Hom_{\CC}(\Top\: \p j_! A, S)
&\simeq& \Hom_{\CC}(\p j_! A, S)
&\text{because the top is the largest}\\
&&&\text{semisimple quotient}\\
&\simeq& \Hom_{\CC_U}(A, j^* S)
&\text{by adjunction of the pair $(\p j_!, j^*)$}\\
\\
\Hom_{\CC}(\Top j_{!*}A, S)
&\simeq& \Hom_{\CC}(j_{!*}A, j_{!*} j^* S)
&\text{by assumption and because the top}\\
&&&\text{is the largest semisimple quotient}\\
&\simeq& \Hom_{\CC_U}(A, j^* S)
&\text{because $j_{!*}$ is fully faithful}\\
\\
\Hom_{\CC}(j_{!*} \Top A, S)
&\simeq& \Hom_{\CC}(j_{!*} \Top A, j_{!*} j^* S)
&\text{by assumption}\\
&\simeq& \Hom_{\CC_U}(\Top A, j^* S)
&\text{because $j_{!*}$ is fully faithful}\\
&\simeq& \Hom_{\CC_U}( A, j^* S)
&\text{because the top is the largest}\\
&&&\text{semisimple quotient}\\
&&&\text{and $j^*S$ is simple}\\
\end{array}
\]
Again, taking dimensions and dividing by $\dim \End_{\CC(U)}(S)$, we find
\[
[\Top\: \p j_! A : S]
= [\Top j_{!*} A : S]
= [j_{!*} \Top A : S]
= [\Top A : j^* S]
\]

Now suppose we are in the second case.
The objects $\p j_* A$ and $j_{!*} A$ have no
non-trivial subobjects in $\CC_F$, hence
\[
[\Soc \p j_* A : S] = [\Soc j_{!*} A : S] = 0
\]
Besides, $\Soc A$ is semisimple, therefore
$j_{!*} \Soc A$ is semisimple and has no non-trivial subobject
in $\CC_F$, hence
\[
[j_{!*} \Soc A : S] = 0 
\]

Similarly, the objects $\p j_* A$ and $j_{!*} A$ have no
non-trivial quotients in $\CC_F$, hence
\[
[\Top \p j_! A : S] = [\Top j_{!*} A : S] = 0
\]
Besides, since $\Top A$ is semisimple, and therefore
$j_{!*} \Top A$ is semisimple and has no non-trivial quotient
in $\CC_F$, hence
\[
[j_{!*} \Top A : S] = 0 
\]

Since $\Soc \p j_* A$, $\Soc j_{!*} A$ and $j_{!*} \Soc A$ (respectively
$\Top \p j_! A$, $\Top j_{!*} A$ and $j_{!*} \Top A$) are semisimple, they
are isomorphic if and only if the multiplicity of each simple object
is the same in each of them, hence the result.
\end{proof}

\subsection{Perverse extensions and multiplicities}\label{subsec:IC mult}

Let $\SC$ (resp. $\SC_U$, $\SC_F$) denote the set of (isomorphisms
classes of) simple objects in $\CC$ (resp. $\CC_U$, $\CC_F$). We have
$\SC = \p j_{!*} \SC_U \cup \p i_* \SC_F$.  Since these three hearts
are assumed to be noetherian and artinian, the multiplicities of the
simple objects and the notion of composition length are well-defined.
Thus, if $B$ is an object in $\CC$, then we have the following
relation in the Grothendieck group $K_0(\CC)$.
\begin{equation}\label{eq:def mult}
[B] = \sum_{T \in\SC} [B : T] \cdot [T]
\end{equation}

\begin{prop}\label{prop:IC mult}
If $B$ is an object in $\CC$, then we have
\begin{equation}\label{eq:j^*}
[B : \p j_{!*} S] = [j^* B : S]
\end{equation}
for all simple objects $S$ in $\CC_U$. 
In particular, if $A$ is an object in $\CC_U$, then we have
\begin{equation}\label{eq:j_!*}
[\p j_! A : \p j_{!*} S] = [\p j_{!*}A : \p j_{!*}S]
= [\p j_* A : \p j_{!*} S] = [A:S]
\end{equation}
\end{prop}

\begin{proof}
The functor $j^*$ is exact, and sends a simple object $T$ on a simple
a simple object if $T\in \p j_{!*}\SC_U$, or on zero if $T \in \p i_* \SC_F$.
Moreover, it sends non-isomorphic simple objects in $\p j_{!*}\SC_U$ on
non-isomorphic simple objects in $\SC_U$.
Thus, applying $j^*$ to the relation \ref{eq:def mult}, we get
\[
[j^* B] = \sum_{T \in \p j_{!*}(\SC_U)} [j^* B : j^* T] \cdot [j^* T]
= \sum_{S \in \SC_U} [j^* B : S] \cdot [S]
\]
hence \ref{eq:j^*}, and \ref{eq:j_!*} follows.
\end{proof}

\section{Perverse $t$-structures}\label{sec:perv}

Let us go back to the setting of Section \ref{sec:context}. We want to
define the $t$-structure defining the $\EM$-perverse sheaves on $X$
for the middle perversity $p$ (and, in case $\EM = \OM$, also for the
perversity $p_+$). Let us start with the case $\EM = \FM$. We will
consider pairs $(\XG,\LG)$, where
\begin{enumerate}[(i)]
\item\label{XG}
$\XG$ is a partition of $X$ into finitely many locally closed smooth
  pieces, called strata, and the closure of a stratum is a union of strata.

\item\label{LG}
$\LG$ consists in the following data: for each stratum $S$ in $\XG$, 
a finite set $\LG(S)$ of isomorphism classes of irreducible locally constant
sheaves of $\FM$-modules over $S$.

\item\label{cons} For each $S$ in $\XG$ and for each $\FC$ in $\LC(S)$, if $j$
denotes the inclusion of $S$ into $X$, then the $R^nj_* \FC$ are
$(\XG,\LG)$-constructible, with the definition below.
\end{enumerate}

A sheaf of $\FM$-modules is $(\XG,\LG)$-constructible if and only if
its restriction to each stratum $S$ in $\XG$ is locally constant and
a finite iterated extension of irreducible locally constant sheaves
whose isomorphism class is in $\LG(S)$. 
We denote by $D^b_{\XG,\LG}(X,\FM)$ the full subcategory of $D^b(X,\FM)$
consisting in the $(\XG,\LG)$-constructible complexes, that is, whose
cohomology sheaves are $(\XG,\LG)$-constructible.

We say that $(\XG',\LG')$ refines $(\XG,\LG)$ if each stratum $S$ in $\XG$
is a union of strata in $\XG'$, and all the $\FC$ in $\LG(S)$ are
$(\XG',\LG')$-constructible, that is, $(\XG'_{|S},\LG_{\left|\XG'_{|S}\right.})$-constructible.

The condition (\ref{cons})
ensures that for $U \stackrel{j}{\injto} V \subset X$ locally closed
and unions of strata, the functors $j_*$, $j_!$ (resp. $j^*$, $j^!$) send
$D^b_{\XG,\LG}(U,\FM)$ into $D^b_{\XG,\LG}(V,\FM)$ (resp.
$D^b_{\XG,\LG}(V,\FM)$ into $D^b_{\XG,\LG}(U,\FM)$).
It follows from the constructibility theorem for $j_*$ (SGA 4$\frac{1}{2}$)
that any pair $(\XG',\LG')$ satisfying (\ref{XG}) and (\ref{LG}) can be refined
into a pair $(\XG,\LG)$ satisfying (\ref{XG}), (\ref{LG}) and
(\ref{cons}) (see \cite[\S 2.2.10]{BBD}).

So let us fix a pair $(\XG,\LG)$ as above. Then we define the full subcategories
$\p D_{\XG,\LG}^{\leqslant 0}(X,\FM)$ and $\p D_{\XG,\LG}^{\geqslant 0}(X,\FM)$
of $D^b_{\XG,\LG}(X,\FM)$ by
\begin{gather*}
A\in \p D^{\leqslant 0}_{\XG,\LG}(X,\EM) \Iff
\text{for all strata $S$ in $\XG$, }
\HC^i i_S^* A = 0
\text{ for all } i > - \dim(S)\\
A\in \p D^{\geqslant 0}_{\XG,\LG}(X,\EM) \Iff
\text{for all strata $S$ in $\XG$, }
\HC^i i_S^! A = 0
\text{ for all } i < - \dim(S)
\end{gather*}
for any $A$ in $D^b_{\XG,\LG}(X,\FM)$, where $i_S$ is the inclusion of the stratum $S$.

One can show by induction on the number of strata that
this defines a $t$-structure on $D^b_{\XG,\LG}(X,\FM)$, by repeated
applications of Theorem \ref{th:recollement}. On a stratum, we
consider the natural $t$-structure shifted by $\dim S$, and we glue
these $t$-structures successively.

The $t$-structure on $D^b_{\XG',\LG'}(X,\FM)$ for a finer pair
$(\XG,\LG)$ induces the same $t$-structure on $D^b_{\XG,\LG}(X,\FM)$,
so passing to the limit we obtain a $t$-structure on $D^b_c(X,\FM)$.

Over $\OM / \varpi^n$, we proceed similarly. An object $K$ of
$D^b_c(X,\OM/\varpi^n)$ is $(\XG,\LG)$-constructible if the
$\varpi^i \HC^j K / \varpi^{i + 1} \HC^j K $ are $(\XG,\LG)$-constructible
as $\FM$-sheaves.

Over $\OM$, since our field $k$ is finite or algebraically closed, we
can use Deligne's definition of $D^b_c(X,\OM)$ as the projective
$2$-limit of the triangulated categories $D^b_c(X,\OM/\varpi^n)$. The
assumption insures that it is triangulated. We have triangulated functors
$\OM/\varpi^n \otimes^\LM_\OM (-) : D^b_c(X,\OM) \to D^b_c(X,\OM/\varpi^n)$,
and in particular $\FM \otimes^\LM_\OM (-)$. We will often omit from
the notation $\otimes^\LM_\OM$ and simply write $\FM(-)$. The functor
$\HC^i : D^b_c(X,\OM) \to \Sh(X,\OM)$ is defined by sending an object
$K$ to the projective system of the $\HC^i(\OM/\varpi^n \otimes^\LM_\OM K)$.
We have exact sequences
\begin{equation}
0 \longto \OM/\varpi^n \otimes_\OM \HC^i(K)
\longto \HC^i(\OM/\varpi^n \otimes^\LM_\OM K)
\longto \Tor_1^\OM(\OM/\varpi^n,\HC^{i + 1}(K))
\longto 0
\end{equation}

Let $D^b_{\XG,\LG}(X,\OM)$ be the full subcategory of $D^b_c(X,\FM)$
consisting in the objects $K$ such that for some (or any) $n$,
$\OM/\varpi^n \otimes^\LM_\OM K$ is in $D^b_{\XG,\LG}(X,\OM/\varpi^n)$, or
equivalently, such that the $\FM\ \HC^i K$ are $(\XG,\LG)$-constructible. 
We define the $t$-structure for the perversity $p$ on
$D^b_{\XG,\LG}(X,\OM)$ as above. Its heart is the abelian category
$\p \MC_{\XG,\LG}(X,\OM)$. Since it is $\OM$-linear, it is endowed
with a natural torsion theory, and we can define another
$t$-structure as in \ref{sec:tt ts}, and we will say that it
is associated to the perversity $p_+$. By \ref{sec:tt recollement},
it can also be obtained by recollement.
Passing to the limit, we get two $t$-structures on $D^b_c(X,\OM)$, for
the perversities $p$ and $p_+$, which can be characterized by
\ref{def:p <= 0}, \ref{def:p >= 0}, \ref{def:p+ <= 0} and \ref{def:p+ >= 0}.

An object $A$ of $D^b_c(X,\OM)$ is in $\p D^{\leqslant 0}(X,\OM)$
(resp. $\pp D^{\geqslant 0}(X,\OM)$)
if and only if $\FM\ A$ is in $\p D^{\leqslant 0}(X,\FM)$
(resp. $\pp D^{\geqslant 0}(X,\FM)$).

If $A$ is an object in $\p\MC(X,\OM)$, then
$\FM\ A$ is in $\p\MC(X,\FM)$ if and only if $A$ is
torsion-free (that is, if and only if $A$ is also $p_+$-perverse).
Then we have $\FM\ A = \Coker \varpi.1_A$
(the cokernel being taken in $\p\MC(X,\OM)$).

Similarly, if $A$ is an object in $\pp\MC(X,\OM)$, then
$\FM\ A$ is in $\p\MC(X,\FM)$ if and only if $A$ is
divisible (that is, if and only if $A$ is also $p$-perverse).
Then we have $\FM\ A = \Ker \varpi.1_A [1]$
(the kernel being taken in $\pp\MC(X,\OM)$).

To pass from $\OM$ to $\KM$, we simply apply $\KM \otimes_\OM
(-)$. Thus $D^b_c(X,\KM)$ has the same objects as
$D^b_c(X,\OM)$, and
$\Hom_{D^b_c(X,\KM)}(A,B) = \KM \otimes_\OM \Hom_{D^b_c(X,\OM)}(A,B)$.
We write $D^b_c(X,\KM) = \KM\otimes_\OM D^b_c(X,\OM)$. We also have
$\Sh(X,\KM) = \KM\otimes_\OM \Sh(X,\OM)$. Then we define the full
subcategory $D^b_{\XG,\LG}(X,\KM)$ of $D^b_c(X,\KM)$ as the image of 
$D^b_{\XG,\LG}(X,\OM)$. The $t$-structures $p$ and $p_+$ on
$D^b_{\XG,\LG}(X,\OM)$ give rise to a single $t$-structure $p$ on
$D^b_{\XG,\LG}(X,\KM)$, because torsion objects are killed by $\KM
\otimes_\OM (-)$. This perverse $t$-structure can be defined by
recollement. Passing to the limit, we get the perverse $t$-structure
on $D^b_c(X,\KM)$ defined by \ref{def:p <= 0} and \ref{def:p >= 0}. We
have $\p \MC(X,\KM) = \KM \otimes_\OM \p\MC(X,\OM)$.

\section{Modular reduction and truncation functors}\label{sec:F recollement}

Modular reduction does not commute with truncation functors.
To simplify the notation, we will write $\FM (-)$ for $\FM\otimes^\LM_\OM(-)$.

\begin{prop}\label{prop:Ft}
For $A \in D^b_c(X,\OM)$ and $i \in \ZM$, we have distinguished triangles
\begin{equation}\label{tri:Fi iF}
\FM\ \t_{\leqslant i}\ A \longto
\t_{\leqslant i}\ \FM\ A \longto
\HC^{-1}(\FM\ \HC^{i+1}_\tors A) [-i] \rightsquigarrow
\end{equation}
\begin{equation}\label{tri:iF Fi+}
\t_{\leqslant i}\ \FM\ A \longto
\FM\  \t_{\leqslant i_+} A \longto
\HC^0 ( \FM\ \HC^{i+1}_\tors A) [-i - 1] \rightsquigarrow
\end{equation}
\begin{equation}\label{tri:Fi+ Fi+1}
\FM\  \t_{\leqslant i_+} A \longto
\FM\  \t_{\leqslant i+1}\ A \longto
\FM\ \HC^{i+1}_\free A [-i-1] \rightsquigarrow
\end{equation}

In particular,
\begin{equation}\label{isom:Fi iF Fi+}
\HC^{i+1}_\tors A = 0 \quad \Imp \quad
\FM\ \t_{\leqslant i}\ A \isom \t_{\leqslant i}\ \FM\ A \isom \FM\  \t_{\leqslant i_+} A\\
\end{equation}
\begin{equation}\label{isom:Fi+ Fi+1}
\HC^{i+1}_\free A = 0 \quad \Imp \quad
\FM\  \t_{\leqslant i_+} A \isom \FM\  \t_{\leqslant i+1}\ A
\end{equation}
\end{prop}

\begin{proof}
We have a distinguished triangle
\[
\t_{\leqslant i_+} A \to \t_{\leqslant i + 1}\ A \to \HC^{i + 1}_\free A[-i - 1] \rtordu
\]
in $D^b_c(X,\OM)$. This follows form \cite[Prop. 1.3.15]{BBD}, which
is proved using the octahedron axiom. Applying $\FM(-)$, we get the
triangle (\ref{tri:Fi+ Fi+1}), and (\ref{isom:Fi+ Fi+1}) follows.

By definition, we have a distinguished triangle
\[
\t_{\leqslant i}\ A \to \t_{\leqslant i_+} A \to \HC^{i + 1}_\tors A[-i - 1] \rtordu
\]
in $D^b_c(X,\OM)$. Applying $\FM(-)$, we get a distinguished triangle in
$D^b_c(X,\FM)$
\begin{equation}
\FM\ \t_{\leqslant i}\ A \to \FM\ \t_{\leqslant i_+} A \to
\FM\ \HC^{i + 1}_\tors A[-i - 1] \rtordu
\end{equation}
On the other hand, we have a distinguished triangle
\begin{equation}
\Tor_1^\OM(\FM, \HC^{i + 1}_\tors A)[-i]
\to \FM\ \HC^{i + 1}_\tors A[-i - 1]
\to \FM\otimes_\OM\HC^{i + 1}_\tors A[-i - 1]
\end{equation}

By the TR 4' axiom, we have an octahedron diagram
\begin{equation}\label{octa:Fi iF Fi+}
\octa
{\FM\ \t_{\leqslant i}\ A}
{B}
{\Tor_1^\OM(\FM, \HC^{i + 1}_\tors A)[-i]}
{\FM\ \t_{\leqslant i_+} A}
{\FM\ \HC^{i + 1}_\tors A[-i - 1]}
{\FM\otimes_\OM\HC^{i + 1}_\tors A[-i - 1]}
\end{equation}
for some $B$ in $D^b_c(X,\OM)$.

The triangle
$(\FM\ \t_{\leqslant i}\ A,\  B,\ \Tor_1^\OM(\FM, \HC^{i + 1}_\tors A)[-i])$
shows that
$B \in D^{\leqslant i}_c(X,\FM)$, and then
the triangle
$(B,\ \FM\ \t_{\leqslant i_+} A,\ \FM\otimes_\OM\HC^{i + 1}_\tors A[-i - 1])$
shows that $B$ is (uniquely) isomorphic to
$\t_{\leqslant i}\ \FM\ \t_{\leqslant i_+} A$.
Let us now show that
$\t_{\leqslant i}\ \FM\ \t_{\leqslant i_+} A \simeq \t_{\leqslant i}\ \FM\ A$.

By the TR 4 axiom, we have an octahedron diagram
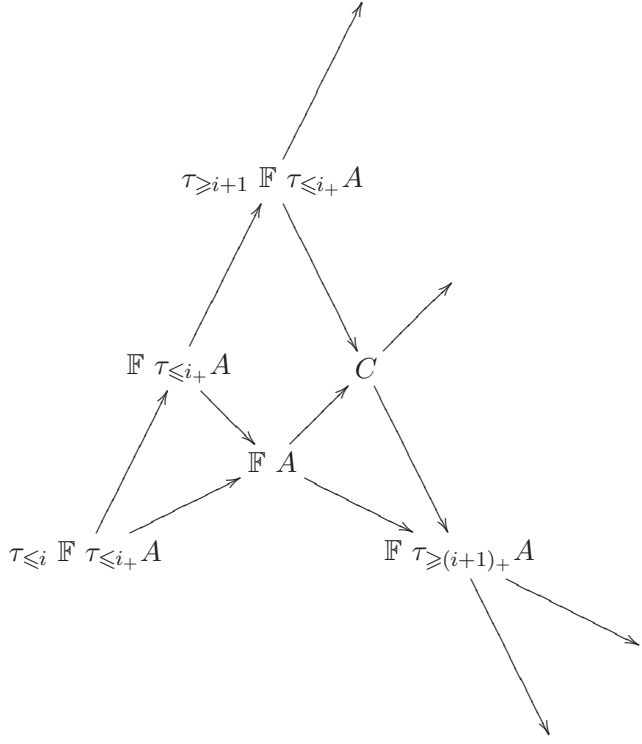
\[
\octa
{\t_{\leqslant i}\ \FM\ \t_{\leqslant i_+} A}
{\FM\ \t_{\leqslant i_+} A}
{\t_{\geqslant i + 1}\ \FM\ \t_{\leqslant i_+} A}
{\FM\ A}
{C}
{\FM\ \t_{\geqslant (i + 1)_+} A}
\]
for some $C$ in $D^b_c(X,\OM)$.

The triangle
$(\t_{\geqslant i + 1}\ \FM\ \t_{\leqslant i_+} A,\ C,\ \FM\ \t_{\geqslant (i + 1)_+} A)$
shows that
$C \in D^{\geqslant i + 1}_c(X,\FM)$, and then the triangle
$(\t_{\leqslant i}\ \FM\ \t_{\leqslant i_+} A,\ \FM\ A,\ C)$
shows that
$B \simeq \t_{\leqslant i}\ \FM\ \t_{\leqslant i_+} A \simeq \t_{\leqslant i}\ \FM\ A$
and $C \simeq \t_{\geqslant i + 1}\ \FM\ A$.

Hence the octahedron diagram (\ref{octa:Fi iF Fi+}) contains the triangles (\ref{tri:Fi iF}) and
(\ref{tri:iF Fi+}). If $\HC^{i + 1}_\tors A = 0$, the diagram reduces to the isomorphisms
(\ref{isom:Fi iF Fi+}).
\end{proof}

We have the same result if we replace $\t_{\leqslant i}$ by $\p
\t_{\leqslant i}$, and $\HC^i$ by $\p \HC^i$. The same remark applies
for the functors $\t^F_{\leqslant i}$ and $\p\t^F_{\leqslant i}$.

\section{Modular reduction and recollement}

Let us fix an open subvariety $j : U \to X$, with closed complement
$i : F\to X$. We want to see how the modular reduction behaves with
respect to this recollement situation.

For $A$ in $\p\MC(U,\OM) \cap \pp\MC(U,\OM)$, we have nine interesting
extensions of $\FM\,A$, out of which seven are automatically
perverse. These correspond to nine ways to truncate
$\FM\; j_*\; A = j_*\; \FM\; A$, three for each degree between $-2$
and $0$. Indeed, each degree is ``made of'' three parts:
the $\p\HC^0\,\FM(-)$ of the torsion part of the cohomology of $A$ of
the same degree, the reduction of the torsion-free part of the
cohomology of $A$ of the same degree, and the $\p\HC^{-1}\,\FM(-)$ of
the torsion part of the next degree (like a $\Tor_1$).

\[
\begin{array}{|c|c|c|c|}
\hline
-2 & -1 & 0 & 1\\
\hline
\begin{array}{c|c|c}
\phantom{\FM} & \FM \ \p j_! & \p j_!\ \FM
\end{array}
&
\begin{array}{c|c|c}
\FM\ \pp j_! & \FM\ \p j_{!*} & \p j_{!*}\ \FM
\end{array}
&
\begin{array}{c|c|c}
\FM\ \pp j_{!*} & \FM\ \p j_* & \p j_*\ \FM
\end{array}
&
\begin{array}{c|c|c}
\FM\ \pp j_* & \phantom{\FM} & \phantom{\FM}
\end{array}
\\
\hline
\end{array}
\]

Using Proposition \ref{prop:Ft} with the functors $\p\t^F_{\leqslant i}$, we
obtain the following distinguished triangles.
\begin{equation}
\FM \; \p j_! \longto \p j_!\; \FM  
\longto \p\HC^{-1}\; \FM\; \p i_*\; \p\HC^{-1}_\tors\; i^* j_* [2]
\rightsquigarrow
\end{equation}
\begin{equation}
\p j_!\; \FM \longto \FM \; \pp j_!
\longto \p\HC^0\; \FM\; \p i_*\; \p\HC^{-1}_\tors\; i^* j_* [1]
\rightsquigarrow
\end{equation}
\begin{equation}
\FM \; \pp j_! \longto \FM \; \p j_{!*}
\longto \FM\; \p i_*\; \p\HC^{-1}_\free\; i^* j_* [1]
\rightsquigarrow
\end{equation}
\begin{equation}
\FM \; \p j_{!*} \longto \p j_{!*}\; \FM 
\longto \p\HC^{-1}\; \FM\; \p i_*\; \p\HC^0_\tors\; i^* j_* [1]
\rightsquigarrow
\end{equation}
\begin{equation}
\p j_{!*}\; \FM   \longto \FM \; \pp j_{!*}
\longto \p\HC^0\; \FM\; \p i_*\; \p\HC^0_\tors\; i^* j_*
\rightsquigarrow
\end{equation}
\begin{equation}
\FM \; \pp j_{!*} \longto \FM \; \p j_*
\longto \FM\; \p i_*\; \p\HC^0_\free\; i^* j_*
\rightsquigarrow
\end{equation}
\begin{equation}
\FM \; \p j_* \longto \p j_*\; \FM 
\longto \p\HC^{-1}\; \FM\; \p i_*\; \p\HC^1_\tors\; i^* j_*
\rightsquigarrow
\end{equation}
\begin{equation}
\p j_*\; \FM   \longto \FM \; \pp j_*
\longto \p\HC^0\; \FM\; \p i_*\; \p\HC^1_\tors\; i^* j_* [-1]
\rightsquigarrow
\end{equation}

In particular, for $A$ in $\p\MC(U,\OM) \cap \pp\MC(U,\OM)$, we have
\begin{equation}
\p\HC^{-1}_\tors\; i^*\; j_*\; A = 0
\quad \Imp \quad 
\FM \; \p j_!\; A \elem{\sim} \p j_!\; \FM\; A \elem{\sim} \FM \; \pp j_!\; A
\end{equation}
\begin{equation}
\p\HC^{-1}_\free\; i^*\; j_*\; A = 0
\quad \Imp \quad 
\FM \; \pp j_!\; A \elem{\sim} \FM \; \p j_{!*}\; A
\end{equation}
\begin{equation}
\p\HC^0_\tors\; i^*\; j_*\; A = 0
\quad \Imp \quad 
\FM \; \p j_{!*}\; A \elem{\sim} \p j_{!*}\; \FM\; A \elem{\sim} \FM \; \pp j_{!*}\; A
\end{equation}
\begin{equation}
\p\HC^0_\free\; i^*\; j_*\; A = 0
\quad \Imp \quad 
\FM \; \pp j_{!*}\; A \elem{\sim} \FM \; \p j_*\; A
\end{equation}
\begin{equation}
\p\HC^1_\tors\; i^*\; j_*\; A = 0
\quad \Imp \quad 
\FM \; \p j_*\; A \elem{\sim} \p j_*\; \FM\; A \elem{\sim} \FM \; \pp j_*\; A
\end{equation}

\COUIC{
Finally, let us remark that in all the triangles above, the third term
is supported on $F$. This has the following consequence. For
$A$ in $\p\MC(U,\OM) \cap \pp\MC(U,\OM)$ and $S$ a simple object in
$\p\MC(U,\FM)$, the multiplicities $[E : \p j_{!*} S]$, where
$E$ is one of the nine extensions of $A$ as above, are all equal
to each other. Using Proposition \ref{subsec:IC mult}, we find that
they are all equal to $[\FM\, A : S]$.

\[
[\FM \; \p j_!\; A : \p j_{!*}\, S]
= [\p j_!\; \FM\; A : \p j_{!*} \, S]
= [\FM \; \pp j_!\; A : \p j_{!*} \, S]
= [\FM \; \p j_{!*}\; A : \p j_{!*} \, S]
= [\p j_{!*}\; \FM\; A : \p j_{!*} \, S]
= [\FM \; \pp j_{!*}\; A : \p j_{!*} \, S]
= [\FM \; \p j_*\; A : \p j_{!*} \, S]
= [\p j_*\; \FM\; A : \p j_{!*} \, S]
= [\FM \; \pp j_*\; A : \p j_{!*}\, S]
\]
}

\section{Decomposition numbers}\label{sec:decnum}

Let $X$ be endowed with a pair $(\XG,\LG)$ satisfying the conditions
(\ref{XG}), (\ref{LG}) and (\ref{cons}) of Section \ref{sec:perv}.
Let $\PG$ be the set of pairs $(\OC,\LC)$ where $\OC \in \XG$ and $\LC \in \LG(\OC)$.
Let $K_0^{\XG,\LG}(X,\FM)$ be the Grothendieck group of the triangulated category
$D^b_{\XG,\LG}(X,\FM)$.

For $\OC \in \XG$, let $j_\OC : \OC \to X$ denote the inclusion.
For $(\OC,\LC) \in \PG$, let us denote by
\begin{equation}
\0 \JC_!(\OC, \LC) = \0 {j_\OC}_!\ (\LC[\dim \OC])
\end{equation}
the extension by zero of the local system $\LC$, shifted by $\dim \OC$.
We also introduce the following notation for the three perverse extensions.
\begin{gather}
\p \JC_!(\OC, \LC) = \p {j_\OC}_!\ (\LC[\dim \OC])\\
\p \JC_{!*}(\OC, \LC) = \p {j_\OC}_{!*} (\LC[\dim \OC])\\
\p \JC_*(\OC, \LC) = \p {j_\OC}_*\ (\LC[\dim \OC])
\end{gather}

We have
\begin{equation}
K_0^{\XG,\LG}(X,\FM) 
\simeq K_0(\Sh_{\XG,\LG}(X,\FM)) 
\simeq K_0(\p\MC_{\XG,\LG}(X,\FM))
\end{equation}
If $K \in D^b_{\XG,\LG}(X,\FM)$, then we have
\[
[K] = \sum_{i \in \ZM} (-1)^i [\HC^i(K)] = \sum_{j \in \ZM} (-1)^j [\p\HC^j(K)]
\]
in $K_0^{\XG,\LG}(X,\FM)$.

This Grothendieck group is free over $\ZM$, and admits the following bases
\begin{gather*}
\BC_0 = (\0\JC_!(\OC,\LC))_{(\OC,\LC)\in\PG}\\
\BC_! = (\p\JC_!(\OC,\LC))_{(\OC,\LC)\in\PG}\\
\BC_{!*} = (\p\JC_{!*}(\OC,\LC))_{(\OC,\LC)\in\PG}\\
\BC_* = (\p\JC_*(\OC,\LC))_{(\OC,\LC)\in\PG}
\end{gather*}

For $C \in K_0^{\XG,\LG}(X,\FM)$, let us define the integers
$\chi_{(\OC,\LC)}(C)$, for $(\OC,\LC)\in\PG$, by
the relations
\begin{equation}
C = \sum_{(\OC,\LC)\in\PG} \chi_{(\OC,\LC)}(C)\ [\0\JC_!(\OC,\LC)]
\end{equation}

For $? \in \{!,!*,*\}$, the complex $\p\JC_?(\OC,\LC)$ extends
the shifted local system $\LC[\dim \OC]$, and is supported on $\ov\OC$. This implies
\begin{equation}
\chi_{(\OC',\LC')}(\p\JC_?(\OC,\LC)) = 0
\text{ unless } \ov\OC' \subsetneq \ov\OC \text{ or } (\OC',\LC') = (\OC,\LC)
\end{equation}
and
\begin{equation}
\chi_{(\OC,\LC)}(\p\JC_?(\OC,\LC)) = 1
\end{equation}

In other words, the three bases $\BC_!$, $\BC_{!*}$ and $\BC_*$
are unitriangular with respect to the basis $\BC_0$.
This implies that they are also unitriangular with respect to each other.
In fact, we already knew it by the results of Paragraph \ref{subsec:top socle},
since $\p\JC_!(\OC,\LC)$ (resp. $\p\JC_*(\OC,\LC)$)
has a top (resp. socle) isomorphic to $\p\JC_{!*}(\OC,\LC)$,
and the radical (resp. the quotient by the socle) is supported on $\ov\OC\setminus\OC$.
In particular, for $?\in\{!,*\}$, we have
\begin{equation}
[\p\JC_?(\OC,\LC) : \p\JC_{!*}(\OC',\LC')] = 0
\text{ unless } \ov\OC' \subsetneq \ov\OC \text{ or } (\OC',\LC') = (\OC,\LC)
\end{equation}
and
\begin{equation}
[\p\JC_?(\OC,\LC) : \p\JC_{!*}(\OC,\LC)] = 1
\end{equation}

Let $K_0^{\XG,\LG}(X,\KM)$ be the Grothendieck group of the triangulated category
$D^b_{\XG,\LG}(X,\KM)$. It can be identified with
$K_0(\Sh_{\XG,\LG}(X,\KM))$ and $K_0(\p\MC_{\XG,\LG}(X,\KM))$ as for
the case $\EM = \FM$.

Now, let $K$ be an object of $D^b_{\XG,\LC}(X,\KM)$.  If $K_\OM$ is an
object of $D^b_{\XG,\LC}(X,\OM)$ such that $\KM \otimes_\OM K_\OM
\simeq K$, we can consider $[\FM K_\OM]$ in
$K_0^{\XG,\LC}(X,\FM)$. This class does not depend on the choice of
$K_\OM$ (note that the modular reduction of a torsion object has a
zero class in the Grothendieck group: if we assume, for simplicity,
that we have only finite monodromy, then by dévissage we can reduce to
the analogue result for finite groups). In fact, it depends only on
the class $[K]$ of $K$ in $K_0^{\XG,\LC}(X,\KM)$. So we have a
well-defined morphism
\begin{equation}
d : K_0^{\XG,\LC}(X,\KM) \longto K_0^{\XG,\LC}(X,\FM)
\end{equation}

For $(\OC,\LC)
\in \PG$, we can consider the decomposition number $[\FM K_\OM :
\p\JC_{!*}(\OC,\LC)]$, where $K_\OM$ is any object of
$D^b_{\XG,\LC}(X,\OM)$ such that $\KM K_\OM \simeq K$.

\section{Equivariance}
\label{sec:orbits}

We now introduce $G$-equivariant perverse sheaves in the sense of
\cite[\S 0]{ICC}, \cite[\S 4.2]{LET}.

Let $G$ be a \emph{connected} algebraic group acting on a variety $X$.
Let $\rho : G \times X \to X$ be the morphism defining the action, and let
$p : G \times X \to X$ be the second projection. A sheaf
$F$ on $X$ is $G$-\emph{equivariant} if there is an isomorphism
$\a : p^* F \isom \rho^* F$. In that case, we can choose $\a$ in a unique way
such that the induced isomorphism $i^*(\a) : F \to F$ is the identity,
where $i : X \to G \times X$ is defined by $i(x) = (1_G, x)$.

If $f : X \to Y$ is a $G$-equivariant morphism, the functors $\0 f^*$,
$\0 f_*$ and $\0 f_!$ take $G$-equivariant sheaves to $G$-equivariant sheaves.

Let $\Sh_G(X,\EM)$ be the category whose objects are the $G$-equivariant $\EM$-sheaves
on $X$, and such that the morphisms between two objects $F_1$ and $F_2$
are the morphisms $\phi$ in $\Sh(X,\EM)$ such that the following diagram commutes
\[
\xymatrix{
p^* F_1 \ar[r]^{p^*\phi} \ar[d]_{\a_1} &
p^* F_2 \ar[d]^{\a_2}\\
\rho^* F_1 \ar[r]_{\rho^*\phi} &
\rho^* F_2
}
\]
where $\a_j$ is the unique isomorphism such that $i^*(\a_j)$ is the
identity for $j = 1,2$.
Then it turns out that $\Sh_G(X,\EM)$ is actually a full subcategory of
$\Sh(X,\EM)$.

For a general complex in $D^b_c(X,\EM)$, the notion of $G$-equivariance is more
delicate. However, for a perverse sheaf we can take the same definition as above,
and again the isomorphism $\a$ can be normalized with the same condition.
If $f$ is a $G$-equivariant morphism, then the functors $\p\HC^j\ f^*$,
$\p\HC^j\ f^!$, $\p\HC^j\ f_*$ and $\p\HC^j\ f^!$ take $G$-equivariant perverse
sheaves to $G$-equivariant perverse sheaves.

We define in the same way the category $\p\MC_G(X,\EM)$ of
$G$-equivariant perverse $\EM$-sheaves,
and again it is a full subcategory of $\p\MC(X,\EM)$. Moreover, it is stable
by subquotients. The simple objects in $\p\MC_G(X,\EM)$ are the intermediate
extensions of irreducible $G$-equivariant $\EM$-local systems on $G$-stable locally closed
smooth irreducible subvarieties of $X$.

Suppose $\EM$ is a field.
If $\OC$ is a homogeneous space for $G$, let $x$ be a point in $\OC$,
and let $A_G(x) = C_G(x)/C_G^0(x)$. Then the set of isomorphism
classes of irreducible $G$-equivariant $\EM$-local systems is in
bijection with the set $\Irr \EM A_G(x)$ of isomorphism classes of
irreducible representations of the group algebra $\EM A_G(x)$.

Suppose $X$ is a $G$-variety with finitely many orbits.  Then we can
take the stratification $\XG$ of $X$ by its $G$-orbits.  The orbits
are indeed locally closed by \cite[Lemma 2.3.1]{SPR}, and they are
smooth. For each $G$-orbit $\OC$ in $X$, let $x_\OC$ be a closed point
in $\OC$.  For $\LG(\OC)$ we take all the irreducible
$G$-equivariant $\FM$-local systems, so that we can identify
$\LG(\OC)$ with $\Irr \FM A_G(x_\OC)$.

Suppose $\EM$ is a field. Let $K_0^G(X,\EM)$ be the Grothendieck group
of the triangulated category $D^b_{\XG,\LG}(X,\EM)$. Then we have
\begin{equation}
K_0^G(X,\EM) = K_0(\p\MC_G(X,\EM)) = K_0(\Sh_G(X,\EM)) \simeq 
\bigoplus_{\OC} K_0(\Irr \EM A_G(x_\OC))
\end{equation}
If $K \in D^b_{\XG,\LG}(X,\EM)$, then we have
\[
[K] = \sum_{i \in \ZM} (-1)^i [\HC^i(K)] = \sum_{j \in \ZM} (-1)^j [\p\HC^j(K)]
\]
in $K_0^G(X,\EM)$.

Let $\PG_\EM$ be the set of pairs $(\OC,\LC)$ with $\OC \in \XG$ and
$\LC$ an irreducible $G$-equivariant $\EM$-local system on $\OC$
(corresponding to an irreducible representation $L$ of $\EM
A_G(x_\OC)$). Then we have bases
$\BC_0^\EM = (\0 j_!(\OC,\LC))_{(\OC,\LC)\in\PG_\EM}$,
$\BC_!^\EM = (\p j_!(\OC,\LC))_{(\OC,\LC)\in\PG_\EM}$,
$\BC_{!*}^\EM = (\p j_{!*}(\OC,\LC))_{(\OC,\LC)\in\PG_\EM}$,
$\BC_*^\EM = (\p j_*(\OC,\LC))_{(\OC,\LC)\in\PG_\EM}$.
Note that, if $\ell$ does not divide the $|A_G(x_\OC)|$, then we can
identify $\PG_\KM$ with $\PG_\FM$.

The transition matrices from $\BC_0^\EM$ to $\BC_?^\EM$ (for $?\in\{!,!*,*\}$)
are unitriangular, and also the transition matrices from
$\BC_{!*}^\EM$ to $\BC_?^\EM$ (for $?\in\{!,*\}$).

As in the last section, we have a morphism
\[
d : K_0^G(X,\KM) \longto K_0^G(X,\FM)
\]

The matrix of $d$ with respect to the bases $\BC_0^\EM$ is just
a product of blocks indexed by the orbits $\OC$, the block
corresponding to $\OC$ being the decomposition matrix of the finite
group $A_G(x_\OC)$. If $\ell$ does not divide the $|A_G(x_\OC)|$, this
is just the identity matrix.

We are interested in the matrix of $d$ in the bases 
$\BC_{!*}^\EM$. That is, we want to study the decomposition numbers
$d_{(\OC,\LC),(\OC',\LC')}
= [\FM\JC_{!*}(\OC,\LC_\OM) : \JC_{!*}(\OC',\LC')]$
for $(\OC,\LC) \in \PG_\KM$ and $(\OC',\LC') \in \PG_\FM$,
where $\LC_\OM$ is an integral form for $\LC$.
Recall that, if $\ell$ does not divide the $A_G(x)$, then we can
identify $\PG_\KM$ with $\PG_\FM$.

We will see in Chapter \ref{chap:springer} that, when $X$ is the nilpotent
variety $\NC$, part of these numbers can be interpreted as
decomposition numbers for the Weyl group, and we expect that the whole
decomposition matrix coincides with the decomposition matrix for the
Schur algebra. In types other than type $A$, we would have to make
clear which Schur algebra should show up, since several of them appear
naturally. Moreover, when $\ell$ divides some $A_G(x)$, we cannot
expect to have a correspondence with a quasi-hereditary algebra, so
interesting things should happen for these primes.

\chapter{Examples}\label{chap:examples}

\section{Semi-small morphisms}\label{sec:small}

\begin{defi}
A morphism $\pi : \Xti \to X$ is \emph{semi-small} is there is a stratification
$\XG$ of $X$ such that the for all strata $S$ in $\XG$, and for all closed
points $s$ in $S$, we have $\dim \pi^{-1}(s) \leqslant \frac{1}{2}\codim_X(S)$.
If moreover these inequalities are strict for all strata of positive codimension,
we say that $\pi$ is \emph{small}.
\end{defi}

Recall that $\Loc(S,\EM)$ is the full subcategory of $\Sh(X,\EM)$
consisting in the $\EM$-local systems. It is the heart of the
$t$-category $D^b_{\Loc}(S,\EM)$ which is the full subcategory of
$D^b_c(S,\EM)$ of objects $A$ such that all the $\HC^i A$ are local
systems, with the $t$-structure induced by the natural $t$-structure
on $D^b_c(S,\EM)$. For $\EM = \OM$, according to the definition given
after Proposition \ref{prop:tt ts}, we have an abelian category
$\Loc^+(S,\OM)$, which is the full subcategory of $D^b_c(S,\OM)$
consisting in the objects $A$ such that $\HC^0 A$ is a torsion-free
$\OM$-local system, and $\HC^1 A$ is a torsion $\OM$-local system.

\begin{prop}\label{prop:small}
Let $\pi : \Xti \to X$ be a surjective, proper and separable morphism,
with $\Xti$ smooth of pure dimension $d$.
Let $\LC$ be in $\Loc(\Xti,\EM)$. Let us consider the complex $K = \pi_!\ \LC[d]$.
\begin{enumerate}[(i)]
\item If $\pi$ is semi-small, then $\dim X = d$ and $K$ is $p$-perverse.
\item If $\pi$ is small, then $K = \p j_{!*}\ j^*\ K$
for any inclusion $j : U \to X$ of a smooth open dense subvariety over which
$\pi$ is étale.
\end{enumerate}

In the case $\EM = \OM$, we can take $\LC$ in $\Loc^+(X,\OM)$ and
replace $p$ by $p_+$.
\end{prop}

\begin{proof}
\begin{enumerate}[(i)]
\item

Let us choose stratifications $\XG$, $\tilde \XG$ such that $\pi$ is
stratified relatively to $\XG$, $\tilde \XG$. By refining $\XG$, we
can assume that for any stratum $S$ in $\XG$, we have $2
\dim(\pi^{-1}(s)) \leqslant d - \dim S$ for all closed points $s$ in
$S$. Over a stratum of maximal dimension, $\pi$ is an étale covering, so $X$ is of
dimension $d$.

For the sequel, first assume that $\EM$ is $\KM$ or $\FM$. For each
stratum $S$ in $\XG$, and for any closed point $s$ in $S$, the fiber
$K_s$ is isomorphic to $\rgc(\pi^{-1}(s), \LC) [d]$, which is
concentrated in degrees $[- d, - d + 2\dim \pi^{-1}(s)] \subset [- d,
- \dim S]$.  Hence $K \in \p D^{\leqslant 0} (X, \EM)$.

Now $\DC_{X,\EM} (K) = \pi_!\ \LC^\vee [d]$, where $\LC^\vee$ is the local system dual
to $\LC$, so we can apply the same argument to show that
$\DC_{X,\EM} (K) \in \p D^{\leqslant 0} (X, \EM)$, and thus
$K \in \p D^{\geqslant 0} (X, \EM)$. Consequently, $K$ is perverse.

For $\EM = \OM$, let us first treat the case of the perversity $p$.
If $\LC$ is a torsion-free local system (so that it is $p$ and
$p_+$-perverse), then the same argument applies, since $\DC_{\Xti,\OM}
(\LC[d]) = \LC^\vee[d]$ is still a local system shifted by $d$.

If $\LC$ is a torsion sheaf, the same argument as above shows that
$K$ is in $\p D^{\leqslant 0}(X,\OM)$.
But $\DC_{\Xti,\OM} (\LC[d])$ is in degree $- d + 1$, so the same argument shows
that, for $s\in S$, $K_s$ is concentrated in degrees $[- d + 1, -
  \dim S + 1]$, and $\HC^{-\dim S + 1} K_s$ is torsion,
so $\DC_{X,\OM} K$ is in $\pp D^{\leqslant 0}(X,\OM)$.
This implies that $K$ is in $\p D^{\geqslant 0}(X,\OM)$.
So $K$ is $p$-perverse in this case.

For a general $\LC$ in $\Loc(X,\OM)$, the result follows from the
above, using the distinguished triangle $(\LC_\tors, \LC, \LC_\free)$.
For $\LC$ in $\Loc^+(X,\OM)$, the result follows by duality.

\item
First assume that $\EM$ is $\KM$ or $\FM$.
If $\pi$ is small, with the notation above, $K_s$ and $\DC_{X,\EM} (K)_s$ are concentrated
in degrees $[- d, - \dim S - 1]$ on all strata of positive codimension.
If $S$ is a stratum of dimension $d$, then the morphism $\pi_S : \Sti
\to S$ obtained by base change $S \to X$ is a finite covering,
so $K_{|S} = ((\pi_S)_*\ \LC_{|\Sti}[d])$ is a local system shifted by $d$.
Hence we have
$K = j_{!*} j^* K$, where $j : U \to X$ is the inclusion of
the union of all strata of dimension $d$ in $\XG$.

If $\EM = \OM$, we can treat the perversities $p$ and $p_+$ as in the
first part of the Proposition.
\end{enumerate}
\end{proof}

\begin{remark}
\emph{
We will apply this Proposition, in Section \ref{sec:springer}, to the 
surjective and proper morphisms
$\pi : \tilde\gG \to \gG$ (which is small)
and $\pi_\NC : \NCt \to \NC$ (which is semi-small), to show that 
$\EM\KC$ is an intersection cohomology complex, and that $\EM\KC_\NC$ is perverse
(see the notation there).
}
\end{remark}

We note that, in the case of a small resolution, the intersection
cohomology complex can be obtained by a direct image.

\section{Equivalent singularities}\label{sec:equiv}

\begin{defi}\label{def:equiv}
Given $X$ and $Y$ two varieties, and two points $x\in X$ and $y\in Y$,
we say that the singularity of $X$ at $x$ and the singularity of $Y$
at $y$ are smoothly equivalent, and we write $\Sing(X,x) =
\Sing(Y,y)$, if there exist a variety $Z$, a point $z\in Z$, and two
maps $\varphi : Z \to X$ and $\psi : Z \to Y$, smooth at $z$, with
$\varphi(z) = x$ and $\psi(z) = y$.

If an algebraic group $G$ acts on $X$, then $\Sing(X,x)$
depends only on the orbit $\OC$ of $x$. In that case, we write
$\Sing(X,\OC) := \Sing(X,x)$.
\end{defi}

In fact, there is an open subset $U$ of $Z$ containing $z$ where
$\varphi$ and $\psi$ are smooth, so after replacing $Z$ by $U$,
we can assume that $\varphi$ and $\psi$ are smooth on $Z$.

We have the following result (it follows from the remarks after Lemma
4.2.6.1. in \cite{BBD}).
\begin{prop}\label{prop:equiv}
Suppose that $\Sing(X,x) = \Sing(Y,y)$. Then
${\bf IC}(X,\EM)_x \simeq {\bf IC}(Y,\EM)_y$
as $\EM$-modules.
\end{prop}

\begin{remark}\label{rem:equiv}
\emph{
Suppose we have a stratification $\XG$ (resp. $\YG$) of $X$ (resp. Y) adapted to
${\bf IC} (X,\EM)$ (resp. ${\bf IC} (Y,\EM)$), and let $\OC(x)$ 
(resp. $\OC(y)$) be the stratum of $x$ (resp. $y$).
Suppose we know ${\bf IC} (X,\EM)_x$ as a representation
of $\EM\pi_1(\OC(x),x)$. The proposition gives us
${\bf IC} (Y,\EM)_y$ as an $\EM$-module, but not as an
$\EM\pi_1(\OC(y),y)$-module. To determine the latter structure,
one needs more information.}
\end{remark}

\section{Cones}\label{sec:cones}

Let $X \subset \PM^{N - 1}$ be a smooth projective variety of dimension $d - 1$.
We denote by $\pi : \AM^N \setminus \{0\} \to \PM^{N - 1}$
the canonical projection. Let $U = \pi^{-1}(X) \subset \AM^N \setminus \{0\}$
and $C = \ov U = U \cup \{0\} \subset \AM^N$. They have dimension $d$.

We have a smooth open immersion $j : U \injto C$ and a closed immersion
$i : \{0\} \injto C$. If $d > 1$, then $j$ is not affine.

\begin{prop}\label{prop:cone}
With the preceding notations, we have
\[
i^* j_* \EM \simeq \rg(U, \EM)
\]
\end{prop}

Truncating appropriately, one deduces the fiber at $0$ of the complexes
$\p j_?\ \EM[d]$, where $? \in \{!,!*,*\}$, and similarly for $p_+$
if $\EM = \OM$.

More generally, we have the following result, which is contained in
\cite[Lemma 4.5 (a)]{KL2}. As indicated there,
in the complex case, this follows easily from topological considerations.

\begin{prop}\label{prop:gencone}
Let $C$ be an irreducible closed subvariety of $\AM^N$ stable under the
$\GM_m$-action defined by
$\l(z_1,\ldots,z_N) = (\l^{a_1}z_1,\ldots,\l^{a_N}z_N)$, where
$a_1 > 0$, \ldots, $a_N > 0$.
Let $j : U = C \setminus \{0\} \to C$ be the open immersion, and
$i : \{0\} \to C$ the closed immersion. Then we have
\[
i^* j_* \EM \simeq \rg(U, \EM)
\]
\end{prop}

So, if $U$ is smooth, the calculation of the intersection cohomology
complex stalks for $C$ is reduced to the calculation of the cohomology of $U$.

\section{$\EM$-smoothness}\label{sec:rat smooth}

\subsection{Definition and remarks}

The following notion was introduced by Deligne for $\EM = \ov\QM_\ell$ in \cite{DELIGNE}.

\begin{defi}
Let $X$ be a $k$-variety, purely of dimension $n$, with structural morphism
$a : X \to \Spec k$. Then $X$ is $\EM$-smooth if and only if
the adjoint $\EM(n)[2n] \to a^! \EM$ of the trace morphism is an isomorphism.
\end{defi}

This condition ensures that $X$ satisfies Poincaré duality with $\EM$ coefficients.
It is equivalent to the following condition: for all closed points $x$ in $X$,
we have $\EM(n)[2n] \isom i_x^!\EM$, that is, $H^i_{<x>}(X,\EM)$ is zero if
$i \neq n$, and isomorphic to $\EM(-n)$ if $i = n$.

Then the shifted constant sheaf $\EM_X[n]$ is self-dual (up to twist),
and $\EM_X[n]$ is an intersection cohomology complex. Indeed, it is
a complex extending the shifted constant sheaf on an everywhere dense open
subvariety, trivially satisfying the support condition of the intersection cohomology
complex. Since it is self-dual, it must be the intersection cohomology complex.
If moreover $\EM$ is a field, then $\EM_X[n]$ is a simple perverse sheaf.

Note that, in general, the fact that $\EM_X[n]$ is perverse does not imply that
$X$ is $\EM$-smooth. For example, $\EM_X[n]$ is perverse if $X$ is a complete
intersection \cite{KW}. But a complete intersection can have several irreducible components
intersecting at the same point, or more generally a branched point, and then
$X$ cannot be $\EM$-smooth. Moreover, we will shortly see, in Paragraph
\ref{subsec:cone curve}, that the cone over a smooth
projective curve of genus $g$ cannot be $\EM$-smooth if $g > 0$,
but on such a surface the shifted constant sheaf if always perverse.

\subsection{A stability property}

\begin{prop}\label{prop:rat smooth}
Let $\pi : X \to Y$ be a finite surjective and separable morphism of
$k$-varieties, with $X$ and $Y$ irreducible of dimension $n$ and
normal.  If $X$ is $\KM$-smooth then $Y$ is also $\KM$-smooth. If $X$
is $\FM$-smooth and $\ell$ does not divide the cardinality $d$ of the
generic fiber of $\pi$, then $Y$ is also $\FM$-smooth.
\end{prop}

\begin{proof}
Since $\pi$ is separable, we can choose an open dense subset $Y_0$ of
$Y$ over which $\pi$ is étale. Let $\pi_0 : X_0 \to Y_0$ be the
morphism deduced by base change.  Since $\pi_0$ is finite étale, we
can find a Galois covering $\Xti_0$ of $X_0$, such that the composite
$\Xti \to Y_0$ is also a Galois covering. Let $H = \Gal(\Xti_0/X_0)$
and $G = \Gal(\Xti_0/Y_0)$. We have $|G:H| = d$. Hence, if $\EM =
\KM$, or $\EM = \FM$ and $\ell\nmid d$, then the local system
${\pi_0}_* \EM_{X_0}$, corresponding to the representation $\Ind_{\EM
H}^{\EM G} \EM$ of $G$ (which is a finite quotient of the fundamental
group of $Y_0$), has the constant sheaf $\EM$ as a direct summand.
Hence $\JC_{!*}(Y_0, \EM)$ is a direct summand of $\JC_{!*}(Y_0,
{\pi_0}_* \EM_{X_0}) = \pi_*\EM_X[n]$ since $\pi$ is finite (in
particular, it is small). But $\pi_*\EM_X[n]$ is concentrated in
degree $-n$, so $\JC_{!*}(Y_0, \EM)$ is so as well, and it is
isomorphic to $\0\JC_*(Y_0,\EM)$ which is the shifted constant sheaf
$\EM_Y[n]$ since $Y$ is normal.
\end{proof}

In particular, the quotient of a smooth irreducible variety by a
finite group $H$ is $\KM$-smooth, and also $\FM$-smooth if $\ell$ does
not divide $H$. This will be illustrated with the simple singularities
in Section \ref{sec:simple}.

\subsection{Cone over a smooth projective curve}\label{subsec:cone curve}

With the notations of Section \ref{sec:cones}, suppose
$X \subset \PM^N$ is an irreducible smooth projective curve of genus $g$.
Let us denote by $H^j$, for $j \in \ZM$, the cohomology group $H^j(X,\ZM_\ell)$
(it is $0$ for $j < 0$ or $j > 2$).
We have $H^0 \simeq \ZM_\ell$, $H^1 \simeq \ZM_\ell^{2g}$,
and $H^2 \simeq \ZM_\ell$ (non-canonically).

We would like to compute the fiber at $0$ of the intersection cohomology complex
$K = j_{!*} (\ZM_\ell [2])$. Since we have a cone singularity, we just have to compute
the cohomology of $U$.

Now, $U$ is the line bundle corresponding to the invertible sheaf $\OC(-1)$
on $X$, minus the null section. We denote by $c$ the first Chern class of $\OC(-1)$ .
We have the Gysin sequence
\[
H^{i - 2} \elem{c} H^i \longto H^i(U, \ZM_\ell) \longto H^{i - 1} \elem{c} H^{i + 1}
\]
and hence a short exact sequence
\[
0 \longto \Coker(c : H^{i - 2} \to H^i) \longto H^i(U, \ZM_\ell)
\longto \Ker(c : H^{i - 1} \to H^{i + 1}) \longto 0
\]
We deduce that the cohomology of $U$ is
\[
\rg(U, \ZM_\ell) \simeq \ZM_\ell
                        \oplus \ZM_\ell^{2g} [-1]
                        \oplus (\ZM_\ell^{2g} \oplus \ZM/c) [-2]
                        \oplus \ZM_\ell      [-3]
\]
(a bounded complex of $\ZM$-modules is quasi-isomorphic to
the direct sum of its shifted cohomology sheaves, because $\ZM$ is a
hereditary ring).

Hence we have
\begin{gather*}
i^*\; \p j_{!}\; \ZM_\ell [2]
= i^*\; \pp j_{!}\; \ZM_\ell [2]
\simeq \ZM_\ell [2] \\
i^*\; \p j_{!*}\; \ZM_\ell [2]
\simeq \ZM_\ell [2] \oplus \ZM_\ell^{2g} [1] \\
i^*\; \pp j_{!*}\; \ZM_\ell [2]
\simeq \ZM_\ell [2] \oplus \ZM_\ell^{2g} [1] \oplus \ZM/c\\
i^*\; \p j_{*}\; \ZM_\ell [2]
= i^*\; \pp j_{*}\; \ZM_\ell [2]
\simeq \ZM_\ell [2] \oplus \ZM_\ell^{2g} [1] \oplus (\ZM/c \oplus \ZM_\ell^{2g})
\end{gather*}

Thus $C$ cannot rationally smooth (resp. $\ZM/\ell$-smooth) if
$g > 0$. If $g = 0$, then $C$ is rationally smooth, and it is
$\ZM/\ell$ smooth if $\ell$ does not divide $c$.
If one takes $N = 1$ and $X = \PM^1$ embedded in
$\PM^1$ by the identity, then $C = \AM^2$ is actually smooth.

But the constant perverse sheaf $\EM_C [2]$ is perverse in any case, since it
is equal to $\p j_{!}\; \EM [2]$ (in the case $\EM = \ZM_\ell$, it is both
$p$ and $p_+$-perverse).

\section{Simple singularities}\label{sec:simple}

In this section, we will calculate the intersection cohomology
complexes over $\KM$, $\OM$ and $\FM$ for simple singularities, and
the corresponding decomposition numbers. We will also consider the
case of simple singularities of inhomogeneous type, that is, simple
singularities with an associated group of symmetries. For the
convenience of the reader, we will recall the main points in the
theory of simple singularities, following \cite{SLO2}, to which we
refer for more details. We will use the results of this section in
Chapter \ref{chap:dec}, to calculate the decomposition numbers
involving the regular class and the subregular class in a simple Lie
algebra. Indeed, by the work of Brieskorn and Slodowy, the singularity
of the nilpotent variety along the subregular class is a simple
singularity of the corresponding type.

\subsection{Rational double points}

We assume that $k$ is algebraically closed. Let $(X,x)$ be the
spectrum of a two-dimensional normal local $k$-algebra, where $x$ denotes
the closed point of $X$. Then $(X,x)$ is \emph{rational} if there is a
resolution $\pi : \wt X \to X$ of the singularities of $X$ such that
the higher direct images of the structural sheaf of $\wt X$ vanish,
that is, $R^q \pi_*(\OC_{\wt X}) = 0$ for $q > 0$. In fact, this
property is independent of the choice of a resolution. The rationality
property is stronger than the Cohen-Macaulay property.

If $\pi : \wt X \to X$ is a resolution, then the reduced exceptional
divisor $E = \pi^{-1}(x)_\text{red}$ is a finite union of irreducible
curves. Since $X$ is a surface, there is a minimal resolution, unique
up to isomorphism, through which all other resolutions must factor.
For the minimal resolution of a simple singularity, these
curves will have a very special configuration.

Let $\G$ be an irreducible homogeneous Dynkin diagram, with set of
vertices $\D$.  We recall that a Dynkin diagram is \emph{homogeneous},
or \emph{simply-laced}, when the corresponding root system has only
roots of the same length. Thus $\G$ is of type $A_n$ ($n \geqslant
1$), $D_n$ ($n \geqslant 4$), $E_6$, $E_7$ or $E_8$.  The Cartan
matrix $C = (n_{\a,\b})_{\a,\b\in\D}$ of $\G$ satisfies $n_{\a,\a} =
2$ for all $\a$ in $\D$, and $n_{\a,\b} \in \{0,-1\}$ for all $\a \neq
\b$ in $\D$.

A resolution $\pi : \wt X \to X$ of the surface $X$, as above, has an
\emph{exceptional configuration of type $\G$} if all the irreducible
components of the exceptional divisor $E$ are projective lines, and if
there is a bijection $\a \mapsto E_\a$ from $\D$ to the set $\Irr(E)$
of these components such that the intersection numbers
$E_\a \cdot E_\b$ are given by the opposite of the Cartan matrix $C$,
that is, $E_\a \cdot E_\b = - n_{\a,\b}$ for $\a$ and $\b$ in $\D$. 
Thus we have a union of projective lines whose normal bundles in $\wt
X$ are isomorphic to the cotangent bundle $T^*\PM^1$, and two of them
intersect transversely in at most one point.

The minimal resolution is characterized by the fact that it has no
exceptional curves with self-intersection $-1$. Therefore, if the
resolution $\pi$ of the surface $X$ has an exceptional configuration
of type $\G$, then it is minimal.

\begin{theo}
The following properties of a normal surface $(X,x)$ are equivalent.
\begin{enumerate}[(i)]
\item $(X,x)$ is rational of embedding dimension $3$ at $x$.
\item $(X,x)$ is rational of multiplicity $2$ at $x$.
\item $(X,x)$ is of multiplicity $2$ at $x$ and it can be resolved by
  successive blowing up of points.
\item The minimal resolution of $(X,x)$ has the exceptional
  configuration of an irreducible homogeneous Dynkin diagram.
\end{enumerate}
\end{theo}

\begin{defi}
If any (hence all) of the properties of the preceding theorem is
satisfied, then $(X,x)$ is called a \emph{rational double point} or a
\emph{simple singularity}.
\end{defi}

\begin{theo}
Let the characteristic of $k$ be good for the irreducible homogeneous
Dynkin diagram $\G$. Then there is exactly one rational double point
of type $\G$ up to isomorphism of Henselizations. Representatives of
the individual classes are given by the local varieties at $0 \in
\AM^3$ defined by the equations in the table below.

In each case, this equation is the unique relation (syzygy) between
three suitably chosen generators $X$, $Y$, $Z$ of the algebra
$k[\AM^2]^H$ of the invariant polynomials of $\AM^2$ under the action
of a finite subgroup $H$ of $SL_2$, given in the same table.

\[
\begin{array}{llcrc}
H && |H| & \text{equation of } \AM^2/H \subset \AM^3 & \G
\\
\hline
\CG_{n+1} & \text{cyclic} & n + 1 & X^{n + 1} + YZ = 0 & A_n
\\
\DG_{4(n-2)} & \text{dihedral} & 4(n - 2) & X^{n - 1} + X Y^2 + Z^2 = 0 & D_n
\\
\TG & \text{binary tetrahedral} & 24 & X^4 + Y^3 + Z^2 = 0 & E_6
\\
\OG & \text{binary octahedral} & 48 & X^3 Y + Y^3 + Z^2 = 0 & E_7 
\\
\IG & \text{binary icosahedral} & 120 & X^5 + Y^3 + Z^2 = 0 & E_8
\end{array}
\]

Moreover, if $k$ is of characteristic $0$, these groups are, up to
conjugation, the only finite subgroups of $SL_2$.
\end{theo}

Thus, in good characteristic, every rational double point is, after
Henselization at the singular point, isomorphic to the corresponding
quotient $\AM^2/H$. When $p$ divides $n + 1$ (resp. $4(n - 2)$), the
group $\CG_{n+1}$ (resp. $\DG_{4(n-2)}$) is not reduced. We have the
following exact sequences
\begin{gather}
1 \longto \DG_8 \longto \TG \longto \CG_3 \longto 1\\
1 \longto \TG \longto \OG \longto \CG_2 \longto 1\\
1 \longto \DG_8 \longto \OG \longto \SG_3 \longto 1
\end{gather}
when the characteristic of $k$ is good for the Dynkin diagram attached
to each of the groups involved.

\subsection{Symmetries on rational double points}

To each inhomogeneous irreducible Dynkin diagram $\G$ we associate a
homogeneous diagram $\wh \G$ and a group $A(\G)$ of automorphisms of $\wh
\G$, as follows.

\[
\begin{array}{|c|c|c|c|c|}
\hline
\G & B_n & C_n & F_4 & G_2\\
\hline
\wh\G &A_{2n-1} & D_{n+1} & E_6 & D_4\\
\hline  
A(\G) & \ZM/2 & \ZM/2 & \ZM/2 & \SG_3\\
\hline
\end{array}
\]

In general, there is a unique (in case $\G = C_3$
or $G_2$ : up to conjugation by $\Aut(\wh\G) = \SG_3$) faithful action
of $A(\G)$ on $\wh\G$. One can see $\G$ as the quotient of $\wh\G$ by
$A(\G)$.

In all cases but $\G = C_3$, the group $A(\G)$ is the full group of
automorphisms of $\wh\G$. Note that $D_4$ is associated to $C_3$ and
$G_2$. For a homogeneous diagram, it will be convenient to set $\wh\G
= \G$ and $A(\G) = 1$.

A rational double point may be represented as the quotient $\AM^2/H$
of $\AM^2$ by a finite subgroup $H$ of $SL_2$ provided the
characteristic of $k$ is good for the corresponding Dynkin diagram. If
$\wh H$ is another finite subgroup of $SL_2$ containing $H$ as a
normal subgroup, then the quotient $\wh H/H$ acts naturally on
$\AM^2/H$.

\begin{defi}
Let $\G$ be an inhomogeneous irreducible Dynkin diagram and let the
characteristic of $k$ be good for $\G$. A couple $(X,A)$ consisting of
a normal surface singularity $X$ and a group $A$ of automorphisms of
$X$ is called a \emph{simple singularity of type $\G$} if it is
isomorphic (after Henselization) to a couple $(\AM^2/H, \wh H/H)$
according to the following table.

\[
\begin{array}{|c|c|c|c|c|}
\hline
\G & B_n & C_n & F_4 & G_2\\
\hline
H & \CG_{2n} & \DG_{4(n - 1)} & \TG & \DG_8\\
\hline  
\wh H & \DG_{4n} & \DG_{8(n - 1)} & \OG & \OG\\
\hline
\end{array}
\]
\end{defi}

Then $X$ is a rational double point of type $\wh\G$ and $A$ is
isomorphic to $A(\G)$. The action of $A$ on $X$ lifts in a unique way
to an action of $A$ on the resolution $\wt X$ of $X$. As $A$ fixes
the singular point of $X$, the exceptional divisor in $\wt X$ will be
stable under $A$. In this way, we recover the action of $A$ on $\wh \G$.
The simple singularities of inhomogeneous type can be characterized in
the following way.

\begin{prop}
Let $\G$ be a Dynkin diagram of type $B_n$, $C_n$, $F_4$ or $G_2$, and
let the characteristic of $k$ be good for $\G$. Let $X$ be a rational
double point of type $\wh\G$ endowed with an action of $A(\G)$, free
on the complement of the singular point, and such that the induced
action on the dual diagram of the minimal resolution of $X$ coincides
with the associated action of $A(\G)$ on $\wh\G$. Then $(X,A)$ is a
simple singularity of type $\G$.
\end{prop}

\subsection{Perverse extensions and decomposition numbers}

Let $\G$ be any irreducible Dynkin diagram, and suppose the
characteristic of $k$ is good for $\G$. Let $\wh\G$ be the
associated homogeneous Dynkin diagram, $A(\G)$ the associated symmetry
group, and $H \subset \wh H$ the corresponding
finite subgroups of $SL_2$. We recall that, if $\G$ is already
homogeneous, then we take $\wh \G = \G$, $A(\G) = 1$ and $\wh H = H$.
We stratify the simple singularity $X =
\AM^2/H$ into two strata: the origin $\{0\}$ (the singular point), and
its complement $U$, which is smooth since $H$ acts freely on
$\AM^2\setminus\{0\}$.  We want to determine the stalks of the three
perverse extensions of the (shifted) constant sheaf $\EM$ on $U$, for
$\EM$ in $(\KM,\OM,\FM)$, and for the two perversities $p$ and $p_+$
in the case $\EM = \OM$. By the results of Chapter
\ref{chap:preliminaries}, this will allow us to determine a
decomposition number.

By the quasi-homogeneous structure of the equation defining $X$ in
$\AM^3$, we have a $\GM_m$-action on $X$ contracting $X$ to the
origin. We are in the situation of Proposition
\ref{prop:gencone}, with $C = X$. Thus it is enough to calculate the
cohomology of $U$ with $\OM$ coefficients. The cases $\EM = \KM$ or
$\FM$ will follow.

Let $\wh\Phi$ be the root system corresponding to $\wh\G$,  in a real vector
space $\wh V$ of dimension equal to the rank $n$ of $\wh\G$. We identify the
set $\wh\D$ of vertices of $\wh\G$ with a basis of $\wh\Phi$. We denote by
$P(\wh\Phi)$ and $Q(\wh\Phi)$ the weight lattice and the root lattice of
$\wh V$. The finite abelian group $P(\wh\Phi)/Q(\wh\Phi)$ is the fundamental group of
the corresponding adjoint group, and also the center of the
corresponding simply-connected group. Its order is called the connection
index of $\wh\Phi$. The coweight lattice $P^\vee(\wh\Phi)$ (the weight lattice
of the dual root system $\wh\Phi^\vee$ in $\wh V^*$) is in duality with
$Q(\wh\Phi)$, and the coroot lattice $Q^\vee(\wh\Phi)$ is in duality with
$P(\wh\Phi)$. Thus the finite abelian group $P^\vee(\wh\Phi)/Q^\vee(\wh\Phi)$ is
dual to $P(\wh\Phi)/Q(\wh\Phi)$.

Let $\pi : \wt X \to X$ be the minimal resolution of $X$.
The exceptional divisor $E$ is the union of projective lines $E_\a$,
$\a\in\wh\D$. Then we have an isomorphism
$H^2(\wt X, \OM) \isom \OM \otimes_\ZM P(\wh\Phi)$
such that, for each $\a$ in $\wh\D$, the cohomology class of the
subvariety $E_\a$ is identified with $1 \otimes \a$, and such that the
intersection pairing is the opposite of the pullback of the $W$-invariant pairing on
$P$ normalized by the condition $(\a,\a) = 2$ for $\a$ in $\wh\D$ \cite{ItoNa}. 
Thus the natural map $H^2_c(\wt X, \OM) \to H^2_c(E,\OM)$ is
identified with the opposite of the map
$\OM \otimes_\ZM Q^\vee(\wh\Phi) \to \OM \otimes_\ZM P^\vee(\wh\Phi)$
induced by the inclusion.

By Poincaré duality ($U$ is smooth), it is enough to compute the
cohomology with proper support of $U$, and to do this we will use the
long exact sequence in cohomology with proper support for the open
subvariety $U$ with closed complement $E$ in $\wt X$.
The following table gives the $H^i_c(-,\OM)$ of the three varieties
(the first column is deduced from the other two).

\[
\begin{array}{c|c|c|c}
i & U & \wt X & E\\
\hline
0 & 0 & 0 & \OM\\
1 & \OM & 0 & 0\\
2 & 0 &
    \OM \otimes_\ZM Q^\vee(\wh\Phi) & \OM \otimes_\ZM P^\vee(\wh\Phi)\\
3 & \OM \otimes_\ZM P^\vee(\wh\Phi)/Q^\vee(\wh\Phi) & 0 & 0\\
4 & \OM & \OM & 0
\end{array}
\]


By (derived) Poincaré duality, we obtain the cohomology of $U$.

\begin{prop}
The cohomology of $U$ is given by
\begin{equation}
\rg(U,\OM) \simeq \OM \oplus \OM \otimes_\ZM P(\wh\Phi)/Q(\wh\Phi) [-2] \oplus \OM [-3]
\end{equation}
\end{prop}

The closed stratum is a point, and for complexes on the point the
perverse $t$-structures for $p$ and $p_+$ are the usual ones (there is no
shift since the point is $0$-dimensional).
With the notations of Section \ref{sec:cones}, we have
\begin{gather}
\label{simple H -1}
H^{-1} i^* j_*(\OM[2]) \simeq H^1(U,\OM) = 0\\
\label{simple H 0}
H^0 i^* j_*(\OM[2]) \simeq H^2(U,\OM) \simeq \OM \otimes_\ZM P(\wh\Phi)/Q(\wh\Phi)\\
\label{simple H 1}
H^1 i^* j_*(\OM[2]) \simeq H^3(U,\OM) \simeq \OM
\end{gather}

By our analysis in the sections \ref{sec:tt recollement} and
\ref{sec:F recollement}, we obtain the following results.

\begin{prop}
We keep the preceding notation. In particular, $X$ is a simple
singularity of type $\G$.

Over $\KM$, we have canonical isomorphisms
\begin{equation}
\p j_! (\KM[2]) \simeq \p j_{!*} (\KM[2]) \simeq \p j_* (\KM[2])
\simeq \KM_X[2]
\end{equation}
In particular, $X$ is $\KM$-smooth.

Over $\OM$, we have canonical isomorphisms
\begin{gather}
\label{simple iso IC}
\p j_! (\OM[2])
\simeq \pp j_! (\OM[2])
\simeq \p j_{!*} (\OM[2])
\simeq \OM_X[2]
\\
\label{simple iso IC+}
\pp j_{!*} (\OM[2])
\simeq \p j_* (\OM[2])
\simeq \pp j_* (\OM[2])
\end{gather}
and a short exact sequence in $\p\MC(X,\OM)$
\begin{equation}
\label{simple ses}
0 \longto \p j_{!*} (\OM[2])
\longto \pp j_{!*} (\OM[2])
\longto i_* \OM \otimes_\ZM (P(\wh\Phi)/Q(\wh\Phi))
\longto 0
\end{equation}

Over $\FM$, we have canonical isomorphisms
\begin{gather}
\label{simple iso F IC}
\FM\ \p j_!\ (\OM[2])
\isom \p j_!\ (\FM[2])
\isom \FM\ \pp j_!\ (\OM[2])
\isom \FM\ \p j_{!*}\ (\OM[2])
\isom \FM_X[2]
\\
\label{simple iso F IC+}
\FM\ \pp j_{!*}\ (\OM[2])
\isom \FM\ \p j_*\ (\OM[2])
\isom \p j_*\ (\FM[2])
\isom \FM\ \pp j_*\ (\OM[2])
\end{gather}
and short exact sequences
\begin{gather}
\label{simple ses F IC}
0 \longto i_*\ \FM \otimes_\ZM (P(\wh\Phi)/Q(\wh\Phi))
\longto \FM\ \p j_{!*}\ (\OM[2])
\longto \p j_{!*}\ (\FM[2])
\longto 0
\\
\label{simple ses F IC+}
0 \longto \p j_{!*}\ (\FM[2])
\longto \FM\ \pp j_{!*}\ (\OM[2])
\longto i_*\ \FM \otimes_\ZM (P(\wh\Phi)/Q(\wh\Phi))
\longto 0
\end{gather}

We have
\begin{equation}\label{simple dec}
[\FM\ \p {j}_{!*}\ (\OM[2]) : i_*\ \FM]
= [\FM\ \pp {j}_{!*}\ (\OM[2]) : i_*\ \FM]
= \dim_\FM \FM \otimes_\ZM (P(\wh\Phi)/Q(\wh\Phi))
\end{equation}

In particular, $\FM\ \p j_{!*}\ (\OM[2])$ is simple (and equal to
$\FM\ \pp j_{!*}\ (\OM[2])$) if and only if $\ell$ does not divide the
connection index $|P(\wh\Phi)/Q(\wh\Phi)|$ of $\wh\Phi$. The variety
$X$ is $\FM$-smooth under the same condition.
\end{prop}

Let us give this decomposition number in each type.

\[
\begin{array}{l|l|l}
\wh\G & P(\wh\Phi)/Q(\wh\Phi) &
[\FM\ \p {j}_{!*}\ (\OM[2]) : i_*\ \FM]\\
\hline
A_n & \ZM / (n + 1) &
      1 \text{ if } \ell \mid n + 1,\ 0 \text{ otherwise}\\
D_n\ (n \text{ even}) & (\ZM/2)^2 &
      2 \text{ if } \ell = 2,\ 0\text{ otherwise}\\
D_n\ (n \text{ odd}) & \ZM/4 &
      1 \text{ if } \ell = 2,\ 0\text{ otherwise}\\
E_6 & \ZM/3 &
      1 \text{ if } \ell = 3,\ 0\text{ otherwise}\\
E_7 & \ZM/2 &
      1 \text{ if } \ell = 2,\ 0\text{ otherwise}\\
E_8 & 0 & 0
\end{array}
\]

Let us note that for $\G = E_8$, the variety $X$ is $\FM$-smooth for
any $\ell$. However, it is not smooth, since it has a double point.

What about the action of $A(\G)$ ?  Let us first recall some facts
from \cite{BOUR456}. Let $\Aut(\wh\Phi)$ denote the group of
automorphisms of $\wh V$ stabilizing $\wh\Phi$. The subgroup of
$\Aut(\wh\Phi)$ of the elements stabilizing $\wh\D$ is identified with
$\Aut(\wh\G)$. The Weyl group $W(\wh\Phi)$ is a normal subgroup of
$\Aut(\wh\Phi)$, and $\Aut(\wh\Phi)$ is the semi-direct product of
$\Aut(\wh\G)$ and $W(\wh\Phi)$ \cite[Chap. VI, \S 1.5, Prop. 16]{BOUR456}.

The group $\Aut(\wh\Phi)$ stabilizes $P(\wh\Phi)$ and $Q(\wh\Phi)$,
thus it acts on the quotient $P(\wh\Phi)/Q(\wh\Phi)$. By
\cite[Chap. VI, \S 1.10, Prop. 27]{BOUR456}, the group $W(\wh\Phi)$
acts trivially on $P(\wh\Phi)/Q(\wh\Phi)$. Thus, the quotient group
$\Aut(\wh\Phi)/W(\wh\Phi) \simeq \Aut(\wh\G)$ acts canonically on
$P(\wh\Phi)/Q(\wh\Phi)$.

Now $A(\G)$ acts on $X$, $\wt X$, $E$ and $U$, and hence on their
cohomology (with or without supports). Moreover, the action of $A(\G)$
on $H^2_c(E,\OM) \simeq \OM \otimes_\ZM P^\vee(\wh\Phi)$ is the one
induced by the inclusions
$A(\G) \subset \Aut(\wt\G) \subset \Aut(\wh\Phi)$.
The inclusions of $E$ and $U$ in $\wt X$ are $A(\G)$-equivariant,
hence the maps in the long exact sequence in cohomology with compact
support that we considered earlier (to calculate $H^3_c(U,\OM)$)
are $A(\G)$-equivariant. Thus the action of $A(\G)$ on
$H^3_c(U,\OM) \simeq P^\vee(\wh\Phi)/Q^\vee(\wh\Phi)$ is induced by the
inclusion $A(\G) \subset \Aut(\G) \simeq \Aut(\wh\Phi)/W(\wh\Phi)$
from the canonical action. It follows that the action of $A(\G)$ on
$H^2(U,\G) \simeq P(\wh\Phi)/Q(\wh\Phi)$ also comes from the canonical
action of $\Aut(\wh\Phi)/W(\wh\Phi)$.

Why do we care about this action, since it does not appear in the
proposition ? Well, here the closed stratum is a point, so $\EM$-local
systems are necessarily trivial. We can regard them as mere $\EM$-modules.
However, when we use the results of this section for the subregular
class, we will need to determine the local system involved on the
regular class. If we just have an equivalence of singularities, we can
deal with the $\EM$-module structure, but for the action of the
fundamental group we need more information. This is why we also
studied the action of $A(\G)$ for inhomogeneous Dynkin diagrams.

\chapter{Cohomology of the minimal nilpotent orbit}\label{chap:min}

This chapter is taken from the preprint \cite{cohmin}.
We compute the integral cohomology of the minimal non-trivial
nilpotent orbit in a complex simple (or quasi-simple) Lie algebra.
We express it in terms of the root system.
We find by a uniform approach that the middle cohomology group
is isomorphic to the fundamental group of the 
root subsystem generated by the long simple roots.
We compute the rest of the cohomology in a case-by-case analysis,
which shows in particular that the primes dividing the torsion of
the rest of the cohomology are bad primes.

All the results and proofs of this chapter remain valid for
$G$ a quasi-simple reductive group over $\ov \FM_p$,
with $p$ good for $G$, using the \'etale cohomology with $\OM$
coefficients.

We will use the middle cohomology to determine the decomposition
number corresponding to the minimal and trivial orbits in Chapter
\ref{chap:dec}. If $\Phi$ is the root system of $G$, and $W$ its Weyl
group, let $\Phi'$ be the root subsystem generated by the long simple
roots (for some given basis of $\Phi$), and let $W'$ be its Weyl group
(a parabolic subgroup, but also a quotient, of $W$). With the
knowledge of the middle cohomology, we will be able to show that this
decomposition number $d_{(x_\mini,1),(0,1)}$ is equal to 
$\dim_\FM \FM \otimes_\ZM P^\vee(\Phi')/Q^\vee(\Phi')$.

In the Fourier transform approach to the (ordinary) Springer
correspondence, the trivial orbit corresponds to the trivial
representation of $W$. When $G$ is of simply-laced type, the minimal
orbit corresponds to the natural (reflection) representation of
$W$. In general, the minimal orbit corresponds to the natural
representation of $W'$, lifted to $W$. In the modular Springer
correspondence, which we will define in Chapter \ref{chap:springer},
the trivial class still corresponds to the trivial representation (in
particular, the pair $(0,1)$ is always in the image of this
correspondence). In Section \ref{sec:decomposition}, we will see that
$d_{(x_\mini,1),(0,1)}$ can be interpreted as the corresponding
decomposition number
for the Weyl group. Note that the trivial representation is involved
in the reduction of the natural representation of $W'$ if and only if
$\ell$ divides the determinant of the Cartan matrix of $W'$, which is
precisely the connection index $P^\vee(\Phi')/Q^\vee(\Phi')$ of
$\Phi'$. I wrote this section before proving the equality of
decomposition numbers in general, and the wish to find the right
decomposition number allowed me to predict that the Cartan matrix of
$W'$ (to be precise, without minus signs), should appear in the Gysin
sequence for the middle cohomology group ! A GAP session with my
supervisor Cédric Bonnafé confirmed this guess. Then it was not
difficult to prove it.

Even if we have now a general theorem relating the decomposition
numbers for the Weyl group and for the nilpotent variety, I think this
chapter is still useful, at least to show that concrete calculations are
possible on the geometric side. Besides, we will find torsion not only in
the middle (the part which controls the decomposition number), but
also in other places for bad primes. Could this torsion have a
representation theoretic interpretation ?

\section*{Introduction}

Let $G$ be a quasi-simple complex Lie group, with Lie algebra $\gG$.
We denote by $\NC$ the nilpotent variety of $\gG$.
The group $G$ acts on $\NC$ by the adjoint action,
with finitely many orbits. If $\OC$ and $\OC'$ are two orbits,
we write $\OC \le \OC'$ if $\OC \incl \ov{\OC'}$. This defines
a partial order on the adjoint orbits. It is well known that
there is a unique minimal non-zero orbit $\OC_\mini$ (see for example
\cite{CM}, and the introduction of \cite{KP2}). 
The aim of this article is to compute the integral cohomology
of $\OC_\mini$.

The nilpotent variety $\NC$ is a cone in $\gG$: it is closed under
multiplication by a scalar. Let us consider its image $\PM(\NC)$
in $\PM(\gG)$. It is a closed subvariety of this projective space,
so it is a projective variety. Now $G$ acts on $\PM(\NC)$,
and the orbits are the $\PM(\OC)$, where $\OC$ is a non-trivial adjoint
orbit in $\NC$. The orbits of $G$ in $\PM(\NC)$ are ordered in the same
way as the non-trivial orbits in $\NC$. Thus $\PM(\OC_\mini)$ is the
minimal orbit in $\PM(\NC)$, and therefore it is closed:
we deduce that it is a projective variety.
Let $x_\mini \in \OC_\mini$, and let $P = N_G(\CM x_\mini)$
(the letter $N$ stands for normalizer, or setwise stabilizer).
Then $G/P$ can be identified to $\PM(\OC_\mini)$, which is a projective
variety. Thus $P$ is a parabolic subgroup of $G$.
Now we have a resolution of singularities (see Section \ref{res spectral})
\[
G \times_P \CM x_\mini \longto \ov{\OC_\mini} = \OC_\mini \cup \{0\}
\]
which restricts to an isomorphism
\[
G \times_P \CM^* x_\mini \elem{\sim} \OC_\mini.
\]

From this isomorphism, one can already deduce that the dimension of
$\OC_\mini$ is equal to one plus the dimension of $G/P$. If we
fix a maximal torus $T$ in $G$ and a Borel subgroup $B$ containing it,
we can take for $x_\mini$ a highest weight vector for the adjoint action
on $\gG$. Then $P$ is the standard parabolic subgroup corresponding
to the simple roots orthogonal to the highest root, and the dimension
of $G/P$ is the number of positive roots not orthogonal to the highest
root, which is $2h - 3$ in the simply-laced types, where
$h$ is the Coxeter number (see \cite[chap. VI, \S 1.11, prop. 32]{BOUR456}).
So the dimension of $\OC_\mini$ is $2h - 2$ is that case. In \cite{WANG},
Wang shows that this formula is still valid if we replace $h$ 
by the dual Coxeter number $h^\vee$ (which is equal to $h$ only in the
simply-laced types).

We found a similar generalization of a result of Carter
(see \cite{CARTER}), relating
the height of a long root to the length of an element of minimal length
taking the highest root to that given long root, in the simply-laced
case: the result extends to all types, if we take the height of the
corresponding coroot instead (see Section
\ref{philg xi}, and Theorem \ref{conclu racines}). 

To compute the cohomology of $\OC_\mini$,
we will use the Gysin sequence
associated to the $\CM^*$-fibration $G \times_P \CM^* x_\mini \longto G/P$.
The Pieri formula of Schubert calculus gives an answer in terms
of the Bruhat order (see Section \ref{res spectral}).
Thanks to the results of Section \ref{philg xi},
we translate this in terms of the combinatorics of the root system
(see Theorem \ref{spectral roots}). As a consequence, we obtain the
following results (see Theorem \ref{middle}):

\medskip

\noindent{\bf Theorem~} 
{\it
(i)~
The middle cohomology of $\OC_\mini$ is given by
$$H^{2h^\vee - 2}(\OC_\mini, \ZM) \simeq P^\vee({\Phi'})/Q^\vee({\Phi'})$$
where $\Phi'$ is the root subsystem of $\Phi$ generated by the long
simple roots,
and $P^\vee({\Phi'})$ (resp. $Q^\vee({\Phi'})$) is its coweight lattice
(resp. its coroot lattice).

(ii)~
If $\ell$ is a good prime for $G$, then there is no $\ell$-torsion
in the rest of the cohomology of $\OC_\mini$.
}

\medskip

Part $(i)$ is obtained by a general argument, while $(ii)$ is obtained
by a case-by-case analysis (see Section \ref{case by case},
where we give tables for each type).

In Section \ref{a bis},
we explain a second method for the type $A_{n - 1}$, based on another
resolution of singularities: this time, it is a cotangent bundle
on a projective space (which is also a generalized flag variety).
This cannot be applied to other types, because the minimal class
is a Richardson class only in type $A$.

Note that we are really interested in the torsion.
The rational cohomology must already be known to the experts
(see Remark \ref{rat}).

\section{Long roots and distinguished coset representatives}\label{philg xi}

The Weyl group $W$ of an irreducible and reduced root system $\Phi$
acts transitively on the set $\Phi_\longue$ of long roots in $\Phi$,
hence if $\a$ is an element of $\Phi_\longue$,
then the long roots are in bijection with $W/W_\a$,
where $W_\a$ is the stabilizer of $\a$ in $W$ (a parabolic subgroup).
Now, if we fix a basis $\D$ of $\Phi$, and if we choose for $\a$
the highest root $\alpt$, we find a relation between the partial orders
on $W$ and $\Phi_\longue$ defined by $\D$, and between the length
of a distinguished coset representative and the (dual) height
of the corresponding long root. After this section was written,
I realized that the result was already proved by Carter in the
simply-laced types in \cite{CARTER} (actually, this result is quoted
in \cite{SPRTRIG}). We extend it to any type and
study more precisely the order relations involved.
I also came across \cite[\S 4.6]{BB}, where the depth of a positive
root $\b$ is defined as the minimal integer $k$ such that
there is an element $w$ in $W$ of length $k$ such that $w(\b) < 0$.
By the results of this section, the depth of a positive long root is nothing
but the height of the corresponding coroot (and the depth
of a positive short root is equal to its height).

For the classical results about root systems that are used throughout
this section, the reader may refer to \cite[Chapter VI, \S 1]{BOUR456}.
It is now available in English \cite{BOUR456ENG}.

\subsection{Root systems}\label{root systems} 
Let $V$ be a finite dimensional $\RM$-vector space and
$\Phi$ a root system in $V$.
We note $V^*=\Hom(V,\RM)$ and, if $\a \in \Phi$, we denote by $\a^\vee$ 
the corresponding coroot
and by $s_\a$ the reflexion $s_{\a,\a^\vee}$ of 
\cite[chap. VI, \S 1.1, d\'ef. 1, $(\Srm\Rrm_{\Irm\Irm})$]{BOUR456}. 
Let $W$ be the Weyl group of $\Phi$. 
The perfect pairing between $V$ and $V^*$ will be denoted by $\langle,\rangle$. 
Let $\Phi^\vee=\{\a^\vee~|~\a \in \Phi\}$. 
In all this section, we will assume that $\Phi$ is {\it irreducible} 
and {\it reduced}. 
Let us fix a scalar product $(~|~)$ on $V$, invariant under $W$,
such that
$$\min_{\a \in \Phi} (\a|\a) = 1.$$
We then define the integer 
\[
r=\max_{\a \in \Phi} (\a|\a).
\]
Let us recall that, since $\Phi$ is irreducible and reduced,
we have $r \in \{1,2,3\}$ 
and $(\a|\a) \in \{1,r\}$ if $\a \in \Phi$ 
(see \cite[chap. VI, \S 1.4, prop. 12]{BOUR456}). We define
\[
\Phi_\longue=\{\a \in \Phi~|~(\a|\a) = r\}
\]
and
\[
\Phi_\courte=\{\a \in \Phi~|~(\a|\a) < r \}
= \Phi\setminus \Phi_\longue.
\]
If $\a$ and $\b$ are two roots, then 
\begin{equation}\label{scalaire}
\langle \a,\b^\vee \rangle = \frac{2(\a|\b)}{(\b|\b)}.
\end{equation}
In particular, if $\a$ and $\b$ belong to $\Phi$, then
\begin{equation}\label{entier}
2(\a|\b) \in \ZM
\end{equation}
and, if $\a$ or $\b$ belongs to $\Phi_\longue$, then
\begin{equation}\label{r entier}
2(\a|\b) \in r\ZM
\end{equation}
The following classical result says that $\Phi_\longue$ is a closed
subset of $\Phi$.

\begin{lem}\label{somme longue}
If $\a$, $\b \in \Phi_\longue$ are such that $\a+\b \in \Phi$, 
then $\a+\b \in \Phi_\longue$.
\end{lem}

\begin{proof}
We have $(\a+\b~|~\a+\b)=(\a|\a)+(\b|\b)+2(\a|\b)$.
Thus, by \ref{r entier}, 
we have $(\a+\b~|~\a+\b) \in r\ZM$, which implies the desired result.
\end{proof}

\subsection{Basis, positive roots, height}\label{basis}
Let us fix a basis $\D$ of $\Phi$ and let $\Phi^+$ be the set of roots
$\a \in \Phi$ whose coefficients in the basis $\D$ are non-negative. 
Let $\D_\longue = \Phi_\longue \cap \D$ and
$\D_\courte = \Phi_\courte \cap \D$. 
Note that $\D_\longue$ need not be a basis of $\Phi_\longue$. 
Indeed, $\Phi_\longue $ is a root system of rank equal to the rank
of $\Phi$, whereas $\D_\longue$ has fewer elements than $\D$ if
$\Phi$ is of non-simply-laced type.
Let us recall the following well-known result
\cite[chap. VI, \S 1, exercice 20 (a)]{BOUR456}:

\begin{lem}\label{caracterisation longue}
Let $\g \in \Phi$ and write $\g=\sum_{\a \in \D} n_\a \a$, with
$n_\a \in \ZM$. Then $\g \in \Phi_\longue$ \iff \
$r$ divides all the $n_\a$, $\a \in \D_\courte$.
\end{lem}

\begin{proof}
Let $\Phi'$ be the set of roots $\g' \in \Phi$ such that,
if $\g'=\sum_{\a \in \D} n_\a' \a$, then $r$ divides $n_\a'$ for all
$\a \in \D_\courte$. We want to show that $\Phi_\longue = \Phi'$. 

Suppose that $r$ divides all the $n_\a$, $\a \in \D_\courte$. Then
$n_\a^2 (\a|\a) \in r\ZM$ for all $\a\in\D$, and by
\ref{entier} and \ref{r entier}, we have $2n_\a n_\b(\a|\b) \in r\ZM$ 
for all $(\a,\b) \in \D \times \D$ such that $\a \neq \b$. 
Thus $(\g|\g) \in r\ZM$, which implies that $\g \in \Phi_\longue$. 
Thus $\Phi' \subset \Phi_\longue$.

Since $W$ acts transitively on $\Phi_\longue$, 
it suffices to show that $W$ 
stabilizes $\Phi'$. In other words, it is enough to show that,
if $\a \in \D$ and $\g \in \Phi'$, then $s_\a(\g) \in \Phi'$. But
$s_\a(\g)=\g -\langle \g,\a^\vee \rangle \a$. 
If $\a \in \D_\longue$, then $s_\a(\g) \in \Phi'$ because $\g \in \Phi'$. 
If $\a \in \D_\courte$, then $\langle \g,\a^\vee\rangle = 2(\g|\a) \in r\ZM$ 
because $\g \in \Phi' \subset \Phi_\longue$
(see \ref{scalaire} and \ref{r entier}). 
Thus $s_\a(\g) \in \Phi'$.
\end{proof}

If $\g = \sum_{\a \in \D} n_\a \a \in \Phi$, the {\it height} of $\g$ 
(denoted by $\hauteur(\g)$) is defined by
$\hauteur(\g)=\sum_{\a \in \D} n_\a$. 
One defines the height of a coroot similarly.

If $\g$ is long, we have
$$\g^\vee =
\sum_{\a \in \D_\longue} n_\a \a^\vee
 + \frac{1}{r} \sum_{\a\in\D_\courte} n_\a \a^\vee.$$
Let
$$\hauteur^\vee(\g) := \hauteur(\g^\vee)
= \sum_{\a \in \D_\longue} n_\a + \frac{1}{r} \sum_{\a\in\D_\courte} n_\a.$$
In particular, the right-hand side of the last equation is an integer,
which is also a consequence of Lemma \ref{caracterisation longue}.

If $\a$ and $\b$ are long roots such that $\a + \b$ is a
(long) root, then $(\a + \b)^\vee = \a^\vee + \b^\vee$, so
$\hauteur^\vee$ is additive on long roots.

\subsection{Length}\label{length}
Let $l : W \to \NM=\{0,1,2,\dots\}$ be the {\it length} function associated 
to $\D$: if we let
$$N(w)=\{\a \in \Phi^+~|~w(\a) \in - \Phi^+\},$$
then we have 
\begin{equation}\label{longueur}
l(w)=|N(w)|.
\end{equation}
If $\a \in \Phi^+$ and if $w \in W$, then we have
\begin{equation}\label{ajout}
\text{\it $l(ws_\a) > l(w)$ \iff\  $w(\a) \in \Phi^+$.}
\end{equation}
Replacing $w$ by $w^{-1}$, and using the fact that an element
of $W$ has the same length as its inverse, we get
\begin{equation}\label{ajout gauche}
\text{\it $\ell(s_\a w) > \ell(w)$ \iff\  $w^{-1}(\a) \in \Phi^+$.}
\end{equation}
More generally, it is easy to show that,
if $x$ and $y$ belong to $W$, then
\begin{equation}\label{dyer}
N(xy) = N(y) \stackrel{.}{+} {}^{y^{-1}}N(x)
\end{equation}
where $\stackrel{.}{+}$ denotes the symmetric difference
(there are four cases to consider), and therefore
\begin{equation}\label{ajout bis}
\text{\it $l(xy)=l(x)+l(y)$ \iff\  $N(y) \subset N(xy)$}.
\end{equation}

Let $w_0$ be the longest element of $W$.
Recall that
\begin{equation}\label{wo}
l(w_0w)=l(ww_0)=l(w_0)-l(w)
\end{equation}
for all $w \in W$. If $I$ is a subset of $\D$, we denote by 
$\Phi_I$ the set of the roots $\a$ which belong
to the sub-vector space of $V$ generated by $I$ and we let
$$\Phi_I^+ = \Phi_I \cap\Phi^+ \qquad\text{and}\qquad
W_I=<s_\a~|~\a \in I>.$$
We also define
$$X_I=\{w \in W~|~w(\Phi_I^+) \subset \Phi^+\}.$$
Let us recall that $X_I$ is a set of coset representatives of $W/W_I$ and that
$w \in X_I$ \iff\  $w$ is of minimal length in $wW_I$. 
Moreover, we have
\begin{equation}\label{somme longueurs}
l(xw)=l(x)+l(w)
\end{equation}
if $x \in X_I$ and $w \in W_I$. 
We denote by $w_I$ the longest element of $W_I$. Then $w_0w_I$ is
the longest element of $X_I$ (this can be easily deduced from \ref{wo} 
and \ref{somme longueurs}). 
Finally, if $i$ is an integer, we denote by $W^i$ the set of elements
of $W$ of length $i$, and similarly $X_I^i$ is the set of elements
of $X_I$ of length $i$. To conclude this section,
we shall prove the following result, which should be well known:

\begin{lem}\label{longueur reflexion}
If $\b \in \Phi_\longue^+$, then
$l(s_\b) = 2~\hauteur^\vee(\b) - 1$. 
\end{lem}

\begin{proof}
We shall prove the result by induction on
$\hauteur^\vee(\b)$. The case where $\hauteur^\vee(\b) = 1$ is clear. 
Suppose $\hauteur^\vee(\b) > 1$ and suppose the result holds
for all positive long roots whose dual height is strictly smaller.

First, there exists a $\g \in \D$ such that $\b - \g \in \Phi^+$ 
(see \cite[chap. VI, \S 1.6, prop. 19]{BOUR456}). Let $\a = \b - \g$. 
There are two possibilities:

\medskip

$\bullet$ If $\g \in \D_\longue$, then $\a = \b - \g \in \Phi_\longue$
by Lemma \ref{somme longue}. 
Moreover, $\hauteur^\vee(\a) = \hauteur^\vee(\b) - 1$. Thus
$l(s_\a) = 2~\hauteur^\vee(\a)-1$. We have $(\a|\g) \neq 0$
(otherwise $\b = \a + \g$ 
would be of squared length $2r$, which is impossible). By 
\cite[chap. VI, \S 1.3]{BOUR456}, we have $(\a|\g) = - r/2$. 
Thus $\b = s_\g(\a) = s_\a(\g)$, and $s_\b = s_\g s_\a s_\g$. Since 
$s_\a(\g) > 0$, we have $l(s_\b s_\a) = l(s_\b) +1$ (see \ref{ajout}).
Since $s_\g s_\a (\g) = s_\g(\b) = \a > 0$, we have
$l(s_\g s_\a s_\g) = l(s_\a s_\g) + 1 = l(s_\a) + 2$ (see \ref{ajout}),
as expected.

\medskip

$\bullet$ If $\g \in \D_\courte$, then, by 
\cite[chap. VI, \S 1.3]{BOUR456}, we have 
$\a = \b - r\g \in \Phi_\longue^+$, $(\a|\g) = -r/2$, and 
$\hauteur^\vee(\a) = \hauteur^\vee(\b) - 1$.
As in the first case, we have $\b = s_\g(\a)$. Thus 
$s_\b = s_\g s_\a s_\g$ and the same argument applies.
\end{proof}

\begin{remark}\label{short}
\emph{By duality, if $\b \in \Phi_\courte^+$, we have}
$$l(s_\b) = 2~\hauteur(\b) - 1.$$
\end{remark}

\subsection{Highest root}\label{highest root} 
Let $\alpt$ be {\it the highest root} of $\Phi$ relatively to $\D$ 
(see \cite[chap. VI, \S 1.8, prop. 25]{BOUR456}). 
It is of height $h - 1$, where $h$ is the Coxeter number of $\Phi$.
The \emph{dual Coxeter number} $h^\vee$ can be defined as
$1 + \hauteur^\vee(\alpt)$. Let us recall the following facts:
\begin{equation}\label{grand long}
\alpt \in \Phi_\longue
\end{equation}
and
\begin{equation}\label{positif}
\text{\it If $\a \in \Phi^+ \setminus\{\alpt\}$, 
then $\langle \a,\alpt^\vee \rangle \in \{0,1\}$.}
\end{equation}
In particular,
\begin{equation}\label{plus}
\text{\it If $\a \in \Phi^+$, then $\langle \alpt,\a^\vee \rangle \ge 0$}
\end{equation}
and
\begin{equation}\label{chambre grand}
\alpt \in \Cba,
\end{equation}
where $C$ is the chamber associated to $\D$.

From now on, ${\Iti}$ will denote the subset of $\D$ defined by
\begin{equation}\label{Iti}
{\Iti}=\{\a \in \D~|~(\alpt|\a)=0\}.
\end{equation}
By construction, ${\Iti}$ is stable under any automorphism of $V$
stabilizing $\D$. 
In particular, it is stable under $-w_0$. 
By \ref{plus}, we have
\begin{equation}\label{phii}
\Phi_{\Iti} = \{\a \in \Phi~|~(\alpt|\a) = 0\}.
\end{equation}
From \ref{phii} and \cite[chap. V, \S 3.3, prop. 2]{BOUR456}, we
deduce that
\begin{equation}\label{stab}
W_{\Iti}=\{w \in W~|~w(\alpt)=\alpt\}.
\end{equation}
Note that $w_0$ and $w_{\Iti}$ commute (because $-w_0({\Iti})= {\Iti}$). 
We have
\begin{equation}\label{wo wi}
N(w_0w_{\Iti}) = \Phi^+ \setminus \Phi_{\Iti}^+.
\end{equation}

Let us now consider the map $W \to \Phi_\longue$, $w \mapsto w(\alpt)$. 
It is surjective \cite[chap. VI, \S 1.3, prop. 11]{BOUR456} 
and thus induces a bijection $W/W_\Iti \to \Phi_\longue$ by \ref{stab}. 
It follows that the map
\begin{equation}\label{bijection}
\fonctio{X_{\Iti}}{\Phi_\longue}{x}{x(\alpt)}
\end{equation}
is a bijection. If $\a \in \Phi_\longue$, we will denote by $x_\a$ the unique 
element of $X_{\Iti}$ such that $x_\a(\alpt)=\a$. We have
\begin{equation}\label{conjugaison}
x_\a s_\alpt = s_\a x_\a.
\end{equation}

\begin{lem}\label{salpt}
We have $w_0 w_{\Iti} = w_{\Iti} w_0 = s_\alpt$.
\end{lem}

\begin{proof}
We have already noticed that $w_0$ and $w_{\Iti}$ commute.

In view of \cite[chap. VI, \S 1, exercice 16]{BOUR456}, 
it suffices to show that $N(w_0w_{\Iti}) = N(s_\alpt)$, 
that is, $N(s_\alpt) = \Phi^+ \setminus \Phi_{\Iti}^+$ (see \ref{wo wi}). 
First, if $\a \in \Phi_{\Iti}^+$, then $s_\alpt(\a) = \a$, so that 
$\a \not\in N(s_\alpt)$. This shows that $N(s_\alpt) \subset
\Phi^+ \setminus \Phi_{\Iti}^+$. 

Let us show the other inclusion. If $\a \in \Phi^+ \setminus \Phi_{\Iti}^+$,
then $\langle \alpt,\a^\vee \rangle > 0$ 
by \ref{plus} and \ref{phii}.
In particular, $s_\alpt(\a) = \a-\langle \a,\alpt^\vee \rangle \alpt$
cannot belong to $\Phi^+$ since $\alpt$ is the highest root.
\end{proof}

\begin{prop}\label{palindrome}
Let $\a\in\Phi_\longue^+$. Then we have
$$l(x_\a s_\alpt) = l(s_\alpt) - l(x_\a)$$
\end{prop}

\begin{proof}
We have
\[
\begin{array}{rcll}
l(x_\a s_\alpt) &=& l(x_\a w_{\Iti} w_0) & \textrm{by Lemma \ref{salpt}}\\
&=& l(w_0) - l(x_\a w_{\Iti}) & \textrm{by \ref{wo}}\\
&=& l(w_0) - l(w_{\Iti}) - l(x_\a) & \textrm{by \ref{somme longueurs}}\\
&=& l(w_0w_{\Iti}) - l(x_\a) & \textrm{by \ref{wo}}\\
&=& l(s_\alpt) - l(x_\a) & \textrm{by Lemma \ref{salpt}.}
\end{array}
\]
\end{proof}

\begin{prop}\label{saxa}
If $\a\in\Phi_\longue^+$, then $x_{-\a} = s_\a x_\a$ and
$l(x_{-\a}) = l(s_\a x_\a) = l(s_\a) + l(x_\a)$.
\end{prop}

\begin{proof}
We have $s_\a x_\a(\alpt) = s_\a(\a) = -\a$, so to show
that $x_{-\a} = s_\a x_\a$, it is enough to show that
$s_\a x_\a \in X_{\Iti}$. But, if $\b \in \Phi_{\Iti}^+$, we have
(see \ref{conjugaison}) $s_\a x_\a(\b) = x_\a s_\alpt(\b)
= x_\a(\b) \in \Phi^+$. Hence the first result.

Let us now show that the lengths add up. By \ref{ajout bis},
it is enough to show that $N(x_\a) \incl N(s_\a x_\a)$.
Let then $\b\in N(x_\a)$. Since $x_\a \in X_{\Iti}$, $\b$ cannot
be in $\Phi_{\Iti}^+$. Thus $\langle \b, \alpt^\vee \rangle > 0$.
Therefore, $\langle x_\a(\b), \a^\vee \rangle > 0$.
Now, we have
$s_\a x_\a(\b) = x_\a(\b) - \langle x_\a(\b), \a^\vee \rangle \a < 0$
(remember that $x_\a(\b) < 0$ since $\b\in N(x_\a)$). 
\end{proof}

\begin{prop}\label{xa}
For $\a\in\Phi_\longue^+$, we have
$$l(x_\a) = \frac{l(s_\alpt) - l(s_\a)}{2}
          = \hauteur^\vee(\alpt) - \hauteur^\vee(\a)$$
$$l(x_{-\a}) = \frac{l(s_\alpt) + l(s_\a)}{2}
           = \hauteur^\vee(\alpt) + \hauteur^\vee(\a) - 1$$
\end{prop}

\begin{proof}
This follows from Propositions \ref{palindrome} and \ref{saxa},
and \ref{conjugaison}.
\end{proof}

\subsection{Orders}\label{orders}

The choice of $\D$ determines an order relation on $V$.
For $x$, $y\in V$, we have
$y \le x$ \iff\  $y - x$ is a linear combination
of the simple roots with non-negative coefficients.

For $\a\in\Phi_\longue$, it will be convenient to define the \emph{level}
$L(\a)$ of $\a$ as follows:
\begin{equation}\label{level}
L(\a) =
\left\{
\begin{array}{ll}
\hauteur^\vee(\alpt) - \hauteur^\vee(\a)
& \textrm{if $\a > 0$}\\
\hauteur^\vee(\alpt) - \hauteur^\vee(\a) - 1
& \textrm{if $\a < 0$}
\end{array}
\right.
\end{equation}
If $i$ is an integer, let $\Phi_\longue^i$ be the set of long
roots of level $i$. Then Proposition \ref{xa} says that
the bijection \ref{bijection} maps $X_\Iti^i$ onto $\Phi_\longue^i$.

For $\g\in\Phi^+$, we write
\begin{equation}\label{elem}
\b \elem{\g} \a \textrm{ \iff\  }
\a = s_\g(\b) \textrm{ and } L(\a) = L(\b) + 1.
\end{equation}
In that case, we have $\b - \a = \langle \beta, \g^\vee \rangle \g > 0$, so
$\b > \a$.

If $\a$ and $\b$ are two long roots, we say that
there is a \emph{path} from $\b$ to $\a$, and we write
$\a \preceq \b$, \iff\  there exists a sequence $(\b_0, \b_1, \ldots, \b_k)$ of
long roots,  and a sequence $(\g_1,\ldots,\g_k)$ of positive roots,
such that
\begin{equation}\label{path}
\b = \b_0 \elem{\g_1} \b_1 \elem{\g_2} \ldots \elem{\g_k} \b_k = \a.
\end{equation}
In that case, we have $L(\b_i) = L(\b) + i$ for 
$i \in \{0,~\ldots,~k\}$. If moreover all the roots $\g_i$ are simple,
we say that there is a \emph{simple path} from $\b$ to $\a$.

On the other hand, we have the Bruhat order on $W$ defined by
the set of simple reflections $S = \{s_\a~\mid~\a\in\D\}$.
If $w$  and $w'$ belong to $W$, we write
$w\longto w'$ if $w' = s_\g w$ and $l(w') = l(w) + 1$, for some
positive root $\g$. In that case, we write $w\elem{\g}w'$
(the positive root $\g$ is uniquely determined). The Bruhat
order $\le$ is the reflexive and transitive closure of the
relation $\longto$. On $X_\Iti$, we will consider the restriction of
the Bruhat order on $W$.

Let us now consider the action of a simple reflection on a long root.

\begin{lem}\label{action simple}
Let $\b\in\Phi_\longue$ and $\g\in\D$. Let $\a = s_\g(\b)$.

\begin{enumerate}
\item
\begin{enumerate}[(i)]
\item If $\b\in\D_\longue$ and $(\b|\g) > 0$, then $\g = \b$,
$\a = -\b$ and $\langle \b, \g^\vee \rangle = 2$.
\item If $\b\in -\D_\longue$ and $(\b|\g) < 0$, then $\g = -\b$,
$\a = -\b$ and $\langle \b, \g^\vee \rangle = - 2$.
\item Otherwise, $\a$ and $\langle \b, \g^\vee \rangle$ are
given by the following table:
$$
\begin{array}{|c|c|c|}
\hline
& \g\in\D_\longue & \g\in\D_\courte\\
\hline
\begin{array}{c}
(\b|\g) > 0\\
(\b|\g) = 0\\
(\b|\g) < 0
\end{array}
&
\begin{array}{l|l}
\a = \b - \g & \langle \b, \g^\vee \rangle = 1\\
\a = \b      & \langle \b, \g^\vee \rangle = 0\\
\a = \b + \g & \langle \b, \g^\vee \rangle = - 1\\
\end{array}
&
\begin{array}{l|l}
\a = \b - r\g & \langle \b, \g^\vee \rangle = r\\
\a = \b       & \langle \b, \g^\vee \rangle = 0\\
\a = \b + r\g & \langle \b, \g^\vee \rangle = - r\\
\end{array}\\
\hline
\end{array}
$$
\end{enumerate}

\item
\begin{enumerate}[(i)]
\item If $(\b|\g) > 0$ then $L(\a) = L(\b) + 1$, so that $\b \elem{\g} \a$.
\item If $(\b|\g) = 0$ then $L(\a) = L(\b)$, and in fact $\a = \b$.
\item If $(\b|\g) < 0$ then $L(\a) = L(\b) - 1$, so that $\a \elem{\g} \b$.
\end{enumerate}
\end{enumerate}
\end{lem}
 
\begin{proof}
Part 1 follows from inspection of the possible cases in
\cite[Chapitre VI, \S 1.3]{BOUR456}.

Part 2 is a consequence of part 1. Note that there is a special case
when we go from positive roots to negative roots, and \emph{vice versa}.
This is the reason why there are two cases in the definition of the
level.
\end{proof}

To go from a long simple root to the opposite of a long simple root,
one sometimes needs a non-simple reflection.

\begin{lem}\label{through}
Let $\b\in\D_\longue$, $\a\in -\D_\longue$ and $\g\in\Phi^+$.
Then $\b\elem{\g}\a$ \iff\  we are in one of the following cases:
\begin{enumerate}[(i)]
\item $\a = -\b$ and $\g = \b$. In this case,
$\langle \b, \g^\vee \rangle = 2$.
\item  $\b + (-\a)$ is a root and $\g = \b + (-\a)$. In this case,
$\langle \b, \g^\vee \rangle = 1$.
\end{enumerate}
\end{lem}

\begin{proof}
This is straightforward.
\end{proof}

But otherwise, one can use simple roots at each step.

\begin{prop}\label{simple path}
Let $\a$ and $\b$ be two long roots such that $\a \le \b$.
Write $\b = \sum_{\s\in J} n_\s \s$, and $\a = \sum_{\t\in K} m_\t \t$,
where $J$ (resp. $K$) is a non-empty subset of $\D$, and the $n_\s$
(resp. the $m_\t$) are non-zero integers, all of the same sign.
\begin{enumerate}[(i)]
\item\label{path positive} If $0 < \a \le \b$,
then there is a simple path from $\b$ to $\a$.
\item\label{path through} If $\a < 0 < \b$,
then there is a simple path from $\b$ to $\a$
\iff\  there is a long root which belongs to both $J$ and $K$.
Moreover, there is a path from $\b$ to $\a$ \iff\  there is a long root
$\s$ in $J$, and a long root $\t$ in $K$, such that $(\s|\t)\neq 0$.
\item\label{path negative} If $\a \le \b < 0$,
then there is a simple path from $\b$ to $\a$.
\end{enumerate}
\end{prop}

\begin{proof}
We will prove (\ref{path positive}) by induction on
$m = \hauteur^\vee(\b) - \hauteur^\vee(\a)$.

\medskip

If $m = 0$, then $\b = \a$, and there is nothing to prove.

So we may assume that $m > 0$ and that the results holds
for $m - 1$. Thus $\a < \b$ and we have
$$\b - \a = \sum_{\g \in J} n_\g \g$$
where $J$ is a non-empty subset of $\D$, and the $n_\g$, $\g\in J$,
are positive integers. We have
$$(\b - \a~|~\b-\a) = \sum_{\g\in J} n_\g (\b|\g)
- \sum_{\g\in J} n_\g (\a|\g) > 0.$$
So there is a $\g$ in $J$ such that $(\b|\g) > 0$ or $(\a|\g) < 0$.
In the first case, let $\b' = s_\g(\b)$. It is a long root.
If $\g$ is long (resp. short),
then $\b' = \b - \g$ (resp. $\b' = \b - r\g$), so that
$\a \le \b' < \b$ (see Lemma \ref{caracterisation longue}).
We have $\b \elem{\g} \b'$ and $\hauteur^\vee(\b')
= \hauteur^\vee(\b) - 1$, so we can conclude by the induction
hypothesis.
The second case is similar: if $\a' = s_\g(\a) \in\Phi_\longue$,
then $\a < \a' \le \b$, $\a' \elem{\g} \a$, $\hauteur^\vee(\a')
= \hauteur^\vee(\a) + 1$, and we can conclude by the induction
hypothesis. This proves (\ref{path positive}).

Now (\ref{path negative}) follows, applying (\ref{path positive})
to $-\a$ and $-\b$ and using the symmetry $-1$.

Let us prove (\ref{path through}). If there is a long simple root $\s$
which belongs to $J$ and $K$, we have $\a \le -\s < \s \le \b$.
Using (\ref{path positive}), we find a simple path from $\b$ to $\s$,
then we have $\s \elem{\s} -\s$, and using (\ref{path negative}) we
find a simple path from $-\s$ to $\a$. So there is a simple path
from $\b$ to $\a$.

Suppose there is a long root
$\g$ in $J$, and a long root $\g'$ in $K$, such that $(\s|\t)\neq 0$.
Then either we are in the preceding case, or
there are long simple roots $\s\in J$ and $\t \in K$,
such that $\a \le -\t < \s \le \b$ and $\g = \s + \t$ is a root.
By Lemma \ref{through}, we have $\s \elem{\g} -\t$. Using
(\ref{path positive}) and (\ref{path negative}), we can find
simple paths from $\b$ to $\s$ and from $-\t$ to $\a$.
So there is a path from $\b$ to $\a$.

Now suppose there is a path from $\b$ to $\a$.
In this path, we must have a unique step of the form
$\s \elem{\g} -\t$, with $\s$ and $\t$ in $\D_\longue$.
We have $\s\in J$, $\t \in K$, and $(\s|\t) \neq 0$.
If moreover it is a simple path from $\b$ to $\a$,
then we must have $\t = -\s$. This completes the proof.
\end{proof}

The preceding analysis can be used to study the length and the
reduced expressions of some elements of $W$.

\begin{prop}\label{minimal length}
Let $\a$ and $\b$ be two long roots. If $x$ is an element
of $W$ such that $x(\b) = \a$, then we have $l(x) \ge |L(\a) - L(\b)|$.

Moreover, there is an $x\in W$ such that $x(\b) = \a$ and
$l(x) = |L(\a) - L(\b)|$ \iff\ 
$\a$ and $\b$ are linked by a simple path, either
from $\b$ to $\a$, or from $\a$ to $\b$. In this case,
there is only one such $x$, and we denote it by $x_{\a\b}$.
The reduced expressions of $x_{\a\b}$ correspond bijectively
to the simple paths from $\b$ to $\a$.

If $\a \le \b \le \g$ are such that $x_{\a\b}$ and $x_{\b\g}$ are
defined, then $x_{\a\g}$ is defined, and we have $x_{\a\g} =
x_{\a\b} x_{\b\g}$ with $l(x_{\a\g}) = l(x_{\a\b}) + l(x_{\b\g})$.

The element $x_{-\a,\a}$ is defined for all $\a \in \Phi_\longue^+$,
and is equal to $s_\a$.

The element $x_{\a,\alpt}$ is defined for all $\a \in \Phi_\longue$,
and is equal to $x_\a$.
\end{prop}

\begin{proof}
Let $(s_{\g_k},\ldots,s_{\g_1})$ be a reduced expression of $x$,
where $k = l(x)$. For $i \in \{0,~\ldots,~k\}$, let
$\b_i = s_{\g_i} \ldots s_{\g_1}(\b)$.
For each $i \in \{0,~\ldots,~k - 1\}$, we have
$|L(\b_{i + 1}) - L(\b_i)| \le 1$ by Lemma \ref{action simple}.
Then we have
$$|L(\a) - L(\b)| \le \sum_{i = 0}^{k - 1} |L(\b_{i + 1}) - L(\b_i)|
\le k = l(x)$$

If we have an equality, then all the $L(\b_{i + 1}) - L(\b_i)$ must
be of absolute value one and of the same sign, so either they are
all equal to $1$, or they are all equal to $-1$. Thus, either
we have a simple path from $\b$ to $\a$, or we have a simple path
from $\b$ to $\a$.

Suppose there is a simple path from $\b$ to $\a$.
Let $(\b_0,\ldots,\b_k)$ be a sequence of long roots,
and $(\g_1,\ldots,\g_k)$ a sequence of simple roots,
such that
$$\b = \b_0 \elem{\g_1} \b_1 \elem{\g_2} \ldots \elem{\g_k} \b_k = \a.$$
Let $x = s_{\g_k}\ldots s_{\g_1}$.
Then we have $x(\b) = \a$, and $l(x) \le k$. But we have seen that
$l(x) \ge L(\a) - L(\b) = k$. So we have equality. The case
where there is a simple path from $\a$ to $\b$ is similar.

Let $\a\in\Phi_\longue$. If $\a > 0$, we have $0 < \a \le \alpt$, so
by Proposition \ref{simple path} (\ref{path positive}), there is a simple
path from $\alpt$ to $\a$. If $\a < 0$, we have $\a < -\a \le \alpt$,
so by Proposition
\ref{simple path} (\ref{path positive}) and (\ref{path through}),
there is also a simple path from $\alpt$ to $\a$ in this case.
Let $x$ be the product of the simple reflections it involves.
Then $l(x) = L(\a)$, so $x$ is of
minimal length in $xW_{\Iti}$, and $x\in X_{\Iti}$. Thus $x = x_\a$
is uniquely determined, and $x_{\a,\alpt}$ is defined. It is equal
to $x_\a$ and is of length $L(\a)$.

Let $\a$ and $\b$ be two long roots such there is a simple path
from $\b$ to $\a$, and let $x$ be the product of the simple
reflections it involves. We have $x x_\b(\alpt) = x(\b) = \a$,
and it is of length $L(\a)$, so it is of minimal length in its coset
modulo $W_{\Iti}$. Thus $x x_\b = x_\a$, and $x = x_\a x_\b^{-1}$ is
uniquely determined. Therefore, $x_{\a\b}$ is defined and equal
to $x_\a x_\b^{-1}$. Any simple path from $\b$ to $\a$ gives
rise to a reduced expression of $x_{\a\b}$, and every reduced expression
of $x_{\a\b}$ gives rise to a simple path from $\b$ to $\a$.
These are inverse bijections.

If $\a \le \b \le \g$ are such that $x_{\a\b}$ and $x_{\b\g}$ are
defined, one can show that $x_{\a\g}$ is defined, and that
we have $x_{\a\g} = x_{\a\b} x_{\b\g}$
with $l(x_{\a\g}) = l(x_{\a\b}) + l(x_{\b\g})$,
by concatenating simple paths from $\g$ to $\b$ and from $\b$
to $\a$.

If $\a$ is a positive long root, then there is a simple path
from $\a$ to $-\a$. We can choose a symmetric path (so that the
simple reflections form a palindrome). So $x_{-\a,\a}$ is defined,
and is a reflection: it must be $s_\a$. It is of length
$L(-\a) - L(\a) = 2 \hauteur^\vee(\a) - 1$.
\end{proof}

\begin{remark}\label{second proof}
\emph{We have seen in the proof that, if $\a\in\Phi_\longue^+$, then
$$l(s_\a) = l(x_{-\a,\a}) = L(-\a) - L(\a) = 2 \hauteur^\vee(\a) - 1$$
and, if $\a\in\Phi_\longue$, then
$$l(x_\a) = l(x_{\a,\alpt}) = L(\a)$$
thus we have a second proof of Lemma \ref{longueur reflexion}
and Proposition \ref{xa}.
Similarly, the formulas $x_{\a\g} = x_{\a\b} x_{\b\g}$
and $l(x_{\a\g}) = l(x_{\a\b}) + l(x_{\b\g})$, applied to
the triple $(-\a,\a,\alpt)$, give another proof of Proposition
\ref{saxa}.}
\end{remark}

\bigskip

To conclude this section, let us summarize the results which
we will use in the sequel.

\begin{theo}\label{conclu racines}
The bijection \ref{bijection} is an anti-isomorphism between the posets
$(\Phi_\longue,\preceq)$ and $(X_\Iti,\le)$
(these orders were defined at the beginning of \ref{orders}), 
and a root of level $i$ corresponds to an element of length $i$
in $X_\Iti$.

If $\b$ and $\a$ are long roots, and $\g$ is a positive root,
then we have
$$\b \elem{\g} \a \quad \textrm{ \iff\  } \quad x_\b \elem{\g} x_\a$$
(these relations have been defined at the beginning of \ref{orders}).

Moreover, in the above situation, the integer
$\partial_{\a\b} = \langle \b, \g^\vee \rangle$ is
determined as follows:
\begin{enumerate}[(i)]
\item if $\b\in\D_\longue$ and $\a\in -\D_\longue$, then
$\partial_{\a\b}$ is equal to $2$ if $\a = -\b$, and to $-1$
if $\b + (-\a)$ is a root;

\item otherwise, $\partial_{\a\b}$ is equal to $1$ if $\g$ is long,
and to $r$ if $\g$ is short (where $r = \max_{\a\in\Phi} (\a|\a)$).
\end{enumerate}
\end{theo}

If $\b$ and $\a$ are two long roots such that $L(\a) = L(\b) + 1$,
then we set $\partial_{\a\b} = 0$ if there is no simple root $\g$
such that $\b \elem{\g} \a$.

The numbers $\partial_{\a\b}$ will appear in Theorem \ref{spectral roots}
as the coefficients of the matrices of some maps appearing in the
Gysin sequence associated to the $\CM^*$-fibration $\OC_\mini \simeq
G\times_P \CM^* x_\mini$ over $G/P$, giving the cohomology of $\OC_\mini$.
By Theorem \ref{conclu racines}, these coefficients are explicitly
determined in terms of the combinatorics of the root system.

\section{Resolution of singularities, Gysin sequence}\label{res spectral}

Let us choose a maximal torus $T$ of $G$, with Lie algebra $\tG \incl \gG$.
We then denote by $X(T)$ its group of characters, and $X^\vee(T)$
its group of cocharacters. For each $\a\in\Phi$, there is a
closed subgroup $U_\a$ of $G$, and an isomorphism $u_\a : \GM_a \to U_\a$
such that, for all $t\in T$ and for all $\l\in \CM$, we have
$t u_\a(\l) t^{-1} = u_\a(\a(t) \l)$.
We are in the set-up of \ref{root systems}, with $\Phi$ equal to
the root system of $(G,T)$ in $V = X(T)\otimes_\ZM \RM$.
We denote $X(T) \times_\ZM \QM$ by $V_\QM$, and the symmetric
algebra $S(V_\QM)$ by $S$.

There is a root subspace decomposition
$$ \gG = \tG \oplus \left(\bigoplus_{\a\in\Phi} \gG_\a\right)$$
where $\gG_\a$ is the (one-dimensional) weight subspace
$\{x\in\gG~\mid~\forall t\in T,~\Ad(t).x = \a(t)x\}$.
We denote by $e_\alpha$ a non-zero vector in $\gG_\a$.
Thus we have $\gG_\a = \CM e_\a$.

Let $W = N_G(T)/T$ be the Weyl group. It acts on $X(T)$, and hence
on $V_\QM$ and $S$.

Let us now fix a Borel subgroup $B$ of $G$ containing $T$, with Lie
algebra $\bG$. This choice determines a basis $\D$, the subset of
positive roots $\Phi^+$, and the height (and dual height) function,
as in \ref{basis}, the length function $l$ as in \ref{length},
the highest root $\alpt$ and the subset $\Iti$ of $\D$ as in
\ref{highest root}, and the orders on $\Phi_\longue$ and $X_\Iti$
as in \ref{orders}.
So we can apply all the notations and results of Section \ref{philg xi}.

Let $H$ be a closed subgroup of $G$, and $X$ a variety with a left
$H$-action. Then $H$ acts on $G \times X$
on the right by $(g,x).h = (gh,h^{-1}x)$.
If the canonical morphism $G\to G/H$ has local sections, then
the quotient variety $(G \times X)/H$ exists (see \cite[\S 5.5]{SPR}).
The quotient is denoted by $G\times_H X$. One has a morphism
$G\times_H X \to G/H$ with local sections, whose fibers are isomorphic
to $X$. The quotient is the \emph{fiber bundle over $G/H$ associated to
$X$}. We denote the image of $(g,x)$ in this quotient by $g *_H x$,
or simply $g * x$ if the context is clear. Note that $G$ acts on the left
on $G \times_H X$, by $g'.g *_H x = g'g *_H x$.

In \ref{G/B}, we describe the cohomology of $G/B$, both in terms Chern
classes of line bundles and in terms of fundamental classes
of Schubert varieties, and we state the Pieri formula (see
\cite{BGG, DEM, HILLER} for a description of Schubert calculus).
In \ref{parab inv}, we explain how this generalizes to the parabolic
case. In \ref{coh}, we give an algorithm to compute 
the cohomology of any line bundle minus the zero section,
on any generalized flag variety. To do this, we need the
Gysin sequence (see for example \cite{BT, HUS}, or \cite{MILNE}
in the \'etale case). In \ref{res sing}, we will see that
the computation of the cohomology of $\OC_\mini$ is a particular case.
Using the results of Section \ref{philg xi}, we give a description
in terms of the combinatorics of the root system.

\subsection{Line bundles on $G/B$, cohomology of $G/B$}\label{G/B}

Let $\BC = G/B$ be the flag variety. It is a smooth projective
variety of dimension $|\Phi^+|$. The map $G \longto G/B$
has local sections (see \cite[\S 8.5]{SPR}).
If $\a$ is a character of $T$, one can lift it to $B$:
let $\CM_\a$ be the corresponding one-dimensional representation 
of $B$. We can then form the $G$-equivariant line bundle
\begin{equation}
\LC(\a) = G \times_B \CM_\a \longto G/B.
\end{equation}

Let $c(\a)\in H^2(G/B,\ZM)$ denote the first Chern class of $\LC(\a)$.
Then $c : X(T) \longto H^2(G/B,\ZM)$ is a morphism of $\ZM$-modules.
It extends to a morphism of $\QM$-algebras, which we still
denote by $c : S \longto H^*(G/B,\QM)$. The latter is surjective
and has kernel $\IC$, where $\IC$ is the ideal of $S$ generated
by the $W$-invariant homogeneous elements in $S$ of positive degree.
So it induces an isomorphism of $\QM$-algebras
\begin{equation}
\ov{c}:S/\IC \simeq H^*(G/B,\QM)
\end{equation}
which doubles degrees.

The algebra $S/\IC$ is called the coinvariant algebra.
As a representation of $W$, it is isomorphic to the regular
representation. We also have an action of $W$ on $H^*(G/B,\QM)$,
because $G/B$ is homotopic to $G/T$, and $W$ acts on the right on $G/T$ by
the formula $gT.w = gnT$, where $n\in N_G(T)$ is a 
representative of $w \in W$, and $g \in G$. One can show
that $\ov c$ commutes with the actions of $W$.

On the other hand, we have the Bruhat decomposition \cite[\S 8.5]{SPR}
\begin{equation}
G/B = \bigsqcup_{w\in W} C(w)
\end{equation}
where the $C(w) = BwB/B \simeq \CM^{l(w)}$ are the Schubert cells.
Their closures are the Schubert varieties $S(w) = \ov{C(w)}$.
Thus the cohomology of $G/B$ is concentrated in even degrees,
and $H^{2i}(G/B,\ZM)$ is free with basis $(Y_w)_{w\in W^i}$,
where $Y_w$ is the cohomology class of the Schubert variety $S(w_0 w)$
(which is of codimension $l(w)=i$).
The object of Schubert calculus is to describe
the multiplicative structure of $H^*(G/B,\ZM)$ in these
terms (see \cite{BGG, DEM, HILLER}).
We will only need the following result (known
as the Pieri formula, or Chevalley formula):
if $w\in W$ and $\a\in X(T)$, then 
\begin{equation}\label{pieri}
c(\a)\ .\  Y_w\quad = \sum_{w\elem{\g}w'}
                   \left<w(\a), \g^\vee\right>  Y_{w'}
\end{equation}

\subsection{Parabolic invariants}\label{parab inv}

Let $I$ be a subset of $\D$.
Let $P_I$ be the parabolic subgroup of $G$ containing $B$ corresponding
to $I$. It is generated by $B$ and the subgroups $U_{-\a}$, for
$\a\in I$. Its unipotent radical $U_{P_I}$ is generated by the $U_\a$,
$\a\in \Phi^+\setminus\Phi_I^+$. And it has a Levi complement
$L_I$, which is generated by $T$ and the $U_\a$, $\a\in\Phi_I$.
One can generalize the preceding constructions to the parabolic case.

If $\a\in X(T)^{W_I}$ (that is, if $\a$ is a character orthogonal
to $I$), then we can form the $G$-equivariant line bundle
\begin{equation}
\LC_I(\a) = G \times_{P_I} \CM_\alpha \longto G/{P_I}
\end{equation}
because the character $\a$ of $T$, invariant by $W_I$, can be extended to $L_I$
and lifted to ${P_I}$.

We have a surjective morphism $q_I : G/B \longto G/P_I$, which induces
an injection
\[
q_I^* : H^*(G/P_I,\ZM) \hookrightarrow H^*(G/B,\ZM)
\]
in cohomology,
which identifies $H^*(G/P_I,\ZM)$ with $H^*(G/B,\ZM)^{W_I}$.

The isomorphism $\ov{c}$ restricts to
\begin{equation}
(S/\IC)^{W_I} \simeq H^*(G/{P_I},\QM)
\end{equation}

We have cartesian square
\[
\xymatrix{
\LC(\a) \ar[r] \ar[d] & \LC_I(\a) \ar[d]\\
G/B \ar[r]_{q_I} & G/P_I
}
\]

That is, the pullback by $q_I$  of $\LC_I(\a)$ is $\LC(\a)$.
By functoriality of Chern classes, we have $q_I^*(c_I(\a)) = c(\a)$.

We still have a Bruhat decomposition
\begin{equation}
G/P_I = \bigsqcup_{w\in X_I} C_I(w)
\end{equation}
where $C_I(w) =  BwP_I/P_I \simeq \CM^{l(w)}$ for $w$ in $X_I$.
We note
\[
\text{
$S_I(w) = \ov{C_I(w)}$ and
$Y_{I,w} = [\ov{Bw_0 wP_I/P_I}] = [\ov{Bw_0 w w_I P_I/P_I}]
= [S_I(w_0 w w_I)]$
for $w$ in $X_I$}
\]
Note that
\begin{equation}
\text{if $w$ is in $X_I$, then $w_0 w w_I$ is also in $X_I$}
\end{equation}
since for any root $\b$ in $\Phi_I^+$, we have
$w_I(\b) \in \Phi_I^-$, hence $ww_I(\b)$ is also negative, and thus
$w_0 w w_I(\b)$ is positive. Moreover, we have
\begin{equation}
\text{if $w \in X_I$, then
$
l(w_0 w w_I)
= l(w_0) - l(w w_I)
= l(w_0)-l(w_I)-l(w)
= \dim G/P_I - l(w)$}
\end{equation}
so that $Y_{I,w} \in H^{2l(w)}(G/P_I,\ZM)$.
We have $q_I^*(Y_{I,w}) = Y_w$.

The cohomology of $G/P_I$ is concentrated in even degrees,
and $H^{2i}(G/P_I,\ZM)$ is free with basis $(Y_{I,w})_{w\in X_I^i}$.
The cohomology ring $H^*(G/P_I,\ZM)$ is identified \emph{via} $q_I^*$ to
a subring of $H^*(G/B,\ZM)$,
so the Pieri formula can now be written as follows.
If $w\in X_I$ and $\a\in X(T)^{W_I}$, then we have
\begin{equation}\label{pieri p}
c_I(\a)\ .\  Y_{I,w}
\quad
=
\sum_{w\elem{\g}w'\in X_I}
  \left<w(\a), \g^\vee\right>  Y_{I,w'}
\end{equation}

\subsection{Cohomology of a $\CM^*$-fiber bundle on $G/P_I$}\label{coh}

Let $I$ be a subset of $\D$, and $\a$ be a $W_I$-invariant character
of $T$. Let us consider
\begin{equation}\label{fib}
\LC^*_I(\a) = G \times_{P_I} \CM^*_\alpha \longto G/{P_I},
\end{equation}
that is, the line bundle $\LC_I(\a)$ minus the zero section.
In the sequel, we will have to calculate the cohomology
of $\LC^*_\Iti(\alpt)$, but we can explain how to calculate
the cohomology of $\LC^*_I(\a)$ for any given $I$ and $\a$
(the point is that the answer for the middle cohomology will turn out to be
nicer in our particular case, thanks to the results of Section
\ref{philg xi}).

We have the Gysin exact sequence
{\small
\[
H^{n-2}(G/P_I,\ZM) \elem{c_I(\a)} H^n(G/P_I,\ZM)
\longto H^n(\LC^*_I(\a),\ZM) \longto 
H^{n-1}(G/P_I,\ZM) \elem{c_I(\a)} H^{n+1}(G/P_I,\ZM)
\]
}
where $c_I(\a)$ means the multiplication by $c_I(\a)$,
so we have a short exact sequence
{\small
\[
0\longto \Coker\ (c_I(\a):H^{n-2}\to H^n)
\longto H^n(\LC^*_I(\a),\ZM)
\longto \Ker\ (c_I(\a):H^{n-1}\to H^{n+1})
\longto 0
\]
}
where $H^j$ stands for $H^j(G/P_I,\ZM)$.
Thus all the cohomology of $\LC^*_I(\a)$ can be
explicitly computed, thanks to the results of \ref{parab inv}.

Let us now assume that $\alpha$ is dominant and regular for $P_I$,
so that $\LC_I(\alpha)$ is ample.
Then, by the hard Lefschetz theorem, $c_I(\a):\QM\otimes_\ZM H^{n-2}
\to \QM\otimes_\ZM H^n$
is injective for $n \leqslant d_I = \dim \LC^*_I(\a) = \dim G/P_I + 1$,
and surjective for $n \geqslant d_I$. By the way, we see that we could
immediately determine the rational cohomology of $\OC_\mini$,
using only the results in this paragraph and the cohomology of
$G/P_I$.

But we can say more.
The cohomology of $G/P_I$ is free and concentrated in even degrees.
In fact, $c_I(\a): H^{n-2} \to H^n$ is injective for $n \leqslant d_I$,
and has free kernel and finite cokernel for $n \geqslant d_I$.

We have 
\begin{equation}
\text{if $n$ is even, then } H^n(\LC^*_I(\a),\ZM)\simeq
\Coker\ (c_I(\a):H^{n-2}\to H^n)
\end{equation}
which is finite for $n \geqslant d_I$, and
\begin{equation}
\text{if $n$ is odd, then } H^n(\LC^*_I(\a),\ZM)\simeq
\Ker\ (c_I(\a):H^{n-1}\to H^{n+1})
\end{equation}
which is free (it is zero if $n \leqslant d_I - 1$).

\subsection{Resolution of singularities}\label{res sing}

Let $\Iti$ be the subset of $\D$ defined in \ref{Iti}.
There is a resolution of singularities (see for example the introduction
of \cite{KP2})

\begin{equation}
\fonctio{
\LC_\Iti(\alpt) = G \times_{P_\Iti} \CM_\alpt }{
\overline{\OC_\mini} = \OC_\mini \cup \{0\} }{
g*\l }{ \Ad g.(\l . e_\alpt)}
\end{equation}

It is the one mentioned in the introduction, with $P = P_\Iti$ and
$x_\mini = e_\alpt$.
It induces an isomorphism

\begin{equation}
\LC^*_\Iti(\alpt) = G \times_{P_\Iti} \CM^*_\alpt \simeq \OC_\mini
\end{equation}

Set $d = d_\Iti = 2 h^\vee -2$. For all integers $j$, let
$H^j$ denote $H^j(G/P_\Iti,\ZM)$.
For $\a\in\Phi_\longue$, we have $x_\a \in X_\Iti$, so
$Z_\a := Y_{\Iti,x_\a}$  is an element of $H^{2i}(G/P_\Iti,\ZM)$,
where $i = l(x_\a) = L(\a)$.
Then $H^*(G/P_\Iti,\ZM)$ is concentrated in even degrees, and
$H^{2i}(G/P_\Iti,\ZM)$ is free with basis $(Z_\a)_{\a\in \Phi_\longue^i}$.
Combining Theorem \ref{conclu racines} and the analysis of \ref{coh},
we get the following description of the cohomology of $\OC_\mini$ (the
highest root $\alpt$ is dominant and regular with respect to $P_\Iti$).

\begin{theo}\label{spectral roots}
We have
\begin{equation}
\text{if $n$ is even, then } H^n(\OC_\mini,\ZM)\simeq
\Coker\ (c_\Iti(\alpt):H^{n-2}\to H^n)
\end{equation}
which is finite for $n \geqslant d$, and
\begin{equation}
\text{if $n$ is odd, then } H^n(\OC_\mini,\ZM)\simeq
\Ker\ (c_\Iti(\alpt):H^{n-1}\to H^{n+1})
\end{equation}
which is free (it is zero if $n \leqslant d - 1$).

Moreover, if $\b\in\Phi_\longue^i$, then we have
\begin{equation}
c_\Iti(\alpt).Z_\b = 
\sum_{\b\elem{\g}\a} \left<\b, \g^\vee\right>  Z_\a
= \sum_{\a\in\Phi_\longue^{i + 1}} \partial_{\a\b} Z_\b
\end{equation}
where the $\partial_{\a\b}$ are the integers defined in
Theorem \ref{conclu racines}.
\end{theo}

As a consequence, we obtain the following results.

\begin{theo}\label{middle}
\begin{enumerate}[(i)]
\item
The middle cohomology of $\OC_\mini$ is given by
$$H^{2h^\vee - 2}(\OC_\mini, \ZM) \simeq P^\vee({\Phi'})/Q^\vee({\Phi'})$$
where $\Phi'$ is the root subsystem of $\Phi$ generated by $\D_\longue$,
and $P^\vee({\Phi'})$ (resp. $Q^\vee({\Phi'})$) is its coweight lattice
(resp. its coroot lattice).

\item
The rest of the cohomology of $\OC_\mini$ is as described
in Section \ref{case by case}. In particular,
if $\ell$ is a good prime for $G$, then there is no $\ell$-torsion
in the rest of the cohomology of $\OC_\mini$.
\end{enumerate}
\end{theo}

\begin{proof}
The map
$c_\Iti(\alpt) : H^{2 h^\vee - 4} \longto H^{2 h^\vee -2}$
is described as follows. By Theorem \ref{spectral roots},
the cohomology group $H^{2 h^\vee - 4}$ is free with basis
$(Z_\b)_{\b\in \D_\longue}$ (the long roots
of level $h^\vee - 2$ are of the long roots dual height $1$, so they are
the long simple roots). Similarly,
$H^{2 h^\vee - 4}$ is free with basis
$(Z_{-\a})_{\a\in \D_\longue}$. Besides, the matrix of
$c_\Iti(\alpt):H^{2 h^\vee - 4} \longto H^{2 h^\vee -2}$
in these bases is $(\partial_{-\a,\b})_{\a,\b\in\D_\longue}$.
We have $\partial_{-\a,\a} = 2$ for $\a\in\D_\longue$, and
for distinct $\a$ and $\b$ we have $\partial_{-\a,\b} = 1$ if
$\a + \b$ is a (long) root, $0$ otherwise.

Thus the matrix of $c_\Iti(\alpt):H^{2 h^\vee - 4} \longto H^{2 h^\vee -2}$
is the Cartan matrix of $\Phi'$ without minus signs.
This matrix is equivalent to the Cartan matrix of $\Phi'$:
since the Dynkin diagram of $\Phi'$ is a tree, it is bipartite.
We can write $\D_\longue = J \cup K$, so that all the edges 
in the Dynkin diagram of $\Phi'$ link an element of $J$ to an element of $K$.
If we replace the $Z_{\pm \a}$, $\a \in J$, by their opposites, 
then the matrix of $c_\Iti(\alpt) : H^{2 h^\vee - 4} \longto H^{2 h^\vee -2}$
becomes the Cartan matrix of $\Phi'$.

Now, the Cartan matrix of $\Phi'$ is transposed to the matrix
of the inclusion of $Q(\Phi')$ in $P(\Phi')$ in the bases
$\D_\longue$ and $(\varpi_\a)_{\a\in\D_\longue}$ (see
\cite[Chap. VI, \S 1.10]{BOUR456}), so it is in fact
the matrix of the inclusion of $Q^\vee(\Phi')$ in $P^\vee(\Phi')$
in the bases $(\b^\vee)_{\b\in\D_\longue}$ and
$(\varpi_{\a^\vee})_{\a\in\D_\longue}$.

The middle cohomology group $H^{2h^\vee - 2}(\OC_\mini, \ZM)$
is isomorphic to the cokernel of the map
$c_\Iti(\alpt):H^{2 h^\vee - 4} \longto H^{2 h^\vee -2}$.
This proves $(i)$.

Part $(ii)$ follows from a case-by-case analysis which will be done
in Section \ref{case by case}.
\end{proof}

\begin{remark}\label{poincare}
\emph{Besides, we have $\partial_{\a\b} = \partial_{-\b,-\a}$, so the
maps ``multiplication by $c_\Iti(\alpt)$''
in complementary degrees are transposed to each other.
This accounts for the fact that $\OC_\mini$ satisfies Poincar\'e duality,
since $\OC_\mini$ is homeomorphic to $\RM^+_*$ times a smooth compact
manifold of (real) dimension $2 h^\vee - 5$ (since we deal with
integral coefficients, one should take the derived dual for the Poincar\'e
duality).}
\end{remark}

\begin{remark}\label{rat}
\emph{
For the first half of the rational cohomology of $\OC_\mini$,
we find
\[
\bigoplus_{i=1}^k \QM[-2(d_i - 2)]
\]
where $k$ is the number of long simple roots, and
$d_1 \leqslant \ldots \leqslant d_k \leqslant \ldots \leqslant d_n$
are the degrees of $W$ ($n$ being the total number of simple roots).
This can be observed case by case, or related to the corresponding
Springer representation. The other half is determined by Poincar\'e
duality.
}
\end{remark}

\section{Case-by-case analysis}\label{case by case}

In the preceding section, we have explained how to compute the cohomology
of the minimal class in any given type in terms of root systems,
and we found a description of
the middle cohomology with a general proof.
However, for the rest of the cohomology, we need a case-by-case analysis.
It will appear that the primes dividing the torsion of the rest of the
cohomology are bad.
We have no \emph{a priori} explanation for this fact. Note that,
for the type $A$, we have an alternative method, which will be explained
in the next section.

For all types, first we give the Dynkin diagram, to fix the numbering
$(\a_i)_{1 \le i \le r}$ of the vertices, where $r$ denotes
the semisimple rank of $\gG$,
and to show the part $\Iti$ of $\D$ (see \ref{Iti}).
The corresponding vertices are represented in black. They are
exactly those that are not linked to the additional vertex in the
extended Dynkin diagram.

Then we give a diagram whose vertices are the positive long roots;
whenever $\b \elem{\g} \a$, we put an edge between $\b$ (above) and
$\a$ (below), and the multiplicity of the edge is equal to
$\partial_{\a\b} = \langle \b, \g^\vee \rangle$. In this diagram, the long root
$\sum_{i = 1}^r n_i \a_i$ (where the $n_i$ are non-negative integers)
is denoted by $n_1\ldots n_r$.
The roots in a given line appear in lexicographic order.

For $1 \leqslant i \leqslant d - 1$, let $\DC_i$ be the
matrix of the map $c_\Iti(\alpt):H^{2i - 2} \to H^{2i}$ 
in the bases $\Phi_\longue^{i - 1}$ and $\Phi_\longue^{i}$ (the roots
being ordered in lexicographic order, as in the diagram).
We give the matrices $\DC_i$
for $i = 1 \ldots h^\vee - 2$.
The matrix $\DC_{h^\vee - 1}$ is equal to
the Cartan matrix without minus signs of the root system $\Phi'$
(corresponding to $\D_\longue$).
The last matrices can be deduced
from the first ones by symmetry, since (by Remark \ref{poincare}) we have
$\DC_{d - i} = {}^t \DC_i$.

Then we give the cohomology of the minimal class with $\ZM$ coefficients
(one just has to compute the elementary divisors of the matrices $\DC_i$).

It will be useful to introduce some notation for the matrices in classical
types. Let $k$ be an integer. We set
\[
M(k) = 
\begin{pmatrix}
1 & 0      & \ldots & 0 & 0\\
1 & 1      & \ddots & \vdots & \vdots\\
0 & 1      & \ddots & 0 & 0\\
\vdots & \ddots & \ddots & 1 & 0\\
0 & \ldots & 0      & 1 & 1
\end{pmatrix}
\hspace{3cm}
N(k) = 
\begin{pmatrix}
1 & 0      & \ldots & 0\\
1 & 1      & \ddots & \vdots\\
0 & 1      & \ddots & 0\\
\vdots & \ddots & \ddots & 1\\
0 & \ldots & 0      & 1
\end{pmatrix}
\]
where $M(k)$ is a square matrix of size $k$, and $N(k)$ is
of size $(k + 1) \times k$.

Now let $k$ and $l$ be non-negative integers. For $i$ and $j$
any integers, we define a $k \times l$ matrix $E_{i,j}(k,l)$ as follows.
If $(i, j)$ is not in the range $[1,k] \times [1,l]$, then
we set $E_{i,j}(k,l) = 0$, otherwise it will denote the
$k\times l$ matrix whose only non-zero entry is a $1$ in the intersection
of line $i$ and column $j$. If the size of the matrix is clear from the
context, we will simply write $E_{i,j}$.

First, the calculations of the elementary divisors of the matrices 
$\DC_i$ were
done with GAP3 (see \cite{GAP}). We used the data on roots systems 
of the package CHEVIE. But actually, all the calculations can be
done by hand.

\subsection{Type $A_{n - 1}$}

\begin{center}
\begin{picture}(280,30)(-35,0)
\put(  5, 10){\circle{10}}
\put( 45, 10){\circle*{10}}
\put( 85, 10){\circle*{10}}
\put(165, 10){\circle*{10}}
\put(205, 10){\circle*{10}}
\put(245, 10){\circle{10}}
\put( 10, 10){\line(1,0){30}}
\put( 50, 10){\line(1,0){30}}
\put( 90, 10){\line(1,0){20}}
\put(120,9.5){$\ldots$}
\put(140, 10){\line(1,0){20}}
\put(170, 10){\line(1,0){30}}
\put(210, 10){\line(1,0){30}}
\put(  0, 20){$\a_1$}
\put( 40, 20){$\a_2$}
\put( 80, 20){$\a_3$}
\put(155, 20){$\a_{n - 3}$}
\put(195, 20){$\a_{n - 2}$}
\put(235, 20){$\a_{n - 1}$}
\end{picture}
\end{center}

We have $h = h^\vee = n$ and $d = 2n - 2$.

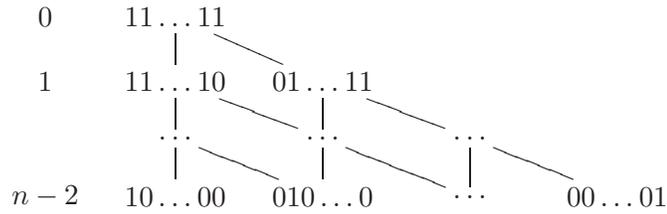
\[
\xymatrix @=.4cm{
0 &
11\ldots 11 \ar@{-}[d] \ar@{-}[dr]\\
1 &
11\ldots 10 \ar@{-}[d] \ar@{-}[dr]&
01\ldots 11 \ar@{-}[d] \ar@{-}[dr]&
{}\phantom{01\ldots 10}
\\
&
\ldots \ar@{-}[d] \ar@{-}[dr]&
\ldots \ar@{-}[d] \ar@{-}[dr]&
\ldots \ar@{-}[d] \ar@{-}[dr]\\
n - 2 &
10\ldots 00 &
010\ldots 0 &
\ldots &
00\ldots 01
}
\]

The odd cohomology of $G/P_\Iti$ is zero, and we have
\[
H^{2i}(G/P_\Iti) =
\begin{cases}
\ZM^{i + 1}
& \text{if $0 \le i \le n - 2$}\\
\ZM^{2n - 2 - i}
& \text{if $n - 1 \le i \le 2n - 3$}\\
0 & \text{otherwise}
\end{cases}
\]

For $1 \le i \le n - 2$, we have $\DC_i = N(i)$; the cokernel is isomorphic
to $\ZM$. We have
\[
\DC_{n - 1} =
\begin{pmatrix}
2      & 1      & 0      & \ldots & 0      \\
1      & 2      & 1      & \ddots & \vdots \\
0      & \ddots & \ddots & \ddots & 0      \\
\vdots & \ddots & 1      & 2      & 1      \\
0      & \ldots & 0      & 1      & 2
  \end{pmatrix}
\]
Its cokernel is isomorphic to $\ZM/n$. The last matrices are
transposed to the first ones, so the corresponding
maps are surjective. From this, we deduce the cohomology
of $\OC_\mini$ in type $A_{n - 1}$. We will see another method in
Section \ref{a bis}.

\[
H^i(\OC_\mini, \ZM) =
\begin{cases}
\ZM & \text{if $0 \le i \le 2n -4$ and $i$ is even,}\\
& \text{or $2n - 1 \le i \le 4n - 5$ and $i$ is odd}\\
\ZM/n & \text{if $i = 2n - 2$}\\
0 & \text{otherwise}
\end{cases}
\]

\subsection{Type $B_n$}

\begin{center}
\begin{picture}(280,30)(-35,0)
\put(  5, 10){\circle*{10}}
\put( 45, 10){\circle{10}}
\put( 85, 10){\circle*{10}}
\put(125, 10){\circle*{10}}
\put(205, 10){\circle*{10}}
\put(245, 10){\circle*{10}}
\put( 10, 10){\line(1,0){30}}
\put( 50, 10){\line(1,0){30}}
\put( 90, 10){\line(1,0){30}}
\put(130, 10){\line(1,0){20}}
\put(160,9.5){$\ldots$}
\put(180, 10){\line(1,0){20}}
\put(209, 11.5){\line(1,0){32}}
\put(209,  8.5){\line(1,0){32}}
\put(220,  5.5){\LARGE{$>$}}
\put(  0, 20){$\a_1$}
\put( 40, 20){$\a_2$}
\put( 80, 20){$\a_3$}
\put(120, 20){$\a_4$}
\put(195, 20){$\a_{n - 1}$}
\put(240, 20){$\a_n$}
\end{picture}
\end{center}

We have $h = 2n$, $h^\vee = 2n - 1$, and $d = 4n - 4$.

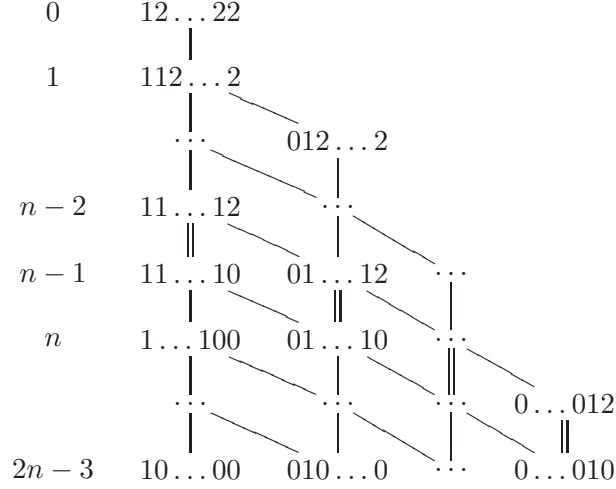
\[
\xymatrix @=.4cm{
0 & 
12\ldots 22 \ar@{-}[d]\\
1 & 
112\ldots 2 \ar@{-}[d] \ar@{-}[dr]\\
& 
\ldots \ar@{-}[d] \ar@{-}[dr]&
012\ldots 2 \ar@{-}[d]\\
n - 2 & 
11\ldots 12 \ar@{=}[d] \ar@{-}[dr]&
\ldots \ar@{-}[d] \ar@{-}[dr]&\\
n - 1 &
11\ldots 10 \ar@{-}[d] \ar@{-}[dr]&
01\ldots 12 \ar@{=}[d] \ar@{-}[dr]&
\ldots \ar@{-}[d]\\
n &
1\ldots 100 \ar@{-}[d] \ar@{-}[dr]&
01\ldots 10 \ar@{-}[d] \ar@{-}[dr]&
\ldots \ar@{=}[d] \ar@{-}[dr]\\
&
\ldots \ar@{-}[d] \ar@{-}[dr]&
\ldots \ar@{-}[d] \ar@{-}[dr]&
\ldots \ar@{-}[d] \ar@{-}[dr]&
0\ldots 012 \ar@{=}[d]\\
2n - 3 &
10\ldots 00&
010\ldots 0&
\ldots&
0\ldots 010
}
\]

There is a gap at each even line (the length of the line
increases by one). The diagram can be a little bit misleading if
$n$ is even: in that case, there is a gap at the line $n - 2$.
Let us now describe the matrices $\DC_i$.

First suppose $1 \le i \le n - 2$. If $i$ is odd, then we have
$\DC_i = M\left(\tfrac{i+1}{2}\right)$ (an isomorphism); 
if $i$ is even, then we have
$\DC_i = N\left(\tfrac{i}{2}\right)$ and the cokernel is isomorphic to $\ZM$.

Now suppose $n - 1 \le i \le 2n - 3$. If $i$ is odd,
then we have
$\DC_i = M\left(\tfrac{i+1}{2}\right) +  E_{i + 2 - n, i + 2 - n}$
and the cokernel is isomorphic to $\ZM/2$.
If $i$ is even, then we have
$\DC_i = N\left(\tfrac{i}{2}\right) + E_{i + 2 - n, i + 2 - n}$
and the cokernel is isomorphic to $\ZM$.

The long simple roots generate a root system of type $A_{n - 1}$.
Thus the matrix $\DC_{2n - 2}$ is the Cartan matrix without minus
signs of type $A_{n - 1}$, which has cokernel $\ZM/n$.

So the cohomology of $\OC_\mini$ is described as follows.

\[ H^i(\OC_\mini, \ZM) \simeq
\begin{cases}
\ZM & \text{if $0 \le i \le 4n - 8$ and $i \equiv 0 \mod 4$,}\\
& \text{or $4n - 1 \le i \le 8n  - 9$ and $i \equiv -1 \mod 4$}\\
\ZM/2 & \text{$2n - 2 \le i \le 6n - 6$ and $i \equiv 2 \mod 4$}\\
\ZM/n & \text{if $i = 4n - 4$}\\
0 & \text{otherwise}
\end{cases}
\]

\subsection{Type $C_n$}

\begin{center}
\begin{picture}(280,30)(-35,0)
\put(  5, 10){\circle{10}}
\put( 45, 10){\circle*{10}}
\put( 85, 10){\circle*{10}}
\put(125, 10){\circle*{10}}
\put(205, 10){\circle*{10}}
\put(245, 10){\circle*{10}}
\put( 10, 10){\line(1,0){30}}
\put( 50, 10){\line(1,0){30}}
\put( 90, 10){\line(1,0){30}}
\put(130, 10){\line(1,0){20}}
\put(160,9.5){$\ldots$}
\put(180, 10){\line(1,0){20}}
\put(209, 11.5){\line(1,0){32}}
\put(209,  8.5){\line(1,0){32}}
\put(220,  5.5){\LARGE{$<$}}
\put(  0, 20){$\a_1$}
\put( 40, 20){$\a_2$}
\put( 80, 20){$\a_3$}
\put(120, 20){$\a_4$}
\put(195, 20){$\a_{n - 1}$}
\put(240, 20){$\a_n$}
\end{picture}
\end{center}

We have $h = 2n$, $h^\vee = n + 1$, and $d = 2n$.
The root system $\Phi'$ is of type $A_1$. Its Cartan matrix is $(2)$.

\[
\xymatrix @=.5cm{
0 &
22\ldots 21 \ar@{=}[d]\\
1 &
02\ldots 21 \ar@{=}[d]\\
&
\vdots      \ar@{=}[d]\\
n - 2 &
0\ldots 021 \ar@{=}[d]\\
n - 1 &
00\ldots 01
}
\]

The matrices $\DC_i$ are all equal to $(2)$.

\[
H^i(\OC_\mini, \ZM) \simeq
\begin{cases}
\ZM & \text{if $i = 0$ or $4n - 1$}\\
\ZM/2 & \text{if $2 \le i \le 4n - 2$ and $i$ is even}\\
0 & \text{otherwise}
\end{cases}
\]

\subsection{Type $D_n$}

\begin{center}
\begin{picture}(280,80)(-35,-20)
\put(  5, 10){\circle*{10}}
\put( 45, 10){\circle{10}}
\put( 85, 10){\circle*{10}}
\put(125, 10){\circle*{10}}
\put(205, 10){\circle*{10}}
\put(245, 30){\circle*{10}}
\put(245,-10){\circle*{10}}
\put( 10, 10){\line(1,0){30}}
\put( 50, 10){\line(1,0){30}}
\put( 90, 10){\line(1,0){30}}
\put(130, 10){\line(1,0){20}}
\put(160,9.5){$\ldots$}
\put(180, 10){\line(1,0){20}}
\put(209, 12){\line(2,1){32}}
\put(209,  8){\line(2,-1){32}}
\put(  0, 20){$\a_1$}
\put( 40, 20){$\a_2$}
\put( 80, 20){$\a_3$}
\put(120, 20){$\a_4$}
\put(195, 20){$\a_{n - 2}$}
\put(255, 28){$\a_{n - 1}$}
\put(255,-13){$\a_n$}
\end{picture}
\end{center}

We have $h = h^\vee = 2n - 2$, and $d = 4n -6$.
We have
\[
\DC_{2n-3} =
\begin{pmatrix}
2      & 1      & 0      & \ldots & 0      & 0\\
1      & 2      & \ddots & \ddots & \vdots & \vdots\\
0      & \ddots & \ddots & 1      & 0      & 0\\
\vdots & \ddots & 1      & 2      & 1      & 1\\
0      & \ldots & 0      & 1      & 2      & 0\\
0      & \ldots & 0      & 1      & 0      & 2
\end{pmatrix}
\]
Its cokernel is $(\ZM/2)^2$ when $n$ is even, $\ZM/4$ when $n$ is odd.

As in the $B_n$ case, the reader should be warned that there is a gap
at line $n - 4$ if $n$ is even. Besides, not all dots are meaningful.
The entries $0\ldots 01211$ and $00\ldots 0111$ are on the right
diagonal, but usually they are not on the lines $n - 1$ and $n$.

{\scriptsize
\[
\xymatrix @=.3cm{
0&
&122\ldots 211 \ar@{-}[d]\\
1&
&112\ldots 211 \ar@{-}[d] \ar@{-}[dr]\\
&
&\ldots \ar@{-}[d] \ar@{-}[dr]&
012\ldots 211  \ar@{-}[d]\\
n - 4&
&11\ldots 1211 \ar@{-}[d] \ar@{-}[dr]&
\ldots \ar@{-}[d] \ar@{-}[dr]\\
n - 3&
&111\ldots 111 \ar@{-}[dl] \ar@{-}[d] \ar@{-}[dr]&
01\ldots 1211 \ar@{-}[d] \ar@{-}[dr]&
\ldots \ar@{-}[d]\\
n - 2&
111\ldots 110 \ar@{-}[d] \ar@{-}[dr]&
111\ldots 101 \ar@{-}[dl] \ar@{-}[dr]&
011\ldots 111 \ar@{-}[dl] \ar@{-}[d] \ar@{-}[dr]&
\ldots \ar@{-}[d] \ar@{-}[dr]\\
n - 1&
111\ldots 100 \ar@{-}[d] \ar@{-}[dr]&
011\ldots 110 \ar@{-}[d] \ar@{-}[dr]&
011\ldots 101 \ar@{-}[dl] \ar@{-}[dr]&
\dots \ar@{-}[dl] \ar@{-}[d] \ar@{-}[dr]&
0\ldots 01211  \ar@{-}[d]\\
&
\ldots \ar@{-}[d] \ar@{-}[dr]&
011\ldots 100 \ar@{-}[d] \ar@{-}[dr]&
001\ldots 110 \ar@{-}[d] \ar@{-}[dr]&
\ldots \ar@{-}[dl] \ar@{-}[dr]&
00\ldots 0111 \ar@{-}[dl] \ar@{-}[d] \ar@{-}[dr]\\
2n - 6&
1110\ldots 00 \ar@{-}[d] \ar@{-}[dr]&
\ldots \ar@{-}[d] \ar@{-}[dr]&
\ldots \ar@{-}[d] \ar@{-}[dr]&
\ldots \ar@{-}[d] \ar@{-}[dr]&
0\ldots 01101 \ar@{-}[dl] \ar@{-}[dr]&
00\ldots 0111 \ar@{-}[dl] \ar@{-}[d]\\
2n - 5&
110\ldots 000 \ar@{-}[d] \ar@{-}[dr]&
0110\ldots 00 \ar@{-}[d] \ar@{-}[dr]&
\ldots \ar@{-}[d] \ar@{-}[dr]&
0\ldots 01100 \ar@{-}[d] \ar@{-}[dr]&
00\ldots 0110 \ar@{-}[d] \ar@{-}[dr]&
00\ldots 0101 \ar@{-}[dl] \ar@{-}[dr]\\
2n - 4&
100\ldots 000&
010\ldots 000&
0010\ldots 00&
\ldots&
00\ldots 0100&
000\ldots 010&
000\ldots 001
}
\]
}

First suppose $i \le i \le n - 3$. We have
\[
\DC_i =
\begin{cases}
M\left(\tfrac{i + 1}{2}\right)
& \text{if $i$ is odd}\\
N\left(\tfrac{i}{2}\right)
& \text{if $i$ is even}
\end{cases}
\]
Then the cokernel is zero if $i$ is odd, $\ZM$ if $i$ is even.

Let $V$ be the $1 \times \tfrac{n - 1}{2}$ matrix $(1,0,\ldots,0)$.
We have
\[\DC_{n - 2} =
\begin{cases}
\begin{pmatrix}
V\\
N\left(\frac{n - 2}{2}\right)
\end{pmatrix}
& \text{if $n$ is even}\\
\\
\begin{pmatrix}
V\\
M\left(\frac{n - 1}{2}\right)
\end{pmatrix}
& \text{if $n$ is odd}\\
\end{cases}
\]
The cokernel is $\ZM^2$ if $n$ is even, $\ZM$ if $i$ is odd.

Now suppose $n - 1 \le i \le 2n - 4$. We have
\[
\DC_i =
\begin{cases}
M\left(\tfrac{i + 3}{2}\right)
+ E_{i + 2 - n, i + 3 - n}
- E_{i + 3 - n, i + 3 - n}
+ E_{i + 3 - n, i + 4 - n}
& \text{if $i$ is odd}\\
N\left(\tfrac{i + 2}{2}\right)
+ E_{i + 2 - n, i + 3 - n}
- E_{i + 3 - n, i + 3 - n}
+ E_{i + 3 - n, i + 4 - n}
& \text{if $i$ is even}\\
\end{cases}
\]
Then the cokernel is $\ZM/2$ if $i$ is odd, $\ZM$ if $i$ is even.

\[
H^i(\OC_\mini, \ZM) \simeq
\begin{cases}
\ZM & \text{if $0 \le i \le 4n - 8$ and $i \equiv 0 \mod 4$,}\\
& \text{or $4n - 5 \le i \le 8n - 13$ and $i \equiv -1 \mod 4$;}\\
\ZM/2 & \text{if $2n - 4 < i < 4n - 6$ and $i \equiv 2 \mod 4$;}\\
& \text{or $4n - 6 < i < 6n - 8$ and $i \equiv 2 \mod 4$;}\\
(\ZM/2)^2 & \text{if $i = 4n - 6$ and $n$ is even;}\\
\ZM/4 & \text{if $i = 4n - 6$ and $n$ is odd;}\\
0 & \text{otherwise.}
\end{cases}
\oplus
\ZM\ 
\begin{array}{ll}
\text{if $i = 2n - 4$}\\
\text{or $i =  6n - 7$.}
\end{array}
\]

\subsection{Type $E_6$}

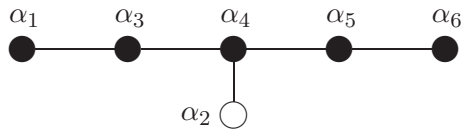
\begin{center}
\begin{picture}(200,45)
\put( 40, 30){\circle*{10}}
\put( 45, 30){\line(1,0){30}}
\put( 80, 30){\circle*{10}}
\put( 85, 30){\line(1,0){30}}
\put(120, 30){\circle*{10}}
\put(125, 30){\line(1,0){30}}
\put(160, 30){\circle*{10}}
\put(165, 30){\line(1,0){30}}
\put(200, 30){\circle*{10}}
\put(120,  5){\circle{10}}
\put(120, 25){\line(0,-1){15}}
\put( 35, 40){$\a_1$}
\put( 75, 40){$\a_3$}
\put(115, 40){$\a_4$}
\put(155, 40){$\a_5$}
\put(195, 40){$\a_6$}
\put(100,  3){$\a_2$}
\end{picture}
\end{center}

We have $h = h^\vee = 12$, and $d = 22$. The Cartan matrix has cokernel
isomorphic to $\ZM/3$.

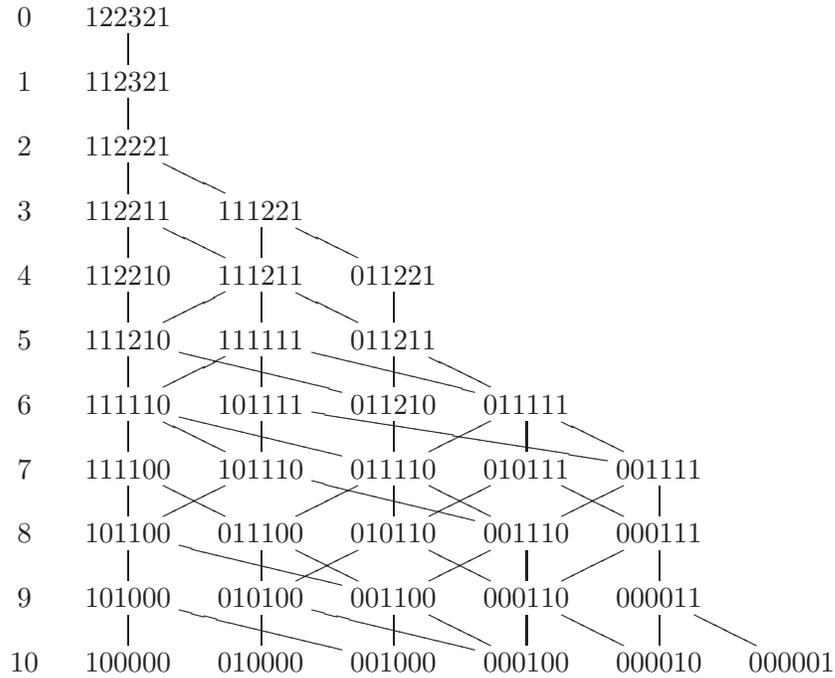
\[
\xymatrix @=.4cm{
0 &
122321\ar@{-}[d]\\
1 &
112321\ar@{-}[d]\\
2 &
112221\ar@{-}[d]\ar@{-}[dr]\\
3 &
112211\ar@{-}[d]\ar@{-}[dr]&
111221\ar@{-}[d]\ar@{-}[dr]\\
4 &
112210\ar@{-}[d]&
111211\ar@{-}[dl]\ar@{-}[d]\ar@{-}[dr]&
011221\ar@{-}[d]\\
5 &
111210\ar@{-}[d]\ar@{-}[drr]&
111111\ar@{-}[dl]\ar@{-}[d]\ar@{-}[drr]&
011211\ar@{-}[d]\ar@{-}[dr]\\
6 &
111110\ar@{-}[d]\ar@{-}[dr]\ar@{-}[drr]&
101111\ar@{-}[d]\ar@{-}[drrr]&
011210\ar@{-}[d]&
011111\ar@{-}[dl]\ar@{-}[d]\ar@{-}[dr]\\
7 &
111100\ar@{-}[d]\ar@{-}[dr]&
101110\ar@{-}[dl]\ar@{-}[drr]&
011110\ar@{-}[dl]\ar@{-}[d]\ar@{-}[dr]&
010111\ar@{-}[dl]\ar@{-}[dr]&
001111\ar@{-}[dl]\ar@{-}[d]\\
8 &
101100\ar@{-}[d]\ar@{-}[drr]&
011100\ar@{-}[d]\ar@{-}[dr]&
010110\ar@{-}[dl]\ar@{-}[dr]&
001110\ar@{-}[dl]\ar@{-}[d]&
000111\ar@{-}[dl]\ar@{-}[d]\\
9 &
101000\ar@{-}[d]\ar@{-}[drr]&
010100\ar@{-}[d]\ar@{-}[drr]&
001100\ar@{-}[d]\ar@{-}[dr]&
000110\ar@{-}[d]\ar@{-}[dr]&
000011\ar@{-}[d]\ar@{-}[dr]\\
10&
100000&
010000&
001000&
000100&
000010&
000001
}
\]

\[
\DC_1 = 
\DC_2 = \begin{pmatrix}  1 \end{pmatrix}\quad
\DC_3 = \begin{pmatrix}  1 \\
     1 \end{pmatrix}\quad
\DC_4 = \begin{pmatrix}  1 & 0 \\
     1 & 1 \\
     0 & 1 \end{pmatrix}\quad
\DC_5 = \begin{pmatrix}  1 & 1 & 0 \\
     0 & 1 & 0 \\
     0 & 1 & 1 \end{pmatrix}\] \[
\DC_6 = \begin{pmatrix}  1 & 1 & 0 \\
     0 & 1 & 0 \\
     1 & 0 & 1 \\
     0 & 1 & 1 \end{pmatrix}\quad
\DC_7 = \begin{pmatrix}  1 & 0 & 0 & 0 \\
     1 & 1 & 0 & 0 \\
     1 & 0 & 1 & 1 \\
     0 & 0 & 0 & 1 \\
     0 & 1 & 0 & 1 \end{pmatrix}\quad
\DC_8 = \begin{pmatrix}  1 & 1 & 0 & 0 & 0 \\
     1 & 0 & 1 & 0 & 0 \\
     0 & 0 & 1 & 1 & 0 \\
     0 & 1 & 1 & 0 & 1 \\
     0 & 0 & 0 & 1 & 1 \end{pmatrix}\] \[
\DC_9 = \begin{pmatrix}  1 & 0 & 0 & 0 & 0 \\
     0 & 1 & 1 & 0 & 0 \\
     1 & 1 & 0 & 1 & 0 \\
     0 & 0 & 1 & 1 & 1 \\
     0 & 0 & 0 & 0 & 1 \end{pmatrix}\quad
\DC_{10} = \begin{pmatrix}  1 & 0 & 0 & 0 & 0 \\
     0 & 1 & 0 & 0 & 0 \\
     1 & 0 & 1 & 0 & 0 \\
     0 & 1 & 1 & 1 & 0 \\
     0 & 0 & 0 & 1 & 1 \\
     0 & 0 & 0 & 0 & 1 \end{pmatrix}
\]

\[
H^i(\OC_\mini, \ZM) \simeq
\begin{cases}
\ZM & \text{for $i = 0,\ 6,\ 8,\ 12,\ 14,\ 20,\ 23,\ 29,\ 31,\ 35,\ 37,\ 43$}\\
\ZM/3 & \text{for $i = 16,\ 22,\ 28$}\\
\ZM/2 & \text{for $i = 18,\ 26$}\\
0 & \text{otherwise}
\end{cases}
\]

\subsection{Type $E_7$}

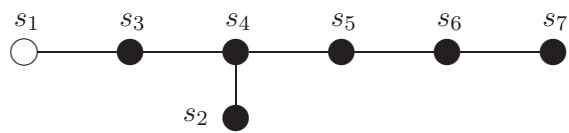
\begin{center}
\begin{picture}(240,45)
\put( 40, 30){\circle{10}}
\put( 45, 30){\line(1,0){30}}
\put( 80, 30){\circle*{10}}
\put( 85, 30){\line(1,0){30}}
\put(120, 30){\circle*{10}}
\put(125, 30){\line(1,0){30}}
\put(160, 30){\circle*{10}}
\put(165, 30){\line(1,0){30}}
\put(200, 30){\circle*{10}}
\put(205, 30){\line(1,0){30}}
\put(240, 30){\circle*{10}}
\put(120, 5){\circle*{10}}
\put(120, 25){\line(0,-1){15}}
\put( 36, 40){$s_1$}
\put( 76, 40){$s_3$}
\put(116, 40){$s_4$}
\put(156, 40){$s_5$}
\put(196, 40){$s_6$}
\put(236, 40){$s_7$}
\put(100, 4){$s_2$}
\end{picture}
\end{center}

We have $h = h^\vee = 18$, and $d = 34$. The Cartan matrix has cokernel
isomorphic to $\ZM/2$.

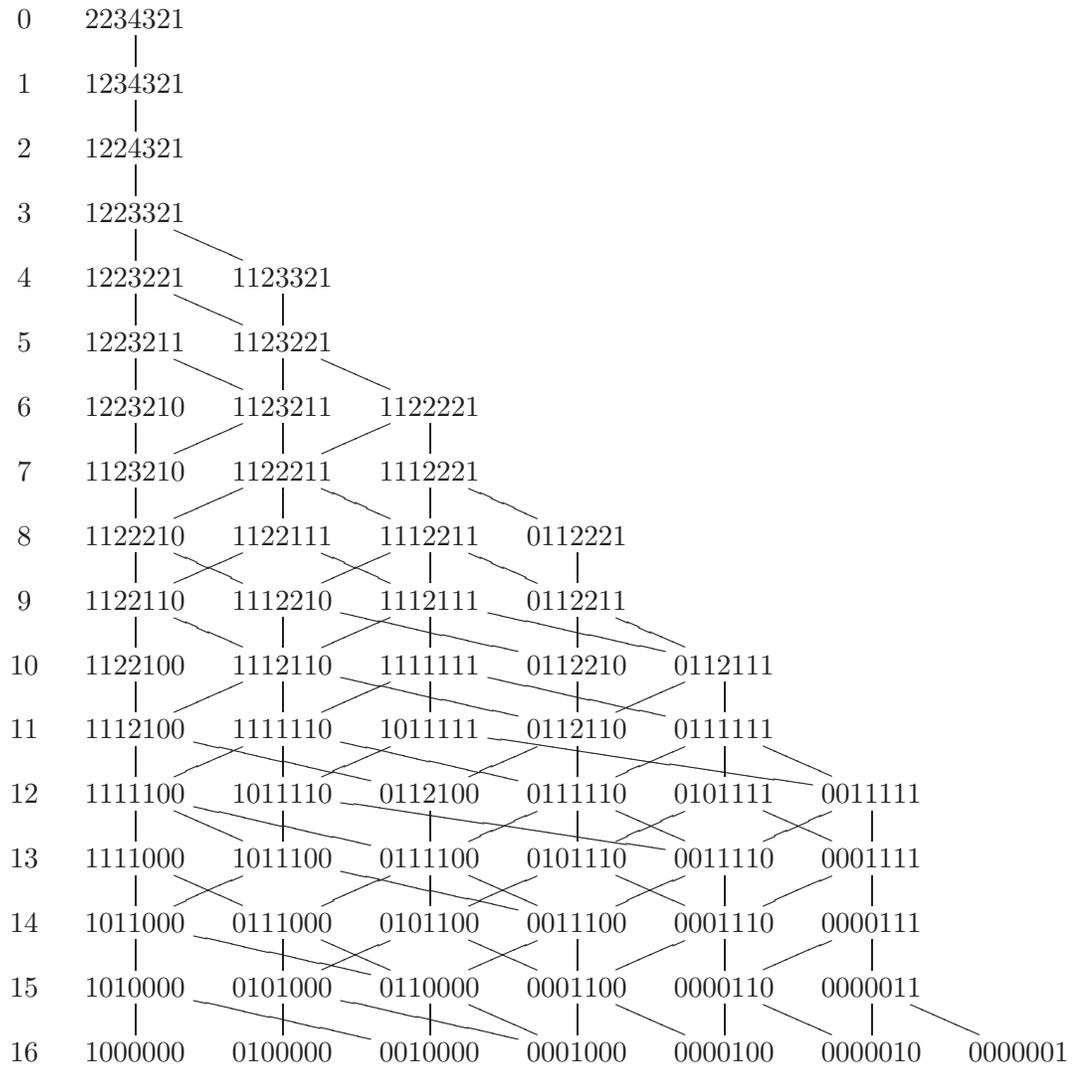
\[
\xymatrix @=.4cm{
0 &
2234321\ar@{-}[d]\\
1 &
1234321\ar@{-}[d]\\
2 &
1224321\ar@{-}[d]\\
3 &
1223321\ar@{-}[d]\ar@{-}[dr]\\
4 &
1223221\ar@{-}[d]\ar@{-}[dr]&
1123321\ar@{-}[d]\\
5 &
1223211\ar@{-}[d]\ar@{-}[dr]&
1123221\ar@{-}[d]\ar@{-}[dr]\\
6 &
1223210\ar@{-}[d]&
1123211\ar@{-}[dl]\ar@{-}[d]&
1122221\ar@{-}[dl]\ar@{-}[d]\\
7 &
1123210\ar@{-}[d]&
1122211\ar@{-}[dl]\ar@{-}[d]\ar@{-}[dr]&
1112221\ar@{-}[d]\ar@{-}[dr]\\
8 &
1122210\ar@{-}[d]\ar@{-}[dr]&
1122111\ar@{-}[dl]\ar@{-}[dr]&
1112211\ar@{-}[dl]\ar@{-}[d]\ar@{-}[dr]&
0112221\ar@{-}[d]\\
9 &
1122110\ar@{-}[d]\ar@{-}[dr]&
1112210\ar@{-}[d]\ar@{-}[drr]&
1112111\ar@{-}[dl]\ar@{-}[d]\ar@{-}[drr]&
0112211\ar@{-}[d]\ar@{-}[dr]\\
10&
1122100\ar@{-}[d]&
1112110\ar@{-}[dl]\ar@{-}[d]\ar@{-}[drr]&
1111111\ar@{-}[dl]\ar@{-}[d]\ar@{-}[drr]&
0112210\ar@{-}[d]&
0112111\ar@{-}[dl]\ar@{-}[d]\\
11&
1112100\ar@{-}[d]\ar@{-}[drr]&
1111110\ar@{-}[dl]\ar@{-}[d]\ar@{-}[drr]&
1011111\ar@{-}[dl]\ar@{-}[drrr]&
0112110\ar@{-}[dl]\ar@{-}[d]&
0111111\ar@{-}[dl]\ar@{-}[d]\ar@{-}[dr]\\
12&
1111100\ar@{-}[d]\ar@{-}[dr]\ar@{-}[drr]&
1011110\ar@{-}[d]\ar@{-}[drrr]&
0112100\ar@{-}[d]&
0111110\ar@{-}[dl]\ar@{-}[d]\ar@{-}[dr]&
0101111\ar@{-}[dl]\ar@{-}[dr]&
0011111\ar@{-}[dl]\ar@{-}[d]\\
13&
1111000\ar@{-}[d]\ar@{-}[dr]&
1011100\ar@{-}[dl]\ar@{-}[drr]&
0111100\ar@{-}[dl]\ar@{-}[d]\ar@{-}[dr]&
0101110\ar@{-}[dl]\ar@{-}[dr]&
0011110\ar@{-}[dl]\ar@{-}[d]&
0001111\ar@{-}[dl]\ar@{-}[d]\\
14&
1011000\ar@{-}[d]\ar@{-}[drr]&
0111000\ar@{-}[d]\ar@{-}[dr]&
0101100\ar@{-}[dl]\ar@{-}[dr]&
0011100\ar@{-}[dl]\ar@{-}[d]&
0001110\ar@{-}[dl]\ar@{-}[d]&
0000111\ar@{-}[dl]\ar@{-}[d]\\
15&
1010000\ar@{-}[d]\ar@{-}[drr]&
0101000\ar@{-}[d]\ar@{-}[drr]&
0110000\ar@{-}[d]\ar@{-}[dr]&
0001100\ar@{-}[d]\ar@{-}[dr]&
0000110\ar@{-}[d]\ar@{-}[dr]&
0000011\ar@{-}[d]\ar@{-}[dr]\\
16&
1000000&
0100000&
0010000&
0001000&
0000100&
0000010&
0000001
}
\]

\newpage

\[ \DC_1 = \DC_2 = \DC_3 = \begin{pmatrix}  1 \end{pmatrix}\quad
\DC_4 = \begin{pmatrix}  1 \\
     1 \end{pmatrix}\quad
\DC_5 = \begin{pmatrix}  1 & 0 \\
     1 & 1 \end{pmatrix}\quad
\DC_6 = \begin{pmatrix}  1 & 0 \\
     1 & 1 \\
     0 & 1 \end{pmatrix}\] \[
\DC_7 = \begin{pmatrix}  1 & 1 & 0 \\
     0 & 1 & 1 \\
     0 & 0 & 1 \end{pmatrix}\quad
\DC_8 = \begin{pmatrix}  1 & 1 & 0 \\
     0 & 1 & 0 \\
     0 & 1 & 1 \\
     0 & 0 & 1 \end{pmatrix}\quad
\DC_9 = \begin{pmatrix}  1 & 1 & 0 & 0 \\
     1 & 0 & 1 & 0 \\
     0 & 1 & 1 & 0 \\
     0 & 0 & 1 & 1 \end{pmatrix}\quad
\DC_{10} = \begin{pmatrix}  1 & 0 & 0 & 0 \\
     1 & 1 & 1 & 0 \\
     0 & 0 & 1 & 0 \\
     0 & 1 & 0 & 1 \\
     0 & 0 & 1 & 1 \end{pmatrix}\] \[
\DC_{11} = \begin{pmatrix}  1 & 1 & 0 & 0 & 0 \\
     0 & 1 & 1 & 0 & 0 \\
     0 & 0 & 1 & 0 & 0 \\
     0 & 1 & 0 & 1 & 1 \\
     0 & 0 & 1 & 0 & 1 \end{pmatrix}\quad
\DC_{12} = \begin{pmatrix}  1 & 1 & 0 & 0 & 0 \\
     0 & 1 & 1 & 0 & 0 \\
     1 & 0 & 0 & 1 & 0 \\
     0 & 1 & 0 & 1 & 1 \\
     0 & 0 & 0 & 0 & 1 \\
     0 & 0 & 1 & 0 & 1 \end{pmatrix}\] \[
\DC_{13} = \begin{pmatrix}  1 & 0 & 0 & 0 & 0 & 0 \\
     1 & 1 & 0 & 0 & 0 & 0 \\
     1 & 0 & 1 & 1 & 0 & 0 \\
     0 & 0 & 0 & 1 & 1 & 0 \\
     0 & 1 & 0 & 1 & 0 & 1 \\
     0 & 0 & 0 & 0 & 1 & 1 \end{pmatrix}\quad
\DC_{14} = \begin{pmatrix}  1 & 1 & 0 & 0 & 0 & 0 \\
     1 & 0 & 1 & 0 & 0 & 0 \\
     0 & 0 & 1 & 1 & 0 & 0 \\
     0 & 1 & 1 & 0 & 1 & 0 \\
     0 & 0 & 0 & 1 & 1 & 1 \\
     0 & 0 & 0 & 0 & 0 & 1 \end{pmatrix}\] \[
\DC_{15} = \begin{pmatrix}  1 & 0 & 0 & 0 & 0 & 0 \\
     0 & 1 & 1 & 0 & 0 & 0 \\
     1 & 1 & 0 & 1 & 0 & 0 \\
     0 & 0 & 1 & 1 & 1 & 0 \\
     0 & 0 & 0 & 0 & 1 & 1 \\
     0 & 0 & 0 & 0 & 0 & 1 \end{pmatrix}\quad
\DC_{16} = \begin{pmatrix}  1 & 0 & 0 & 0 & 0 & 0 \\
     0 & 1 & 0 & 0 & 0 & 0 \\
     1 & 0 & 1 & 0 & 0 & 0 \\
     0 & 1 & 1 & 1 & 0 & 0 \\
     0 & 0 & 0 & 1 & 1 & 0 \\
     0 & 0 & 0 & 0 & 1 & 1 \\
     0 & 0 & 0 & 0 & 0 & 1 \end{pmatrix}
\]

\[
H^i(\OC_\mini, \ZM) \simeq
\begin{cases}
\ZM & \text{for } i = 0,\ 8,\ 12,\ 16,\ 20,\ 24,\ 32,\\ 
    & \qquad \quad 35,\ 43,\ 47,\ 51,\ 55,\ 59,\ 67\\
\ZM/2 & \text{for } i = 18,\ 26,\ 30,\ 34,\ 38,\ 42,\ 50\\
\ZM/3 & \text{for } i = 28,\ 40\\
0 & \text{otherwise}
\end{cases}
\]

\subsection{Type $E_8$}

\begin{center}
\begin{picture}(280,45)
\put( 40, 30){\circle*{10}}
\put( 45, 30){\line(1,0){30}}
\put( 80, 30){\circle*{10}}
\put( 85, 30){\line(1,0){30}}
\put(120, 30){\circle*{10}}
\put(125, 30){\line(1,0){30}}
\put(160, 30){\circle*{10}}
\put(165, 30){\line(1,0){30}}
\put(200, 30){\circle*{10}}
\put(205, 30){\line(1,0){30}}
\put(240, 30){\circle*{10}}
\put(245, 30){\line(1,0){30}}
\put(280, 30){\circle{10}}
\put(120,  5){\circle*{10}}
\put(120, 25){\line(0,-1){15}}
\put( 35, 40){$\a_1$}
\put( 75, 40){$\a_3$}
\put(115, 40){$\a_4$}
\put(155, 40){$\a_5$}
\put(195, 40){$\a_6$}
\put(235, 40){$\a_7$}
\put(275, 40){$\a_8$}
\put(100,  3){$\a_2$}
\end{picture}
\end{center}

{\scriptsize
\[
\xymatrix @=.3cm{
0 &
23465432\ar@{-}[d]\\
1 &
23465431\ar@{-}[d]\\
2 &
23465421\ar@{-}[d]\\
3 &
23465321\ar@{-}[d]\\
4 &
23464321\ar@{-}[d]\\
5 &
23454321\ar@{-}[d]\ar@{-}[dr]\\
6 &
23354321\ar@{-}[d]\ar@{-}[dr]&
22454321\ar@{-}[dl]\\
7 &
22354321\ar@{-}[d]\ar@{-}[dr]&
13354321\ar@{-}[d]\\
8 &
22344321\ar@{-}[d]\ar@{-}[dr]&
12354321\ar@{-}[d]\\
9 &
22343321\ar@{-}[d]\ar@{-}[dr]&
12344321\ar@{-}[d]\ar@{-}[dr]\\
10&
22343221\ar@{-}[d]\ar@{-}[dr]&
12343321\ar@{-}[d]\ar@{-}[dr]&
12244321\ar@{-}[d]\\
11&
22343211\ar@{-}[d]\ar@{-}[dr]&
12343221\ar@{-}[d]\ar@{-}[dr]&
12243321\ar@{-}[d]\ar@{-}[dr]\\
12&
22343210\ar@{-}[d]&
12343211\ar@{-}[dl]\ar@{-}[d]&
12243221\ar@{-}[dl]\ar@{-}[d]&
12233321\ar@{-}[dl]\ar@{-}[d]\\
13&
12343210\ar@{-}[d]&
12243211\ar@{-}[dl]\ar@{-}[d]&
12233221\ar@{-}[dl]\ar@{-}[d]\ar@{-}[dr]&
11233321\ar@{-}[d]\\
14&
12243210\ar@{-}[d]&
12233211\ar@{-}[dl]\ar@{-}[d]\ar@{-}[dr]&
12232221\ar@{-}[dl]\ar@{-}[dr]&
11233221\ar@{-}[dl]\ar@{-}[d]\\
15&
12233210\ar@{-}[d]\ar@{-}[drr]&
12232211\ar@{-}[dl]\ar@{-}[d]\ar@{-}[drr]&
11233211\ar@{-}[d]\ar@{-}[dr]&
11232221\ar@{-}[d]\ar@{-}[dr]\\
16&
12232210\ar@{-}[d]\ar@{-}[dr]&
12232111\ar@{-}[dl]\ar@{-}[dr]&
11233210\ar@{-}[dl]&
11232211\ar@{-}[dll]\ar@{-}[dl]\ar@{-}[d]&
11222221\ar@{-}[dl]\ar@{-}[d]\\
17&
12232110\ar@{-}[d]\ar@{-}[dr]&
11232210\ar@{-}[d]\ar@{-}[dr]&
11232111\ar@{-}[dl]\ar@{-}[dr]&
11222211\ar@{-}[dl]\ar@{-}[d]\ar@{-}[dr]&
11122221\ar@{-}[d]\ar@{-}[dr]\\
18&
12232100\ar@{-}[d]&
11232110\ar@{-}[dl]\ar@{-}[d]&
11222210\ar@{-}[dl]\ar@{-}[dr]&
11222111\ar@{-}[dll]\ar@{-}[dl]\ar@{-}[dr]&
11122211\ar@{-}[dl]\ar@{-}[d]\ar@{-}[dr]&
01122221\ar@{-}[d]\\
19&
11232100\ar@{-}[d]&
11222110\ar@{-}[dl]\ar@{-}[d]\ar@{-}[dr]&
11221111\ar@{-}[dl]\ar@{-}[dr]&
11122210\ar@{-}[dl]\ar@{-}[dr]&
11122111\ar@{-}[dll]\ar@{-}[dl]\ar@{-}[dr]&
01122211\ar@{-}[dl]\ar@{-}[d]\\
20&
11222100\ar@{-}[d]\ar@{-}[dr]&
11221110\ar@{-}[dl]\ar@{-}[dr]&
11122110\ar@{-}[dl]\ar@{-}[d]\ar@{-}[drr]&
11121111\ar@{-}[dl]\ar@{-}[d]\ar@{-}[drr]&
01122210\ar@{-}[d]&
01122111\ar@{-}[dl]\ar@{-}[d]\\
21&
11221100\ar@{-}[d]\ar@{-}[dr]&
11122100\ar@{-}[d]\ar@{-}[drrr]&
11121110\ar@{-}[dl]\ar@{-}[d]\ar@{-}[drrr]&
11111111\ar@{-}[dl]\ar@{-}[d]\ar@{-}[drrr]&
01122110\ar@{-}[d]\ar@{-}[dr]&
01121111\ar@{-}[d]\ar@{-}[dr]\\
22&
11221000\ar@{-}[d]&
11121100\ar@{-}[dl]\ar@{-}[d]\ar@{-}[drr]&
11111110\ar@{-}[dl]\ar@{-}[d]\ar@{-}[drr]&
10111111\ar@{-}[dl]\ar@{-}[drrr]&
01122100\ar@{-}[dl]&
01121110\ar@{-}[dll]\ar@{-}[dl]&
01111111\ar@{-}[dll]\ar@{-}[dl]\ar@{-}[d]\\
23&
11121000\ar@{-}[d]\ar@{-}[drr]&
11111100\ar@{-}[dl]\ar@{-}[d]\ar@{-}[drr]&
10111110\ar@{-}[dl]\ar@{-}[drrr]&
01121100\ar@{-}[dl]\ar@{-}[d]&
01111110\ar@{-}[dl]\ar@{-}[d]\ar@{-}[dr]&
01011111\ar@{-}[dl]\ar@{-}[dr]&
00111111\ar@{-}[dl]\ar@{-}[d]\\
24&
11111000\ar@{-}[d]\ar@{-}[dr]\ar@{-}[drr]&
10111100\ar@{-}[d]\ar@{-}[drrr]&
01121000\ar@{-}[d]&
01111100\ar@{-}[dl]\ar@{-}[d]\ar@{-}[dr]&
01011110\ar@{-}[dl]\ar@{-}[dr]&
00111110\ar@{-}[dl]\ar@{-}[d]&
00011111\ar@{-}[dl]\ar@{-}[d]\\
25&
11110000\ar@{-}[d]\ar@{-}[dr]&
10111000\ar@{-}[dl]\ar@{-}[drr]&
01111000\ar@{-}[dl]\ar@{-}[d]\ar@{-}[dr]&
01011100\ar@{-}[dl]\ar@{-}[dr]&
00111100\ar@{-}[dl]\ar@{-}[d]&
00011110\ar@{-}[dl]\ar@{-}[d]&
00001111\ar@{-}[dl]\ar@{-}[d]\\
26&
10110000\ar@{-}[d]\ar@{-}[drr]&
01110000\ar@{-}[d]\ar@{-}[dr]&
01011000\ar@{-}[dl]\ar@{-}[dr]&
00111000\ar@{-}[dl]\ar@{-}[d]&
00011100\ar@{-}[dl]\ar@{-}[d]&
00001110\ar@{-}[dl]\ar@{-}[d]&
00000111\ar@{-}[dl]\ar@{-}[d]\\
27&
10100000\ar@{-}[d]\ar@{-}[drr]&
01010000\ar@{-}[d]\ar@{-}[drr]&
00110000\ar@{-}[d]\ar@{-}[dr]&
00011000\ar@{-}[d]\ar@{-}[dr]&
00001100\ar@{-}[d]\ar@{-}[dr]&
00000110\ar@{-}[d]\ar@{-}[dr]&
00000011\ar@{-}[d]\ar@{-}[dr]\\
28&
10000000&
01000000&
00100000&
00010000&
00001000&
00000100&
00000010&
00000001
}
\]
}

{\scriptsize
\[
\DC_1 = 
\DC_2 = 
\DC_3 = 
\DC_4 = 
\DC_5 = \begin{pmatrix}  1 \end{pmatrix}\quad
\DC_6 = \begin{pmatrix}  1 \\
     1 \end{pmatrix}\quad
\DC_7 = \begin{pmatrix}  1 & 1 \\
     1 & 0 \end{pmatrix}\quad
\DC_8 = 
\DC_9 = \begin{pmatrix}  1 & 0 \\
     1 & 1 \end{pmatrix}\] \[
\DC_{10} = \begin{pmatrix}  1 & 0 \\
     1 & 1 \\
     0 & 1 \end{pmatrix}\quad
\DC_{11} = \begin{pmatrix}  1 & 0 & 0 \\
     1 & 1 & 0 \\
     0 & 1 & 1 \end{pmatrix}\quad
\DC_{12} = \begin{pmatrix}  1 & 0 & 0 \\
     1 & 1 & 0 \\
     0 & 1 & 1 \\
     0 & 0 & 1 \end{pmatrix}\quad
\DC_{13} = \begin{pmatrix}  1 & 1 & 0 & 0 \\
     0 & 1 & 1 & 0 \\
     0 & 0 & 1 & 1 \\
     0 & 0 & 0 & 1 \end{pmatrix}\] \[
\DC_{14} = \begin{pmatrix}  1 & 1 & 0 & 0 \\
     0 & 1 & 1 & 0 \\
     0 & 0 & 1 & 0 \\
     0 & 0 & 1 & 1 \end{pmatrix}\quad
\DC_{15} = \begin{pmatrix}  1 & 1 & 0 & 0 \\
     0 & 1 & 1 & 0 \\
     0 & 1 & 0 & 1 \\
     0 & 0 & 1 & 1 \end{pmatrix}\quad
\DC_{16} = \begin{pmatrix}  1 & 1 & 0 & 0 \\
     0 & 1 & 0 & 0 \\
     1 & 0 & 1 & 0 \\
     0 & 1 & 1 & 1 \\
     0 & 0 & 0 & 1 \end{pmatrix}\] \[
\DC_{17} = \begin{pmatrix}  1 & 1 & 0 & 0 & 0 \\
     1 & 0 & 1 & 1 & 0 \\
     0 & 1 & 0 & 1 & 0 \\
     0 & 0 & 0 & 1 & 1 \\
     0 & 0 & 0 & 0 & 1 \end{pmatrix}\quad
\DC_{18} = \begin{pmatrix}  1 & 0 & 0 & 0 & 0 \\
     1 & 1 & 1 & 0 & 0 \\
     0 & 1 & 0 & 1 & 0 \\
     0 & 0 & 1 & 1 & 0 \\
     0 & 0 & 0 & 1 & 1 \\
     0 & 0 & 0 & 0 & 1 \end{pmatrix}\quad
\DC_{19} = \begin{pmatrix}  1 & 1 & 0 & 0 & 0 & 0 \\
     0 & 1 & 1 & 1 & 0 & 0 \\
     0 & 0 & 0 & 1 & 0 & 0 \\
     0 & 0 & 1 & 0 & 1 & 0 \\
     0 & 0 & 0 & 1 & 1 & 0 \\
     0 & 0 & 0 & 0 & 1 & 1 \end{pmatrix}\] \[
\DC_{20} = \begin{pmatrix}  1 & 1 & 0 & 0 & 0 & 0 \\
     0 & 1 & 1 & 0 & 0 & 0 \\
     0 & 1 & 0 & 1 & 1 & 0 \\
     0 & 0 & 1 & 0 & 1 & 0 \\
     0 & 0 & 0 & 1 & 0 & 1 \\
     0 & 0 & 0 & 0 & 1 & 1 \end{pmatrix}\quad
\DC_{21} = \begin{pmatrix}  1 & 1 & 0 & 0 & 0 & 0 \\
     1 & 0 & 1 & 0 & 0 & 0 \\
     0 & 1 & 1 & 1 & 0 & 0 \\
     0 & 0 & 0 & 1 & 0 & 0 \\
     0 & 0 & 1 & 0 & 1 & 1 \\
     0 & 0 & 0 & 1 & 0 & 1 \end{pmatrix}\quad
\DC_{22} = \begin{pmatrix}  1 & 0 & 0 & 0 & 0 & 0 \\
     1 & 1 & 1 & 0 & 0 & 0 \\
     0 & 0 & 1 & 1 & 0 & 0 \\
     0 & 0 & 0 & 1 & 0 & 0 \\
     0 & 1 & 0 & 0 & 1 & 0 \\
     0 & 0 & 1 & 0 & 1 & 1 \\
     0 & 0 & 0 & 1 & 0 & 1 \end{pmatrix}\] \[
\DC_{23} = \begin{pmatrix}  1 & 1 & 0 & 0 & 0 & 0 & 0 \\
     0 & 1 & 1 & 0 & 0 & 0 & 0 \\
     0 & 0 & 1 & 1 & 0 & 0 & 0 \\
     0 & 1 & 0 & 0 & 1 & 1 & 0 \\
     0 & 0 & 1 & 0 & 0 & 1 & 1 \\
     0 & 0 & 0 & 0 & 0 & 0 & 1 \\
     0 & 0 & 0 & 1 & 0 & 0 & 1 \end{pmatrix}\quad
\DC_{24} = \begin{pmatrix}  1 & 1 & 0 & 0 & 0 & 0 & 0 \\
     0 & 1 & 1 & 0 & 0 & 0 & 0 \\
     1 & 0 & 0 & 1 & 0 & 0 & 0 \\
     0 & 1 & 0 & 1 & 1 & 0 & 0 \\
     0 & 0 & 0 & 0 & 1 & 1 & 0 \\
     0 & 0 & 1 & 0 & 1 & 0 & 1 \\
     0 & 0 & 0 & 0 & 0 & 1 & 1 \end{pmatrix}\] \[
\DC_{25} = \begin{pmatrix}  1 & 0 & 0 & 0 & 0 & 0 & 0 \\
     1 & 1 & 0 & 0 & 0 & 0 & 0 \\
     1 & 0 & 1 & 1 & 0 & 0 & 0 \\
     0 & 0 & 0 & 1 & 1 & 0 & 0 \\
     0 & 1 & 0 & 1 & 0 & 1 & 0 \\
     0 & 0 & 0 & 0 & 1 & 1 & 1 \\
     0 & 0 & 0 & 0 & 0 & 0 & 1 \end{pmatrix}\quad
\DC_{26} = \begin{pmatrix}  1 & 1 & 0 & 0 & 0 & 0 & 0 \\
     1 & 0 & 1 & 0 & 0 & 0 & 0 \\
     0 & 0 & 1 & 1 & 0 & 0 & 0 \\
     0 & 1 & 1 & 0 & 1 & 0 & 0 \\
     0 & 0 & 0 & 1 & 1 & 1 & 0 \\
     0 & 0 & 0 & 0 & 0 & 1 & 1 \\
     0 & 0 & 0 & 0 & 0 & 0 & 1 \end{pmatrix}\] \[
\DC_{27} = \begin{pmatrix}  1 & 0 & 0 & 0 & 0 & 0 & 0 \\
     0 & 1 & 1 & 0 & 0 & 0 & 0 \\
     1 & 1 & 0 & 1 & 0 & 0 & 0 \\
     0 & 0 & 1 & 1 & 1 & 0 & 0 \\
     0 & 0 & 0 & 0 & 1 & 1 & 0 \\
     0 & 0 & 0 & 0 & 0 & 1 & 1 \\
     0 & 0 & 0 & 0 & 0 & 0 & 1 \end{pmatrix}\quad
\DC_{28} = \begin{pmatrix}  1 & 0 & 0 & 0 & 0 & 0 & 0 \\
     0 & 1 & 0 & 0 & 0 & 0 & 0 \\
     1 & 0 & 1 & 0 & 0 & 0 & 0 \\
     0 & 1 & 1 & 1 & 0 & 0 & 0 \\
     0 & 0 & 0 & 1 & 1 & 0 & 0 \\
     0 & 0 & 0 & 0 & 1 & 1 & 0 \\
     0 & 0 & 0 & 0 & 0 & 1 & 1 \\
     0 & 0 & 0 & 0 & 0 & 0 & 1 \end{pmatrix}\quad
\]
}

We have $h = h^\vee = 30$, and $d = 58$. The Cartan matrix is an
isomorphism.

\[
H^i(\OC_\mini, \ZM) \simeq
\begin{cases}
\ZM & \text{for } i = 0,\ 12,\ 20,\ 24,\ 32,\ 36,\ 44,\ 56,\\ 
    & \qquad \quad 59,\ 71,\ 79,\ 83,\ 91,\ 95,\ 103, 115\\
\ZM/2 & \text{for } i = 30,\ 42,\ 50,\ 54,\ 62,\ 66,\ 74,\ 86\\
\ZM/3 & \text{for } i = 40,\ 52,\ 64,\ 76\\
\ZM/5 & \text{for } i = 48,\ 68\\
0 & \text{otherwise}
\end{cases}
\]

\subsection{Type $F_4$}

\begin{center}
\begin{picture}(200,30)(-10,0)
\put( 40, 10){\circle{10}}
\put( 45, 10){\line(1,0){30}}
\put( 80, 10){\circle*{10}}
\put( 84, 11.5){\line(1,0){32}}
\put( 84,  8.5){\line(1,0){32}}
\put( 95,  5.5){\LARGE{$>$}}
\put(120, 10){\circle*{10}}
\put(125, 10){\line(1,0){30}}
\put(160, 10){\circle*{10}}
\put( 35, 20){$\a_1$}
\put( 75, 20){$\a_2$}
\put(115, 20){$\a_3$}
\put(155, 20){$\a_4$}
\end{picture}
\end{center}

We have $h = 12$, $h^\vee = 9$ and $d = 16$.

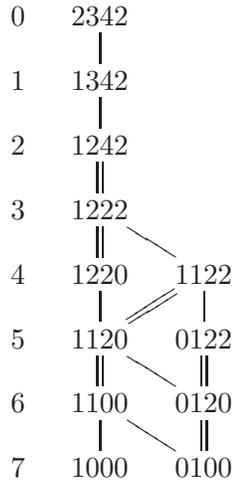
\[
\xymatrix @=.4cm{
0&
2342\ar@{-}[d]\\
1&
1342\ar@{-}[d]\\
2&
1242\ar@{=}[d]\\
3&
1222\ar@{=}[d]\ar@{-}[dr]\\
4&
1220\ar@{-}[d]&
1122\ar@{=}[dl]\ar@{-}[d]\\
5&
1120\ar@{=}[d]\ar@{-}[dr]&
0122\ar@{=}[d]\\
6&
1100\ar@{-}[d]\ar@{-}[dr]&
0120\ar@{=}[d]\\
7&
1000&
0100
}
\]

We have
\[
\DC_1 = \DC_2 = (1)\quad \DC_3 = (2)
\quad \DC_4 = \begin{pmatrix}
2\\
1
\end{pmatrix}
\quad \DC_5 = \begin{pmatrix}
1&2\\
0&1
\end{pmatrix}
\quad \DC_6 = \begin{pmatrix}
2&0\\
1&2
\end{pmatrix}
\quad \DC_7 = \begin{pmatrix}
1&0\\
1&2
\end{pmatrix}
\]

The type of $\Phi'$ is $A_2$, so we have
\[
\DC_8 =
\begin{pmatrix}
2&1\\
1&2
\end{pmatrix}
\]

The matrices of the last differentials are transposed to the first ones.

\[
H^i(\OC_\mini, \ZM) \simeq
\begin{cases}
\ZM & \text{for } i = 0,\ 8,\ 23,\ 31\\
\ZM/2 & \text{for } i = 6,\ 14,\ 18,\ 26\\
\ZM/4 & \text{for } i = 12,\ 20\\
\ZM/3 & \text{for } i = 16\\
0 & \text{otherwise}
\end{cases}
\]

\subsection{Type $G_2$}

\begin{center}
\begin{picture}(100,30)
\put( 40, 10){\circle{10}}
\put(44.5,12){\line(1,0){32}}
\put( 45, 10){\line(1,0){30}}
\put(44.5, 8){\line(1,0){32}}
\put( 55,5.5){\LARGE{$>$}}
\put( 80, 10){\circle*{10}}
\put( 35, 20){$\a_1$}
\put( 75, 20){$\a_2$}
\end{picture}
\end{center}

We have $h = 6$, $h^\vee = 4$, and $d = 6$.
The root system $\Phi'$ is of type $A_1$. Its Cartan matrix
has cokernel $\ZM/2$.

$$
\xymatrix @=.5cm{
0&
23\ar@{-}[d]\\
1&
13\ar@3{-}[d]\\
2&
10
}
$$

We have
\[
\DC_1 = (1) \quad 
\DC_2 = (3) \quad 
\DC_3 = (2) \quad 
\DC_4 = (3) \quad 
\DC_5 = (1)
\]

\[
H^i(\OC_\mini, \ZM) \simeq
\begin{cases}
\ZM & \text{for } i = 0,\ 11\\
\ZM/3 & \text{for } i = 4,\ 8\\
\ZM/2 & \text{for } i = 6\\
0 & \text{otherwise}
\end{cases}
\]

\section{Another method for type $A$}\label{a bis}

Here we will explain a method which applies only in type $A$.
This is because the minimal class is a Richardson class only
in type $A$.

So suppose we are in type $A_{n - 1}$. We can assume $G = GL_n$.
The minimal class corresponds to the partition $(2,1^{n - 2})$.
It consists of the nilpotent matrices of rank $1$ in $\gG\lG_n$, or,
in other words, the matrices of rank $1$ and trace $0$.

Let us consider the set $E$ of pairs $([v],x) \in \PM^{n - 1} \times \gG\lG_n$
such that $\textrm{Im}(x) \incl \CM v$ (so $x$ is either zero or of
rank $1$). Together with
the natural projection, this is a vector bundle on $\PM^{n - 1}$,
corresponding to the locally free sheaf $\EC = \OC(-1)^n$
(we have one copy of the tautological bundle for each column).

There is a trace morphism $\Tr : \EC \to \OC$. Let $\FC$ be its kernel,
and let $F$ be the corresponding sub-vector bundle of $E$.
Then $F$ consists of the pairs $([v],x)$ such that $x$ is either zero
or a nilpotent matrix of rank $1$ with image $\CM v$.
The second projection gives a morphism $\pi : E \to \overline{\OC_\mini}$,
which is a resolution of singularities, with exceptional fiber
the null section. So we have an isomorphism from $F$ minus the null section
onto $\OC_\mini$.

As before, we have a Gysin exact sequence
\[
H^{i-2n+2} \elem{c} H^i \longto H^i(\OC_\mini,\ZM)
\longto H^{i-2n+3} \elem{c} H^{i+1}
\]
where $H^j$ stands for $H^j(\PM^{n-1},\ZM)$ and $c$ is the multiplication
by the last Chern class $c$ of $F$. Thus $H^i(\OC_\mini,\ZM)$ fits
in a short exact sequence
\[
0 \longto \Coker(c:H^{i-2n+2}\to H^i) \longto H^i(\OC_\mini,\ZM)
\longto \Ker(c:H^{i-2n+3}\to H^{i+1}) \longto 0
\]
We denote by $y\in H^2(\PM^{n - 1}, \ZM)$ the first Chern class of $\OC(-1)$.
We have $H^*(\PM^{n - 1},\ZM) \simeq \ZM[y]/(y^n)$ as a ring.
In particular, the cohomology of $\PM^{n-1}$ is free and
concentrated in even degrees.

For $0 \leqslant i \leqslant 2n - 4$, we have
$H^i(\OC_\mini,\ZM) \simeq H^i$ which is isomorphic to $\ZM$ if
$i$ is even, and to $0$ if $i$ is odd.
We have $H^{2n - 3}(\OC_\mini,\ZM) \simeq \Ker(c:H^0\to H^{2n-2})$
and $H^{2n - 2}(\OC_\mini,\ZM) \simeq \Coker(c:H^0\to H^{2n-2})$.
For $2n - 1 \leqslant i \leqslant 4n - 5$, we have
$H^i(\OC_\mini,\ZM) \simeq H^{i - 2n + 3}$ which is isomorphic to $\ZM$ if
$i$ is odd, and to $0$ if $i$ is even.

We have an exact sequence
\[
0 \longto \FC \longto \EC = \OC(-1)^n \longto \OC \longto 0
\]

The total Chern class of $F$ is thus
$$(1+y)^n = \sum_{i = 0}^{n-1} \binom{n}{i} y^i$$
by multiplicativity (remember that $y^n = 0$).
So its last Chern class $c$ is $ny^{n - 1}$.

In fact, $F$ can be identified with the cotangent bundle
$T^*(G/Q)$, where $Q$ is the parabolic subgroup which stabilizes
a line in $\CM^n$, and $G/Q \simeq \PM^{n - 1}$~; then
we can use the fact that the Euler characteristic of $\PM^{n - 1}$
is $n$.

We can now determine the two remaining cohomology groups.
\[
H^{2n - 3}(\OC_\mini,\ZM) \simeq \Ker(\ZM\elem{n}\ZM) = 0
\]
and
\[
H^{2n - 2}(\OC_\mini,\ZM) \simeq \Coker(\ZM\elem{n}\ZM) = \ZM/n
\]

Thus we find the same result as in Section \ref{case by case} for
the cohomology of $\OC_\mini$ in type $A_{n - 1}$.

\chapter{Some decomposition numbers}\label{chap:dec}

In this chapter, we compute in a geometrical way certain decomposition
numbers for the $G$-equivariant perverse sheaves on the nilpotent
variety, using the results of the preceding chapters, and geometrical results
about nilpotent classes that can be found in the literature.

In Chapter \ref{chap:springer}, we will prove that part of this
decomposition matrix is the decomposition matrix for the Weyl group
(see Theorem \ref{th:d}), and we expect that the whole matrix is the
one of the Schur algebra, at least in type $A$ (in the other types one
would have to give a precise meaning to the Schur algebra to be used).
The decomposition numbers we find are in accordance with this
expectation.

We made these calculations to have some examples and to get used to
perverse sheaves over $\KM$, $\OM$ and $\FM$, and to find evidence for
Theorem \ref{th:d} before it was proved. Now that this theorem is
proved, this chapter also shows that some calculations are feasible on
the geometrical side. It should be clear, however, that further
calculations would be increasingly difficult, so it is not clear
whether Theorem \ref{th:d} will help to compute the decomposition
matrices of the symmetric group.

In any case, we can see \emph{a posteriori} that the calculations in
this chapter reveal new connections between the geometry of nilpotent
orbits and the representations of Weyl groups. Next developments in
this directions could include, in type $A$, the treatment of two-column
(resp. two-row) partitions, and the transfer of a generalization of the
row and column removal rule from the group side to the geometrical side.

\section{Subregular class}\label{sec:subreg}

We assume that $G$ is simple and adjoint of type $\G$, and that
the characteristic of $k$ is $0$ or greater than $4h - 2$
(where $h$ is the Coxeter number). This is a serious restriction on
$p$, but it does not matter so much for our purposes. Note that, on
the other hand, we make no concessions on $\ell$ (the only restriction
is $\ell \neq p$).

Let $\OC_\reg$ (resp. $\OC_\subreg$) be the regular
(resp. subregular) orbit in $\NC$.
The  orbit $\OC_\subreg$ is the unique
open dense orbit in $\NC \setminus \OC_\reg$
(we assume that $\gG$ is simple). It is of
codimension 2 in $\NC$. Let $x_\reg \in \OC_\reg$ and
$x_\subreg \in \OC_\subreg$.

The centralizer of $x_\reg$ in $G$ is a connected unipotent subgroup,
hence $A_G(x_\reg) = 1$. The unipotent radical of the centralizer in
$G$ of $x_\subreg$ has a reductive complement $C$ given by the following
table.

\[
\begin{array}{c|c|c|c|c|c|c|c|c|c}
\G & A_n\ (n > 1) & B_n & C_n & D_n & E_6 & E_7 & E_8 & F_4 & G_2\\
\hline
C(x) & \GM_m & \GM_m \rtimes \ZM/2 & \ZM/2 & 1 & 1 & 1 & 1 & \ZM/2 & \SG_3
\end{array}
\]

In type $A_1$, the subregular class is just the trivial class, so in
this case the centralizer is $G = PSL_2$ itself, which is reductive.

We have $A_G(x_\subreg) \simeq C/C^0$. This group is isomorphic to
the associated symmetry group $A(\G)$ introduced in section
\ref{sec:simple}.  

Let $X$ be the intersection $X = S \cap \NC$ of a transverse slice $S$ to the
orbit $\OC_\subreg$ of $x_\subreg$ with the nilpotent variety $\NC$.
The group $C$ acts on $X$. We can find a section $A$ of
$C/C^0 \simeq A_G \simeq A(\G)$ in $C$. In homogeneous types, $A$ is
trivial. If $\G = C_n$, $F_4$ or $G_2$, then $A = C$. If $\G = B_n$,
take $\{1,s\}$ where $s$ is a nontrivial involution (in this case, $A$
is well-defined up to conjugation by $C^0 = \GM_m$).

\begin{theo}\cite{BRI,SLO1}\label{th:subreg}
We keep the preceding notation.
The surface $X$ has a rational double point of type $\wh\G$ at $x_\subreg$.
Thus $\Sing(\ov \OC_\reg, \OC_\subreg) = \wh\G$.

Moreover the couple $(X,A)$ is a simple singularity of type $\G$.
\end{theo}

In fact, the first part of the theorem is already true when the
characteristic of $k$ is very good for $G$. This part is enough to
calculate the decomposition numbers $d_{(x_\reg,1),(x_\subreg,1)}$ for
homogeneous types (then $A = 1$), and even some more decomposition
numbers $d_{(x_\reg,1),(x_\subreg,\rho)}$ for the other
types. Actually, what can be deduced in all types is the following
relation:
\[
\sum_{\rho\in\Irr\FM A} \rho(1) \cdot d_{(x_\reg,1),(x_\subreg,\rho)}
= \dim_\FM \FM \otimes_\ZM P(\wh\Phi)/Q(\wh\Phi)
\]
This is enough, for example, to determine for which $\ell$ we have
$d_{(x_\reg,1),(x_\subreg,\rho)} = 0$ for all $\rho$ (those which do
not divide the connection index of $\wh\Phi$).

Anyway, the second part of the theorem will allow us to deal with the
local systems involved on $\OC_\subreg$.

Let $\jreg : \OC_\reg \injto \OC_\reg \cup \OC_\subreg$ be the open immersion,
and $i_\subreg : \OC_\subreg \injto \OC_\reg \cup \OC_\subreg$ the closed
complement. Finally, let $j$ be the open inclusion of $\OC_\subreg
\cup \OC_\reg$ into $\NC$. Applying the functor $j^*$, we see that
\begin{eqnarray*}
d_{(x_\reg,1),(x_\subreg,\rho)}
&=& [\FM \p \JC_{!*} (\OC_\reg, \OM) :
   \p \JC_{!*} (\OC_\subreg, \rho)]\\
&=& [\FM \jreg_{!*} (\OM[2\nu]) : \isubreg_{*} (\rho [2\nu + 2])]
\end{eqnarray*}

By Slodowy's theorem and the analysis of Section \ref{sec:simple}, we
obtain the following result.

\begin{theo}
We have
\begin{equation}
d_{(x_\reg,1),(x_\subreg,\rho)}
= [\FM \otimes_\ZM P(\wh\Phi)/Q(\wh\Phi) : \rho]
\end{equation}
for all $\rho$ in $\Irr \FM A$.
\end{theo}

For homogeneous types, we recover the decomposition numbers described
in section \ref{sec:simple}. Let us describe in detail all the other
possibilities. The action of $\Aut(\wh\Phi)/W(\wh\Phi)$ on
$P(\wh\Phi)/Q(\wh\Phi)$ is described in all types in
\cite[Chap. VI, \S 4]{BOUR456}.

In the types $B_n$, $C_n$ and $F_4$, we have $A \simeq
\ZM/2$. When $\ell = 2$, we have $\Irr \FM A = \{1\}$. In this case,
we would not even need to know the actual action, since for our
purposes we only need the class in the Grothendieck group
$K_0(\FM A) \simeq \ZM$, that is, the dimension. When $\ell$ is not
$2$, we have $\Irr \FM A = \{1,\e\}$, where $\e$ is the unique
non-trivial character of $\ZM/2$.

\subsection{Case $\G = B_n$}
We have $\wh\G = A_{2n - 1}$ and
$P(\wh\Phi)/Q(\wh\Phi) \simeq \ZM/2n$. The non-trivial element of
$A \simeq \ZM/2$ acts by $-1$. Thus we have

\[
\begin{array}{lllll}
\text{If } \ell = 2, &\text{then}& d_{(x_\reg,1),(x_\subreg,1)} = 1\\
\text{If } 2 \neq \ell \mid n, &\text{then}& d_{(x_\reg,1),(x_\subreg,1)} = 0
&\text{and}& d_{(x_\reg,1),(x_\subreg,\e)} = 1\\
\text{If } 2 \neq \ell \nmid n, &\text{then}& d_{(x_\reg,1),(x_\subreg,1)} = 0
&\text{and}& d_{(x_\reg,1),(x_\subreg,\e)} = 0
\end{array}
\]

\subsection{Case $\G = C_n$}

We have $\wh\G = D_{n + 1}$.

If $n$ is even, then we have
$P(\wh\Phi)/Q(\wh\Phi) \simeq \ZM/4$,
and the nontrivial element of $A \simeq \ZM/2$ acts by $-1$.

If $n$ is odd, then we have
$P(\wh\Phi)/Q(\wh\Phi) \simeq (\ZM/2)^2$,
and the nontrivial element of $A \simeq \ZM/2$ acts by exchanging two
nonzero elements.

Thus we have
\[
\begin{array}{lll}
\text{If } \ell = 2 \text{ and $n$ is even},& \text{then}& d_{(x_\reg,1),(x_\subreg,1)} = 1\\
\text{If } \ell = 2 \text{ and $n$ is odd},& \text{then}& d_{(x_\reg,1),(x_\subreg,1)} = 2\\
\text{If } \ell \neq 2, &\text{then}& 
d_{(x_\reg,1),(x_\subreg,1)} = d_{(x_\reg,1),(x_\subreg,\e)} = 0
\end{array}
\]

\subsection{Case $\G = F_4$}

We have $\wh\G = E_6$ and $P(\wh\Phi)/Q(\wh\Phi) \simeq \ZM/3$.
The nontrivial element of $A \simeq \ZM/2$ acts by $-1$.
Thus we have
\[
\begin{array}{lllll}
\text{If } \ell = 2, &\text{then}&  d_{(x_\reg,1),(x_\subreg,1)} = 0\\
\text{If } \ell = 3, &\text{then}&  d_{(x_\reg,1),(x_\subreg,1)} = 0
&\text{and}& d_{(x_\reg,1),(x_\subreg,\e)} = 1\\
\text{If } \ell > 3, &\text{then}& d_{(x_\reg,1),(x_\subreg,1)} = 0
&\text{and}& d_{(x_\reg,1),(x_\subreg,\e)} = 0
\end{array}
\]

\subsection{Case $\G = G_2$}

We have $\wh\G = D_4$ and $P(\wh\Phi)/Q(\wh\Phi) \simeq (\ZM/2)^2$.
The group $A \simeq \SG_3$ acts by permuting the three non-zero
elements. Let us denote the sign character by $\e$ (it is nontrivial
when $\ell \neq 2$), and the degree two character by $\psi$ (it
remains irreducible for $\ell = 2$, but for $\ell = 3$ it decomposes
as $1 + \e$). We have
\[
\begin{array}{lll}
\text{If } \ell = 2, &\text{then}&  d_{(x_\reg,1),(x_\subreg,1)} = 0
\quad\text{and}\quad d_{(x_\reg,1),(x_\subreg,\psi)} = 1\\
\text{If } \ell = 3, &\text{then}&  d_{(x_\reg,1),(x_\subreg,1)} = 
 d_{(x_\reg,1),(x_\subreg,\e)} = 0\\
\text{If } \ell > 3, &\text{then}& d_{(x_\reg,1),(x_\subreg,1)} = 
d_{(x_\reg,1),(x_\subreg,\e)} = d_{(x_\reg,1),(x_\subreg,\psi)} = 0
\end{array}
\]

\section{Minimal class}


We assume that $G$ is simple. Thus there is a unique minimal
non-zero nilpotent orbit $\OC_\mini$ in $\gG$ (corresponding to the highest
weight of the adjoint representation, the highest root
$\tilde\alpha$). It is of dimension $d = 2 h^\vee - 2$, where $h^\vee$
is the dual Coxeter number (see Chapter \ref{chap:min}).

Its closure $\ov \OC_\mini = \OC_\mini \cup \{0\}$ is a cone with
origin $0$.  Let $\jmin : \OC_\mini \to \ov \OC_\mini$ be the open
immersion, and $\io : \{0\} \to \ov \OC_\mini$ the closed
complement. By Proposition \ref{prop:cone}, we have
\[
\io^* \jmin_* (\OM[d]) \simeq \bigoplus_i H^{i + d}(\OC_\mini,\OM) [-i]
\]
In Chapter \ref{chap:min}, we have calculated all the cohomology of
$\OC_\mini$ when the base field is $\CM$, and by the comparison
theorems we can deduce the cohomology in the étale case.
However we will only need the following results.
Remember that $\Phi'$ is the root subsystem of $\Phi$ generated by the
long simple roots.
\begin{gather}
\label{min H -1}
H^{-1} \io^* \jmin_* (\OM[d]) = H^{2h^\vee - 3} (\OC_\mini, \OM) = 0
\\
\label{min H 0}
H^0 \io^* \jmin_* (\OM[d])
 = H^{2h^\vee - 2} (\OC_\mini, \OM)
 = \OM \otimes_\ZM (P^\vee(\Phi')/Q^\vee(\Phi'))
\\
\label{min H 1}
H^1 \io^* \jmin_* (\OM[d]) = H^{2h^\vee - 1} (\OC_\mini, \OM)
\text{ is torsion-free}
\end{gather}

By the distinguished triangles in Sections
\ref{sec:tt recollement} and \ref{sec:F recollement},
we obtain the following results.

\begin{theo}
Over $\OM$, we have canonical isomorphisms
\begin{gather}
\label{min iso IC}
\p \jmin_! (\OM[d])
\simeq \pp \jmin_! (\OM[d])
\simeq \p \jmin_{!*} (\OM[d])
\\
\label{min iso IC+}
\pp \jmin_{!*} (\OM[d])
\simeq \p \jmin_* (\OM[d])
\simeq \pp \jmin_* (\OM[d])
\end{gather}
and a short exact sequence
\begin{equation}
\label{min ses}
0 \longto \p \jmin_{!*} (\OM[d])
\longto \pp \jmin_{!*} (\OM[d])
\longto \io_* \OM \otimes_\ZM (P^\vee(\Phi')/Q^\vee(\Phi'))
\longto 0
\end{equation}

Over $\FM$, we have canonical isomorphisms
\begin{gather}
\label{min iso F IC}
\FM\ \p \jmin_!\ (\OM[d])
\isom \p \jmin_!\ (\FM[d])
\isom \FM\ \pp \jmin_!\ (\OM[d])
\isom \FM\ \p \jmin_{!*}\ (\OM[d])
\\
\label{min iso F IC+}
\FM\ \pp \jmin_{!*}\ (\OM[d])
\isom \FM\ \p \jmin_*\ (\OM[d])
\isom \p \jmin_*\ (\FM[d])
\isom \FM\ \pp \jmin_*\ (\OM[d])
\end{gather}
and short exact sequences
\begin{gather}
\label{min ses F IC}
0 \longto \io_*\ \FM \otimes_\ZM (P^\vee(\Phi')/Q^\vee(\Phi'))
\longto \FM\ \p \jmin_{!*}\ (\OM[d])
\longto \p \jmin_{!*}\ (\FM[d])
\longto 0
\\
\label{min ses F IC+}
0 \longto \p \jmin_{!*}\ (\FM[d])
\longto \FM\ \pp \jmin_{!*}\ (\OM[d])
\longto \io_*\ \FM \otimes_\ZM (P^\vee(\Phi')/Q^\vee(\Phi'))
\longto 0
\end{gather}

We have
\begin{equation}\label{min dec}
[\FM\ \p {j_\mini}_{!*}\ (\OM[d]) : \io_*\ \FM]
= [\FM\ \pp {j_\mini}_{!*}\ (\OM[d]) : \io_*\ \FM]
= \dim_\FM \FM \otimes_\ZM \left( P^\vee(\Phi')/Q^\vee(\Phi') \right)
\end{equation}

In particular, $\FM\ \p {j_\mini}_{!*}\ (\OM[d])$ is simple 
(and equal to $\FM\ \pp {j_\mini}_{!*}\ (\OM[d])$) if and only if
$\ell$ does not divide the connection index of $\Phi'$.
\end{theo}

Let us give this decomposition number in each type.
We denote the singularity of $\ov\OC_\mini$ at the origin by the lower
case letter $\g$ corresponding to the type $\G$ of $\gG$.

\[
\begin{array}{l|l|l|l}
\g & \G' & P^\vee(\Phi')/Q^\vee(\Phi') & d_{(x_\mini,1),(0,1)}\\
\hline
a_n & A_n & \ZM / (n + 1) &
      1 \text{ if } \ell \mid n + 1,\ 0 \text{ otherwise}\\
b_n & A_{n - 1} & \ZM/n &
      1 \text{ if } \ell \mid n,\ 0 \text{ otherwise}\\
c_n & A_1 & \ZM/2 &
      1 \text{ if } \ell = 2,\ 0 \text{ otherwise}\\
d_n\ (n \text{ even}) & D_n & (\ZM/2)^2 &
      2 \text{ if } \ell = 2,\ 0\text{ otherwise}\\
d_n\ (n \text{ odd}) & D_n & \ZM/4 &
      1 \text{ if } \ell = 2,\ 0\text{ otherwise}\\
e_6 & E_6 & \ZM/3 &
      1 \text{ if } \ell = 3,\ 0\text{ otherwise}\\
e_7 & E_7 & \ZM/2 &
      1 \text{ if } \ell = 2,\ 0\text{ otherwise}\\
e_8 & E_8 & 0 & 0\\
f_4 & A_2 & \ZM/3 &
      1 \text{ if } \ell = 3,\ 0 \text{ otherwise}\\
g_2 & A_1 & \ZM/2 &
      1 \text{ if } \ell = 2,\ 0 \text{ otherwise}
\end{array}
\]

Note that the singularities $c_n$ (for $n \geqslant 1$, including
$c_1 = a_1 = A_1$ and $c_2 = b_2$) and $g_2$ are $\KM$-smooth but not
$\FM_2$-smooth.

\section{Rows and columns}\label{sec:rc}

In this section, $G = GL_n$. The nilpotent orbits are parametrized by
the partitions of $n$, and the order defined by the inclusions of
closures of orbits corresponds to the usual dominance order on
partitions. Let us introduce some notation. The decomposition number
$[\FM \JC_{!*}(\OC_\l,\OM) : \JC_{!*}(\OC_\mu,\FM)]$ will be denoted
by $d_{\l,\mu}$, and we introduce the ``characteristic functions''
\[
\chi_{\l,\mu}
= \sum_{i\in\ZM} (-1)^i \dim_\KM \HC^i_{x_\mu} \JC_{!*}(\OC_\l,\KM)
= \sum_{i\in\ZM} (-1)^i \dim_\FM \HC^i_{x_\mu} \FM \JC_{!*}(\OC_\l,\OM)
\]
and
\[
\phi_{\nu,\mu}
= \sum_{i\in\ZM} (-1)^i \dim_\FM \HC^i_{x_\mu} \JC_{!*}(\OC_\nu,\FM)
\]
These form triangular systems, in the sense that
$\chi_{\l,\mu}$ can be non-zero only if $\mu \leqslant \l$,
and $\phi_{\nu,\mu}$ can be non-zero only if $\mu \leqslant \nu$.
We have
\[
\chi_{\l,\mu} = \sum_{\nu} d_{\l,\nu} \phi_{\nu,\mu}
\]
and $d_{\l,\nu}$ can be non-zero only if $\nu \leqslant \l$.
Moreover, we have $\chi_{\l,\l} = \phi_{\l,\l} = 1$, and
$d_{\l,\l} = 1$.

Kraft and Procesi found a row and column removal rule for the
singularities of the closures of the nilpotent orbits in type
$A_{n-1}$ \cite{KP1}. Actually, they state the result when the base
field is $\CM$, but it is certainly true when $p$ is very good.

\begin{prop}
Let $\OC_\mu \subset \ov\OC_\l$ be a degeneration of nilpotent orbits
in $\gG\lG_n$ and assume that the first $r$ rows and the first $s$
columns of $\l$ and $\mu$ coincide. Denote by $\l_1$ and $\mu_1$ the
Young diagrams obtained from $\l$ and $\mu$ by erasing these rows and
columns. Then $\OC_{\mu_1} \subset \ov\OC_{\l_1}$, and we have
\begin{equation}
\codim_{\ov\OC_{\l_1}} \OC_{\mu_1} = \codim_{\ov\OC_\l} \OC_\mu
\quad\text{ and }\quad
\Sing(\ov\OC_{\l_1}, \OC_{\mu_1}) = \Sing(\ov\OC_\l, \OC_\mu)
\end{equation}
\end{prop}

All the partitions $\nu$ in the interval $[\mu,\l] = \{\nu\mid \mu
\leqslant \nu \leqslant \l\}$ have the same first $r$ rows and first
$s$ columns. For $\nu$ in $[\mu,\l]$, let us denote by $\nu_1$ the
partition obtained from $\nu$ by erasing them.

The proposition implies that, for all $\eta \leqslant \zeta$ in
$[\mu,\l]$, the local intersection cohomology of $\ov\OC_{\eta_1}$ at
$\ov\OC_{\zeta_1}$ is the same as the local intersection cohomology of
$\ov\OC_\eta$ at $\ov\OC_\zeta$, both over $\KM$ and over $\FM$, and
thus $\chi_{\eta,\zeta} = \chi_{\eta_1,\zeta_1}$ and
$\phi_{\eta,\zeta} = \phi_{\eta_1,\zeta_1}$.

Since the decomposition numbers can be deduced from this
information (for $GL_n$, we only have trivial $A_G(x)$), we find a row
and column removal rule for the decomposition numbers for
$GL_n$-equivariant perverse sheaves on the nilpotent variety of
$GL_n$.

\begin{prop}\label{prop:rc}
With the notations above, we have
\begin{equation}
[\FM \JC_{!*}(\OC_\l, \OM) : \JC_{!*}(\OC_\mu, \FM)]
= [\FM \JC_{!*}(\OC_{\l_1}, \OM) : \JC_{!*}(\OC_{\mu_1}, \FM)]
\end{equation}
\end{prop}

\begin{proof}
The decomposition numbers
$(d_{\eta,\zeta})_{\mu \leqslant \zeta \leqslant \eta \leqslant \l}$
are the unique solution of the triangular linear system
\[
\chi_{\eta,\zeta} = \sum_{\mu \leqslant \nu \leqslant \l} d_{\eta,\nu} \phi_{\nu,\zeta}
\]
whereas the decomposition numbers
$(d_{\eta_1,\zeta_1})_{\mu_1 \leqslant \zeta_1 \leqslant \eta_1 \leqslant \l_1}$
are the unique solution of the triangular linear system
\[
\chi_{\eta_1,\zeta_1} = \sum_{\mu \leqslant \nu \leqslant \l}
d_{\eta_1,\nu_1} \phi_{\nu_1,\zeta_1}
\]
Since we have $\chi_{\eta,\zeta} = \chi_{\eta_1,\zeta_1}$ and
$\phi_{\eta,\zeta} = \phi_{\eta_1,\zeta_1}$, the two systems are
identical, so that $d_{\eta,\zeta} = d_{\eta_1,\zeta_1}$ for all
$\zeta \leqslant \eta$ in $[\mu,\l]$. In particular, we have
$d_{\l,\mu} = d_{\l_1,\mu_1}$.
\end{proof}

So we have a rule similar to the one for decomposition matrices for the symmetric groups,
and even for the Schur algebra. Now, if we have two adjacent partitions,
using this rule we can reduce to the extreme cases of the minimal or subregular class
for a smaller general linear group. So we get the following result.

\begin{cor}
The decomposition numbers for the $GL_n$-equivariant perverse sheaves on the nilpotent
variety and for the Schur algebra coincide for adjacent partitions.
\end{cor}

In \ref{sec:decomposition} we will see that, at least for the part of
the decomposition matrix corresponding to the symmetric group, this is
not a coincidence. We hope to find an explanation for the whole
decomposition matrix of the Schur algebra as well. We will come back
to this later.

In \cite{KP2} Kraft and Procesi obtain a similar result for all classical types,
so the same rule should apply for perverse sheaves in the classical
types. The only problem that is left is to deal with the local systems
involved, when the $A_G(x)$ are non-trivial. In many cases, we have
enough information (see the introduction, and the tables in Chapter
\ref{chap:tables}).

\section{Special classes}\label{sec:special}

We will use another geometrical result by Kraft and Procesi, contained
in \cite{KP3}.

In \cite{LusSpec}, Lusztig introduced the special representations of a
finite Weyl group. The \emph{special unipotent classes} of a simple
group are the unipotent classes $C$ such that the representation of
$W$ corresponding to the pair $(C,1)$ is \emph{special}.

On the other hand, Spaltenstein introduced an order-reversing map
$d$ from the set of unipotent classes to itself, such that $d^3 = d$
(it is an involution on its image) in \cite{SpalDual}. The image of
$d$ consists exactly in the special unipotent classes, and the locally
closed subvarieties
\[
\widehat C = \ov C \setminus
\mathop{\bigcup_{C' \text{ special}}}\limits_{\ov C'\subset\ov C} \ov C'
\]
where $C$ runs through the special classes, form a partition of the
unipotent variety (any unipotent class is contained in a
$\widehat C$ for a unique special class $C$).

If $p$ is good for $G$ (or if the base field is $\CM$), then we can
use a $G$-equivariant homeomorphism from the unipotent variety to the
nilpotent variety to transport these notions to nilpotent orbits.

In type $A$, all the unipotent classes (resp. nilpotent orbits) are
special, so this this section gives information only for the other
classical types.

\begin{theo}
Let $C$ be a special unipotent conjugacy class of a classical
group. Define $\widehat C$ as above. Then $\widehat C$ consists of
$2^d$ conjugacy classes, where $d$ is the number of irreducible
components of $\widehat C \setminus C$. There is a smooth variety $Y$
with an action of the group $(\ZM/2)^d$, and an isomorphism
\begin{equation}
Y / (\ZM/2)^d \elem{\sim} \widehat C
\end{equation}
which identifies the stratification of $\widehat C$ with the
stratification of the quotient by isotropy groups. (These are the
$2^d$ subproducts of $(\ZM/2)^d$). In particular $\widehat C$ is a
rational homology manifold.
\end{theo}

But actually their result gives more information.
By Proposition \ref{prop:rat smooth},
the variety $\widehat C$ is not only $\KM$-smooth, but also
$\FM$-smooth for $\ell \neq 2$.

If $p$ is good, the result can be transferred to the nilpotent variety, to fit with
our context. We obtain the following result.

\begin{prop}
Assume $p$ is good. Let $\OC$ be a special nilpotent orbit in a classical group. Let
\[
\widehat\OC = \ov\OC \setminus
\mathop{\bigcup_{\OC' \text{ special}}}\limits_{\ov\OC'\subset\ov \OC} \ov\OC'
\]
Assume $\ell \neq 2$.  Then $\widehat \OC$ is $\FM$-smooth and,
for all orbits $\OC' \subset \widehat \OC \setminus \OC$, we have
\[
[\FM \JC_{!*}(\OC,\OM) : \JC_{!*}(\OC',\FM)] = 0
\]
\end{prop}

\chapter{Fourier-Deligne transform}\label{chap:fourier}

Our aim is to define a Springer correspondence modulo $\ell$.  The
fact that the Weyl group $W$ acts on the cohomology spaces
$H^i(\BC_x,\EM)$ is true for $\EM = \KM$, $\OM$ or $\FM$, with the
same proof. So we have ``Springer representations''. But the main
point, in characteristic zero, is that all $E \in \Irr \KM W$ appear in
such representations, and that we can associate to each $E$ a pair
$(\OC,\LC)$, where $\OC$ is a nilpotent orbit and $\LC$ is a
$G$-equivariant simple perverse sheaf on $\NC$.

With $\KM$ coefficients, there are many versions of the Springer
correspondence. We had to choose a version which can be adapted to
$\FM$ coefficients. The main obstacle is that Gabber's decomposition
theorem \cite{BBD} is no longer true, already for a finite \'etale
covering. In this case taking the direct image is like inducing from
$\FM H$ to $\FM K$ where $K$ is a finite group and $H$ is a subgroup
of $K$. If we start with the constant sheaf $\FM$, we get the local
system corresponding to the representation $\FM [K/H]$ of $K$, which is
not semi-simple if $\ell$ divides $|G:H|$.

However, the approach of Hotta-Kashiwara \cite{HK} uses a Fourier
transform for $\DC$-modules in the complex case, thus avoiding the
decomposition theorem. Brylinski 
\cite{BRY} adapted this method to the case of a base field of
characteristic $p$, with $\ell$-adic coefficients, using the Fourier-Deligne
transform. This transform actually makes sense over $\KM$, $\OM$ and $\FM$.
We are going to use it in chapter
\ref{chap:springer} to define a modular Springer correspondence. It
would be interesting to understand better the modular version of the
Lusztig-Bohro-MacPherson approach \cite{BM,LusGreen,ICC}, where a
restriction to the nilpotent variety is used. We hope that, in type
$A$, it will be possible to involve the Schur algebra.

In this chapter, we define Fourier-Deligne transforms and study their
properties, following Laumon \cite{Lau}, but with coefficients $\EM$
instead of $\ov\QM_\ell$. No particular difficulty arises.  We shall
give more details than in Laumon's article. In particular, when we
have a chain of functorial isomorphisms, we will justify each
step. However, we do not give a proof for the most difficult result,
Theorem \ref{th:oubli support} (SUPP), since it appears in \cite{KaLa}
with torsion coefficients.

The use of the Fourier-Deligne transform is the reason why we had to
consider a base field of characteristic $p$. Otherwise, we could have
used complex varieties instead. From now on, $k$ is the finite field
$\FM_q$. Remember that, in section \ref{sec:context}, we assumed that
all schemes we consider are varieties, that is, separated schemes of
finite type over $k$.

\section{Preliminaries}

Since we will justify each step in the proofs, we are going to use
abbreviations for the main theorems or properties that we use.

If we have a commutative square of varieties
\[
\xymatrix{
X'
\ar[r]^{f'}
\ar[d]_{\pi'}
\ar@{}[dr]|{\DS\boxempty}
&
X \ar[d]^\pi
\\
S' \ar[r]_f
& S
}
\]
then we have canonical isomorphisms
\begin{align*}
&&&&
f'^*\pi^* &\simeq \pi'^* f^* &&&\mathrm{COM^*(\boxempty)}
\\
&&&&
\pi_* f'_* &\simeq f_* \pi'_* &&&\mathrm{COM_*(\boxempty)}
\\
&&&&
\pi_! f'_! &\simeq f_! \pi'_! &&&\mathrm{COM_!(\boxempty)}
\\
&&&&
f'^!\pi^! &\simeq \pi'^! f^! &&&\mathrm{COM^!(\boxempty)}
\\
\end{align*}
For example, we have
$f'^*\pi^* \simeq (\pi \circ f')^* \simeq (f \circ \pi')^* \simeq \pi'^* f^*$,
and similarly for the other relations. We will use similar notations
for other kinds of diagrams, like triangles.

If moreover the square $\boxempty$ is cartesian, then we can apply the
Proper Base Change Theorem (because of the assumptions on the schemes
we made in \ref{sec:context}). Thus we have a canonical isomorphism
\begin{align*}
&&&&
f^* \pi_! &\simeq \pi'_! f'^* &&&\mathrm{PBCT(\boxempty)}
\end{align*}
In other words, the functor $\pi_!$ commutes with any base change.
This theorem is stated and proved in \cite[XVII, Th. 5.2.6]{SGA4} for
$\EM$-sheaves when $\EM$ is a finite ring of order prime to $p$. We
will freely use it for the derived categories $D^b_c(X,\EM)$, where
$\EM = \KM$, $\OM$ or $\FM$. Recall that we omit the $R$ for the
derived functors associated to a morphism.

The same remarks apply for the Projection Formula
\cite[XVII, Prop. 5.2.9]{SGA4}. Let $\pi : X \to S$ be a morphism.
Then we have a functorial isomorphism
\begin{align*}
&&&
L \otimes^\LM_\EM \pi_!(K) \simeq \pi_!\ (\pi^* L \otimes^\LM_\EM K) 
&&\mathrm{PROJ(\pi)}
\end{align*}
for $(K,L)$ in $D^b_c(X,\EM) \times D^b_c(S,\EM)$.

We also have a ``distributivity'' functorial isomorphism
\begin{align*}
&&&
\pi^*(K_1 \otimes^\LM_\EM K_2) \simeq \pi^* K_1 \otimes^\LM_\EM \pi^* K_2 
&&\mathrm{DISTR(\pi)}
\end{align*}
for $(K_1,K_2)$ in $D^b_c(S,\EM) \times D^b_c(S,\EM)$.

More generally, suppose we have a commutative diagram (with cartesian square)
\[
\xymatrix{
& Z
\ar@/_.4cm/[ddl]_{\pi_1}
\ar[d]_\pi
\ar@/^.4cm/[ddr]^{\pi_2}
\\
& X_1 \times_S X_2
\ar[dl]^{p_1}
\ar[dr]_{p_2}
\\
X_1 \ar[dr]^{q_1}
&&
X_2 \ar[dl]_{q_2}
\\
& S
}
\]
For $(K_1,K_2)$ in $D^b_c(X_1,\EM) \times D^b_c(X_2,\EM)$,
we set
\[
K_1 \boxtimes^\LM_S K_2 := p_1^* K_1 \otimes^\LM_\EM p_2^* K_2
\]

Then we have a functorial isomorphism
\[
\pi^* (K_1 \boxtimes^\LM_S K_2) \simeq \pi_1^* K_1 \otimes^\LM_\EM \pi_2^* K_2
\]
for $(K_1, K_2)$ in $D^b_c(X_1, \EM) \times D^b_c(X_2, \EM)$.
Indeed, we have
\[
\pi^* (K_1 \boxtimes^\LM_S K_2)
= \pi^* (p_1^* K_1 \otimes^\LM_\EM p_2^* K_2)
= \pi^* p_1^* K_1 \otimes^\LM_\EM \pi^* p_2^* K_2
= \pi_1^* K_1 \otimes^\LM_\EM \pi_2^* K_2
\]
(using the definition, DISTR$(\pi)$, and $\mathrm{COM^*}$ of the two triangles).

If we use the notation $\pi = (\pi_1,\pi_2)$, this can be rewritten as follows:
\begin{align*}
&&&&
(\pi_1,\pi_2)^* (K_1 \boxtimes^\LM_S K_2) \simeq \pi_1^* K_1 \otimes^\LM_\EM \pi_2^* K_2
&&&\mathrm{DISTR(\pi_1,\pi_2)}
\end{align*}

We will also need the Künneth formula \cite[XVII, Th. 5.4.3]{SGA4}.
Suppose we have a commutative diagram
\[
\xymatrix{
& Y_1 \times_S Y_2
\ar[dl]_{\pi_1}
\ar[d]^f
\ar[dr]^{\pi_2}
\\
Y_1
\ar[d]_{f_1}
& X_1 \times_S X_2
\ar[dl]^{p_1}
\ar[dr]_{p_2}
& Y_2
\ar[d]^{f_2}
\\
X_1 \ar[dr]^{q_1}
&& X_2\ar[dl]_{q_2}
\\
& S
}
\]
Then we have a functorial isomorphism
\begin{align*}
&&&&
f_{1!} K_1 \boxtimes^\LM_S f_{2!} K_2 \simeq f_! (K_1 \boxtimes^\LM_S K_2)
&&&\mathrm{KUNNETH}
\end{align*}
for $(K_1, K_2)$ in $D^b_c(Y_1, \EM) \times D^b_c(Y_2, \EM)$.

Finally, we will need the duality theorems \cite[II.7]{KW}
Let $f : Y \to X$ be a morphism. Then we have a functorial isomorphism
\begin{align*}
&&&&
\RHOM(f_! L, K) \simeq f_* \RHOM(L, f^! K)
&&&\mathrm{DUAL}(f)
\end{align*}
for $(K,L)$ in $D^b_c(X,\EM) \times D^b_c(Y,\EM)$.
As a consequence, we have the second duality theorem: we have a
functorial isomorphism
\begin{align*}
&&&&
f^! \RHOM(K_1, K_2) \simeq \RHOM(f^* K_1, f^! K_2)
&&&\mathrm{DUAL_2}(f)
\end{align*}

\section{Definition, first properties and examples}

\subsection{Definition}

Let us assume that $\EM^\times$ contains a primitive root of unity of order $p$.
We fix a non-trivial character $\psi : \FM_p \to \EM^\times$, that is,
a primitive root $\psi(1)$ of order $p$ in $\EM^\times$.
Composing with $\Tr_{\FM_q/\FM_p}$, we get a character of $\FM_q$.
Let $\LC_\psi$ be the locally constant $\EM$-sheaf of rank $1$ 
on $\GM_{a}$ associated to $\psi$ (the corresponding Artin-Schreier
local system). It is endowed with a ``rigidification'' at the origin:
\begin{align*}
&&&
(\LC_\psi)_{|0} \simeq \EM
&&&\mathrm{RIG}
\end{align*}
Moreover, we have an isomorphism
\begin{align*}
&&&
\LC_\psi \boxtimes^\LM \LC_\psi \simeq m^* \LC_\psi
&&&\mathrm{ADD}
\end{align*}
where $m : \AM^1 \times \AM^1 \to \AM^1$ is the addition morphism.
If $f,g : X \to \AM^1$ are two morphisms, we set
$f + g = m \circ(f,g)$. We have the following commutative diagram:

\[
\xymatrix{
& X
\ar@/_.5cm/[ddl]_f
\ar@/^.5cm/[ddr]^g
\ar[d]^{(f,g)}
\\
& \AM^1 \times \AM^1
\ar[dl]_{\pr_1}
\ar[dr]^{\pr_2}
\ar[d]^m
\\
\AM^1 & \AM^1 & \AM^1
}
\]

Then the following chain of isomorphisms:
\[
\begin{array}{rcll}
(f + g)^* \LC_\psi 
&\simeq& (f,g)^* m^* \LC_\psi
\\
&\simeq& (f,g)^* (\LC_\psi \boxtimes^\LM \LC_\psi)
&\text{by ADD}
\\
&\simeq& f^* \LC_\psi \otimes^\LM_\EM g^* \LC_\psi
&\text{by DISTR}(f,g)
\end{array}
\]
When we use this result, we will simply refer to it as ADD.

Now let $S$ be a variety, and $E\elem{\pi} S$
a vector bundle of constant rank $r \geqslant 1$.
We denote by $E' \elem{\pi'} S$ its dual vector bundle,
by $\mu : E \times_S E' \to \AM^1$ the canonical pairing,
and by $\pr : E \times_S E' \to E$ and $\pr' : E \times_S E' \to E'$
the canonical projections. So we have the following diagram.

\[
\xymatrix{
& E\times_S E'
\ar[dl]_\pr
\ar[dr]^{\pr'}
\ar[rr]^\mu 
&& \AM^1\\
E \ar[dr]_\pi 
\ar@{}[rr] | {\DS\boxempty}
&& E' \ar[dl]^{\pi'}\\
&S
}
\]

\begin{defi}
The Fourier-Deligne transform for $E \elem{\pi} S$, associated to the
character $\psi$, is the triangulated functor
\[
\FC_\psi : D^b_c(E,\EM) \longto D^b_c(E',\EM)
\]
defined by
\[
\FC_\psi(K) = {\pr'}_! (\pr^* K \otimes^\LM_\EM \mu^*\LC_\psi) [r]
\]
\end{defi}

In the sequel, we will drop the indices $\psi$ from the notations
$\FC_\psi$ and $\LC_\psi$ when no confusion may arise.

\subsection{First properties}

Let $E''\elem{\pi''}S$ be the bidual vector bundle of $E\elem{\pi}S$
and $a : E \elem{\sim}E''$ the $S$-isomorphism defined by $a(e) = -\mu(e,-)$
(that is, the opposite of the canonical $S$-isomorphism). We will denote by
$\s:S\to E$, $\s':S\to E'$ and $\s'':S\to E''$ the respective null sections
of $\pi$, $\pi'$ and $\pi''$. Finally, we will denote by
$s:E\times_S E\to E$ the addition of the vector bundle $E\elem{\pi}S$
and by $- 1_E:E\to E$ the opposite for this addition.

The following Proposition is the analogue of the fact that the Fourier transform
of the constant function is a Dirac distribution supported at the origin in
classical Fourier analysis. By the function/sheaf dictionary, this becomes
a functorial isomorphism, to which we will refer as DIRAC.

\begin{prop}\label{prop:dirac}
We have a functorial isomorphism 
\begin{align*}
&&&&
\FC(\pi^* L[r]) \simeq \s'_* L (-r)&&&\mathrm{(DIRAC)}
\end{align*}
for all objects $L$ in $D^b_c(S,\EM)$.
\end{prop}

\begin{proof}
We have the following diagram
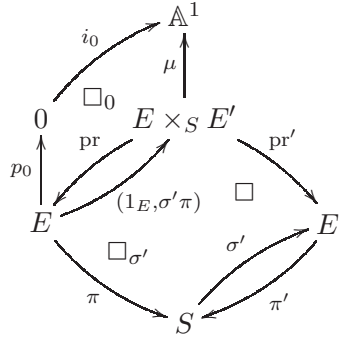
\[
\xymatrix{
& \AM^1 &
\\
0 \ar@/^/[ur]^{i_0}
\ar@{}[urr] |(.2) {\DS{\boxempty_0}}
& E\times_S E'
\ar@/_/[dl]_\pr
\ar@/^/[dr]^{\pr'}
\ar[u]^\mu 
\\
E \ar@/_/[dr]_\pi
\ar@/_/[ur]_{(1_E,\s'\pi)} 
\ar[u]^{p_0}
\ar@{}[drr] |(.3){\DS{\boxempty_{\s'}}}
&& E' \ar@/^/[dl]^{\pi'}
\ar@{}[ull]|(.3) {\DS\boxempty}
\\
&S \ar@/^/[ur]^{\s'} &
}
\]
where $\boxempty$ denotes the cartesian square containing $\pr$ and $\pi'$,
and $\boxempty_{\s'}$ the cartesian square containing $(1_E,\s'\pi)$ and
$\s'$. The square $\boxempty_0$ containing $i_0$ and $(1_E,\s'\pi)$
is commutative. Moreover, $\s'$ is a section of $\pi'$, and
$(1_E,\s'\pi)$ is a section of $\pr$.

We begin by showing that $\pr'_!\:\mu^*\LC = \s'_* \EM_S [-2r] (-r)$.
Let us first prove that $\pr'_!\:\mu^*\LC$ is concentrated on the
closed subset $\s'(S)$ of $E'$.
Let $e' \in E' - \s'(S)$. Set $s = \pi'(e')$.
We have a commutative diagram with cartesian square
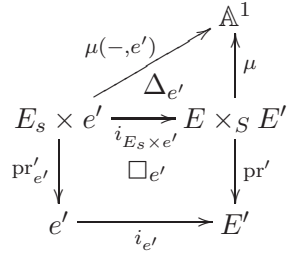
\[
\xymatrix{
& \AM^1\\
E_s\times e' \ar[d]_{\pr'_{e'}} \ar[r]_{i_{E_s \times e'}}
\ar[ur]^{\mu(-,e')}_{\DS{\D_{e'}}}
\ar@{}[dr] | {\DS{\boxempty_{e'}}}
& E\times_S E' \ar[d]^{\pr'} \ar[u]_\mu\\
e' \ar[r]_{i_{e'}} & E'
}
\]

Fix a basis $(e'_1,\ldots,e'_r)$ of $E'_s$ such that $e'_1 = e'$,
and let $(e_1,\ldots,e_r)$ be the dual basis of $E_s$. We have
\[
\begin{array}{rcll}
(\pr'_!\:\mu^*\LC)_{e'}
&=& \rgc(E_s \times e', i_{E_s \times e'}^* \mu^* \LC)
&\text{by PBCT}(\boxempty_{e'})\\
&=& \rgc(E_s,\mu(-,e')^*\LC)
&\text{by COM}^*(\D_{e'})\\
&=&
\rgc((\AM^1)^r,\LC\boxtimes^\LM \EM \boxtimes^\LM\cdots \boxtimes^\LM\EM)
&\text{by the choice of the basis}
\\
&=& \rgc(\AM^1,\LC) \otimes^\LM_\EM \rgc(\AM^1,\EM) \otimes^\LM_\EM
\cdots \otimes^\LM_\EM \rgc(\AM^1,\EM)
&\text{by KUNNETH}\\
&=& 0 
\end{array}
\]
since $\rgc(\AM^1,\LC) = 0$.
Thus $\pr'_!\:\mu^*\LC$ is concentrated in the closed subset
$\s'(S)$ of $E'$, and we have
\begin{equation}
\begin{array}{rcll}
\pr'_!\:\mu^*\LC 
&=& \s'_* \s'^* \pr'_!\:\mu^*\LC
&\text{by the above}
\\
&=& \s'_* \pi_! (1_E,\s'\pi)^* \mu^* \LC
&\text{by PBCT}(\boxempty_{\s'})
\\
&=& \s'_* \pi_! p_0^* i_0^* \LC
&\text{by COM}^*(\boxempty_0)
\\
&=& \s'_* \pi_! p_0^* \EM
&\text{by RIG}
\\
&=& \s'_* \pi_! \EM_E
\\
&=& \s'_* \EM_S [-2r] (-r)
\end{array}
\end{equation}

We are now ready to estimate the Fourier-Deligne transform of $\pi^* L[r]$.
\[
\begin{array}{rcll}
\FC(\pi^* L[r])
&=& \pr'_!\:(\pr^*\ \pi^*L[r] \otimes^\LM_\EM \mu^*\LC) [r]
&\text{by definition}
\\
&=& \pr'_!\:(\pr'^*\:\pi'^*L \otimes^\LM_\EM \mu^*\LC) [2r]
&\text{by COM}^*(\boxempty)
\\
&=& \pi'^* L \otimes^\LM_\EM \pr'_!\:\mu^*\LC [2r]
& \text{by PROJ}(\pr')
\\
&=& \pi'^* L \otimes^\LM_\EM \s'_* \EM_S (-r)
& \text{by the above}
\\
&=& \s'_* (\s'^* \pi'^* L \otimes^\LM_\EM \EM_S) (-r)
& \text{by PROJ}(\s')
\\
&=& \s'_* L (-r)
\end{array}
\]
\end{proof}

\begin{theo}\label{th:inv}
Let $\FC'$ be the Fourier-Deligne transform, associated to the character
$\psi$, of the vector bundle $E'\elem{\pi'}S$. Then we have a functorial
isomorphism
\begin{align*}
&&&&
\FC' \circ \FC (K) \simeq a_* K (-r)&&&\mathrm{(INV)}
\end{align*}
for all objects $K$ in $D^b_c(E,\EM)$.
\end{theo}

\begin{proof}
We have a commutative diagram (with plain arrows) with cartesian
squares $\boxempty_{1,2,3}$
\[
\xymatrix{
E 
&& E\times_S E'' \ar@/_/[ll]_{\pr^{13}_1} \ar@/^/[rr]^{\pr^{13}_3}
\ar@/^-.5cm/@{-->}[rr]^\b
&& E''
\\
&\qquad \boxempty_1 && \boxempty_{3,3'}
\\
&& E\times_S E' \times_S E'' \ar@/^1cm/@{-->}[dr]^\a
\ar[uu]_{\pr_{13}}
\ar@/^/[dl]_{\pr_{12}}
\ar@/_/[dr]^{\pr_{23}}
\\
& E\times_S E' \ar@/^/[uuul]^{\pr^{12}_1} \ar@/_/[ddr]_{\pr^{12}_2}
&& E' \times_S E'' \ar@/_/[uuur]_{\pr^{23}_3} \ar@/^/[ddl]^{\pr^{23}_2}
\\
&&\boxempty_2
\\
&& E'
}
\]
and the square $\boxempty'_3$ containing the dashed arrows $\a$ and
$\b$ is also cartesian, with $\a(e,e',e'') = (e', e'' - a(e))$ and
$\b(e,e'') = e'' - a(e)$.  We have $\mu\:\pr_{12} + \mu'\:\pr_{23} =
\mu'\:\a$, where $\mu':E'\times_S E'' \to \AM^1$ is the canonical
pairing.  Indeed, we have
\[
\mu' \circ \a (e,e',e'') = \mu'(e',e''-a(e)) = \mu'(e',e'') - \mu'(e',a(e))
= \mu'(e',e'') + \mu(e,e')
\]

Therefore we have
\begin{align*}
\pr_{12}^* \: \mu^*\LC \otimes^\LM_\EM \pr_{23}^*\: \mu'^*\LC
&= (\mu\: \pr_{12})^* \LC \otimes^\LM_\EM (\mu'\: \pr_{23})^* \LC\\
&= (\mu\: \pr_{12} + \mu'\: \pr_{23})^* \LC &\text{by ADD}\\
&= (\mu'\: \a)^* \LC\\
&= \a^* \mu'^* \LC
\end{align*}

Let us now start to calculate $\FC' \circ \FC(K)$.
\[
\begin{array}{lll}
\multicolumn{2}{l}{\FC' \circ  \FC (K)}\\
=& {\pr^{23}_3}_{!}\ ({\pr^{23}_2}^* \  
{\pr^{12}_2}_{!}\ ({\pr^{12}_1}^* K \otimes^\LM_\EM \mu^*\LC)
\otimes^\LM_\EM \mu'^*\LC) [2r]
&\text{by definition}
\\
=& {\pr^{23}_3}_{!}\ ({\pr_{23}}_{!} \ {\pr_{12}}^*\
({\pr^{12}_1}^* K \otimes^\LM_\EM \mu^*\LC)
\otimes^\LM_\EM \mu'^*\LC) [2r]
&\text{by  PBCT}(\boxempty_2)
\\
=&
{\pr^{23}_3}_{!}\ ({\pr_{23}}_{!} \
({\pr_{12}}^* \  {\pr^{12}_1}^* K \otimes^\LM_\EM
{\pr_{12}}^* \  \mu^*\LC)
\otimes^\LM_\EM \mu'^*\LC) [2r]
&\text{by DISTR}
\\
=&
{\pr^{23}_3}_{!}\ ({\pr_{23}}_{!} \
({\pr_{13}}^* \  {\pr^{13}_1}^* K \otimes^\LM_\EM
{\pr_{12}}^* \  \mu^*\LC)
\otimes^\LM_\EM \mu'^*\LC) [2r]
&\text{by COM}^*(\boxempty_1)
\\
=&
{\pr^{23}_3}_{!}\  {\pr_{23}}_{!} \
(({\pr_{13}}^* \  {\pr^{13}_1}^* K \otimes^\LM_\EM
{\pr_{12}}^* \  \mu^*\LC)
\otimes^\LM_\EM {\pr_{23}}^*\ \mu'^*\LC) [2r]
&\text{by PROJ}(\pr_{23})
\\
=&
{\pr^{13}_3}_{!}\  {\pr_{13}}_{!} \
({\pr_{13}}^* \  {\pr^{13}_1}^* K \otimes^\LM_\EM
{\pr_{12}}^* \  \mu^*\LC
\otimes^\LM_\EM {\pr_{23}}^*\ \mu'^*\LC) [2r]
&\text{by COM}_!(\boxempty_3)
\\
=&
{\pr^{13}_3}_{!}\
({\pr^{13}_1}^* K \otimes^\LM_\EM
{\pr_{13}}_{!} \
({\pr_{12}}^* \  \mu^*\LC
\otimes^\LM_\EM {\pr_{23}}^*\ \mu'^*\LC)) [2r]
&\text{by PROJ}(\pr_{13})
\\
=& {\pr^{13}_3}_{!}\ ({\pr^{13}_1}^* K \otimes^\LM_\EM {\pr_{13}}_!\ \a^* \mu'^* \LC) [2r]
&\text{by the above}
\\
=& {\pr^{13}_3}_{!}\ ({\pr^{13}_1}^* K \otimes^\LM_\EM \b^* {\pr^{23}_3}_!\  \mu'^* \LC) [2r]
&\text{by PBCT}(\boxempty'_3)
\\
=& {\pr^{13}_3}_{!}\ ({\pr^{13}_1}^* K \otimes^\LM_\EM \b^* 
(\EM_{E''}[r] \otimes^\LM_\EM {\pr^{23}_3}_!\ \mu'^* \LC) [r])
\\
=& {\pr^{13}_3}_{!}\ ({\pr^{13}_1}^* K \otimes^\LM_\EM \b^* {\pr^{23}_3}_!\ 
({\pr^{23}_3}^* \EM_{E''}[r] \otimes^\LM_\EM \mu'^* \LC) [r])
&\text{by PROJ}(\pr^{23}_3)
\\
=& {\pr^{13}_3}_{!}\ ({\pr^{13}_1}^* K \otimes^\LM_\EM \b^* {\pr^{23}_3}_!\ 
(\EM_{E' \times_S E''}[r] \otimes^\LM_\EM \mu'^* \LC) [r])
\\
=& {\pr^{13}_3}_{!}\ ({\pr^{13}_1}^* K \otimes^\LM_\EM \b^* {\pr^{23}_3}_!\ 
({\pr^{23}_2}^* \EM_{E'}[r] \otimes^\LM_\EM \mu'^* \LC) [r])
\\
=& {\pr^{13}_3}_{!}\ ({\pr^{13}_1}^* K \otimes^\LM_\EM \b^*
\FC'(\EM_{E'}[r]) )
\\
=& {\pr^{13}_3}_{!}\ ({\pr^{13}_1}^* K \otimes^\LM_\EM \b^* \s''_* \EM_S (-r))
&\text{by DIRAC}
\end{array}
\]

Now, we have a cartesian square (which we will denote by $\boxempty$)
\[
\xymatrix{
E \ar[r]^\pi \ar[d]_\d & S \ar[d]^{\s^{''}}\\
E \times_S E'' \ar[r]^\b & E''
}
\]
where $\d(e) = (e,a(e))$, and $\s''$ (hence $\d$) is finite and proper since
it is a closed immersion. Thus we have
\[
\begin{array}{rcll}
\FC' \circ \FC (K)
&=& {\pr^{13}_3}_{!}\ ({\pr^{13}_1}^* K \otimes^\LM_\EM \d_* \pi^* \EM_S (-r))
&\text{by PBCT}(\boxempty)
\\
&=& {\pr^{13}_3}_{!}\ ({\pr^{13}_1}^* K \otimes^\LM_\EM \d_* \EM_E (-r))
\\
&=& {\pr^{13}_3}_{!}\ \d_* (\d^* {\pr^{13}_1}^* K \otimes^\LM_\EM \EM_E (-r))
&\text{by PROJ}(\d)
\\
&=& {\pr^{13}_3}_{!}\ \d_* (K(-r))
&\text{since } \pr^{13}_1 \circ \d = 1_E
\\
&=&  a_*(K(-r))
&\text{since } \pr^{13}_3 \circ \d = a
\end{array}
\]

The proof is complete.
\end{proof}

\begin{cor}\label{cor:equiv db}
The triangulated functor $\FC$ is an equivalence
of triangulated categories from
$D^b_c(E,\EM)$ to $D^b_c(E',\EM)$, with quasi-inverse
$a^* \FC'(-)(r)$.
\end{cor}

\begin{theo}\label{th:morphism}
Let $f : E_1 \to E_2$ be a morphism of vector bundles over $S$,
with constant ranks $r_1$ and $r_2$ respectively, and
let $f' : E'_2 \to E'_1$ denote the transposed morphism.
Then we have a functorial isomorphism
\begin{align*}
&&&&
\FC_2(f_! K_1) \simeq f'^* \FC_1(K_1)[r_2 - r_1]
&&&\mathrm{(MOR)}
\end{align*}
for $K_1$ in $D^b_c(E_1,\EM)$.
\end{theo}

\begin{proof}
First, by adjunction of $f$ and $f'$, we have
\[
\mu_2 \circ (f \times_S 1_{E'_2}) = \mu_1 \circ (1_{E_1} \times_S f')
\]
where $\mu_i$, $i = 1,2$, is the pairing $E_i \times_S E'_i \to \AM^1$.
That is, the following diagram, which we will denote by $\boxempty$, is commutative.
\[
\xymatrix{
E_1 \times_S E'_2
\ar[r]^{1_{E_1} \times_S f'}
\ar[d]_{f \times_S 1_{E'_2}}
&
E_1 \times_S E'_1
\ar[d]^{\mu_1}
\\
E_2 \times_S E'_2
\ar[r]_{\mu_2}
&
\AM^1
}
\]

We also have a commutative diagram with cartesian squares and commutative triangles
\[
\xymatrix@=1cm{
&& E_1 \times_S E'_1
\ar@/^1.3cm/[ddrr]^{\pr'_1}_{\DS\D'}
\ar@/_1.3cm/[ddll]_{\pr_1}^{\DS\D}
\ar[dr]^{f \times_S 1_{E'_1}}
\\
& E_1 \times_S E'_2
\ar[ur]^{1_{E_1} \times_S f'}
\ar[dl]_{p_1}
\ar[dr]^{f \times_S 1_{E'_2}}
& \boxempty_1 & E_2 \times_S E'_1
\ar[dr]^{p'_1}
\\
E_1 \ar[dr]^f
& \boxempty_2
& E_2 \times_S E'_2
\ar[ur]^{1_{E_2} \times_S f'}
\ar[dl]^{\pr_2}
\ar[dr]_{\pr'_2}
& \boxempty'_2
& E'_1
\\
& E_2
&& E'_2
\ar[ur]^{f'}
}
\]

The result is proved by successive applications of the proper base change
theorem and of the projection formula, following the diagram:
\[
\begin{array}{cll}
\multicolumn{2}{l}{\FC_2(f_!\ K_1)}\\
=& {\pr'_2}_!\ ({\pr_2}^* f_!\ K_1 \otimes^\LM_\EM \mu_2^* \LC) [r_2]
&\text{by definition}
\\
=& {\pr'_2}_!\ ((f \times_S 1_{E'_2})_!\ {p_1}^* K_1 \otimes^\LM_\EM \mu_2^* \LC) [r_2]
&\text{by PBCT}(\boxempty_2)
\\
=& {\pr'_2}_!\ (f \times_S 1_{E'_2})_!\
    ({p_1}^* K_1 \otimes^\LM_\EM (f \times_S 1_{E'_2})^* \mu_2^* \LC) [r_2]
&\text{by PROJ}(f \times_S 1_{E'_2})
\\
=& {\pr'_2}_!\ (f \times_S 1_{E'_2})_!\
    ((1_{E_1} \times_S f')^* {\pr_1}^* K_1 \otimes^\LM_\EM (1_{E_1} \times_S f')^* \mu_1^* \LC) [r_2]
&\text{by COM}^*(\boxempty,\D)
\\
=& {\pr'_2}_!\ (f \times_S 1_{E'_2})_!\ (1_{E_1} \times_S f')^* 
    ({\pr_1}^* K_1 \otimes^\LM_\EM \mu_1^* \LC) [r_2]
&\text{by DISTR}(1_{E_1} \times_S f')
\\
=& f'^* {p'_1}_!\ (f \times_S 1_{E'_1})_!\ 
    ({\pr_1}^* K_1 \otimes^\LM_\EM \mu_1^* \LC) [r_2]
&\text{by PBCT}(\boxempty_1, \boxempty'_2)
\\
=& f'^* {\pr'_1}_!\ 
    ({\pr_1}^* K_1 \otimes^\LM_\EM \mu_1^* \LC) [r_2]
&\text{by COM}_!(\D')
\\
=& f'^*\: \FC_1(K_1) [r_2 - r_1]
\end{array}
\]

The proof is complete.
\end{proof}

\begin{cor}\label{cor:ev}
We have a functorial isomorphism
\begin{align*}
&&&&
\pi'_! \FC(K) \simeq \s^* K (-r)[-r]
&&&\mathrm{(EV)}
\end{align*}
for $K$ in $D^b_c(E,\EM)$.
\end{cor}

\begin{proof}
Consider the morphism $\pi' : E' \to S$ of vector bundles over $S$.
The transposed morphism is $\s'' : S \to E''$.  We can apply Theorem
\ref{th:morphism} (MOR) to $E_1 = E'$, $E_2 = S$, $f = \pi'$ and $K_1
= \FC(K)$. Then we have $\FC_1 = \FC'$ and $\FC_2 =
\id_{D^b_c(S,\EM)}$. We get
\[
\pi'_!\ \FC(K)
= \FC_2(\pi'_!\ \FC(K))
= {\s''}^* \FC' \circ \FC (K) [-r]
= \s^* a^* a_* K (-r) [-r]
= \s^* K (-r) [-r]
\]
\end{proof}

\begin{defi}
The convolution product for $E \elem{\pi} S$ is the internal operation
\[
* : D^b_c(E,\EM) \times D^b_c(E,\EM) \longto D^b_c(E,\EM)
\]
defined by
\[
K_1 * K_2 = s_! (K_1 \boxtimes^\LM_S K_2)
\]
\end{defi}

\begin{prop}\label{prop:conv}
We have a functorial isomorphism
\begin{align*}
&&&&
\FC(K_1 * K_2) \simeq \FC(K_1) \otimes^\LM_\EM \FC(K_2) [-r]
&&&\mathrm{(CONV)}
\end{align*}
for $(K_1,K_2)$ in $D^b_c(E,\EM) \times D^b_c(E,\EM)$.
\end{prop}

\begin{proof}
We have a commutative diagram
\[
\xymatrix{
&& E \times_S E
\ar[dll]_{p_1}
\ar[drr]^{p_2}
\\
E
& \boxempty_1
& E \times_S E' \times_S E \times_S E'
\ar[dll]_{P_1}
\ar[drr]^{P_2}
\ar[u]^{P}
\ar[d]^{P'}
& \boxempty_2
& E
\\
E \times_S E'
\ar[u]^{\pr}
\ar[d]^{\pr'}
& \boxempty'_1
& E' \times_S E'
\ar[dll]_{p'_1}
\ar[drr]^{p'_2}
& \boxempty'_2
& E \times_S E'
\ar[u]^{\pr}
\ar[d]^{\pr'}
\\
E'
\ar[drr]_{\pi'}
&&&& E'
\ar[dll]^{\pi'}
\\
&& S
}
\]

Let $M : E \times_S E' \times_S E \times_S E' \to \AM^1$
be the pairing defined by $M(e_1, e'_1, e_2, e'_2) = \mu(e_1, e'_1) + \mu(e_2, e'_2)$.
Thus we have $M = \mu\: P_1 + \mu\: P_2$. By ADD, we deduce that
\[
M^* \LC = P_1^*\: \mu^* \LC \otimes^\LM_\EM P_2^*\: \mu^* \LC
\]

Let us still denote by $\FC$ the Fourier-Deligne transform for
$E \times_S E \elem{\pi \times_S \pi} S$. Then we have
\[
\FC(\Kti) = P'_!\ (P^*\Kti \otimes^\LM_\EM M^* \LC)
= P'_!\ (P^*\Kti \otimes^\LM_\EM P_1^*\: \mu^* \LC \otimes^\LM_\EM P_2^*\: \mu^* \LC)
\]
for $\Kti$ in $D^b_c(E \times_S E)$

We are now ready to prove the proposition. The crucial point relies on the K\"unneth formula.
\[
\begin{array}{lll}
\multicolumn{2}{l}{\FC(K_1) \boxtimes^\LM_S \FC(K_2)}\\
=& \pr'_!\ (\pr^* K_1 \otimes^\LM_\EM \mu^* \LC)
\boxtimes^\LM_S \pr'_!\ (\pr^* K_2 \otimes^\LM_\EM \mu^* \LC) [2r]
&\text{by definition}
\\
=& P'_!\ \left(
(\pr^* K_1 \otimes^\LM_\EM \mu^* \LC) \boxtimes^\LM_S (\pr^* K_2 \otimes^\LM_\EM \mu^* \LC)
\right) [2r]
&\text{by KUNNETH}
\\
=& P'_!\ \left(
P_1^*(\pr^* K_1 \otimes^\LM_\EM \mu^* \LC) \otimes^\LM_\EM P_2^*(\pr^* K_2 \otimes^\LM_\EM \mu^* \LC)
\right) [2r]
&\text{by definition}
\\
=& P'_!\ \left(
P_1^* \pr^* K_1 \otimes^\LM_\EM P_1^* \mu^* \LC \otimes^\LM_\EM P_2^* \pr^* K_2 \otimes^\LM_\EM P_2^* \mu^* \LC
\right) [2r]
&\text{by DISTR}(P_{1,2})
\\
=& P'_!\ \left(
P^* p_1^* K_1 \otimes^\LM_\EM P^* p_2^* K_2 \otimes
P_1^*\: \mu^* \LC \otimes^\LM_\EM P_2^*\: \mu^* \LC
\right) [2r]
&\text{by COM}^*(\boxempty_{1,2})
\\
=& P'_!\ \left(
P^* (p_1^* K_1 \otimes^\LM_\EM p_2^* K_2) \otimes
M^*\LC
\right) [2r]
&\text{by DISTR}(P)\text{ and the above}
\\
=& \FC(K_1 \boxtimes^\LM_S K_2)
\end{array}
\]

Now, we just need to apply Theorem \ref{th:morphism} (MOR) to the morphism
$s : E \times_S E \to E$ and the complex $K_1 \boxtimes^\LM_S K_2$.
So we take $E_1 = E \times_S E$, $E_2 = E$, $f = s$ and $K_1 := K_1 \boxtimes^\LM_S K_2$.
Let us remark that $s' : E' \to E' \times_S E'$ is the diagonal embedding $s'(e') = (e',e')$,
so that $p'_i \circ s' = 1_{E'}$ for $i = 1, 2$. We get
\[
\begin{array}{rcll}
\FC(K_1 * K_2)
&=& \FC(s_!\ (K_1 \boxtimes^\LM_S K_2))
&\text{by definition}
\\
&=& s'^* \FC(K_1 \boxtimes^\LM_S K_2) [-r]
&\text{by MOR}
\\
&=& s'^* (\FC(K_1) \boxtimes^\LM_S \FC(K_2)) [-r]
&\text{by the above}
\\
&=& s'^* ({p'_1}^* \FC(K_1) \otimes^\LM_\EM {p'_2}^* \FC(K_2)) [-r]
&\text{by definition}
\\
&=& s'^* {p'_1}^* \FC(K_1) \otimes^\LM_\EM s'^* {p'_2}^* \FC(K_2) [-r]
&\text{by DISTR}(s')
\\
&=& \FC(K_1) \otimes^\LM_\EM \FC(K_2) [-r]
&\text{since } p'_i \circ s' = 1_{E'},\ i = 1,2
\end{array}
\]

The proof is complete.
\end{proof}

\begin{prop}\label{prop:plancherel}
We have a ``Plancherel'' functorial isomorphism
\begin{align*}
&&&&
\pi'_! (\FC(K_1) \otimes^\LM_\EM \FC(K_2)) \simeq \pi_!(K_1 \otimes^\LM_\EM (-1_E)^* K_2)(-r)
&&&\mathrm{(PL)}
\end{align*}
for $(K_1,K_2)$ in $D^b_c(E,\EM) \times D^b_c(E,\EM)$.
\end{prop}

\begin{proof}
We have a cartesian square (which we will denote by $\boxempty$)
\[
\xymatrix@=1.5cm{
E \ar[r]_-{(1_E, - 1_E)} \ar[d]_\pi
& E \times_S E \ar[d]^s
\\
S \ar[r]_\s
& E
}
\]

\[
\begin{array}{rcll}
\pi'_!\ (\FC(K_1) \otimes^\LM_\EM \FC(K_2))
&=& \pi'_!\ \FC(K_1 * K_2) [r]
&\text{by CONV}
\\
&=& \s^* (K_1 * K_2) (-r)
&\text{by EV}
\\
&=& \s^* s_!\ (K_1 \boxtimes^\LM_S K_2) (-r)
&\text{by definition}
\\
&=& \pi_!\ (1_E, - 1_E)^* (K_1 \boxtimes^\LM_S K_2) (-r)
&\text{by PBCT}(\boxempty)
\\
&=& \pi_!\ (K_1 \otimes^\LM_\EM (- 1_E)^* K_2)(-r)
&\text{by DISTR'}
\end{array}
\]
\end{proof}

\begin{prop}\label{prop:bc^*}
The formation of $\FC(K)$, for an object $K$ in $D^b_c(E,\EM)$, commutes
with any base change $S_1 \to S$. That is, if we fix the notations by
the diagram
\[
\xymatrix@=1cm{
&
E_1 \times_{S_1} E'_1
\ar[dl]^{\pr_1}
\ar[dr]_{\pr'_1}
\ar[drrr]^F
\ar[rrrrr]^{\mu_1}
&&&&&
\AM^1
\\
E_1
\ar[dr]_{\pi_1}
\ar[drrr]^(.3){f_E}
&&
E'_1
\ar[dl]_(.3){\pi'_1} |!{[ll];[dr]}\hole
\ar[drrr]^(.3){f_{E'}} |!{[rr];[dr]}\hole
&&
E \times_S E'
\ar[urr]_\mu
\ar[dl]^(.7)\pr
\ar[dr]^{\pr'}
\ar@{}[ul] | {\DS\D}
\\
&
S_1
\ar[drrr]_f
&& E \ar[dr]_\pi
&& E' \ar[dl]^{\pi'}
\\
&&&& S
}
\]
then we have a functorial isomorphism
\begin{align*}
&&&&
\FC_1 (f_E^* K) \simeq f_{E'}^* \FC(K)
&&&\mathrm{(BC^*)}
\end{align*}
for $K$ in $D^b_c(E,\EM)$,
where $\FC_1$ denotes the Fourier-Deligne transform associated to
$E_1 \elem{\pi_1} S$.
\end{prop}

\begin{proof}
Let us denote by $\boxempty_{\pr}$ (respectively $\boxempty_{\pr'}$)
the cartesian square containing $F$ and $\pr$ (respectively $\pr'$).
Then we have
\[
\begin{array}{rcll}
\FC_1(f_E^* K)
&=& {\pr'_1}_!\ (\pr_1^* f_E^* K \otimes^\LM_\EM \mu_1^* \LC) [r]
&\text{by definition}
\\
&=& {\pr'_1}_!\ ( F^* \pr^* K \otimes^\LM_\EM F^* \mu^* \LC) [r]
&\text{by COM}^*(\boxempty_\pr, \D)
\\
&=& {\pr'_1}_!\ F^* (\pr^* K \otimes^\LM_\EM \mu^* \LC) [r]
&\text{by DISTR}(F)
\\
&=& f_{E'}^* {\pr'}_!\ (\pr^* K \otimes^\LM_\EM \mu^* \LC) [r]
&\text{by PBCT}(\boxempty_{\pr'})
\\
&=& f_{E'}^* \FC(K)
&\text{by definition}
\end{array}
\]
\end{proof}

\subsection{Examples}

\begin{prop}\label{prop:sub}
Let $i : F \hookrightarrow E$ be a sub-vector bundle over $S$, with constant
rank $r_F$. We denote by $i^\perp : F^\perp \hookrightarrow E'$ the orthogonal
of $F$ in $E'$. Then we have a canonical isomorphism
\begin{align*}
&&&&
\FC(i_* \EM_F [r_F]) \simeq i^\perp_* \EM_{F^\perp}(-r_F)[r - r_F]
&&&\mathrm{(SUB)}
\end{align*}
\end{prop}

\begin{proof}
First remark that we have a cartesian square (which we will denote by $\boxempty$)
\[
\xymatrix{
F^\perp
\ar[r]^{\pi_{F^\perp}} 
\ar[d]_{i^\perp}
& S \ar[d]^{\s'_F}
\\
E' \ar[r]_-{i'}
& F' = E' / F^\perp
}
\]

Let us denote by $\FC_F$ the Fourier-Deligne transform associated to
$F \elem{\pi_F} S$. Then we have
\[
\begin{array}{rcll}
\FC (i_*\ \EM_F [r_F])
&=& i'^*\ \FC_F\ (\pi_F^*\ \EM_S\ [r_F])\ [r - r_F]
&\text{by MOR}
\\
&=& i'^*\ {\s'_F}_*\ \EM_S (-r_F)\ [r - r_F]
&\text{by DIRAC}
\\
&=& i^\perp_*\ {\pi_{F^\perp}}^* \EM_S (-r_F)\ [r - r_F]
&\text{by PBCT}(\boxempty)
\\
&=& i^\perp_*\ \EM_{F^\perp} (-r_F)\ [r - r_F]
\end{array}
\]
\end{proof}

\begin{prop}\label{prop:trans}
Let $e \in E(S)$ be a section of $E \elem{\pi} S$.
Let $\t_e : E \elem{\sim} E$ denote the translation by $e$.
Finally, let $\mu_e = \mu(e,-) : E' \to \AM^1$. 
Then we have a functorial isomorphism
\begin{align*}
&&&&
\FC({\t_{e}}_* K) \simeq \mu_e^* \LC \otimes^\LM_\EM \FC(K)
&&&\mathrm{(TRANS)}
\end{align*}
for $K$ in $D^b_c(E,\EM)$.
\end{prop}

\begin{proof}
Let us first show that ${\t_e}_* K = (e_* \EM_S) * K$.
We have a commutative diagram
\[
\xymatrix{
& E \times_S E
\ar[dl]^{p_1}
\ar@/_/[dr]_{p_2}
\ar[rr]^s
&& E
\\
E
\ar@/_/[dr]_\pi
&& E
\ar[dl]^\pi
\ar@/_/[ul]_(.4){(e\pi, 1_E)}
\ar[ur]_{\t_e}
\\
& S \ar@/_/[ul]_e
}
\]
We denote by $\boxempty_e$ the cartesian square containing $\pi$ and $e$, and by
$\D_e$ the commutative triangle corresponding to the relation $\t_e = s \circ (e\pi, 1_E)$.
Now we have
\[
\begin{array}{cll}
\multicolumn{2}{l}{(e_* \EM_S) * K}\\
=& s_!\ (e_* \EM_S \boxtimes^\LM_S K)
&\text{by definition}
\\
=& s_!\ (p_1^* e_* \EM_S \otimes^\LM_\EM p_2^* K)
&\text{by definition}
\\
=& s_!\ ((e\pi, 1_E)_* \pi^* \EM_S \otimes^\LM_\EM p_2^* K)
&\text{by PBCT}(\boxempty_e)
\\
=& s_!\ (e\pi, 1_E)_* (\pi^* \EM_S \otimes^\LM_\EM (e\pi, 1_E)^* p_2^* K)
&\text{by PROJ}((e\pi, 1_E))
\\
=& {\t_e}_* K
&\text{by COM}_!(\D_e) \text{ and }p_2 (e\pi, 1_E) = 1_E
\end{array}
\]

Secondly, let us show that $\FC(e_* \EM_S) = \mu_e^* \LC [r]$.
We have a commutative diagram
\[
\xymatrix{
& E \times_S E'
\ar[dl]^\pr
\ar@/_/[dr]_{\pr'}
\ar[rr]^\mu
&& \AM^1
\\
E
\ar@/_/[dr]_\pi
&& E'
\ar[dl]^{\pi'}
\ar@/_/[ul]_(.4){(e\pi', 1_{E'})}
\ar[ur]_{\mu_e}
\\
& S \ar@/_/[ul]_e
}
\]
We denote by $\boxempty'_e$ the cartesian square containing $\pi'$ and $e$, and by 
$\D'_e$ the commutative triangle corresponding to the relation
$\mu_e = \mu \circ (e\pi', 1_{E'})$.
Now we have
\[
\begin{array}{cll}
\multicolumn{2}{l}{\FC(e_* \EM_S)}\\
=& \pr'_!\ (\pr^* e_* \EM_S \otimes^\LM_\EM \mu^* \LC) [r]
&\text{by definition}
\\
=& \pr'_!\ ((e\pi', 1_{E'})_* \pi'^* \EM_S \otimes^\LM_\EM \mu^* \LC) [r]
&\text{by PBCT}(\boxempty'_e)
\\
=& \pr'_!\ (e\pi', 1_{E'})_* (\EM_{E'} \otimes^\LM_\EM (e\pi', 1_{E'})^* \mu^* \LC) [r]
&\text{by PROJ}((e\pi', 1_{E'}))
\\
=& \mu_e^* \LC [r]
&\text{by COM}_!(\D'_e) \text{ and } \mu (e\pi', 1_{E'}) = \mu_e
\end{array}
\]

Finally, the result follows by Proposition \ref{prop:conv} (CONV).
\[
\FC({\t_e}_* K)
= \FC((e_* \EM_S) * K)
= \FC(e_* \EM_S) \otimes^\LM_\EM \FC(K) [-r]
= \mu_e^* \LC \otimes^\LM_\EM \FC(K)
\]
\end{proof}

\begin{prop}\label{prop:G eq}
Let $G$ be a smooth affine group scheme over $S$, acting linearly on
the vector bundle $E \elem{\pi} S$, let $K$ and $L$ be two objects in
$D^b_c(E,\EM)$, and let $M$ be an object in $D^b_c(G,\EM)$. We denote
by $m : G \times_S E \to E$ the action of $G$ on $E$, and by $m' : G
\times_S E' \to E'$ the contragredient action, defined by $m'(g,e') =
{}^t g^{-1}.e'$. Then each isomorphism
\[
m^* K \simeq M \boxtimes^\LM_S L
\]
in $D^b_c(G\times_S E, \EM)$ induces canonically an isomorphism
\begin{align*}
&&&&
m'^*\FC(K) \simeq M \boxtimes^\LM_S \FC(L)
&&&(G \mathrm{-EQ})
\end{align*}
in $D^b_c(G\times_S E', \EM)$.
\end{prop}

\begin{proof}

We have a base change diagram
{\scriptsize
\[
\xymatrix@=.8cm{
&
G \times_S E \times_S E'
\ar[dl]^{(1_G,\pr)}
\ar[dr]_{(1_G,\pr')}
\ar[drrr]^P
\ar[rrrrr]^{\mu_G}
&&&&&
\AM^1
\\
G \times_S E
\ar[dr]_{p_G}
\ar[drrr]^(.3){p_E}
&&
G \times_S E'
\ar[dl]_(.3){p'_G} |!{[ll];[dr]}\hole
\ar[drrr]^(.3){p_{E'}} |!{[rr];[dr]}\hole
&&
E \times_S E'
\ar[urr]_\mu
\ar[dl]^(.7)\pr
\ar[dr]^{\pr'}
\ar@{}[ul] | {\DS\D}
\\
&
G
\ar[drrr]_{\pi_G}
&& E \ar[dr]_\pi
&& E' \ar[dl]^{\pi'}
\\
&&&& S
}
\]
}
and commutative triangles
\[
\xymatrix{
G \times_S E 
\ar[rr]^{(p_G, m)}
\ar[dr]_m
\ar@{}[drr] | {\DS{\D_G}}
&& G \times_S E
\ar[dl]^{p_E}
& \txt{and}
& G \times_S E' 
\ar[rr]^{(p'_G, m')}
\ar[dr]_{m'}
\ar@{}[drr] | {\DS{\D'_G}}
&& G \times_S E'
\ar[dl]^{p_{E'}}
\\
& E &&
&& E'&
}
\]

We will use Theorem \ref{th:morphism} (MOR) for the morphism of
$G$-vector bundles
\[
(p_G, m)^{-1} : G \times_S E' \longto G \times_S E'
\]
whose transposed morphism is
\[
(p'_G, m') : G \times_S E \longto G \times_S E
\]
Both are isomorphisms. We will use the fact that the functor
$(p_G, m)^{-1}_!$ is an equivalence, isomorphic to $(p_G, m)^{-1}_*$
and to $(p_G, m)^*$.

\[
\begin{array}{rcll}
m'^* \FC(K)
&=& (p'_G, m')^* p_{E'}^* \FC(K)
&\text{by COM}^*(\D'_G)
\\
&=& (p'_G, m')^* \FC_{G\times_S E} (p_E^* K)
&\mathrm{by\ BC^*}
\\
&=& \FC_{G\times_S E} ((p_G, m)^{-1}_! p_E^* K)
&\text{by MOR}
\\
&=& \FC_{G\times_S E} ((p_G, m)^* p_E^* K)
&\text{by the above}
\\
&=& \FC_{G\times_S E} (m^* K)
&\text{by COM}^*(\D_G)
\end{array}
\]

Now assume we are given an isomorphism $\phi : m^* K \elem{\sim} M \boxtimes^\LM_S L$.
Then $\phi$ induces an isomorphism
\[
\begin{array}{rcll}
\FC_{G\times_S E} (m^* K)
&\simeq& \FC_{G\times_S E} (M \boxtimes^\LM_S L)
&\text{induced by } \phi
\\
&=& (1_G, \pr')_!\ ((1_G, \pr)^* (M \boxtimes^\LM_S L) \otimes^\LM_\EM P^* \mu^* \LC) [r]
&\text{by COM}^*(\D)
\\
&=& (1_G, \pr')_!\ 
((M \boxtimes^\LM_S \pr^* L) \otimes^\LM_\EM (\EM_G \boxtimes^\LM_S \mu^* \LC)) [r]
&\text{by DISTR}(1_G, \pr)
\\
&=& (1_G, \pr')_!\ (M \boxtimes^\LM_S (\pr^* L \otimes^\LM_\EM \mu^* \LC)) [r]
\\
&=& M \boxtimes^\LM_S \pr'_!\ (\pr^* L \otimes^\LM_\EM \mu^* \LC) [r]
&\text{by KUNNETH}
\\
&=& M \boxtimes^\LM_S \FC(L)
\end{array}
\]

We have applied the K\"unneth formula to the following diagram.
\[
\xymatrix@!=1.5cm{
& G \times_S E \times_S E'
\ar[dl]_{P_G}
\ar[d]^(.6){(1_G, \pr')}
\ar[dr]^P
\\
G \ar[d]_{1_G}
&
G \times_S E'
\ar[dl]^{p'_G}
\ar[dr]_{p_{E'}}
&
E \times_S E'
\ar[d]^{\pr'}
\\
G \ar[dr]_{\pi_G}
&&
E' \ar[dl]^{\pi'}
\\
& S
}
\]
\end{proof}

\begin{prop}\label{prop:bc_!}
Let $f : S_1 \to S$ be an $\FM_q$-morphism of finite type. With the notations of
Proposition \ref{prop:bc^*} $\mathrm{(BC^*)}$, we have a functorial isomorphism
\begin{align*}
&&&&
\FC({f_E}_!\ K_1) \simeq {f_{E'}}_!\ \FC_1(K_1)
&&&\mathrm{(BC_!)}
\end{align*}
for $K_1$ in $D^b_c(E_1,\EM)$.
\end{prop}

\begin{proof}
We have
\[
\begin{array}{rcll}
\FC({f_E}_! K_1)
&=& \pr'_!\ (\pr^* {f_E}_!\ K_1 \otimes^\LM_\EM \mu^* \LC) [r]
&\text{by definition}
\\
&=& \pr'_!\ ( F_!\ \pr_1^* K_1 \otimes^\LM_\EM \mu^* \LC) [r]
&\text{by PBCT}(\boxempty_\pr)
\\
&=& \pr'_!\ F_!\ (\pr_1^* K_1 \otimes^\LM_\EM F^* \mu^* \LC) [r]
&\text{by PROJ}(F)
\\
&=& {f_{E'}}_!\ {\pr'_1}_!\ (\pr_1^* K_1 \otimes^\LM_\EM \mu_1^* \LC) [r]
&\text{by COM}_!(\boxempty_{\pr'},\D)
\\
&=& {f_{E'}}_!\ \FC_1(K_1)
&\text{by definition}
\end{array}
\]
\end{proof}

\section{Fourier-Deligne transform and duality}\label{sec:fourier duality}

We keep the preceding notations. For the proof of the following fundamental theorem,
we refer to \cite{KaLa}. Katz and Laumon state the result for $\ov\QM_\ell$, but
for the proof they make a reduction to torsion coefficients, and prove it in that
context. The crucial point is the one-dimensional case. 

\begin{theo}\label{th:oubli support}
For any object $K$ in $D^b_c(E,\EM)$, the support forgetting morphism
\begin{align*}
&&&&
\pr'_!\ (\pr^* K \otimes^\LM_\EM \mu^* \LC) \longto \pr'_*\ (\pr^* K \otimes^\LM_\EM \mu^* \LC)
&&&\mathrm{(SUPP)}
\end{align*}
is an isomorphism.
\end{theo}

This theorem has the following corollaries.

\begin{theo}\label{th:rhom fourier}
We have a functorial isomorphism
\[
\RHOM(\FC_\psi(K), \pi'^! L) \simeq \FC_{\psi^{-1}}(\RHOM(K,\pi^! L)) (r)
\]
for $(K,L)$ in $D^b_c(E,\EM)^\op \times D^b_c(S,\EM)$.
\end{theo}

\begin{proof}
First recall that $\LC_\psi$ is a rank one local system on $\AM^1$. Applying
$\mu^*$ to the relation $\LC_\psi \otimes^\LM_\EM \LC_{\psi^{-1}} \simeq \EM_{\AM^1}$
and using DISTR, we find that $\mu^* \LC_\psi$ is a rank one local system
with inverse $\mu^* \LC_{\psi^{-1}}$. So $(-) \otimes^\LM_\EM \mu^* \LC_\psi$
is an automorphism of $D^b_c(E \times_S E')$, with quasi-inverse
$(-) \otimes^\LM_\EM \mu^* \LC_{\psi^{-1}}$.

We have
\[
\begin{array}{lll}
\multicolumn{2}{l}{\RHOM(\FC_\psi(K), \pi'^! L)}\\
=& \RHOM(\pr'_!\ (\pr^* K \otimes^\LM_\EM \mu^* \LC_\psi), \pi'^! L) [-r]
&\text{by definition}
\\
=& \pr'_*\ \RHOM(\pr^* K \otimes^\LM_\EM \mu^* \LC_\psi, \pr'^! \pi'^! L) [-r]
&\text{by DUAL}(\pr')
\\
=& \pr'_*\ \left(\RHOM(\pr^* K, \pr^! \pi^! L) \otimes^\LM_\EM \mu^* \LC_{\psi^{-1}}\right) [-r]
&\text{by the above and COM}^! (\boxempty)
\\
=& \pr'_*\ \left(\pr^! \RHOM(K, \pi^! L) \otimes^\LM_\EM \mu^* \LC_{\psi^{-1}}\right) [-r]
&\mathrm{by\ DUAL_2}(\pr)
\\
=& \pr'_*\ \left(\pr^* \RHOM(K, \pi^! L) \otimes^\LM_\EM \mu^* \LC_{\psi^{-1}}\right) [r] (r)
&\text{since $\pr$ is smooth}
\\
=& \FC_{\psi^{-1}}(\RHOM(K,\pi^! L)) (r)
&\text{by SUPP}
\end{array}
\]

\end{proof}

Remember that, if $X$ is a variety, we denote by $\DC_{X,\EM}$ the
duality functor of $D^b_c(X,\EM)$ (see section \ref{sec:context}).  If
$a : X \to \Spec k$ is the structural morphism, we denote by
$D_{X,\EM}$ the dualizing complex $a^! \EM$.

\begin{cor}\label{cor:dual fourier}
We have a functorial isomorphism
\[
\DC_{E',\EM} (\FC_\psi(K)) \simeq \FC_{\psi^{-1}} (\DC_{E,\EM}(K)) (r)
\]
for $K$ in $D^b_c(E,\EM)^\op$.
\end{cor}

\begin{proof}
We have
\[
\begin{array}{lcl}
\DC_{E',\EM} (\FC_\psi(K))
&=& \RHOM(\FC_\psi(K), \pi'^! D_{S,\EM})\\
&=& \FC_{\psi^{-1}} \RHOM(K, \pi^! D_{S,\EM}) (r)\\
&=& \FC_{\psi^{-1}} (\DC_{E,\EM}(K)) (r)
\end{array}
\]

\end{proof}

\begin{theo}\label{th:equiv perv}
$\FC$ maps ${}^p\MC(E,\EM)$ onto ${}^p\MC(E',\EM)$. The functor
\begin{align*}
&&&&
\FC : {}^p\MC(E,\EM) \longto {}^p\MC(E',\EM)
&&&\mathrm{(EQUIV)}
\end{align*}
is an equivalence of abelian categories, with quasi-inverse $a^* \FC'(-)(r)$.
\end{theo}

\begin{proof}
Since $\pr$ is smooth, purely of relative dimension $r$, the functor
$\pr^*(-)[r]$ is $t$-exact \cite[4.2.5]{BBD} and, since $\pr'$ is affine,
the functor $\pr'_!$ is left $t$-exact, whereas the functor $\pr'_*$
is right $t$-exact. By Theorem \ref{th:oubli support}, we deduce that $\FC$ is
$t$-exact.

The second assertion follows from the first and Theorem \ref{th:inv} (INV).

\end{proof}

\begin{cor}\label{cor:fourier simple}
Suppose $\EM = \KM$ or $\FM$. Then $\FC$ transforms simple
$\EM$-perverse sheaves on $E$ into simple $\EM$-perverse sheaves on
$E'$.
\end{cor}

\chapter{Springer correspondence and decomposition matrices}\label{chap:springer}

\section{The geometric context}\label{sec:geometric context}

\subsection{Notation}

Let $G$ be a connected semisimple linear algebraic group of rank $r$
over $k$, and let $\gG$ be its Lie algebra.
Let us fix a Borel subgroup $B$ of $G$, with unipotent radical $U$, and a
maximal torus $T$ contained in $B$. We denote by $\bG$, $\uG$ and $\tG$ the
corresponding Lie algebras. The characters of $T$ form a free abelian group $X(T)$ of rank $r$. 
The Weyl group $W = N_G(T) / T$ acts as a reflection group
on $V = \QM \otimes_\ZM X(T)$.

Let $\Phi \subset X(T)$ be the root system of $(G,T)$,
$\Phi^+$ the set of positive roots defined by $B$, and $\D$ the corresponding basis.
We denote by $\nu_G$ (or just $\nu$) the cardinality
of $\Phi^+$. Then $\dim G = 2\nu +r$, $\dim B = \nu + r$, $\dim T = r$ and $\dim U = \nu$.

\subsection{The finite quotient map}

Let $\phi : \tG \to \tG/W$ be the quotient map, corresponding to the
inclusion $k[\tG]^W \hookrightarrow k[\tG]$. It is finite and surjective.
For $t \in \tG$, we will also denote $\phi(t)$ by $\ov t$.

Let us assume that $p$ is not a torsion prime for $\gG$.
Then $k[\tG]^W = k[\phi_1,\ldots,\phi_r]$ for some algebraically independent
homogeneous polynomials $\phi_1,\ldots,\phi_r$ of degrees
$d_1 \leqslant \ldots \leqslant d_r$, and we have $d_i = m_i + 1$, where
the $m_i$ are the exponents of $W$ (see \cite{DEM}).
Then $\tG/W$ can be identified with $\AM^r$ and $\phi$ with $(\phi_1,\ldots,\phi_r)$.

For example, if $G = SL_n$, we can identify $\tG$ with the hyperplane
$\{ (x_1,\ldots,x_n) \mid x_1 + \cdots + x_n = 0 \}$ of $k^n$, and we can take
$\phi_i = \s_{i + 1}$, for $i = 2, \ldots, n - 1$ (here $r = n - 1$),
that is, the $i + 1^\text{st}$ elementary symmetric function of $k^n$,
restricted to this hyperplane ($\s_1$ does not appear, since its restriction vanishes).

\subsection{The adjoint quotient}

The Chevalley restriction theorem says that the restriction map
$k[\gG]^G \to k[\tG]^W$ is an isomorphism, so $k[\gG]^G$
is also generated by $r$ homogeneous algebraically independent polynomials
$\chi_1,\ldots,\chi_r$. With a suitable ordering, they have the same degrees
$d_1 \leqslant \ldots \leqslant d_r$ as the $\phi_i$'s.

Hence we have a morphism
$\chi = (\chi_1,\ldots,\chi_r) : \gG \to \gG/\!/G \simeq \tG/W \simeq \AM^r$.
It is called the Steinberg map, or the adjoint quotient.
In the last section, we could have taken $\phi_i = {\chi_i}_{|\tG}$.

The morphism $\chi$ has been extensively studied (see \cite{SLO2} and the references therein).
First, it is flat, and its schematic fibers are irreducible, reduced and normal
complete intersections, of codimension $r$ in $\gG$.
If $t \in \tG$, let $\gG_{\ov t}$ be the fiber $\chi^{-1}(\ov t)$.
It is the union of finitely many classes. It contains
exactly one class of regular elements, which is open and dense in $\gG_{\ov t}$,
and whose complement has codimension $\geqslant 2$ in $\gG_{\ov t}$.
This regular class consists exactly in the nonsingular points of $\gG_{\ov t}$.
So $\tG/W$ parametrizes the classes of regular elements.
The fiber $\gG_{\ov t}$ also
contains exactly one class of semisimple elements, the orbit of $t$, which is
the only closed class in $\gG_{\ov t}$, and which lies in the closure of every
other class in $\gG_{\ov t}$.

In fact, $\chi$ can be interpreted as the map which sends $x$ to the intersection
of the class of $x_s$ with $\tG$, which is a $W$-orbit.

For example, for $G = SL_n$, we can define the $\chi_i : \sG\lG_n \to k$
by the formula
\[
\det(\xi - x) = \xi^n + \sum_{i = 0}^{n - 1} (-1)^{i + 1} \chi_i(x) \xi^{n - i - 1} \in k[\xi]
\]
for $x \in \sG\lG_n$. So $\chi(x)$ can be interpreted as the characteristic polynomial
of $x$. Restricting $\chi_i$ to $\tG$, we recover the previous $\phi_i$.

\subsection{Springer's resolution of the nilpotent variety}

Let $\NC$ be the closed subvariety of $\gG$ consisting in its nilpotent elements.
It is the fiber $\gG_{0} = \chi^{-1}(0)$. In particular, it is a complete intersection
in $\gG$, given by the equations $\chi_1(x) = \cdots = \chi_r(x) = 0$.
It is singular. We are going to describe Springer's resolution of the nilpotent variety.

The set $\BC$ of Borel subalgebras of $\gG$ is a homogeneous space under $G$,
in bijection with $G/B$, since the normalizer of $\bG$ in $G$ is $B$. Hence
$\BC$ is endowed with a structure of smooth projective variety, of dimension $\nu$.

Let $\NCt = G\times^B \uG \simeq \{ (x,\bG') \in \NC \times \BC \mid x \in \bG' \}$.
It is a smooth variety: the second projection makes it a vector bundle over $\BC$.
Actually it can be identified to the cotangent bundle $T^*\BC$, since
$T\BC = T(G/B) = G \times^B \gG/\bG$ and $\uG = \bG^\perp$.
Now let $\pi_\NC$ be the first projection. Since $\NCt$ is closed in $\NC \times \BC$
and $\BC$ is projective, the morphism $\pi_\NC$ is projective. Moreover, it is an
isomorphism over the open dense subvariety of $\NC$ consisting in the regular nilpotent
elements. Hence $\pi_\NC$ is indeed a resolution of $\NC$.

\subsection{Grothendieck's simultaneous resolution of the adjoint quotient}

In the last paragraph, we have seen the resolution of the fiber $\chi^{-1}(0)$.
We are now going to explain Grothendieck's simultaneous resolution, which
gives resolutions for all the fibers of $\chi$ simultaneously.

So let $\tilde\gG = G \times^B \bG \simeq \{ (x,\bG') \in \gG \times \BC \mid x \in \bG' \}$.
We define $\pi : \tilde \gG \to \gG$ by $\pi(g * x) = \Ad(g).x$ (in the identification with
pairs $(x,\bG')$, this is just the first projection). Then the commutative diagram
\[
\xymatrix{
\tilde \gG
\ar[r]^\pi
\ar[d]_\theta
& \gG \ar[d]^\chi
\\
\tG \ar[r]_\phi
& \tG/W
}
\]
where $\theta$ is the composition $G\times^B \bG \to \bG/[\bG,\bG] \stackrel{\sim}{\to} \tG$,
is a simultaneous resolution of the singularities of the flat morphism $\chi$.
That is, $\theta$ is smooth, $\phi$ is finite surjective, $\pi$ is proper, and $\pi$ induces
a resolution of singularities $\theta^{-1}(t) \to \chi^{-1}(\phi(t))$ for all $t \in \tG$.

\section{Springer correspondence for $\EM W$}\label{sec:springer}

\subsection{The perverse sheaves $\KC_\rs$, $\KC$ and $\KC_\NC$}

Let us consider the following commutative diagram with cartesian squares.

\[
\xymatrix@=1.5cm{
\tilde\gG_\rs
\ar[d]_{\pi_\rs}
\ar@<-.5ex>@{^{(}->}[r]^{\tilde j_\rs}
\ar@{}[dr] | {\DS{\boxempty_\rs}}
&
\tilde\gG
\ar[d]^\pi
\ar@{}[dr] | {\DS{\boxempty_\NC}}
&
\NCt
\ar@<.5ex>@{_{(}->}[l]_{i_\NCt}
\ar[d]^{\pi_\NC}
\\
\gG_\rs
\ar@<-.5ex>@{^{(}->}[r]_{j_\rs}
&
\gG
&
\NC
\ar@<.5ex>@{_{(}->}[l]^{i_\NC}
}
\]

Let us define the complex
\[
\KC = \pi_! \OM_{\tilde \gG} [2\nu + r]\\
\]

Let $\EM$ be $\KM$, $\OM$ or $\FM$.
Since modular reduction commutes with direct images with proper support,
we have $\EM \KC = \pi_! \EM_{\tilde \gG} [2\nu + r]$.
By the proper base change theorem, the fiber at a point $x$ in $\gG$ of $\EM \KC$
is given by $(\EM \KC)_x = \rgc(\BC_x, \OM)$.

Let $\KC_\rs = {j_\rs}^*\KC$ and $\KC_\NC = {i_\NC}^*\KC[-r]$.
By the proper base change theorem and the commutation between modular reduction
and inverse images, we have
\begin{gather*}
\EM\KC_\rs = {j_\rs}^*\EM \KC = {\pi_\rs}_* \EM_{{\tilde \gG}_\rs} [2\nu + r]\\
\EM\KC_\NC = {i_\NC}^*\EM \KC[-r] = {\pi_\NC}_! \EM_\NCt [2\nu]
\end{gather*}

The morphism $\pi$ is proper and small, hence
$\EM\KC$ is an intersection cohomology complex by Proposition \ref{prop:small}.
Actually $\pi$ is \'etale over the open subvariety $\gG_\rs$. More precisely,
the morphism $\pi_\rs$ obtained after the base change $j_\rs$
is a Galois finite \'etale covering, with Galois group $W$,
so we have $\EM\KC = {j_\rs}_{!*} \EM\KC_\rs = {j_\rs}_{!*} (\EM W\ [2\nu + r])$.
Note that, if $\EM = \OM$, we have
${}^{p_+} {j_\rs}_{!*} \KC_\rs
= \DC_{\gG,\OM} ({}^p {j_\rs}_{!*} \KC_\rs)
= \DC_{\gG,\OM} (\KC) = \KC$
so it does not matter whether we use $p$ or $p_+$ (we have used the fact that
the regular representation is self-dual, and that $\KC$ is self-dual because
$\pi$ is proper).

Thus the endomorphism algebra of $\EM\KC_\rs$ is the group algebra
$\EM W$. Since the functor ${j_\rs}_{!*}$ is fully faithful, it induces an isomorphism
$\End(\KC_\rs) = \EM W \elem{\sim} \End(\KC)$.
In particular, we have an action of $\EM W$ on the stalks
$\HC^i_x(\EM\KC) = H^{i + 2 \nu + r}(\BC_x,\EM)$. 

When $\EM = \KM$, the group algebra $\KM W$ is semisimple. The perverse sheaves
\[
\KM\KC_\rs = \bigoplus_{E \in \Irr \KM W} (E\ [2\nu + r])^{\dim E}
\]
and
\[
\KM\KC = \bigoplus_{E \in \Irr \KM W} {j_\rs}_{!*} (E\ [2\nu + r])^{\dim E}
\]
are semisimple.

If $\ell$ does not divide the order of the Weyl group $W$,
we have a similar decomposition for $\EM = \OM$ or $\FM$.
However, we are mostly interested in the case where $\ell$ divides $|W|$. Then
$\FM\KC_\rs$ and $\FM\KC$ are not semisimple.
More precisely, we have decompositions
\begin{gather*}
\OM W = \bigoplus_{F \in \Irr \FM W} (P_F)^{\dim F}\\
\FM W = \bigoplus_{F \in \Irr \FM W} (\FM P_F)^{\dim F}
\end{gather*}
where $P_F$ is a projective indecomposable $\OM W$-module such that
$\FM P_F$ is a projective cover of $F$.
Besides, $\FM P_F$ has head and socle isomorphic to $F$.

Hence we have a similar decomposition for $\KC_\rs$, and its modular reduction.
\begin{gather*}
\KC_\rs = \bigoplus_{F \in \Irr \FM W} (P_F\ [2\nu + r])^{\dim F}\\
\FM \KC_\rs = \bigoplus_{F \in \Irr \FM W} (\FM P_F\ [2\nu + r])^{\dim F}
\end{gather*}
These are decompositions into indecomposable summands, and the
indecomposable summand $\FM P_F\ [2\nu + r]$ has head and socle
isomorphic to $F\ [2 \nu + r]$.  By Proposition \ref{prop:top socle},
applying ${j_\rs}_{!*}$ we get decompositions into indecomposable
summands, and the indecomposable summand ${j_\rs}_{!*} (\FM P_F\ [2
\nu + r])$ has head and socle isomorphic to ${j_\rs}_{!*} (F\ [2 \nu +
r])$.
\begin{gather*}
\KC = \bigoplus_{F \in \Irr \FM W} {j_\rs}_{!*} (P_F\ [2\nu + r])^{\dim F}\\
\FM \KC = \bigoplus_{F \in \Irr \FM W} {j_\rs}_{!*} (\FM P_F\ [2\nu + r])^{\dim F}
\end{gather*}

The morphism $\pi_\NC$ is proper and semi-small, hence $\EM\KC_\NC$
is perverse. The functor ${i_\NC}^*(-)[-r]$ induces a morphism
\begin{equation}\label{mor:res}
\res : \End(\EM\KC) \longto \End(\EM\KC_\NC)
\end{equation}

\subsection{Springer correspondence for $\KM W$ by restriction}

\begin{theo}[Lusztig]
If $\EM = \KM$, then the morphism $\res$ in (\ref{mor:res}) is an isomorphism.
\end{theo}

The Weyl group $W$ acts on $G/T$ by $gT.w = g n_w T$, where $n_w$ is
any representative of $w$ in $N_G(T)$. So $W$ acts naturally on the
cohomology complex $\rg(G/T,\EM)$. Since the projection $G/T \to G/B$
is a locally trivial $U$-fibration, it induces an isomorphism
$\rg(G/B,\EM) \simeq \rg(G/T,\EM)$, and thus there is a natural action
of $W$ on $\rg(G/B,\EM)$. When $\EM = \KM$, one can show that the
action on the cohomology is the regular representation.  On the other
hand, the stalk at $0$ of $\KC$ is isomorphic to $\rg(G/B,\EM)$, and
$\EM W$ acts on it through $\EM W \simeq \End(\KC) \elem{\res}
\End(\KC_\NC)$. In fact, when $\EM = \KM$, the two actions on the
cohomology coincide. Since the regular representation is faithful,
this implies that the morphism $\res$ is injective.

Then one can show that the two algebras have the same dimension, to
prove that $\res$ is an isomorphism.

We have
\[
\KM\KC = \bigoplus_{E \in \Irr \KM W} {j_\rs}_{!*} \un E^{\dim E}
\]
and $i_\NC^*(\KM\KC) [-r] = \KM\KC_\NC$. In fact, the restriction functor
$i_\NC^* [-r]$ sends each simple constituent ${j_\rs}_{!*} \un E$
on a simple object. The assignment
\[
\RC : E \mapsto i_\NC^* {j_\rs}_{!*} \un E [-r]
\]
is an injective map from $\Irr \KM W$ to the simple $G$-equivariant
perverse sheaves on $\NC$, which are parametrized by the pairs
$(\OC,\LC)$, where $\OC$ is a nilpotent orbit, and $\LC$ is an
irreducible $G$-equivariant $\KM$-local system on $\OC$. This is the
Springer correspondence (by restriction).

We said $G$ was semisimple, but everything can be done for a reductive
group instead, as $GL_n$.
For $G = GL_n$, the Specht module $S^\l$ is sent to
$\p\JC_{!*}(\OC_\l,\KM)$, where $\OC_\l$ is the nilpotent orbit
corresponding to the partition $\l$ by the Jordan normal form.

\subsection{The Fourier-Deligne transform of $\EM\KC$}

We assume that there exists a non-degenerate $G$-invariant symmetric
bilinear form $\mu$ on $\gG$, so that we can identify $\gG$ with its
dual. This is the case, for example, if $p$ is very good for $G$ (take
the Killing form), or if $G = GL_n$ (take $\mu(X,Y) = \tr(XY)$).  For
a more detailed discussion, see \cite{LET}.

\begin{lem}
The root subspace $\gG_\a$ is orthogonal to $\tG$ and to all the root subspaces
$\gG_\b$ with $\b \neq -\a$.
\end{lem}

\begin{proof}
Let $x \in \tG$. For $t \in T$. We have
$\mu(x,e_\a) = \mu(\Ad(t).x, \Ad(t).e_\a) = \a(t) \mu(x, e_\a)$.
Since $\a \neq 0$, we can choose $t$ so that $\a(t) \neq 1$, and thus
$\mu(x, e_\a) = 0$.

Now let $\b$ be a root different from $-\a$. We have
$\mu(e_\b, e_\a) = \mu(\Ad(t).e_\b, \Ad(t).e_\a) = \a(t)\b(t)\mu(e_\b,e_\a)$.
Since $\b \neq -\a$, we may choose $t$ so that $\a(t)\b(t) \neq 1$, and thus
$\mu(e_\b, e_\a) = 0$.
\end{proof}

\begin{cor}\label{b orthogonal}
The orthogonal of $\bG$ is $\uG$.
\end{cor}

\begin{proof}
By the preceding lemma, $\bG$ is orthogonal to $\uG$, and we have
$\dim \bG + \dim \uG = 2\nu + r = \dim \gG$, hence the result, since
$\mu$ is non-degenerate.
\end{proof}

Let $\FC$ be the Fourier-Deligne transform associated to $p : \gG \to \Spec k$
(any vector space can be considered as a vector bundle over a point).
Since we identify $\gG$ with $\gG'$, the functor $\FC$
is an auto-equivalence of the triangulated category $D^b_c(\gG,\EM)$.
The application $a$ of INV, which was defined as the opposite of the canonical
isomorphism from a vector bundle to its bidual, is now multiplication by $-1$.

We will need to consider the base change $f : \BC \to \Spec k$. We will denote by
$\FC_\BC$ the Fourier-Deligne transform associated to $p_\BC : \BC \times \gG \to \BC$.
We have

\[
\xymatrix{
G \times_B \gG
\ar@{=}[d]
&
G \times_B \bG
\ar@{=}[d]
\ar@<.5ex>@{_{(}->}[l]
&
G \times_B \nG
\ar@{=}[d]
\ar@<.5ex>@{_{(}->}[l]
\\
\BC \times \gG
\ar[d]_{p_\BC}
\ar[dr]_F
&
\tilde\gG
\ar@<.5ex>@{_{(}->}[l]_i
\ar[d]^\pi
\ar@{}[ddl] |(.25){\D}
\ar@{}[dr] | {\boxempty_\NC}
&
\NCt
\ar@<.5ex>@{_{(}->}[l]_{i_\NCt}
\ar[d]^{\pi_\NC}
\\
\BC
\ar[dr]_f
&
\gG
\ar[d]_p
&
\NC
\ar@<.5ex>@{_{(}->}[l]_{i_\NC}
\\
&
\Spec k
}
\]

We have
\[
\begin{array}{rcll}
\FC(\EM\KC)
&=& \FC(\pi_!\ \EM_{\tilde\gG}[2\nu + r])
\\
&=& \FC(F_!\ i_*\ \EM_{\tilde\gG}[\nu + r]) [\nu]
&\text{by COM}_!(\D)
\\
&=& F_!\ \FC_\BC(i_*\ \EM_{\tilde\gG}[\nu + r]) [\nu]
&\text{by BC}_!(f)
\\
&=& F_!\ i_*\ {i_\NCt}_*\ \EM_\NCt (-\nu - r) [\nu] [\nu]
&\text{by SUB}
\\
&=& {i_\NC}_*\ {\pi_\NC}_!\ \EM_\NCt (-\nu - r) [2\nu]
&\text{by COM}_!(\D,\boxempty_\NC)
\\
&=& {i_\NC}_*\ \EM\KC_\NC (- \nu - r)
\end{array}
\]

Applying $\FC$ and using Theorem \ref{th:inv} (INV), we get
\[
a_* \EM\KC (-2\nu - r) = \FC({i_\NC}_*\ \EM\KC_\NC) (-\nu - r)
\]
But $a_* \EM\KC \simeq \EM\KC$ since $\EM\KC$ is monodromic, so we have
\[
\FC({i_\NC}_*\ \EM\KC_\NC) \simeq \EM\KC (-\nu)
\]

\begin{theo}
We have
\begin{gather*}
\FC(\EM\KC) \simeq {i_\NC}_*\ \EM\KC_\NC (- \nu - r)\\
\FC({i_\NC}_*\ \EM\KC_\NC) \simeq \EM\KC (-\nu)
\end{gather*}
\end{theo}

Note that this proves a second time that $\KC_\NC$ is perverse.

\begin{cor}
The functors ${j_\rs}_{!*}$, $\FC(-)(\nu + r)$ and ${i_\NC}_*$ induce isomorphisms
\[
\EM W
= \End(\EM \KC_\rs)
\elem{\sim} \End(\EM \KC)
\elem{\sim} \End({i_\NC}_*\ \EM\KC_\NC)
\stackrel{\sim}{\longleftarrow} \End(\EM\KC_\NC)
\]
\end{cor}

For $E \in \EM W\text{-mod}$, let
$\TC(E) = \FC {j_\rs}_{!*}(E\ [2\nu + r])(\nu + r)$.
By the theorem above, we have $\TC(\EM W) = \KC_\NC$. More correctly,
we should say that $\TC(\EM W)$ is supported on $\NC$, and write
$\TC(\EM W) = {i_\NC}_* \KC_\NC$, but we identify the perverse sheaves
on $\NC$ with their extension by zero on $\gG$.

\begin{cor}
The perverse sheaf $\KM \KC_\NC$ is semisimple, and we have the decomposition
\[
\KM \KC_\NC =
\bigoplus_{E \in \Irr \KM W} \TC(E)^{\dim E}
\]
Similarly, we have decompositions into indecomposable summands
\[
\KC_\NC = 
\bigoplus_{F \in \Irr \FM W} \TC(P_F)^{\dim F}
\]
and
\[
\FM \KC_\NC
= \bigoplus_{F \in \Irr \FM W} \TC(\FM P_F)^{\dim F}
\]
The indecomposable summand $\TC(\FM P_F)$ has head and socle isomorphic
to $\TC(F)$.
\end{cor}

\subsection{Springer correspondence by Fourier-Deligne tranform}

\subsubsection{Springer correspondence for $\KM W$}

Let $E$ be a simple $\KM W$-module. Then $\TC(E)$ is a direct
summand of $\KC_\NC$. Hence it is a $G$-equivariant simple perverse
sheaf supported on $\NC$, so it is of the form $\JC_{!*}(\OC_E, \LC_E)$ for
some adjoint orbit $\OC_E$ in $\NC$, and some irreducible $G$-equivariant
local system on $\OC_E$. We may thus associate to the simple $\KM W$-module
$E$ the pair $(\OC_E, \LC_E)$ or equivalently the pair $(x_E, \rho_E)$
(up to G-conjugacy), where
$x_E$ is a representative of the orbit $\OC_E$, and $\rho_E$ is the irreducible
character of $A_G(x_E)$ corresponding to $\LC_E$.

Let $\NG_\KM(G)$ (or simply $\NG_\KM$)
be the set of all pairs $(\OC,\LC)$, where $\OC$ is an adjoint orbit
in $\NC$, and $\LC$ is an irreducible $G$-equivariant local system on $\OC$ (over $\KM$).
This finite set parametrizes the simple $G$-equivariant simple perverse sheaves
on $\NC$.

Let us denote by $\Psi_\KM : \Irr \KM W \to \NG_\KM$ be the map defined above, and
let $\impsi_\KM$ be its image. Then $\Psi_\KM$ induces a bijection from $\Irr \KM W$
to $\impsi$ (that is, $\Psi_\KM$ is injective). Indeed, if we know
$\Psi(E)$ (or equivalently, $\TC(E)$) one can recover
${j_\rs}_{!*} (\LC(E)[2\nu + r]))$ by applying Theorem \ref{th:inv} (INV),
and then restricting to $\gG_\rs$ we get the local system we started with,
and hence the representation $E$.

\begin{theo}
The map $\Psi_\KM$ defined above induces a bijection
$\Irr \KM W \stackrel{\sim}{\to} \impsi_\KM$.
\end{theo}

\subsubsection{Springer correspondence for $\FM W$}

Let $F$ be a simple $\FM W$-module. Then, by Proposition \ref{prop:top
socle}, $\TC(F)$ is the head (and also the socle) of $\TC(\FM P_F)$,
which is a direct summand of $\KC_\NC$. Hence $\TC(F)$ is supported on
$\NC$, and $\TC(F) = \JC_{!*}(\OC_F,\LC_F)$ for some pair $\Psi_\FM(F)
= (\OC_F,\LC_F)$ in the set $\NG_\FM(G)$ (or simply $\NG_\FM$) of all
pairs $(\OC,\LC)$, where $\OC$ is an adjoint orbit in $\NC$, and $\LC$
is an irreducible $G$-equivariant local system on $\OC$ (over
$\FM$). We will denote by $\impsi_\FM$ the image of $\Psi_\FM : \Irr
\FM W \to \NG_\FM$.  Again, $\Psi_\FM$ is injective.

\begin{theo}
The map $\Psi_\FM$ defined above induces a bijection
$\Irr \FM W \stackrel{\sim}{\to} \impsi_\FM$.
\end{theo}

Let us remark that if $\ell$ does not divide the order of any of the finite
groups $A_G(x)$, $x\in\NC$, then all the group algebras $\FM A_G(x)$ are
semisimple, so for each $x$ there is a natural bijection
$\Irr \KM A_G(x) \stackrel{\sim}{\to} \Irr \FM A_G(x)$,
and thus there is a natural bijection
$\NG_\KM \stackrel{\sim}{\to} \NG_\FM$.

\section{Decomposition matrices}\label{sec:decomposition}

\subsection{Comparison of $e$ maps}

\begin{theo}\label{th:e}
Let $F \in \Irr \FM W$. Then $\TC(\KM P_F)$ is supported on $\NC$, and for each
$E \in \Irr \KM W$  we have
\[
[\KM P_F : E] = [\TC(\KM P_F) : \TC(E)]
\]
\end{theo}

\begin{proof}
We have
\[
\bigoplus_{F \in \FM W} \TC(\KM P_F)^{\dim F}
= \TC(\KM W) = \KM\KC_\NC 
\]
hence $\TC(\KM P_F)$ is supported on $\NC$.

\[
\begin{array}{rcll}
[\KM P_F : E]
&=& [{j_\rs}_{!*} \KM \un P_F : {j_\rs}_{!*} \un E]
&\text{by Prop. \ref{prop:IC mult}}
\\
&=& [\FC {j_\rs}_{!*} \KM \un P_F (\nu + r) : \FC {j_\rs}_{!*} \un E (\nu + r)]
\\
&=& [\KM \TC(P_F) : \TC(E)]
\end{array}
\]
\end{proof}

\subsection{Comparison of $d$ maps}

If $E \in \Irr \KM W$ and $F \in \Irr \FM W$, let $d^W_{E,F}$ be the
corresponding decomposition number.

\begin{theo}\label{th:d}
Let $E \in \Irr \KM W$, and let $E_\OM$ be an integral form for $E$.
Then $\TC(E_\OM)$ is supported on $\NC$, and for each $F \in \Irr \FM W$
we have
\[
[\FM E_\OM : F] = [\FM \TC(E_\OM) : \TC(F)]
\]
Thus
\[
d_{\Psi_\KM(E),\Psi_\FM(F)} = d^W_{E,F}
\]
\end{theo}

\begin{proof}

We have a short exact sequence
\begin{equation}\label{T}
0 \longto T \longto \FM {j_\rs}_{!*} \un E_\OM
\longto {j_\rs}_{!*} \FM \un E_\OM \longto 0
\end{equation}
with $T$ supported on $\gG - \gG_\rs$.

We have $\KM \TC(E_\OM) = \TC(E)$, so it is supported by $\NC$.
Let $j : \gG \setminus \NC \to \gG$ be the open immersion.
By what we have just said, $j^* \TC(E_\OM)$ is a torsion perverse sheaf.
Hence ${}^p j_! j^* \TC(E_\OM)$ is torsion, and thus
the adjunction morphism ${}^p j_! j^* \TC(E_\OM) \to \TC(E_\OM)$ is zero,
because $\TC(E_\OM)$ is torsion-free. But the identity of $j^* \TC(E_\OM)$
factors through
$j^* \TC(E_\OM) \to j^*\; {}^p j_!\; j^* \TC(E_\OM) \to j^* \TC(E_\OM)$,
so $j^* \TC(E_\OM)$ is zero, that is, $\TC(E_\OM)$ is supported on $\NC$.

We have
\[
\begin{array}{cll}
\multicolumn{2}{l}{[\FM E_\OM : F]}\\
=& [{j_\rs}_{!*} \FM \underline E_\OM : {j_\rs}_{!*} \underline F]
&\text{by Prop. \ref{prop:IC mult}}
\\
=& [\FM {j_\rs}_{!*} \underline E_\OM : {j_\rs}_{!*} \underline F]
&\text{by (\ref{T}) and } [T : {j_\rs}_{!*} \underline F] = 0
\\
=& [\FM \FC({j_\rs}_{!*} \underline E_\OM) (\nu + r) : \FC({j_\rs}_{!*} \underline F)(\nu + r)]
&\text{by EQUIV and } \FC\FM = \FM\FC
\\
=& [\FM \TC(E_\OM) : \TC(F)]
\end{array}
\]
\end{proof}

This theorem means that we can obtain the decomposition matrix of the
Weyl group $W$ by extracting certain rows (the image $\impsi_\KM$ of
the ordinary Springer correspondence) and certain columns (the image
$\impsi_\FM$ of the modular Springer correspondence) of the
decomposition matrix for $G$-equivariant perverse sheaves on the
nilpotent variety $\NC$.

\section{Modular Springer correspondence for $GL_n$}\label{sec:gl_n}

For the symmetric group $\SG_n$, we have a Specht module theory
compatible with the order on the nilpotent orbits through the
Springer correspondence. This is enough to determine the modular
Springer correspondence for $G = GL_n$.

Since all the groups $A_G(x)$ are trivial, we only need to parametrize
nilpotent orbits, which is done using the Jordan normal form.
So $\NG_\KM = \NG_\FM$ is the set $\PG_n$ of partitions of $n$. If $\l$
is a partition of $n$, we denote by $\OC_\l$ the corresponding orbit,
and by $x_\l$ an element of this orbit. Moreover, to simplify the
notation, we set $\ic(\l,\EM) = \p\JC_{!*}(\OC_\l,\EM)$.
In characteristic zero, it is known that $\impsi_\KM = \PG_n$, and
$\TC(S^\l) = \ic(\l',\KM)$ for $\l$ in $\PG_n$.

\begin{theo}
Suppose $G = GL_n$. If $\mu$ is an $\ell$-regular partition, then we have
\[
\TC(D^\mu) = \ic(\mu',\FM)
\]
where $\mu'$ is the partition dual to $\mu$. Thus we have
\[
\impsi_\FM = \{ \l \vdash n \mid \l \text{ is $\ell$-restricted} \}
\]
and, for two partitions $\l$ and $\mu$, with $\mu$ $\ell$-regular, we
have $d^W_{\l,\mu} = d_{\l',\mu'}$.
\end{theo}

\begin{proof}
We prove by induction on $\l \in \PG_n$ that $\l$ is in
$\impsi_\FM$ if and only if $\l$ is $\ell$-restricted, and
that, in this case, $\TC(D^{\l'}) = \ic(\l,\FM)$.

First, note that the partition $\l = (1^n)$ is always
$\ell$-restricted, and that $\TC(D^{(n)}) = \FC(\FM_\gG[2\nu + r]) =
\ic((1^n),\FM)$ by DIRAC.

Now assume the claim has been proved for all $\nu < \l$.
If $\l$ is in $\impsi_\FM$, let $\mu$ be the $\ell$-regular partition
such that $\Psi_\FM(D^\mu) = \l$. We have
\[
d^{\SG_n}_{\l',\mu} = d_{\Psi_\KM(S^{\l'}),\Psi_\FM(D^\mu)}
= d_{\l,\l} = 1 \neq 0
\]
hence $\l' \leqslant \mu$, and thus $\mu' \leqslant \l$.
If equality holds, we are done, since $\mu'$ is
$\ell$-restricted and $\Psi_\FM(D^{\l'}) = \Psi_\FM(D^{\mu}) = \l$.
But a strict inequality $\mu' < \l$ would lead to a contradiction. Indeed, by the
induction hypothesis, we would have $\Psi_\FM(D^\mu) = \mu'$, and this
would contradict the choice of $\mu$ ($\Psi_\FM(D^\mu) = \l$).

In the other direction, let us assume that  $\l$ is $\ell$-restricted.
Let $\mu = \Psi_\FM(D^{\l'})$. Then we have
\[
d_{\l,\mu}
= d_{\Psi_\KM(S^{\l'}), \Psi_\FM(D^{\l'})}
= d^{\SG_n}_{\l',\l'} = 1 \neq 0
\]
hence
\[
\mu \leqslant \l
\]
We cannot have $\mu < \l$, because this would imply
$\Psi_\FM(\l') = \mu = \Psi_\FM(D^{\mu'})$ by induction, hence $\l' =
\mu'$ since $\Psi_\FM$ is injective, and $\l = \mu$. Thus we must have
$\l = \mu = \Psi_\FM(D^{\l'})$.
\end{proof}

By Proposition \ref{prop:rc}, we can now say that the result of Kraft
and Procesi implies James's row and column removal rule.

\chapter{Tables}\label{chap:tables}

In this chapter, we give tables for types of rank at most $3$.  In
each case, we give the character table of $W$ (including ordinary and
modular characters) and the (ordinary and modular) Springer
correspondence. Be aware that, to get the correspondence obtained by
the Lusztig-Borho-MacPherson approach, one should tensor all the
representations of $W$ by the sign representation. The ordinary characters and
the conjugacy classes of $W$ (resp. the nilpotent orbits) are labeled
in the usual way, as for example in the book \cite{CarterFGLT}. 
I used the tables given there for the ordinary Springer
correspondence (but I had to tensor by the sign character).

Afterwards, we give the decomposition matrix for
$G$-equivariant perverse sheaves on the nilpotent variety (we assume
$G$ is simple of adjoint type, and that $p$ is very good for $G$).
In favorable cases, it is complete. There are more indeterminacies
for $\ell = 2$.

The rows correspond to pairs in $\NG_\KM$, while the columns
correspond to pairs in $\NG_\FM$. When the label of a row or a column
is just a nilpotent orbit, then the default local system is the
constant one. For the others, we use $\e$ for the sign character, and
$\psi$ for the irreducible character of degree $2$ of $\SG_3$.

The entries that we were able to determine
geometrically in chapter \ref{chap:dec} are in italics (see \cite{KP2}
for the description of the minimal degenerations in classical types).
The rows and columns corresponding to the Weyl group (the image of the
ordinary resp. modular Springer correspondence) have bold labels, and
the corresponding entries are underlined and (if you can see this
document in colors) red. The zeros above the diagonal are represented
by dots. The unknown entries are left blank.

\newpage

\section{Type $A_1$}

\vfill

\[
\begin{array}{|l|c|c|r|r|}
\hline
&&\text{cardinal}&1&1\\
&&\text{order}&1&2\\
&&\text{class}&1^2&2\\
\hline
\text{orbit}&\text{char}&a\text{-function}&&\\
\hline
1^2 & \chi_{2} &0 & 1 & 1\\
2 & \chi_{1^2} &1 & 1 &-1\\
\hline
1^2 & \phi_2 & \ov\chi_{2} & 1 & 1\\
\hline
\end{array}
\]

\vfill

$\ell = 2$

\[
\begin{array}{l|c|c|}
& {\bf 1^2} & 2          \\
\hline
{\bf 1^2} &\red{\it 1}&.\\
\hline
{\bf 2}   &\red{\it 1}&{\it 1}
\end{array}
\]

\vfill

\newpage

\section{Type $A_2$}

\vfill

\[
\begin{array}{|l|c|c|r|r|r|}
\hline
&&\text{cardinal}&1&1&2\\
&&\text{order}&1&2&3\\
&&\text{class}&1^3&21&3\\
\hline
\text{orbit}&\text{char}&a\text{-function}&&&\\
\hline
1^3 & \chi_{3} &0 & 1 & 1 & 1\\
21 & \chi_{21} &1 & 1 &-1 & 1\\
3 & \chi_{1^3} &3 & 2 & 0 &-1\\
\hline
1^3 & \phi_3 & \ov\chi_{3}    & 1 && 1 \\
21 & \phi_{21} & \ov\chi_{21} & 2 &&-1 \\
\hline
1^3 & \phi_3 & \ov\chi_3 & 1 & 1 &\\
21  & \phi_{21} & \ov\chi_{1^3} & 1 & -1&\\
\hline
\end{array}
\]

\vfill

$\ell = 2$

\[
\begin{array}{l|c|c|c|}
& {\bf 1^3} & {\bf 21} & 3          \\
\hline
{\bf 1^3} &\red{\it 1}&.&.\\
\hline
{\bf 21}  &\red{\it 0}&\red{\it 1}&.\\
\hline
{\bf 3}   &\red{1}    &\red{\it 0} & {\it 1}
\end{array}
\]

\vfill

$\ell = 3$

\[
\begin{array}{l|c|c|c|}
& {\bf 1^3} & {\bf 21} & 3          \\
\hline
{\bf 1^3} &\red{\it 1}&.&.\\
\hline
{\bf 21}  &\red{\it 1}&\red{\it 1}&.\\
\hline
{\bf 3}   &\red{0}    &\red{\it 1} & {\it 1}
\end{array}
\]

\vfill

\section{Type $A_3$}

\[
\begin{array}{|l|c|c|r|r|r|r|r|}
\hline
&&\text{cardinal}&1&6&3&3&6\\
&&\text{order}&1&2&2&3&4\\
&&\text{class}&1^4&21^2&2^2&31&4\\
\hline
\text{orbit}&\text{char}&a\text{-function}&&&&&\\
\hline
1^4  & \chi_{4}  &0  & 1 & 1 & 1 & 1 & 1\\
21^2 & \chi_{31} &1  & 3 & 1 &-1 & 0 &-1\\
2^2  & \chi_{2^2}&2  & 2 & 0 & 2 &-1 & 0\\
31   & \chi_{21^2}&3 & 3 &-1 &-1 & 0 & 1\\
4    & \chi_{1^4}&6  & 1 &-1 & 1 & 1 &-1\\
\hline
1^4  & \phi_3    & \ov\chi_{3}  & 1 && 1 &&\\
21^2 & \phi_{21} & \ov\chi_{21} & 2 &&-1 &&\\
\hline
1^4 & \phi_4     & \ov\chi_4     & 1 & 1 & 1 && 1\\
21^2& \phi_{31}  & \ov\chi_{31}  & 3 & 1 &-1 &&-1\\
2^2 & \phi_{2^2} & \ov\chi_{1^4} & 1 &-1 & 1 &&-1\\
31  & \phi_{21^2}& \ov\chi_{21^2}& 3 &-1 &-1 && 1\\
\hline
\end{array}
\]

\vfill

$\ell = 2$

\[
\begin{array}{l|c|c|c|c|c|}
& {\bf 1^4} & {\bf 21^2} & 2^2 & 31 & 4\\
\hline
{\bf 1^4} &\red{\it 1}&.&.&.&.\\
\hline
{\bf 21^2}&\red{\it 1}&\red{\it 1}&.&.&.\\
\hline
{\bf 2^2} &\red{0}    &\red{\it 1}&{\it 1}&.&.\\
\hline
{\bf 31}  &\red{1}    &\red{1}    &{\it 1}&{\it 1}&.\\
\hline
{\bf 4}   &\red{1}    &\red{0}    &(1)    &{\it 1}&{\it 1}
\end{array}
\]

The $(1)$ is the decomposition number of the Schur algebra, so it is
expected to be the right decomposition number for perverse sheaves.

\vfill

$\ell = 3$

\[
\begin{array}{l|c|c|c|c|c|}
& {\bf 1^4} & {\bf 21^2} & {\bf 2^2} & {\bf 31} & 4\\
\hline
{\bf 1^4} &\red{\it 1}&.&.&.&.\\
\hline
{\bf 21^2}&\red{\it 0}&\red{\it 1}&.&.&.\\
\hline
{\bf 2^2} &\red{1}    &\red{\it 0}&\red{\it 1}&.&.\\
\hline
{\bf 31}  &\red{0}    &\red{0}    &\red{\it 0}&\red{\it 1}&.\\
\hline
{\bf 4}   &\red{0}    &\red{0}    &\red{1}    &\red{\it 0}&{\it 1}
\end{array}
\]

\vfill

\newpage

\section{Type $B_2$}

\[
\begin{array}{|l|c|c|r|r|r|r|r|}
\hline
&&\text{cardinal}&1&2&1&2&2\\
&&\text{order}   &1&2&2&2&4\\
&&\text{class}&1^2,-&1,1&-,1^2&2,-&-,2\\
\hline
\text{orbit}&\text{char}&a\text{-function}&&&&&\\
\hline
1^5   & \chi_{2,-}   & 0  & 1& 1& 1& 1& 1\\
2^2 1 & \chi_{1^2,-}  & 1  & 1& 1& 1&-1&-1\\
3 1^2 & \chi_{1,1}   & 1  & 2& .&-2& .& .\\
|\ \e & \chi_{-,2}   & 1  & 1&-1& 1& 1&-1\\
5     & \chi_{-,1^2}  & 4  & 1&-1& 1&-1& 1\\
\hline
1^5   & \phi_{2,-}  & \ov\chi_{2,-} & 1&&&&\\
\hline
\end{array}
\]

\vfill

$\ell = 2$

\[
\begin{array}{l|c|c|c|c|c|c|c}
&
{\bf 1^5} &
2^2 1 &
3 1^2 &
5          \\
\hline
{\bf 1^5}  &\red{\it 1}&.      &.      &.      \\
\hline
{\bf 2^2 1}&\red{\it 1}&{\it 1}&.      &.      \\
\hline
{\bf 3 1^2}&    \red{2}&{\it 1}&{\it 1}&.      \\
|\ \epsb   &    \red{1}&       &{\it 1}&.      \\
\hline
{\bf 5}    &    \red{1}&       &{\it 1}&{\it 1}\\
\end{array}
\]

\vfill

For $C_2$, we have the same geometry and the same Weyl group. One only needs
to replace the labels of the nilpotent orbits by 
$1^4$, $2 1^2$, $2^2$, $4$.

\vfill

\newpage

\section{Type $B_3$}

\vfill

\[
\begin{array}{|l|c|c|r|r|r|r|r|r|r|r|r|r|}
\hline
&&\text{cardinal}&1&3&3&1&6&6&6&6&8&8\\
&&\text{order}&1&2&2&2&2&4&2&4&3&6\\
&&\text{class}&1^3,-&1^2,1&1,1^2&-,1^3&21,-&1,2&2,1&-,21&3,-&-,3\\
\hline
\text{orbit}&\text{char}&a\text{-function}&&&&&&&&&&\\
\hline
1^7         & \chi_{3,-}   &0  & 1& 1& 1& 1& 1& 1& 1& 1& 1& 1\\
2^2 1^3     & \chi_{21,-}  &1  & 2& 2& 2& 2& .& .& .& .&-1&-1\\
3 1^4       & \chi_{2,1}   &1  & 3& 1&-1&-3& 1& 1&-1&-1& .& .\\
|\ \e       & \chi_{-,3}   &1  & 1&-1& 1&-1& 1&-1&-1& 1& 1&-1\\
3 2^2       & \chi_{1,2}   &2  & 3&-1&-1& 3& 1&-1& 1&-1& .& .\\
3^2 1       & \chi_{1^2,1} &3  & 3& 1&-1&-3&-1&-1& 1& 1& .& .\\
|\ \e       & \chi_{1^3,-} &4  & 1& 1& 1& 1&-1&-1&-1&-1& 1& 1\\
5 1^2       & \chi_{1,1^2} &4  & 3&-1&-1& 3&-1& 1&-1& 1& .& .\\
|\ \e       & \chi_{-,21}  &4  & 2&-2& 2&-2& .& .& .& .&-1& 1\\
7           & \chi_{-,1^3} &9  & 1&-1& 1&-1&-1& 1& 1&-1& 1&-1\\
\hline
1^6         & \phi_{3,-}   & \ov\chi_{3,-}   & 1&  &  &  &  &  &  &  & 1&  \\
2^2 1^3     & \phi_{21,-}  & \ov\chi_{21,-}  & 2&  &  &  &  &  &  &  &-1&  \\
\hline
1^7         & \phi_{3,-}   & \ov\chi_{3,-}   & 1& 1& 1& 1& 1& 1& 1& 1&  &  \\
2^2 1^3     & \phi_{21,-}  & \ov\chi_{1^3,-} & 1& 1& 1& 1&-1&-1&-1&-1&  &  \\
3 1^4       & \phi_{2,1}   & \ov\chi_{2,1}   & 3& 1&-1&-3& 1& 1&-1&-1&  &  \\
|\ \e       & \phi_{-,3}   & \ov\chi_{-,3}   & 1&-1& 1&-1& 1&-1&-1& 1&  &  \\
3 2^2       & \phi_{1,2}   & \ov\chi_{1,2}   & 3&-1&-1& 3& 1&-1& 1&-1&  &  \\
3^2 1       & \phi_{1^2,1} & \ov\chi_{1^2,1} & 3& 1&-1&-3&-1&-1& 1& 1&  &  \\
5 1^2       & \phi_{1,1^2} & \ov\chi_{1,1^2} & 3&-1&-1& 3&-1& 1&-1& 1&  &  \\
|\ \e       & \phi_{-,21}  & \ov\chi_{-,1^3} & 1&-1& 1&-1&-1& 1& 1&-1&  &  \\
\hline
\end{array}
\]

\vfill

\newpage

\vfill

$\ell = 2$

\[
\begin{array}{l|c|c|c|c|c|c|c}
&
{\bf 1^7} &
{\bf 2^2 1^3} &
3 1^4 &
3 2^2 &
3^2 1 &
5 1^2 &
7          \\
\hline
{\bf 1^7}     &\red{\it 1}&.          &.      &.      &.      &.      &.      \\
\hline
{\bf 2^2 1^3} &\red{\it 0}&\red{\it 1}&.      &.      &.      &.      &.      \\
\hline
{\bf 3 1^4}   &    \red{1}&\red{\it 1}&{\it 1}&.      &.      &.      &.      \\
|\ \epsb   &    \red{1}&    \red{0}&{\it 1}&.      &.      &.      &.      \\
\hline
{\bf 3 2^2}   &    \red{1}&    \red{1}&{\it 2}&{\it 1}&.      &.      &.      \\
\hline
{\bf 3^2 1}   &    \red{1}&    \red{1}&       &{\it 1}&{\it 1}&.      &.      \\
|\ \epsb   &    \red{1}&    \red{0}&       &       &{\it 1}&.      &.      \\
\hline
{\bf 5 1^2}   &    \red{1}&    \red{1}&       &       &{\it 1}&{\it 1}&.       \\
|\ \epsb   &    \red{0}&    \red{1}&       &       &       &{\it 1}&.      \\
\hline
{\bf 7}       &    \red{1}&    \red{0}&       &       &       &{\it 1}&{\it 1}\\
\end{array}
\]

\vfill

$\ell =3$

\[
\begin{array}{l|c|c|cc|c|cc|cc|c}
&
{\bf 1^7}     &
{\bf 2^2 1^3} &
{\bf 3 1^4}   &
\epsb         &
{\bf 3 2^2}   &
{\bf 3^2 1}   &
\e            &
{\bf 5 1^2}   &
{\bf \epsb}   &
7 \\
\hline
{\bf 1^7}     &\red{\it 1}&.&.&.&.&.&.&.&.&.\\
\hline
{\bf 2^2 1^3} &\red{\it 1}&\red{\it 1}&.&.&.&.&.&.&.&.\\
\hline
{\bf 3 1^4}   &    \red{0}&\red{\it 0}&\red{\it 1}&.&.&.&.&.&.&.\\
|\ \epsb   &    \red{0}&    \red{0}&\red{\it 0}&\red{\it 1}&.&.&.&.&.&.\\
\hline
{\bf 3 2^2}   &    \red{0}&    \red{0}&\red{\it 0}&\red{\it 0}&\red{\it 1}&.&.&.&.&.\\
\hline
{\bf 3^2 1}   &    \red{0}&    \red{0}&    \red{0}&    \red{0}&\red{\it 0}&\red{\it 1}&.&.&.&.\\
|\ \epsb   &    \red{0}&    \red{1}&    \red{0}&    \red{0}&    \red{0}&\red{\it 0}&{\it 1}&.&.&.\\
\hline
{\bf 5 1^2}   &    \red{0}&    \red{0}&    \red{0}&    \red{0}&    \red{0}&\red{\it 0}&{\it 0}&
\red{\it 1}&.&.\\
|\ \epsb   &    \red{0}&    \red{0}&    \red{0}&    \red{1}&    \red{0}&    \red{0}&&
\red{\it 0}&\red{\it 1}&.\\
\hline
{\bf 7}       &    \red{0}&    \red{0}&    \red{0}&    \red{0}&    \red{0}&    \red{0}&&
\red{\it 0}&\red{\it 1}&{\it 1}\\
\end{array}
\]

\vfill

\newpage

\section{Type $C_3$}

\vfill

\[
\begin{array}{|l|c|c|r|r|r|r|r|r|r|r|r|r|}
\hline
&&\text{cardinal}&1&3&3&1&6&6&6&6&8&8\\
&&\text{order}&1&2&2&2&2&4&2&4&3&6\\
&&\text{class}&1^3,-&1^2,1&1,1^2&-,1^3&21,-&1,2&2,1&-,21&3,-&-,3\\
\hline
\text{orbit}&\text{char}&a\text{-function}&&&&&&&&&&\\
\hline
1^6         & \chi_{3,-}   &0  & 1& 1& 1& 1& 1& 1& 1& 1& 1& 1\\
2 1^4       & \chi_{-,3}   &1  & 1&-1& 1&-1& 1&-1&-1& 1& 1&-1\\
2^2 1^2     & \chi_{2,1}   &1  & 3& 1&-1&-3& 1& 1&-1&-1& .& .\\
|\ \e       & \chi_{21,-}  &1  & 2& 2& 2& 2& .& .& .& .&-1&-1\\
2^3         & \chi_{1,2}   &2  & 3&-1&-1& 3& 1&-1& 1&-1& .& .\\
3^2         & \chi_{1^2,1} &3  & 3& 1&-1&-3&-1&-1& 1& 1& .& .\\
4 1^2       & \chi_{-,21}  &4  & 2&-2& 2&-2& .& .& .& .&-1& 1\\
42          & \chi_{1,1^2} &4  & 3&-1&-1& 3&-1& 1&-1& 1& .& .\\
|\ \e       & \chi_{1^3,-} &4  & 1& 1& 1& 1&-1&-1&-1&-1& 1& 1\\
6           & \chi_{-,1^3} &9  & 1&-1& 1&-1&-1& 1& 1&-1& 1&-1\\
\hline
1^6         & \phi_{3,-}   & \ov\chi_{3,-}   & 1&  &  &  &  &  &  &  & 1&  \\
2^2 1^2     & \phi_{2,1}   & \ov\chi_{21,-}  & 2&  &  &  &  &  &  &  &-1&  \\
\hline
1^6         & \phi_{3,-}   & \ov\chi_{3,-}   & 1& 1& 1& 1& 1& 1& 1& 1&  &  \\
2 1^4       & \phi_{-,3}   & \ov\chi_{-,3}   & 1&-1& 1&-1& 1&-1&-1& 1&  &  \\
2^2 1^2     & \phi_{2,1}   & \ov\chi_{2,1}   & 3& 1&-1&-3& 1& 1&-1&-1&  &  \\
|\ \e       & \phi_{21,-}  & \ov\chi_{1^3,-} & 1& 1& 1& 1&-1&-1&-1&-1&  &  \\
2^3         & \phi_{1,2}   & \ov\chi_{1,2}   & 3&-1&-1& 3& 1&-1& 1&-1&  &  \\
3^2         & \phi_{1^2,1} & \ov\chi_{1^2,1} & 3& 1&-1&-3&-1&-1& 1& 1&  &  \\
4 1^2       & \phi_{-,21}  & \ov\chi_{-,1^3} & 1&-1& 1&-1&-1& 1& 1&-1&  &  \\
42,\ 1      & \phi_{1,1^2} & \ov\chi_{1,1^2} & 3&-1&-1& 3&-1& 1&-1& 1&  &  \\
\hline
\end{array}
\]

\vfill

\newpage

\vfill

$\ell = 2$

\[
\begin{array}{l|c|c|c|c|c|c|c|c}
& {\bf 1^6} & 2 1^4 & {\bf 2^2 1^2} & 2^3 & 3^2 & 4 1^2 & 42 & 6\\
\hline
{\bf 1^6}     &\red{\it 1}&.      &.          &.      &.      &.      &.      &.  \\
\hline
{\bf 2 1^4}   &\red{\it 1}&{\it 1}&.          &.      &.      &.      &.      &.  \\
\hline
{\bf 2^2 1^2} &\red{1}    &{\it 1}&\red{\it 1}&.      &.      &.      &.      &.  \\
|\ \e   &\red{0}    &       &\red{\it 1}&.      &.      &.      &.      &.  \\
\hline
{\bf 2^3}     &\red{1}    &       &\red{\it 1}&{\it 1}&.      &.      &.      &.  \\
\hline
{\bf 3^2}     &\red{1}    &       &    \red{1}&{\it 1}&{\it 1}&.      &.      &.  \\
\hline
{\bf 4 1^2}   &\red{0}    &       &    \red{1}&{\it 1}&{\it 0}&{\it 1}&.      &.  \\
\hline
{\bf 42}      &\red{1}    &       &    \red{1}&       &{\it 1}&{\it 1}&{\it 1}&.  \\
|\ \epsb   &\red{1}    &       &    \red{0}&       &       &       &{\it 1}&.   \\
\hline
{\bf 6}       &\red{1}    &       &    \red{0}&       &       &       &{\it 2}&{\it 1}
\end{array}
\]

\vfill

$\ell = 3$
\[
\begin{array}{l|c|c|cc|c|c|c|cc|c}
&{\bf 1^6}&{\bf 2 1^4}&{\bf 2^2 1^2}&\epsb&{\bf 2^3}&{\bf 3^2}&{\bf 4 1^2}&{\bf 42}&
\e & 6\\
\hline
{\bf 1^6}     &\red{\it 1}&.&.&.&.&.&.&.&.&.\\
\hline
{\bf 2 1^4}   &\red{\it 0}&\red{\it 1}&.&.&.&.&.&.&.&.\\
\hline
{\bf 2^2 1^2} &\red{0}&\red{\it 0}&\red{\it 1}&.&.&.&.&.&.&.\\
|\ \epsb   &\red{1}&\red{0}   &\red{\it 0}&\red{\it 1}&.&.&.&.&.&.\\
\hline
{\bf 2^3}     &\red{0}&\red{0}&\red{\it 0}&\red{\it 0}&\red{\it 1}&.&.&.&.&.\\
\hline
{\bf 3^2}     &\red{0}&\red{0}&\red{0}&\red{0}&\red{\it 0}&\red{\it 1}&.&.&.&.\\
\hline
{\bf 4 1^2}   &\red{0}&\red{1}&\red{0}&\red{0}&\red{\it 0}&\red{\it 0}&\red{\it 1}&.&.&.\\
\hline
{\bf 42}      &\red{0}&\red{0}&\red{0}&\red{0}&\red{0}&\red{\it 0}&\red{\it 0}&\red{\it 1}&.&.\\
|\ \epsb   &\red{0}&\red{0}&\red{0}&\red{1}&\red{0}&\red{0}&\red{0}&\red{\it 0}&{\it 1}&.\\
\hline
{\bf 6}       &\red{0}&\red{0}&\red{0}&\red{0}&\red{0}&\red{0}&\red{1}&\red{\it 0}&{\it 0}&{\it 1}
\end{array}
\]

\vfill

\newpage

\section{Type $G_2$}

\begin{center}
\begin{picture}(100,30)
\put( 40, 10){\circle{10}}
\put(44.5,12){\line(1,0){32}}
\put( 45, 10){\line(1,0){30}}
\put(44.5, 8){\line(1,0){32}}
\put( 55,5.5){\LARGE{$<$}}
\put( 80, 10){\circle{10}}
\put( 35, 20){$\a$}
\put( 75, 20){$\b$}
\end{picture}
\end{center}

\[
\begin{array}{|l|c|c|r|r|r|r|r|r|}
\hline
&&\text{cardinal}&1&3&3&2&2&1\\
&&\text{order}&1&2&2&6&3&2\\
&&\text{class}&1&s_\a&s_\b&s_\a s_\b&(s_\a s_\b)^2&-1\\
\hline
\text{orbit}&\text{char}&a\text{-function}&&&&&&\\
\hline
1              & \chi_{1,0}   &0        & 1 & 1 & 1 & 1 & 1 & 1\\
A_1            & \chi_{1,3}'  &1        & 1 & 1 &-1 &-1 & 1 &-1\\
\widetilde A_1 & \chi_{2,2}   &1  & 2 & 0 & 0 &-1 &-1 & 2\\
G_2(a_1)       & \chi_{2,1}   &1    & 2 & 0 & 0 & 1 &-1 &-2\\
|\ \psi        & \chi_{1,3}'' &1        & 1 &-1 & 1 &-1 & 1 &-1\\
G_2            & \chi_{1,6}   &6        & 1 &-1 &-1 & 1 & 1 & 1\\
\hline
1              & \phi_{1,0}   & \ov\chi_{1,0} & 1 &&&& 1&\\
\widetilde A_1 & \phi_{2,2}   & \ov\chi_{2,2} & 2 &&&&-1&\\
\hline
1              & \phi_{1,0}   & \ov\chi_{1,0}  & 1 & 1 & 1 &&& 1\\
A_1            & \phi_{1,3}'  & \ov\chi_{1,3}' & 1 & 1 &-1 &&&-1\\
\widetilde A_1 & \phi_{2,2}   & \ov\chi_{1,6}  & 1 &-1 &-1 &&& 1\\
G_2(a_1)       & \phi_{2,1}   & \ov\chi_{1,3}''& 1 &-1 & 1 &&&-1\\
\hline
\end{array}
\]
$\ell = 2$
\[
\begin{array}{l|l|r|r|cr|r}
&
{\bf 1}              &
A_1            &
{\bf \widetilde {A_1}} &
G_2(a_1)       &
\psi                &
G_2 \\
\hline
{\bf 1}              &\red{\it 1}&.&.&.&.&.\\
\hline
{\bf A_1}            &\red{\it 1}&{\it 1}&.&.&.&.\\
\hline
{\bf \widetilde {A_1}} &    \red{0}&    &\red{\it 1}&.&.&.\\
\hline
{\bf G_2(a_1)}       &    \red{0}&           &\red{\it 1}&{\it 1}&.&.\\
|\ \psib        &    \red{1}&           &    \red{0}&{\it 0}&{\it 1}&.\\
|\ \e          &         0 &         0 &         0 &{\it 1}&{\it 0}&.\\
\hline
{\bf G_2}            &    \red{1}&           &    \red{0}&{\it 0}&{\it 1}&{\it 1}
\end{array}
\]
$\ell = 3$
\[
\begin{array}{l|l|r|r|cr|r}
&
{\bf 1}              &
{\bf A_1}            &
{\bf \widetilde {A_1}} &
{\bf G_2(a_1)}       &
\e                &
G_2 \\
\hline
{\bf 1}              &\red{\it 1}&.&.&.&.&.\\
\hline
{\bf A_1}            &\red{\it 0}&\red{\it 1}&.&.&.&.\\
\hline
{\bf \widetilde {A_1}} &    \red{1}&    \red{0}&\red{\it 1}&.&.&.\\
\hline
{\bf G_2(a_1)}       &    \red{0}&    \red{1}&\red{\it 0}&\red{\it 1}&.&.\\
|\ \psib        &    \red{0}&    \red{0}&    \red{0}&\red{\it 1}&{\it 1}&.\\
|\ \e          &         0 &         0 &         0 &    {\it 0}&{\it 1}&.\\
\hline
{\bf G_2}            &    \red{0}&    \red{0}&    \red{1}&\red{\it 0}&{\it 0}&{\it 1}
\end{array}
\]

\chapter{Character sheaves on ${\mathfrak sl}_2$}\label{chap:sl2}

Up to now, we have been paying a lot of attention to $\KC_\NC$, but
not so much to $\KC$ itself. We are going to give a full description
in the case of $G = SL_2$, and we will give a hint of what we
might do in type $A_n$ to construct a Springer correspondence involving
the Schur algebra.

The general idea is to replace the regular representation of $W$,
which is $\Ind_{\EM 1}^{\EM W} \EM$, by a direct sum of all
permutation modules over standard parabolic subgroups:
\[
\bigoplus_{\l \in \NG} \Ind_{\EM\SG_\l}^{\EM\SG_n} \EM
\]
whose indecomposable summands are the Young modules $Y^\l$, with some
multiplicities, and whose endomorphism algebra is the Schur algebra
$S_\EM(n) = S_\EM(n,n)$. This corresponds to a (shifted) local system
$\EM\KC'_\rs$ on $\gG_\rs$, and we call $\EM\KC'$ its intermediate extension
on $\gG$.  We still have $\End(\EM\KC') \simeq S_\EM(n)$. I conjecture that
the restriction functor $i_\NC^*[-1]$ sends $\EM\KC'$ to a perverse sheaf
on $\NC$ with the same endomorphism algebra, which would enable us to
make a direct link between the Schur algebra and the $G$-equivariant perverse
sheaves on $\NC$, and between both decomposition matrices. Let us see
what happens for $\gG = \sG\lG_2$.

Apart from the open stratum $\gG_\rs$, we just
have the two nilpotent orbits $\OC_\reg = \OC_\mini = \OC_{(2)}$
and $\OC_\triv = \OC_\subreg = \OC_{(1^2)} = \{0\}$.
On $\gG_\rs$, we will only consider local systems which become trivial
after a pullback by $\pi_\rs$.
We have $W = \SG_2$. The local system ${\pi_\rs}_* \EM$ corresponds to
the regular representation $\EM\SG_2$.

In characteristic $0$, the group algebra is semi-simple, and the perverse sheaf
$\KM\KC_\rs$ splits as the sum of the constant perverse sheaf $C^\rs$ and
the shifted local system $C_\e^\rs$ corresponding to the sign
representation of $\SG_2$. These two simple components are sent by
${j_\rs}_{!*}$ on two simple perverse sheaves on $\gG$, the constant
perverse sheaf $C$ (since $\gG$ is smooth), and the other one, $C_\e$.
Let us denote by $A$ the simple perverse sheaf supported on $\{0\}$,
and by $B$ the simple perverse sheaf $\p\JC_{!*}(\OC_\reg,\KM)$.
Since $\FC(C) = A$, we must have $\FC(C_\e) = B$. This gives the
Springer correspondence for $\sG\lG_2$ by Fourier transform.

Let us make tables for the stalks of the perverse sheaves involved.
We have a line for each stratum, and one column for each cohomology
degree. If $x$ is a point of a given stratum $\OC$ and $i$ is an integer,
the corresponding entry in the table of a perverse sheaf $\AC$ will be
the class of $\HC^i_x \AC$, seen as a representation of a suitable
group $A(\OC)$, in the Grothendieck group of $\EM A(\OC)$.
There is a column $\chi$ describing the alternating sum of the stalks
of each stratum.

Let us first describe $\KM\KC$. So, over $\gG_\rs$, we have the regular
representation of $\SG_2$. Over $\OC_\reg$, the fibers are single
points, so the cohomology of $\BC_{x_\reg}$ is just $\KM$.
But we have $\BC_0 = \BC = G/B = \PM^1$. We get the following table for
$\KM\KC$.
\[
\begin{array}{|c|c|c|c|c|c|c|}
\hline
\text{Stratum} & \text{Dimension} & \chi & -3 & -2 & -1 & 0\\
\hline
\gG_\rs & 3 & -1 - \e & \KM \oplus \KM_\e&.&.&.\\
\hline
\OC_\reg & 2 & -1 & \KM &.&.&.\\
\hline
\OC_\triv & 0 & -2 & \KM & . & \KM &.
\\
\hline 
\end{array}
\]
It is the direct sum of the two simple perverse sheaves $C$
\[
\begin{array}{|c|c|c|c|c|c|c|}
\hline
\text{Stratum} & \text{Dimension} & \chi & -3 & -2 & -1 & 0\\
\hline
\gG_\rs & 3 & -1  & \KM &.&.&.\\
\hline
\OC_\reg & 2 & -1 & \KM &.&.&.\\
\hline
\OC_\triv & 0 & -1 & \KM & . & . & .
\\
\hline 
\end{array}
\]
and $C_\e$, which we deduce by substraction (we have a direct sum !)
\[
\begin{array}{|c|c|c|c|c|c|c|}
\hline
\text{Stratum} & \text{Dimension} & \chi & -3 & -2 & -1 & 0\\
\hline
\gG_\rs & 3 & - \varepsilon & \KM_\varepsilon&.&.&.\\
\hline
\OC_\reg & 2 & 0 & . &.&.&.\\
\hline
\OC_\triv & 0 & -1 & . & . & \KM &.\\
\hline 
\end{array}
\]
The simple $G$-equivariant perverse sheaves on $\NC$ are
$B = \p\JC_{!*}(\OC_\reg,\KM)$,
\[
\begin{array}{|c|c|c|c|c|c|c|}
\hline
\text{Stratum} & \text{Dimension} & \chi & -3 & -2 & -1 & 0\\
\hline
\gG_\rs & 3 & 0 & .&.&.&.\\
\hline
\OC_\reg & 2 & 1 & . & \KM &.&.\\
\hline
\OC_\triv & 0 & 1 & . & \KM &.&.\\
\hline 
\end{array}
\]
$A = \p\JC_{!*}(\OC_\triv,\KM)$
\[
\begin{array}{|c|c|c|c|c|c|c|}
\hline
\text{Stratum} & \text{Dimension} & \chi & -3 & -2 & -1 & 0\\
\hline
\gG_\rs & 3 & 0 & .&.&.&.\\
\hline
\OC_\reg & 2 & 0 & . & . &.&.\\
\hline
\OC_\triv & 0 & 1 & . & . &.& \KM\\
\hline 
\end{array}
\]
and the cuspidal $B_\e = \p\JC_{!*}(\OC_\reg,\KM_\e)$,
which is clean (its intermediate extension is
just the extension by zero), and stable by the
Fourier-Deligne transform, by the general theory
\[
\begin{array}{|c|c|c|c|c|c|c|}
\hline
\text{Stratum} & \text{Dimension} & \chi & -3 & -2 & -1 & 0\\
\hline
\gG_\rs & 3 & 0 & .&.&.&.\\
\hline
\OC_\reg & 2 & 1 & . & \KM_\e &.&.\\
\hline
\OC_\triv & 0 & 1 & . & . &.&.\\
\hline 
\end{array}
\]

We can check that, applying $i_\NC^*[-1]$ to $\KC$, we recover $\KC_\NC$.
This functor sends $C$ to $B$ and $C_\e$ to $A$. There is a twist by
the sign character between the two versions of the Springer
representations (by Fourier-Deligne transform, and by restriction).

So, to summarize the situation over $\KM$, we have
\[
\begin{array}{|c|c|c|}
\hline
\gG_\rs & \gG & \NC\\
\hline
C^\rs \oplus C^\rs_\e & C \oplus C_\e & B \oplus A\\
\hline
\end{array}
\]

If $\ell \neq 2$, the situation over $\FM$ is similar. Now let us
assume that $\ell = 2$. Then the sign representation becomes
trivial. The regular representation is an extension of the trivial by
the trivial, so $\FM\KC_\rs$ is an extension of the constant $c_\rs$
(reduction of $C_\rs$) by itself.

Now $\FM\KC$ is as follows
\[
\begin{array}{|c|c|c|c|c|c|c|}
\hline
\text{Stratum} & \text{Dimension} & \chi & -3 & -2 & -1 & 0\\
\hline
\gG_\rs & 3 & -2 & \FM \SG_2&.&.&.\\
\hline
\OC_\reg & 2 & -1 & \FM &.&.&.\\
\hline
\OC_\triv & 0 & -2 & \FM & . & \FM &.
\\
\hline 
\end{array}
\]

It must be made of the simple perverse sheaves $c$, $b$ and $c$, where
$c$ is the constant on $\gG$ (it is the reduction of $C$, and has the
same table with $\FM$ instead of $\KM$), $a$ is the constant on the
origin (the reduction of $A$), and $b = \p\JC_{!*}(\OC_\reg,\FM)$
has the following table
\[
\begin{array}{|c|c|c|c|c|c|c|}
\hline
\text{Stratum} & \text{Dimension} & \chi & -3 & -2 & -1 & 0\\
\hline
\gG_\rs & 3 & 0 & .&.&.&.\\
\hline
\OC^\reg & 2 & 1 & . & \FM &.&.\\
\hline
\OC^\triv & 0 & 0 & . & \FM &\FM&.\\
\hline 
\end{array}
\]

Looking at the $\chi$ functions, we see that
$[\FM\KC] = 2[c] + [b]$ in the Grothendieck group of
$\p\MC_G(\NC,\FM)$. We know that the top and the socle of $\FM\KC$
must be $c$, the intermediate extension of $c$, and that $b$ cannot
appear either in the top nor in the socle. Thus there is only one
possible Loewy structure:
\[
\FM\KC = 
\begin{array}{c}
c\\
b\\
c
\end{array}
\]

Similarly, we find
\[
\FM\KC_\NC = 
\begin{array}{c}
a\\
b\\
a
\end{array}
\]

Thus, as we already know, $\FC(c) = a$, but we also deduce that
$\FC(b) = b$.

The restriction functor $i^*_\NC$ sends $c$ onto the reduction of $B$,
which has the following Loewy structure (by Section \ref{sec:simple}): 
\[
\begin{array}{c}
b\\
a
\end{array}
\]

The reduction of $C_\e$ has structure
\[
\begin{array}{c}
c\\
b
\end{array}
\]
and it restricts to $a$ (the reduction of $A$, which is the restriction of $C_\e$).

So we have the following situation.

\[
\begin{array}{|c|c|c|}
\hline
\gG_\rs & \gG & \NC\\
\hline
\begin{array}{c}
c_\rs\\
c_\rs
\end{array}
&
\begin{array}{c}
c\\
b\\
c
\end{array}
&
\begin{array}{c}
a\\
b\\
a
\end{array}\\
\hline
\end{array}
\]

And we check that we can get $\KC_\NC$ either by Fourier-Deligne
transform, or by restriction. For the Springer correspondence, $b$ is
missing. In appears neither in the top nor in the socle.

Now, let us see what happens with the sum of induced modules. What we
lack here is the induction from $\SG_2$ to $\SG_2$, which gives the
trivial module, and hence the constant perverse sheaf $c_\rs$. So,
by restriction, we would get

\[
\begin{array}{|c|c|c|}
\hline
\gG_\rs & \gG & \NC\\
\hline
c_\rs
\oplus
\begin{array}{c}
c_\rs\\
c_\rs
\end{array}
&
c \oplus
\begin{array}{c}
c\\
b\\
c
\end{array}
&
\begin{array}{c}
b\\
a
\end{array}
\oplus
\begin{array}{c}
a\\
b\\
a
\end{array}\\
\hline
\end{array}
\]

We can hope that, in general, each intermediate extension of a Young
module will restrict to an indecomposable on $\NC$ with simple top, and
that all the simple $GL_n$-equivariant perverse sheaves appear in this
way. We would thus obtain a correspondence involving all the
partitions of $n$, and we would certainly explain why the
decomposition matrix for $GL_n$-equivariant perverse sheaves on $\NC$
must be the decomposition matrix for the Schur algebra.
Of course, if we looked at $SL_n$, there would be supercuspidality
phenomena, on top of that.

}

\backmatter

{\selectlanguage{english}
\newcommand{\etalchar}[1]{$^{#1}$}
\def\cprime{$'$}

}

\end{document}